\newcommand{\R}{\mathbb{R}}
\newcommand{\Z}{\mathbb{Z}}
\newcommand{\N}{\mathbb{N}}
\newcommand{\Q}{\mathbb{Q}}
\newcommand{\C}{\mathbb{C}}

\newcommand{\F}{\mathcal{F}}

\newcommand{\wt}{\widetilde}
\newcommand{\wh}{\widehat}

\newcommand{\Hom}{\operatorname{Hom}}
\newcommand{\RHom}{\operatorname{RHom}}
\newcommand{\End}{\operatorname{End}}

\newcommand{\digon}{di-gon}

\newcommand{\RS}{Y}
\newcommand{\lra}[1]{\stackrel{#1}{\longrightarrow}}
\newcommand{\lla}[1]{\stackrel{#1}{\longleftarrow}}
\newcommand{\del}{\partial}
\newcommand{\const}[1]{\underline{#1}}

\newcommand{\isom}{\cong}
\newcommand{\TT}{\mathbb{T}}
\newcommand{\Tg}{\TT}
\newcommand{\Tb}{\TT_{in}}
\newcommand{\Tm}{\TT_{o}}
\newcommand{\Cx}{\C^*}

\newcommand{\Ng}{N}
\newcommand{\Nb}{N_{in}}
\newcommand{\Nm}{N_{o}}

\newcommand{\Mg}{M}

\newcommand{\Mm}{M_{o}}

\newcommand{\Xg}{X}
\newcommand{\Xm}{X_{o}}
\newcommand{\subs}{\subset}
\newcommand{\subg}{\leq}
\newcommand{\Aut}{\operatorname{Aut}}

\newcommand{\pf}{\pi}
\newcommand{\Qcov}{\wt{Q}}
\newcommand{\zzf}{\wt{\eta}}
\newcommand{\zzff}{\wt{\beta}}

\newcommand{\face}{f}
\newcommand{\mof}[1]{\iota_{#1}}

\newcommand{\Meq}[1]{[#1]_\Mg}

\newcommand{\wind}{\operatorname{Wind}_v}
\newcommand{\windf}{\operatorname{Wind}_{\face}}
\newcommand{\moa}{u}
\newcommand{\mob}{y}

\newcommand{\bp}{p}
\newcommand{\bq}{q}

\newcommand{\Rsym}{R}

\newcommand{\sbf}{S}
\newcommand{\hp}{\mathcal{H}}

\newcommand{\mew}{\mu}
\newcommand{\nat}{\alpha}
\newcommand{\tg}{d}

\newcommand{\pth}{p}

\newcommand{\ZF}[1]{\mathcal{Z}(#1)}
\newcommand{\ZC}{\mathcal{X}}
\newcommand{\Moth}{M^{gen}}

\newcommand{\finitep}{p^{(\leq n)}}
\newcommand{\finiteq}{q^{(\leq k)}}

\newcommand{\finiteph}{\wh{p}^{(\leq n)}}
\newcommand{\finiteqh}{\wh{q}^{(\leq k)}}

\newcommand{\prii}{{\bf p}}

\DeclareMathAlphabet{\mathpzc}{OT1}{pzc}{mb}{it}
\documentclass{memo-l}

\usepackage{enumerate, amsmath, amssymb, amscd, paralist,array,url,amsthm}
\usepackage[all,web,poly,curve,frame]{xy}
\usepackage{graphicx}
\usepackage{pstricks,pstricks-add,pst-math,pst-xkey}

\newtheorem{theorem}{Theorem}[chapter]
\newtheorem{lemma}[theorem]{Lemma}
\newtheorem{proposition}[theorem]{Proposition}
\newtheorem{corollary}[theorem]{Corollary}

\theoremstyle{definition}
\newtheorem{definition}[theorem]{Definition}
\newtheorem{example}[theorem]{Example}

\theoremstyle{remark}
\newtheorem{remark}[theorem]{Remark}

\numberwithin{section}{chapter}
\numberwithin{equation}{chapter}

\makeindex

\begin{document}

\frontmatter

\title{Dimer models and Calabi-Yau algebras}


\author{Nathan Broomhead}
\address{Department of Mathematical Sciences, University of Bath, Bath, BA2 7AY, UK}
\curraddr{Institute of Algebraic Geometry, Leibniz Universität Hannover, Welfengarten 1, 30167 Hannover, Germany}
\email{broomhead@math.uni-hannover.de}
\thanks{}

\date{June 8, 2009}



\begin{abstract}
In this article we use techniques from algebraic geometry and homological algebra, together with ideas from string theory to construct a class of 3-dimensional Calabi-Yau algebras. The Calabi-Yau property appears throughout geometry and string theory and is increasingly being studied in algebra. We further show that the algebras constructed are examples of non-commutative crepant resolutions (NCCRs), in the sense of Van den Bergh, of Gorenstein affine toric threefolds.

Dimer models, first studied in theoretical physics, give a way of writing down a class of non-commutative algebras, as the path algebra of a quiver with relations obtained from a `superpotential'. Some examples are Calabi-Yau and some are not. We consider two types of `consistency' condition on dimer models, and show that a `geometrically consistent' dimer model is `algebraically consistent'. We prove that the algebras obtained from algebraically consistent dimer models are 3-dimensional Calabi-Yau algebras. This is the key step which allows us to prove that these algebras are NCCRs of the Gorenstein affine toric threefolds associated to the dimer models. 
\end{abstract}

\maketitle


\setcounter{page}{4}

\tableofcontents

\chapter*{Acknowledgements}
First and foremost I would like to thank my supervisor Alastair King for his support and guidance. I will always appreciate our long and stimulating discussions, scattered with his deep and insightful questions and comments. I would also like to thank my other supervisor Gregory Sankaran for his help and advice. I am grateful to Ami Hanany for his insights conveyed during visits to the Perimeter Institute, Waterloo and the Newton Institute, Cambridge, and to Ali Craw, Bal$\acute{\text{a}}$zs Szendr\"oi, Kasushi Ueda, Raf Bocklandt and Joe Chuang for their thoughts.

Thanks also to members of the geometry group in Bath, particularly Fran Burstall and David Calderbank, for their stimulating mathematical conversation, and the postgraduates who I've had the pleasure of studying with. Finally I'd like to thank my friends and family for keeping me sane and making my time in Bath such a happy one.

\mainmatter
\chapter{Introduction}

\section{Overview}

In this article we use techniques from algebraic geometry and homological algebra, together with ideas from string theory to construct a class of 3-dimensional Calabi-Yau algebras which are non-commutative crepant resolutions of Gorenstein affine toric threefolds. The Calabi-Yau property appears throughout geometry and string theory where, for example, 3-dimensional Calabi-Yau manifolds play an important role in mirror symmetry. A characteristic property of an $n$-dimensional Calabi-Yau manifold $X$, is that the $n$th power of the shift functor on $D(X):=D^b(coh(X))$, the bounded derived category of coherent sheaves on $X$, is a Serre functor. That is, there exists a natural isomorphism
$$ \Hom_{D(X)}(A,B) \cong \Hom_{D(X)}(B, A[n])^* \qquad \forall A,B \in D(X)$$
This property is not restricted to categories of the form $D(X)$.
The idea behind Calabi-Yau algebras is to write down conditions on the algebra $A$ such that $D(A):=D^b(mod(A))$, the bounded derived category of modules over $A$, has the same property. In this form the Calabi-Yau property is increasingly being studied in algebra.

Calabi-Yau algebras are important in the context of non-commutative resolutions of singularities.
One way of studying resolutions of singularities is by considering their derived categories.  Of particular interest are crepant (i.e. suitably minimal) resolutions of toric Gorenstein singularities. Crepant resolutions do not always exist and if they do, they are not in general unique. However it is conjectured (Bondal and Orlov) that if $f_1:Y_1 \rightarrow X$ and $f_2:Y_2 \rightarrow X$ are crepant resolutions then there is a derived equivalence $D(Y_1) \cong D(Y_2)$. Thus the derived category of a crepant resolution is an invariant of the singularity.

One way to try to understand the derived category of a toric crepant resolution $Y$, is to look for a tilting bundle $T$, i.e. a bundle which determines a derived equivalence $D(Y) \cong D(A)$, where $A=\End(T)$. If an equivalence of this form exists, then one could consider the algebra $A$ as a type of non-commutative crepant resolution (NCCR) of the singularity. Van den Bergh \cite{VdB} formalised this idea with a definition of an NCCR which depends only on the singularity.
Given an NCCR $A$, a (commutative) crepant resolution $Y$ such that $D(Y)\cong D(A)$ can be constructed as a moduli space of certain stable $A$-representations (Van den Bergh \cite{VdB}). This is a generalisation of the approach to the McKay correspondence in \cite{BKR}.

If $X=\operatorname{Spec}(R)$ is a Gorenstein singularity, then any crepant resolution is a Calabi-Yau variety. 
Therefore if $A$ is an NCCR it must be a Calabi-Yau algebra. The center of $A$ must also be the coordinate ring $R$ of the singularity.
Thus any 3-dimensional Calabi-Yau algebra whose center $R$ is the coordinate ring of toric Gorenstein 3-fold, is potentially an NCCR of $X=\operatorname{Spec}(R)$.

In \cite{Bocklandt}, Bocklandt proved that every graded Calabi-Yau algebra of global dimension 3 is isomorphic to a superpotential algebra. A superpotential algebra
$A = \C Q/(dW)$ is the quotient of the path algebra of a quiver $Q$ by an ideal of relations $(dW)$, where the relations are generated by taking (formal) partial derivatives of a single element $W$ called the `superpotential'. 
The superpotential not only gives a very concise way of writing down the relations in a non-commutative algebra, it also encodes some information about the syzygies, i.e. relations between the relations. The Calabi-Yau condition is actually equivalent to saying that all the syzygies can be obtained from the superpotential.
Not all superpotential algebras are Calabi-Yau, and it is an open question to understand which ones are.

\section{Structure of the article and main results}
Chapter 2 acts as a general introduction to dimer models. A dimer model is a finite bipartite tiling of a compact (oriented) Riemann surface $\RS$. The examples that will be of particular interest are tilings of the 2-torus, and in this case, we can consider the dimer model as a doubly periodic tiling of the plane. 
Given a dimer model we also consider its dual tiling, where faces are dual to vertices and edges dual to edges. The edges of this dual tiling inherit an orientation from the bipartiteness of the dimer model. This is usually chosen so that the arrows go clockwise around a face dual to a white vertex. Therefore the dual tiling is actually a quiver $Q$, with faces. In the example below we draw the dimer model and the quiver together, with the dimer model highlighted in the left hand diagram. The dotted lines show a fundamental domain.

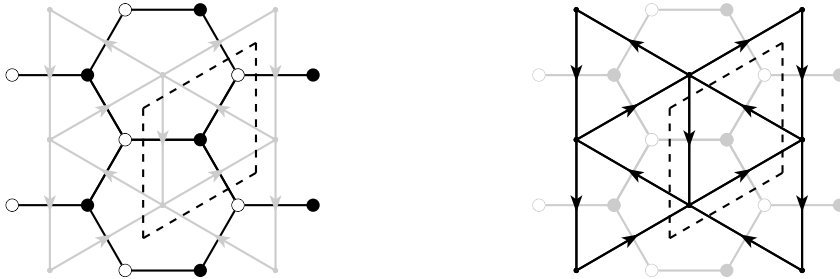
\begin{center}
\newrgbcolor{zzzzzz}{0.8 0.8 0.8}
\psset{xunit=1.0cm,yunit=1.0cm,algebraic=true,dotstyle=*,dotsize=5pt 0,linewidth=0.8pt,arrowsize=3pt 2,arrowinset=0.25}
\begin{pspicture*}(-3,-1)(10,4)
\psline(-2,1)(-1,1)
\psline(-1,1)(-0.5,1.87)
\psline(-0.5,1.87)(-1,2.73)
\psline(-1,2.73)(-2,2.73)
\psline(-1,2.73)(-0.5,1.87)
\psline(-0.5,1.87)(0.5,1.87)
\psline(0.5,1.87)(1,2.73)
\psline(1,2.73)(0.5,3.6)
\psline(0.5,3.6)(-0.5,3.6)
\psline(-0.5,3.6)(-1,2.73)
\psline(0.5,1.87)(-0.5,1.87)
\psline(-0.5,1.87)(-1,1)
\psline(-1,1)(-0.5,0.13)
\psline(-0.5,0.13)(0.5,0.13)
\psline(0.5,0.13)(1,1)
\psline(1,1)(0.5,1.87)
\psline(0.5,1.87)(1,1)
\psline(1,1)(2,1)
\psline(2,2.73)(1,2.73)
\psline(1,2.73)(0.5,1.87)
\psline(5.5,0.13)(7,1)
\psline(7,1)(8.5,0.13)
\psline(8.5,1.87)(8.5,0.13)
\psline(7,1)(8.5,1.87)
\psline(8.5,1.87)(7,2.73)
\psline(7,2.73)(7,1)
\psline(7,1)(5.5,1.87)
\psline(5.5,1.87)(5.5,0.13)
\psline(5.5,1.87)(7,2.73)
\psline(7,2.73)(5.5,3.6)
\psline(5.5,3.6)(5.5,1.87)
\psline(7,2.73)(8.5,3.6)
\psline(8.5,3.6)(8.5,1.87)
\psline[ArrowInside=->, linecolor=zzzzzz](-1.5,0.13)(0,1)
\psline[ArrowInside=->, linecolor=zzzzzz](1.5,0.13)(0,1)
\psline[ArrowInside=->, linecolor=zzzzzz](1.5,1.87)(1.5,0.13)
\psline[ArrowInside=->, linecolor=zzzzzz](0,1)(1.5,1.87)
\psline[ArrowInside=->, linecolor=zzzzzz](1.5,1.87)(0,2.73)
\psline[ArrowInside=->, linecolor=zzzzzz](0,2.73)(0,1)
\psline[ArrowInside=->, linecolor=zzzzzz](0,1)(-1.5,1.87)
\psline[ArrowInside=->, linecolor=zzzzzz](-1.5,1.87)(-1.5,0.13)
\psline[ArrowInside=->, linecolor=zzzzzz](-1.5,1.87)(0,2.73)
\psline[ArrowInside=->, linecolor=zzzzzz](0,2.73)(-1.5,3.6)
\psline[ArrowInside=->, linecolor=zzzzzz](-1.5,3.6)(-1.5,1.87)
\psline[ArrowInside=->, linecolor=zzzzzz](0,2.73)(1.5,3.6)
\psline[ArrowInside=->, linecolor=zzzzzz](1.5,3.6)(1.5,1.87)
\psline[linecolor=zzzzzz](5,1)(6,1)
\psline[linecolor=zzzzzz](6,1)(6.5,1.87)
\psline[linecolor=zzzzzz](6.5,1.87)(6,2.73)
\psline[linecolor=zzzzzz](6,2.73)(5,2.73)
\psline[linecolor=zzzzzz](6,2.73)(6.5,1.87)
\psline[linecolor=zzzzzz](6.5,1.87)(7.5,1.87)
\psline[linecolor=zzzzzz](7.5,1.87)(8,2.73)
\psline[linecolor=zzzzzz](8,2.73)(7.5,3.6)
\psline[linecolor=zzzzzz](7.5,3.6)(6.5,3.6)
\psline[linecolor=zzzzzz](6.5,3.6)(6,2.73)
\psline[linecolor=zzzzzz](7.5,1.87)(6.5,1.87)
\psline[linecolor=zzzzzz](6.5,1.87)(6,1)
\psline[linecolor=zzzzzz](6,1)(6.5,0.13)
\psline[linecolor=zzzzzz](6.5,0.13)(7.5,0.13)
\psline[linecolor=zzzzzz](7.5,0.13)(8,1)
\psline[linecolor=zzzzzz](8,1)(7.5,1.87)
\psline[linecolor=zzzzzz](7.5,1.87)(8,1)
\psline[linecolor=zzzzzz](8,1)(9,1)
\psline[linecolor=zzzzzz](9,2.73)(8,2.73)
\psline[linecolor=zzzzzz](8,2.73)(7.5,1.87)
\psline[ArrowInside=->](5.5,0.13)(7,1)
\psline[ArrowInside=->](8.5,0.13)(7,1)
\psline[ArrowInside=->](8.5,1.87)(8.5,0.13)
\psline[ArrowInside=->](7,1)(8.5,1.87)
\psline[ArrowInside=->](8.5,1.87)(7,2.73)
\psline[ArrowInside=->](7,2.73)(7,1)
\psline[ArrowInside=->](7,1)(5.5,1.87)
\psline[ArrowInside=->](5.5,1.87)(5.5,0.13)
\psline[ArrowInside=->](5.5,1.87)(7,2.73)
\psline[ArrowInside=->](7,2.73)(5.5,3.6)
\psline[ArrowInside=->](5.5,3.6)(5.5,1.87)
\psline[ArrowInside=->](7,2.73)(8.5,3.6)
\psline[ArrowInside=->](8.5,3.6)(8.5,1.87)
\psline[linestyle=dashed,dash=3pt 3pt](-0.26,2.29)(-0.26,0.56)
\psline[linestyle=dashed,dash=3pt 3pt](-0.26,0.56)(1.24,1.43)
\psline[linestyle=dashed,dash=3pt 3pt](1.24,1.43)(1.24,3.16)
\psline[linestyle=dashed,dash=3pt 3pt](1.24,3.16)(-0.26,2.29)
\psline[linestyle=dashed,dash=3pt 3pt](6.74,2.29)(6.74,0.56)
\psline[linestyle=dashed,dash=3pt 3pt](6.74,0.56)(8.24,1.43)
\psline[linestyle=dashed,dash=3pt 3pt](8.24,1.43)(8.24,3.16)
\psline[linestyle=dashed,dash=3pt 3pt](8.24,3.16)(6.74,2.29)
\psdots[dotstyle=o](-2,1)
\psdots[dotstyle=o](-0.5,1.87)
\psdots(-1,2.73)
\psdots[dotstyle=o](-2,2.73)
\psdots(0.5,1.87)
\psdots[dotstyle=o](1,2.73)
\psdots(0.5,3.6)
\psdots[dotstyle=o](-0.5,3.6)
\psdots(-1,1)
\psdots[dotstyle=o](-0.5,0.13)
\psdots(0.5,0.13)
\psdots[dotstyle=o](1,1)
\psdots(2,1)
\psdots(2,2.73)
\psdots[dotstyle=o](1,2.73)
\psdots[dotsize=2pt 0,linecolor=zzzzzz](-1.5,1.87)
\psdots[dotsize=2pt 0,linecolor=zzzzzz](0,1)
\psdots[dotsize=2pt 0](5.5,1.87)
\psdots[dotsize=2pt 0](7,1)
\psdots[dotsize=2pt 0](7,2.73)
\psdots[dotsize=2pt 0](8.5,1.87)
\psdots[dotsize=2pt 0](5.5,3.6)
\psdots[dotsize=2pt 0](5.5,0.13)
\psdots[dotsize=2pt 0](8.5,0.13)
\psdots[dotsize=2pt 0](8.5,3.6)
\psdots[dotsize=2pt 0,linecolor=zzzzzz](-1.5,0.13)
\psdots[dotsize=2pt 0,linecolor=zzzzzz](0,1)
\psdots[dotsize=2pt 0,linecolor=zzzzzz](0,1)
\psdots[dotsize=2pt 0,linecolor=zzzzzz](1.5,0.13)
\psdots[dotsize=2pt 0,linecolor=zzzzzz](1.5,1.87)
\psdots[dotsize=2pt 0,linecolor=zzzzzz](1.5,0.13)
\psdots[dotsize=2pt 0,linecolor=zzzzzz](0,1)
\psdots[dotsize=2pt 0,linecolor=zzzzzz](1.5,1.87)
\psdots[dotsize=2pt 0,linecolor=zzzzzz](1.5,1.87)
\psdots[dotsize=2pt 0,linecolor=zzzzzz](0,2.73)
\psdots[dotsize=2pt 0,linecolor=zzzzzz](0,2.73)
\psdots[dotsize=2pt 0,linecolor=zzzzzz](0,1)
\psdots[dotsize=2pt 0,linecolor=zzzzzz](0,1)
\psdots[dotsize=2pt 0,linecolor=zzzzzz](-1.5,1.87)
\psdots[dotsize=2pt 0,linecolor=zzzzzz](-1.5,1.87)
\psdots[dotsize=2pt 0,linecolor=zzzzzz](-1.5,0.13)
\psdots[dotsize=2pt 0,linecolor=zzzzzz](-1.5,1.87)
\psdots[dotsize=2pt 0,linecolor=zzzzzz](0,2.73)
\psdots[dotsize=2pt 0,linecolor=zzzzzz](0,2.73)
\psdots[dotsize=2pt 0,linecolor=zzzzzz](-1.5,3.6)
\psdots[dotsize=2pt 0,linecolor=zzzzzz](-1.5,3.6)
\psdots[dotsize=2pt 0,linecolor=zzzzzz](-1.5,1.87)
\psdots[dotsize=2pt 0,linecolor=zzzzzz](0,2.73)
\psdots[dotsize=2pt 0,linecolor=zzzzzz](1.5,3.6)
\psdots[dotsize=2pt 0,linecolor=zzzzzz](1.5,3.6)
\psdots[dotsize=2pt 0,linecolor=zzzzzz](1.5,1.87)
\psdots[dotstyle=o,linecolor=zzzzzz](5,1)
\psdots[linecolor=zzzzzz](6,1)
\psdots[dotstyle=o,linecolor=zzzzzz](6.5,1.87)
\psdots[linecolor=zzzzzz](6,2.73)
\psdots[dotstyle=o,linecolor=zzzzzz](5,2.73)
\psdots[linecolor=zzzzzz](6,2.73)
\psdots[dotstyle=o,linecolor=zzzzzz](6.5,1.87)
\psdots[linecolor=zzzzzz](7.5,1.87)
\psdots[dotstyle=o,linecolor=zzzzzz](8,2.73)
\psdots[linecolor=zzzzzz](7.5,3.6)
\psdots[dotstyle=o,linecolor=zzzzzz](6.5,3.6)
\psdots[linecolor=zzzzzz](7.5,1.87)
\psdots[dotstyle=o,linecolor=zzzzzz](6.5,1.87)
\psdots[linecolor=zzzzzz](6,1)
\psdots[dotstyle=o,linecolor=zzzzzz](6.5,0.13)
\psdots[linecolor=zzzzzz](7.5,0.13)
\psdots[dotstyle=o,linecolor=zzzzzz](8,1)
\psdots[linecolor=zzzzzz](7.5,1.87)
\psdots[dotstyle=o,linecolor=zzzzzz](8,1)
\psdots[linecolor=zzzzzz](9,1)
\psdots[linecolor=zzzzzz](9,2.73)
\psdots[dotstyle=o,linecolor=zzzzzz](8,2.73)
\psdots[dotsize=2pt 0](5.5,0.13)
\psdots[dotsize=2pt 0](7,1)
\psdots[dotsize=2pt 0](7,1)
\psdots[dotsize=2pt 0](8.5,0.13)
\psdots[dotsize=2pt 0](8.5,1.87)
\psdots[dotsize=2pt 0](8.5,0.13)
\psdots[dotsize=2pt 0](7,1)
\psdots[dotsize=2pt 0](8.5,1.87)
\psdots[dotsize=2pt 0](8.5,1.87)
\psdots[dotsize=2pt 0](7,2.73)
\psdots[dotsize=2pt 0](5.5,1.87)
\psdots[dotsize=2pt 0](5.5,0.13)
\psdots[dotsize=2pt 0](5.5,1.87)
\psdots[dotsize=2pt 0](7,2.73)
\psdots[dotsize=2pt 0](7,2.73)
\psdots[dotsize=2pt 0](5.5,3.6)
\psdots[dotsize=2pt 0](5.5,3.6)
\psdots[dotsize=2pt 0](5.5,1.87)
\psdots[dotsize=2pt 0](7,2.73)
\psdots[dotsize=2pt 0](8.5,3.6)
\psdots[dotsize=2pt 0](8.5,3.6)
\psdots[dotsize=2pt 0](8.5,1.87)
\end{pspicture*}
\end{center}

The faces of the quiver encode a superpotential $W$, and so there is a superpotential algebra $A= \C Q/(dW)$ associated to every dimer model.

In \cite{HananyKen} Hanany {\it et al} describe a way of using `perfect matchings' to construct a commutative ring $R$ from a dimer model. A perfect matching is subset of the edges of a dimer model with the property that every vertex of the dimer model is the end point of precisely one of these edges. For example the following diagram shows the three perfect matchings of the hexagonal tiling, where the edges in the perfect matchings are shown as thick grey lines.

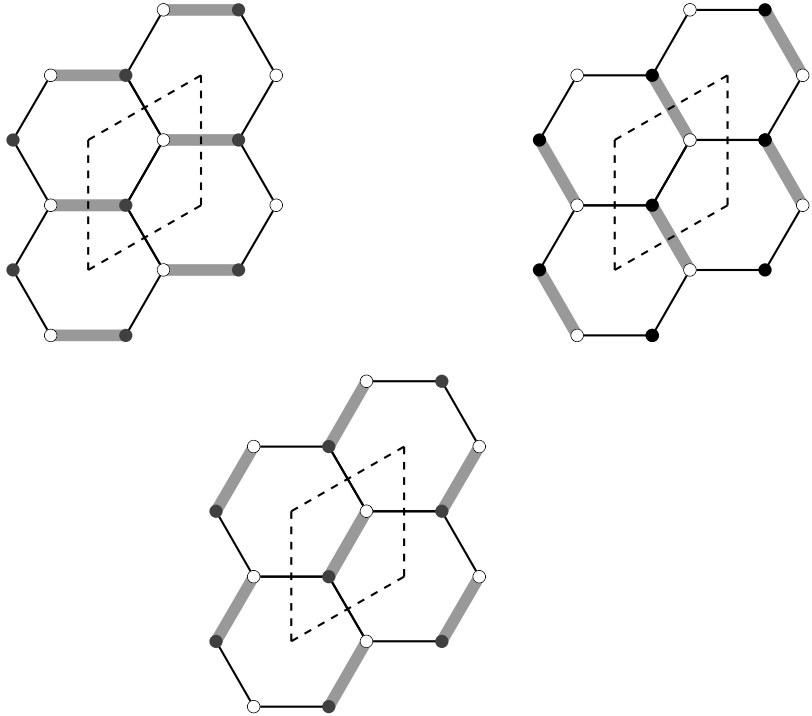
\begin{center}
\newrgbcolor{zzzzzz}{0.6 0.6 0.6}
\psset{xunit=1.0cm,yunit=1.0cm,algebraic=true,dotstyle=*,dotsize=5pt 0,linewidth=0.8pt,arrowsize=3pt 2,arrowinset=0.25}
\begin{pspicture*}(-3,-4.5)(8.5,6)
\psline(5,1)(6,1)
\psline(6,1)(6.5,1.87)
\psline[linewidth=4.4pt,linecolor=zzzzzz](6.5,1.87)(6,2.73)
\psline(6,2.73)(5,2.73)
\psline(5,2.73)(4.5,1.87)
\psline[linewidth=4.4pt,linecolor=zzzzzz](4.5,1.87)(5,1)
\psline[linewidth=4.4pt,linecolor=zzzzzz](6,2.73)(6.5,1.87)
\psline(6.5,1.87)(7.5,1.87)
\psline(7.5,1.87)(8,2.73)
\psline[linewidth=4.4pt,linecolor=zzzzzz](8,2.73)(7.5,3.6)
\psline(7.5,3.6)(6.5,3.6)
\psline(6.5,3.6)(6,2.73)
\psline(5,2.73)(6,2.73)
\psline(6,2.73)(6.5,3.6)
\psline(6.5,3.6)(6,4.46)
\psline(6,4.46)(5,4.46)
\psline(5,4.46)(4.5,3.6)
\psline[linewidth=4.4pt,linecolor=zzzzzz](4.5,3.6)(5,2.73)
\psline[linewidth=4.4pt,linecolor=zzzzzz](6,4.46)(6.5,3.6)
\psline(6.5,3.6)(7.5,3.6)
\psline(7.5,3.6)(8,4.46)
\psline[linewidth=4.4pt,linecolor=zzzzzz](8,4.46)(7.5,5.33)
\psline(7.5,5.33)(6.5,5.33)
\psline(6.5,5.33)(6,4.46)
\psline[linestyle=dashed,dash=3pt 3pt](5.5,3.6)(7,4.46)
\psline[linestyle=dashed,dash=3pt 3pt](7,4.46)(7,2.73)
\psline[linestyle=dashed,dash=3pt 3pt](7,2.73)(5.5,1.87)
\psline[linestyle=dashed,dash=3pt 3pt](5.5,1.87)(5.5,3.6)
\psline[linewidth=4.4pt,linecolor=zzzzzz](-2,1)(-1,1)
\psline(-1,1)(-0.5,1.87)
\psline(-0.5,1.87)(-1,2.73)
\psline[linewidth=4.4pt,linecolor=zzzzzz](-1,2.73)(-2,2.73)
\psline(-2,2.73)(-2.5,1.87)
\psline(-2.5,1.87)(-2,1)
\psline(-1,2.73)(-0.5,1.87)
\psline[linewidth=4.4pt,linecolor=zzzzzz](-0.5,1.87)(0.5,1.87)
\psline(0.5,1.87)(1,2.73)
\psline(1,2.73)(0.5,3.6)
\psline[linewidth=4.4pt,linecolor=zzzzzz](0.5,3.6)(-0.5,3.6)
\psline(-0.5,3.6)(-1,2.73)
\psline[linewidth=4.4pt,linecolor=zzzzzz](-2,2.73)(-1,2.73)
\psline(-1,2.73)(-0.5,3.6)
\psline(-0.5,3.6)(-1,4.46)
\psline[linewidth=4.4pt,linecolor=zzzzzz](-1,4.46)(-2,4.46)
\psline(-2,4.46)(-2.5,3.6)
\psline(-2.5,3.6)(-2,2.73)
\psline(-1,4.46)(-0.5,3.6)
\psline[linecolor=zzzzzz](-0.5,3.6)(0.5,3.6)
\psline(0.5,3.6)(1,4.46)
\psline(1,4.46)(0.5,5.33)
\psline[linewidth=4.4pt,linecolor=zzzzzz](0.5,5.33)(-0.5,5.33)
\psline(-0.5,5.33)(-1,4.46)
\psline[linestyle=dashed,dash=3pt 3pt](-1.5,3.6)(0,4.46)
\psline[linestyle=dashed,dash=3pt 3pt](0,4.46)(0,2.73)
\psline[linestyle=dashed,dash=3pt 3pt](0,2.73)(-1.5,1.87)
\psline[linestyle=dashed,dash=3pt 3pt](-1.5,1.87)(-1.5,3.6)
\psline(0.7,-3.94)(1.7,-3.94)
\psline[linewidth=4.4pt,linecolor=zzzzzz](1.7,-3.94)(2.2,-3.07)
\psline(2.2,-3.07)(1.7,-2.21)
\psline(1.7,-2.21)(0.7,-2.21)
\psline[linewidth=4.4pt,linecolor=zzzzzz](0.7,-2.21)(0.2,-3.07)
\psline(0.2,-3.07)(0.7,-3.94)
\psline(1.7,-2.21)(2.2,-3.07)
\psline(2.2,-3.07)(3.2,-3.07)
\psline[linewidth=4.4pt,linecolor=zzzzzz](3.2,-3.07)(3.7,-2.21)
\psline(3.7,-2.21)(3.2,-1.34)
\psline(3.2,-1.34)(2.2,-1.34)
\psline[linewidth=4.4pt,linecolor=zzzzzz](2.2,-1.34)(1.7,-2.21)
\psline(0.7,-2.21)(1.7,-2.21)
\psline[linewidth=4.4pt,linecolor=zzzzzz](1.7,-2.21)(2.2,-1.34)
\psline(2.2,-1.34)(1.7,-0.48)
\psline(1.7,-0.48)(0.7,-0.48)
\psline[linewidth=4.4pt,linecolor=zzzzzz](0.7,-0.48)(0.2,-1.34)
\psline(0.2,-1.34)(0.7,-2.21)
\psline(1.7,-0.48)(2.2,-1.34)
\psline(2.2,-1.34)(3.2,-1.34)
\psline[linewidth=4.4pt,linecolor=zzzzzz](3.2,-1.34)(3.7,-0.48)
\psline(3.7,-0.48)(3.2,0.39)
\psline(3.2,0.39)(2.2,0.39)
\psline[linewidth=4.4pt,linecolor=zzzzzz](2.2,0.39)(1.7,-0.48)
\psline[linestyle=dashed,dash=3pt 3pt](1.2,-1.34)(2.7,-0.48)
\psline[linestyle=dashed,dash=3pt 3pt](2.7,-0.48)(2.7,-2.21)
\psline[linestyle=dashed,dash=3pt 3pt](2.7,-2.21)(1.2,-3.07)
\psline[linestyle=dashed,dash=3pt 3pt](1.2,-3.07)(1.2,-1.34)
\psdots[dotstyle=o](-2,1)
\psdots[dotstyle=o](5,1)
\psdots(6,1)
\psdots(6,1)
\psdots[dotstyle=o](6.5,1.87)
\psdots(6,2.73)
\psdots[dotstyle=o](5,2.73)
\psdots(4.5,1.87)
\psdots(7.5,1.87)
\psdots[dotstyle=o](8,2.73)
\psdots(7.5,3.6)
\psdots[dotstyle=o](6.5,3.6)
\psdots[dotstyle=o](6.5,3.6)
\psdots(6,4.46)
\psdots[dotstyle=o](5,4.46)
\psdots(4.5,3.6)
\psdots(7.5,3.6)
\psdots[dotstyle=o](8,4.46)
\psdots(7.5,5.33)
\psdots[dotstyle=o](6.5,5.33)
\psdots[linecolor=darkgray](-1,1)
\psdots[dotstyle=o](-0.5,1.87)
\psdots[linecolor=darkgray](-1,2.73)
\psdots[dotstyle=o](-2,2.73)
\psdots[linecolor=darkgray](-2.5,1.87)
\psdots[linecolor=darkgray](0.5,1.87)
\psdots[dotstyle=o](1,2.73)
\psdots[linecolor=darkgray](0.5,3.6)
\psdots[dotstyle=o](-0.5,3.6)
\psdots[linecolor=darkgray](-1,4.46)
\psdots[dotstyle=o](-2,4.46)
\psdots[linecolor=darkgray](-2.5,3.6)
\psdots[dotstyle=o](1,4.46)
\psdots[linecolor=darkgray](0.5,5.33)
\psdots[dotstyle=o](-0.5,5.33)
\psdots[linecolor=darkgray](1.7,-3.94)
\psdots[dotstyle=o](0.7,-3.94)
\psdots[dotstyle=o](2.2,-3.07)
\psdots[linecolor=darkgray](1.7,-2.21)
\psdots[dotstyle=o](0.7,-2.21)
\psdots[linecolor=darkgray](0.2,-3.07)
\psdots[linecolor=darkgray](3.2,-3.07)
\psdots[dotstyle=o](3.7,-2.21)
\psdots[linecolor=darkgray](3.2,-1.34)
\psdots[dotstyle=o](2.2,-1.34)
\psdots[linecolor=darkgray](1.7,-0.48)
\psdots[dotstyle=o](0.7,-0.48)
\psdots[linecolor=darkgray](0.2,-1.34)
\psdots[dotstyle=o](3.7,-0.48)
\psdots[linecolor=darkgray](3.2,0.39)
\psdots[dotstyle=o](2.2,0.39)
\end{pspicture*}
\end{center}

The difference of two perfect matchings defines a homology class of the 2-torus and so, by choosing a fixed `reference matching' to subtract from each perfect matching, we obtain a set of points in the integer homology lattice of the 2-torus $H_1(T) \cong \Z^{2}$. The convex hull of these points is a lattice polygon. Taking the cone on the polygon and using the machinery of toric geometry, this defines $R=\C [X]$, the coordinate ring of an affine toric Gorenstein 3-fold $X$. Given a lattice point in the polygon, its multiplicity is defined to be the number of perfect matchings corresponding to that point. A perfect matching is said to be extremal if it corresponds to a vertex of the polygon.
Looking at the example above we see that the polygon is a triangle, with each of the three perfect matching corresponding to a vertex, and no other lattice points. The ring $R$ is the polynomial ring in three variables.

Since we have a superpotential algebra associated to every dimer model, it is natural to ask if these algebras are Calabi-Yau. In fact there are some examples which are Calabi-Yau and some which are not. Therefore we ask what conditions can be placed on a dimer model so that its superpotential algebra is Calabi-Yau. In Chapter~3 we discuss the two ideas of `consistency' first understood by Hanany and Vegh \cite{HananyVegh}. Consistency conditions are a strong type of non-degeneracy condition.  
Following Kenyon and Schlenker \cite{Kenyon}, we give necessary and sufficient conditions for a dimer model to be `geometrically consistent' in terms of the intersection properties of special paths called zig-zag flows on the universal cover $\Qcov$ of the quiver $Q$. Geometric consistency amounts to saying that zig-zag flows behave effectively like straight lines.

In Chapter 4 we study some properties of zig-zag flows in a geometrically consistent dimer model. The homology class of a zig-zag flow encodes information about the `direction' of that flow. We show that at each quiver face $\face$, there is a `local zig-zag fan' in the homology lattice of the torus generated by the homology classes of zig-zag flows which intersect the boundary of $\face$. Furthermore, the cyclic order of the intersections around the face, is the same as the order of the rays in the local zig-zag fan. There is also a `global zig-zag fan' generated by the homology classes of all zig-zag flows. Using these fans we construct, in a very explicit way, a collection of perfect matchings indexed by the 2-dimensional cones in the global zig-zag fan. We prove that these perfect matchings are extremal perfect matchings and they are all the extremal perfect matchings. We also see that each perfect matching of this form corresponds to a different vertex of the polygon described above, so these vertices are of multiplicity one. 

In Chapter 5 we introduce the concept of (non-commutative, affine, normal) toric algebras. 
These are non-commutative algebras which have an underlying combinatoric structure. They are defined by knowledge of a lattice containing a strongly convex rational polyhedral cone, a set and a lattice map. It is hoped that they may play a similar role in non-commutative algebraic geometry to that played by toric varieties in algebraic geometry. We prove that toric algebras are examples of toric orders which we define. Finally we show that there is a toric algebra $B$ naturally associated to every dimer model, and moreover, the center of this algebra is the ring $R$ associated to the dimer model in the way we described above.

Therefore a given dimer model has two non-commutative algebras $A$ and $B$ and there is a natural algebra map $ \mathpzc{h}: A \longrightarrow B$. We call a dimer model `algebraically consistent' if this map is an isomorphism.
Algebraic and geometric consistency are the two consistency conditions that we will study in the core of the article.

In Chapter 6 we prove the following main theorem
\begin{theorem}
A geometrically consistent dimer model is algebraically consistent.
\end{theorem}
The proof relies on the explicit description of extremal perfect matchings from Chapter 5. We actually prove the following proposition, which provides the surjectivity of the map $\mathpzc{h}: A \longrightarrow B$, while the injectivity is provided by a result of Hanany, Herzog and Vegh (see Theorem~\ref{Uniqueness} and Remark~\ref{hinject}).
\begin{proposition}
Given a geometrically consistent dimer model, for all vertices $i,j$ in the universal cover of $Q$, there exists a path from $i$ to $j$ which avoids some extremal perfect matching.
\end{proposition}
Thus we see that the extremal perfect matchings play a key role in the theory.

In Chapter 7 we discuss Ginzburg's definition of a Calabi-Yau algebra and prove the following main theorem
\begin{theorem} \label{introcy}
If a dimer model on a torus is algebraically consistent then the algebra $A$ obtained from it is a Calabi-Yau algebra of global dimension 3.
\end{theorem}
Therefore we have shown that for dimer models on the torus, algebraic consistency  and consequently geometric consistency, is a sufficient condition to obtain a Calabi-Yau algebra. Thus we have produced a class of superpotential algebras which are Calabi-Yau.  

Finally, in Chapter~8 we show explicitly that the algebra $A$, with centre $R$, obtained from an algebraically consistent dimer model on a torus, is isomorphic to the endomorphism algebra of a given reflexive $R$-module. Together with Theorem~\ref{introcy} and a result of Stafford and Van den Bergh from \cite{VdB2}, this is enough to prove the following Theorem.
\begin{theorem}
Given an algebraically consistent dimer model on a torus, the algebra $A$ obtained from it is an NCCR of the (commutative) ring $R$ associated to that dimer model. 
\end{theorem}
 
We note in Section~\ref{stiengul} that work of Gulotta \cite{Gulotta} and Stienstra \cite{Stienstra2} proves that for any lattice polygon $V$, there exists a geometrically consistent dimer model which has $V$ as its perfect matching polygon. Thus to every Gorenstein affine toric threefold, there is an associated geometrically consistent model. Therefore, using results from this article, we are able to conclude the following result.
\begin{theorem}
Every Gorenstein affine toric threefold admits an NCCR, which can be obtained via a geometrically consistent dimer model.
\end{theorem}


\section{Related results}
Recently there have been several papers proving results that are related
to parts of this article.
In \cite{Mozgovoy} Mozgovoy and Reineke prove that if an algebra
obtained from a dimer model satisfies two conditions then it is a 3
dimensional Calabi-Yau algebra. Davison \cite{Davison} goes on to show
that the second of these conditions is actually a consequence of the
first. The first condition states that the algebra should satisfy a
cancellation property. This holds in algebraically consistent cases as
cancellation is a property of toric algebras. Therefore the condition
also holds in the geometrically consistent case. In fact, Ishii and Ueda in \cite{Ishii} and Bocklandt in \cite{Bocklandt2} have recently shown that this first condition is equivalent to a slightly weakened form of geometric consistency, also
defined in terms of intersection properties of zig-zag flows.


In \cite{Bocklandt2}, Bocklandt also states a result in the opposite direction. He shows that any graded ``toric order" which is a Calabi-Yau algebra of global dimension 3 comes from a dimer model on the torus which satisfies the first condition of Mozgovoy and Reineke. Bocklandt's definition of a toric order and the one we give in Chapter~\ref{ch:ToricA} are closely related, however his definition includes an additional condition. The idea of a toric order was first brought to our attention during conversations with Bocklandt, however the precise definitions were developed separately. We note that his additional condition is certainly satisfied by any toric algebra. 

In his paper, Bocklandt actually proves a more general result where he considers algebras with a certain cancellation property. He shows that on this class of algebras, the CY3 property is equivalent to the existence of a ``weighted quiver polyhedron" which he introduces as a generalisation of a dimer model. 

There are also several recent papers which give methods of constructing dimer  
models with desired properties.
We will describe in Section~\ref{stiengul} a
method of Gulotta \cite{Gulotta}, which associates a
geometrically consistent dimer model to any lattice polygon. This is  
done iteratively by removing triangles from a larger polygon (with a known geometrically
consistent dimer model), and considering how this affects the dimer
model. This can be done in such a way that
geometric consistency is preserved. The process 
has been reformulated 
in terms of
adjacency matrices by Stienstra in \cite{Stienstra2}. More recently,
Ishii and Ueda \cite{Ishii} have generalised this method, describing
the process of changing a dimer model as one alters the polygon by removing  
any vertex and taking the convex hull of the remaining vertices.
This generalisation is more complicated, and relies on use of the
special McKay correspondence. The resulting dimer model is not
necessarily geometrically consistent but can be shown to satisfy the weakened version of geometric consistency mentioned above. 
This is not the only property they show is preserved by the process. Given a lattice polygon and a corresponding (non-degenerate) dimer model, the same authors proved in \cite{Ishii1} that the moduli space of $\theta$-stable $(1, \dots,1)$-dimensional representataions of the associated quiver $Q$, for some generic character $\theta$, is a smooth toric Calabi-Yau threefold. It is a crepant resolution of the Gorenstein affine toric threefold corresponding to a lattice
polygon. In particular, this resolution comes with a tautological bundle. If this 
is a tilting
bundle whose endomorphism algebra is the superpotential algebra of $Q$, they prove that this property also holds for the corresponding objects after changing the dimer model by an application of their process.
Using this method, they can also describe the bounded derived categories of crepant resolutions of all Gorenstein
toric threefold singularities.


\chapter{Introduction to the dimer model} \label{Introto}
The aim of this chapter is to provide a mathematical introduction to the theory of dimer models as introduced by Hanany {\it et al} (see
\cite{Kennaway, HananyKen, FHKVW} for example). It is intended to be self contained and should require very little prior knowledge. Most of
the ideas and concepts that appear here, are covered in some form in the physics literature and we will endeavor to provide references where
appropriate. However we note that, in order to produce a fluent mathematical introduction, the way we present these ideas will often be different.
\section{Quivers and algebras from dimer models}\label{Quivalgdim}
The theory begins with a finite bipartite tiling of a compact (oriented) Riemann surface $\RS$.
By `bipartite tiling' we mean a polygonal cell decomposition of the surface, whose vertices and edges form a bipartite graph i.e. the vertices may be coloured black and white in such a way that all edges join a black vertex to a white vertex. In particular we note that each face must have an even number of vertices (and edges) and each vertex has valence at least 2.
We call a tiling of this type, a dimer model. This definition is quite general and as we progress we will describe additional non-degeneracy conditions, the strongest of which is `consistency'.
We note here that in principle faces in a dimer model which have two edges (\digon s) and bivalent vertices are allowed. However we will see that models with these features fail to satisfy certain of the non-degeneracy conditions.
Furthermore we shall observe that the `consistency' condition forces the Riemann surface $\RS$ to be a 2-torus $T$.
Therefore in the majority of this article, we shall focus on bipartite tilings of $T$. In this case, we may (and shall) consider the dimer model instead as a doubly periodic tiling of the plane.
\begin{remark} \label{balanced} In some of the literature it is included as part of the definition of a bipartite graph, that there are the same number of black and white vertices. We do not impose this as a condition here, however if it does hold we call the dimer model `balanced'.
\end{remark}
\subsection{Examples} \label{basicexamples}
The two simplest bipartite tilings of the 2-torus $T$, are the regular ones by squares and by hexagons. In each case, a fundamental domain is indicated by the dotted line.
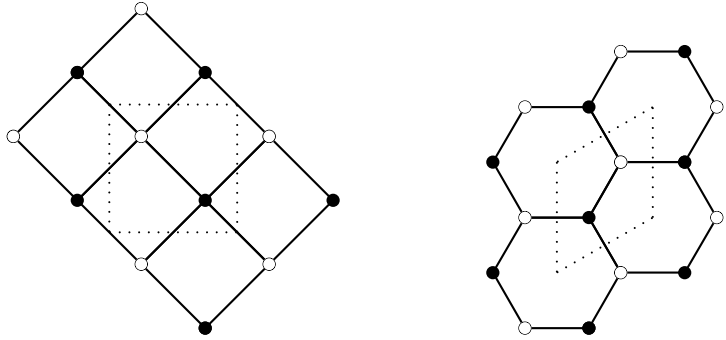
\begin{center}
\psset{xunit=0.85cm,yunit=0.85cm,algebraic=true,dotstyle=*,dotsize=5pt 0,linewidth=0.8pt,arrowsize=3pt 2,arrowinset=0.25}
\begin{pspicture*}(-3.5,0.4)(9,7)
\psline(-2,3)(-1,2)
\psline(-1,2)(0,3)
\psline(0,3)(-1,4)
\psline(-1,4)(-2,3)
\psline(0,3)(-1,2)
\psline(-1,2)(0,1)
\psline(0,1)(1,2)
\psline(1,2)(0,3)
\psline(-1,4)(0,3)
\psline(0,3)(1,4)
\psline(1,4)(0,5)
\psline(0,5)(-1,4)
\psline(0,3)(1,2)
\psline(1,2)(2,3)
\psline(2,3)(1,4)
\psline(1,4)(0,3)
\psline(-2,3)(-1,4)
\psline(-1,4)(-2,5)
\psline(-2,5)(-3,4)
\psline(-3,4)(-2,3)
\psline(-1,4)(0,5)
\psline(0,5)(-1,6)
\psline(-1,6)(-2,5)
\psline(-2,5)(-1,4)
\psline[linestyle=dotted](-1.5,4.5)(0.5,4.5)
\psline[linestyle=dotted](0.5,4.5)(0.5,2.5)
\psline[linestyle=dotted](0.5,2.5)(-1.5,2.5)
\psline[linestyle=dotted](-1.5,2.5)(-1.5,4.5)
\psline(5,1)(6,1)
\psline(6,1)(6.5,1.87)
\psline(6.5,1.87)(6,2.73)
\psline(6,2.73)(5,2.73)
\psline(5,2.73)(4.5,1.87)
\psline(4.5,1.87)(5,1)
\psline(6,2.73)(6.5,1.87)
\psline(6.5,1.87)(7.5,1.87)
\psline(7.5,1.87)(8,2.73)
\psline(8,2.73)(7.5,3.6)
\psline(7.5,3.6)(6.5,3.6)
\psline(6.5,3.6)(6,2.73)
\psline(5,2.73)(6,2.73)
\psline(6,2.73)(6.5,3.6)
\psline(6.5,3.6)(6,4.46)
\psline(6,4.46)(5,4.46)
\psline(5,4.46)(4.5,3.6)
\psline(4.5,3.6)(5,2.73)
\psline(6,4.46)(6.5,3.6)
\psline(6.5,3.6)(7.5,3.6)
\psline(7.5,3.6)(8,4.46)
\psline(8,4.46)(7.5,5.33)
\psline(7.5,5.33)(6.5,5.33)
\psline(6.5,5.33)(6,4.46)
\psline[linestyle=dotted](5.5,3.6)(7,4.46)
\psline[linestyle=dotted](7,4.46)(7,2.73)
\psline[linestyle=dotted](7,2.73)(5.5,1.87)
\psline[linestyle=dotted](5.5,1.87)(5.5,3.6)
\psdots[dotstyle=o](-1,2)
\psdots(-2,3)
\psdots(0,3)
\psdots[dotstyle=o](-1,4)
\psdots(0,1)
\psdots(0,1)
\psdots[dotstyle=o](1,2)
\psdots[dotstyle=o](1,4)
\psdots(0,5)
\psdots(2,3)
\psdots[dotstyle=o](1,4)
\psdots[dotstyle=o](-3,4)
\psdots(-2,5)
\psdots[dotstyle=o](-3,4)
\psdots[dotstyle=o](-1,6)
\psdots(-2,5)
\psdots[dotstyle=o](5,1)
\psdots(6,1)
\psdots(6,1)
\psdots[dotstyle=o](6.5,1.87)
\psdots(6,2.73)
\psdots[dotstyle=o](5,2.73)
\psdots(4.5,1.87)
\psdots(7.5,1.87)
\psdots[dotstyle=o](8,2.73)
\psdots(7.5,3.6)
\psdots[dotstyle=o](6.5,3.6)
\psdots[dotstyle=o](6.5,3.6)
\psdots(6,4.46)
\psdots[dotstyle=o](5,4.46)
\psdots(4.5,3.6)
\psdots(7.5,3.6)
\psdots[dotstyle=o](8,4.46)
\psdots(7.5,5.33)
\psdots[dotstyle=o](6.5,5.33)
\end{pspicture*}
\end{center}
For a tiling of $T$ which is not balanced, consider the following tiling by three rhombi:
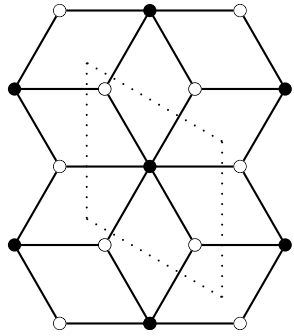
\begin{center}
\psset{xunit=1.2cm,yunit=1.2cm}
\begin{pspicture*}(-2,1.3)(3,6.3)
\psset{xunit=1.2cm,yunit=1.2cm,runit=1.2cm, algebraic=true,dotstyle=*,dotsize=5pt 0,linewidth=0.8pt,arrowsize=3pt 2,arrowinset=0.25}
\psline(0.5,3.87)(0,3)
\psline(0.5,3.87)(1,3)
\psline(0.5,3.87)(1.5,3.87)
\psline(0.5,3.87)(1,4.73)
\psline(0.5,3.87)(0,4.73)
\psline(0.5,3.87)(-0.5,3.87)
\psline(1,3)(2,3)
\psline(1.5,3.87)(2,3)
\psline(1.5,3.87)(2,4.73)
\psline(1,4.73)(2,4.73)
\psline(0.5,5.6)(0,4.73)
\psline(0.5,5.6)(1,4.73)
\psline(0,4.73)(-1,4.73)
\psline(-1,4.73)(-0.5,3.87)
\psline(-0.5,3.87)(-1,3)
\psline(-1,3)(0,3)
\psline(-0.5,5.6)(0.5,5.6)
\psline(-0.5,5.6)(-1,4.73)
\psline(0.5,5.6)(1.5,5.6)
\psline(1.5,5.6)(2,4.73)
\psline(0,3)(0.5,2.13)
\psline(1,3)(0.5,2.13)
\psline(1.5,2.13)(0.5,2.13)
\psline(1.5,2.13)(2,3)
\psline(-1,3)(-0.5,2.13)
\psline(-0.5,2.13)(0.5,2.13)
\psline[linestyle=dotted](-0.2,5.02)(1.3,4.15)
\psline[linestyle=dotted](1.3,4.15)(1.3,2.42)
\psline[linestyle=dotted](1.3,2.42)(-0.2,3.29)
\psline[linestyle=dotted](-0.2,3.29)(-0.2,5.02)
\psdots[dotstyle=o](0,3)
\psdots(1,3)
\psdots[dotstyle=o](1.5,3.87)
\psdots(1,4.73)
\psdots(0,4.73)
\psdots[dotstyle=o](-0.5,3.87)
\psdots(0.5,3.87)
\psdots(2,4.73)
\psdots[dotstyle=o](1.5,5.6)
\psdots(0.5,5.6)
\psdots[dotstyle=o](0,4.73)
\psdots[dotstyle=o](1,4.73)
\psdots[dotstyle=o](-0.5,5.6)
\psdots(-1,4.73)
\psdots(-1,3)
\psdots[dotstyle=o](-0.5,2.13)
\psdots(0.5,2.13)
\psdots[dotstyle=o](1,3)
\psdots[dotstyle=o](0,3)
\psdots(0.5,2.13)
\psdots[dotstyle=o](1.5,2.13)
\psdots(2,3)
\end{pspicture*}
\end{center}

On the 2-sphere, the only regular (Platonic) bipartite tiling is the cube (shown here in stereographic projection):

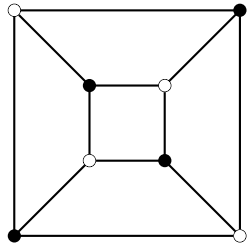
\begin{center}
\psset{xunit=1.0cm,yunit=1.0cm,algebraic=true,dotstyle=*,dotsize=5pt 0,linewidth=0.8pt,arrowsize=3pt 2,arrowinset=0.25}
\begin{pspicture*}(-1,1)(4,6)
\psline(0,5)(0,2)
\psline(0,2)(3,2)
\psline(3,2)(3,5)
\psline(3,5)(0,5)
\psline(1,4)(0,5)
\psline(2,4)(3,5)
\psline(2,3)(3,2)
\psline(1,3)(0,2)
\psline(1,3)(2,3)
\psline(2,3)(2,4)
\psline(2,4)(1,4)
\psline(1,4)(1,3)
\psdots[dotstyle=o](0,5)
\psdots(3,5)
\psdots[dotstyle=o](3,2)
\psdots(0,2)
\psdots(1,4)
\psdots[dotstyle=o](2,4)
\psdots(2,3)
\psdots[dotstyle=o](1,3)
\end{pspicture*}
\end{center}

\subsection{The quiver} \label{quiversect}
Given a dimer model one can consider the dual tiling (or dual cell decomposition) with a vertex dual to every face, an edge dual to every edge and a face dual to every vertex. Crucially, since the dimer model is bipartite, the edges of the dual tiling inherit a consistent choice of orientation. In particular, it is the convention that faces dual to white vertices are oriented clockwise and faces dual to black vertices are oriented anti-clockwise. This is equivalent to requiring that black vertices are on the left and white vertices on the right of every arrow dual to a dimer edge.
Thus the dual graph is a quiver $Q$ (i.e. a directed graph), with the additional structure that it provides
a tiling of the Riemann surface $\RS$ with oriented faces. We will refer to the faces of the quiver dual to black/white vertices of the dimer model, as black/white faces.

In the usual way, we denote by $Q_0$ and $Q_1$ the sets of vertices and arrows (directed edges) of the quiver and by $h,t \colon Q_1 \rightarrow Q_0$ the maps which take an arrow to its head and tail.
To this information we add the set $Q_2$ of oriented faces. We may write down a homological chain
complex for the Riemann surface $\RS$, using the fact that this `quiver with faces' forms a cell
decomposition,
\begin{equation}
\label{eq:chain}
\Z_{Q_2} \lra{\del} \Z_{Q_1} \lra{\del} \Z_{Q_0}
\end{equation}
where $\Z_{Q_i}$ denotes the free abelian group generated by the elements in $Q_i$.
We also have the dual cochain complex
\begin{equation}
\label{eq:cochain}
\Z^{Q_0} \lra{d} \Z^{Q_1} \lra{d} \Z^{Q_2}
\end{equation}
where $\Z^{Q_i}$ denotes the $\Z$-linear functions on $\Z_{Q_i}$.
Note that, because of the way the faces are oriented, the coboundary
map $d\colon \Z^{Q_1}\to \Z^{Q_2}$ simply sums any function of the edges around each face (without any signs).
Let $\face \mapsto (-1)^f$ be the map which takes the value +1 on black faces of $Q$, and -1 on white faces. Then the cycle
$$
\sum_{\face \in Q_2} (-1)^\face \face \in \Z_{Q_2}
$$
is a generator for the kernel of the boundary map $\del$. Thus it defines a fundamental class, i.e. a choice of generator of $H_{2}(\RS)\isom \Z$. We note that the function $\const{1}\in\Z^{Q_2}$, which takes the constant value 1 on every face, evaluates to $|B|-|W|$ on this fundamental cycle (where $|B|$ and $|W|$ denote the number of black and white faces respectively). If we have a balanced dimer model, i.e. $|B|=|W|$, then this implies that the function $\const{1}\in\Z^{Q_2}$ is exact.

\subsection{The superpotential algebra} \label{quivalgsec}
We construct the path algebra $\C Q$ of the quiver; this is a complex algebra with generators $\{ e_i \mid i \in Q_0 \}$ and $\{ x_a \mid a \in Q_1 \}$ subject to the relations $e_i^{2} = e_i  $, $e_ie_j =0 $ for $i \neq j$ and $e_{ta} x_a e_{ha} = x_a$. The $e_i$ are idempotents of the algebra which, we observe, has a monomial basis of paths in $Q$. Following Ginzburg \cite{Ginz}, let $[ \C Q , \C Q ]$ be the complex vector space in $\C Q$ spanned by commutators and denote by $\C Q_\text{cyc}:= \C Q /[ \C Q , \C Q ] $ the quotient space. This space has a basis of elements corresponding to cyclic paths in the quiver.  The consistent orientation of any face $f$ of the quiver means that we may interpret $\del f$ as a cyclic path in the quiver. Therefore the set of faces determines an element of $\C Q_\text{cyc}$ which, following the physics literature \cite{FHKVW}, we refer to as the `superpotential'
\begin{equation}
\label{eq:superpot}
W=\sum_{f\in Q_2} (-1)^f \del f.
\end{equation}
\begin{remark}
In this article we will use the notation $\del f$ for several distinct objects, namely the boundary of $f$ considered as an element of $\Z_{Q_1}$, the boundary of $f$ considered as an element of $\C Q_\text{cyc}$, and the set of arrows which are contained in the boundary of $f$. It should be clear from context which of these we mean in any given instance.
\end{remark}
For each arrow $a \in Q_1$ there is a linear map $\frac{\del}{\del x_a}: \C Q_\text{cyc} \rightarrow \C Q $ which is a (formal) cyclic derivative. The image of a cyclic path is obtained by taking all the representatives in $\C Q$ which start with $x_a$, removing this and then summing them. We note that because of the cyclicity, this is equivalent to taking all the representatives in $\C Q$ which end with $x_a$, removing this and then summing.
The image of a cycle is in fact an element in $e_{ha} \C Q e_{ta}$, and if the cycle contains no repeated arrows then this is just the path which begins at $ha$ and follows the cycle around to $ta$.

This superpotential $W$ determines an ideal of relations in the path algebra $\C Q$
\begin{equation}
\label{eq:ideal}
I_W= \left ( \frac{\del}{\del x_a} W : a\in Q_1 \right ).
\end{equation}
The quotient of the path algebra $\C Q$ by this ideal is the superpotential (or `Jacobian', or quiver) algebra
$$A= \C Q/ I_W$$
From an algebraic point of view, the output of a dimer model is this algebra, and we are interested in understanding the special properties that such algebras exhibit.

\begin{remark} \label{Ftermremark}
Each arrow $a \in Q_1$ occurs in precisely two oppositely oriented faces $\face^+ ,\face^- \in Q_2$ (this also implies that no arrow is repeated in the boundary of any one face). Therefore each relation $\frac{\del}{\del x_a} W$ can be written explicitly as a difference of two paths,
$$ \frac{\del}{\del x_a} W = p_a^+ - p_a^-  $$
where $p_a^\pm$ is the path from $ha$ around the boundary of $\face^\pm$ to $ta$. The relations $p_a^+ = p_a^-$ for $a \in Q_1$ are called `F-term' relations. They generate an equivalence relation on paths in the quiver such that the equivalence classes form a natural basis for the algebra $A$.
\end{remark}

\subsection{Examples}
\begin{example}
We return to the tilings of the torus by regular hexagons and by squares that we saw in Section~\ref{basicexamples}. We start by considering the hexagonal tiling of the torus. The figures below both show the bipartite tiling and the dual quiver, drawn together so it is clear how they are related. The left hand figure highlights the bipartite tiling, and the right hand figure highlights the quiver.
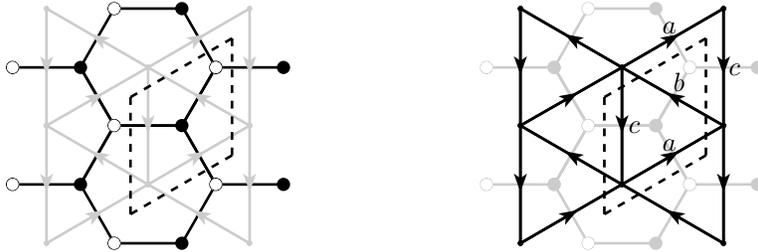
\begin{center}
\newrgbcolor{zzzzzz}{0.8 0.8 0.8}
\psset{xunit=0.9cm,yunit=0.9cm,algebraic=true,dotstyle=*,dotsize=5pt 0,linewidth=1.0pt,arrowsize=3pt 2,arrowinset=0.25}
\begin{pspicture*}(-3,-1)(10,4.3)
\psline(-2,1)(-1,1)
\psline(-1,1)(-0.5,1.87)
\psline(-0.5,1.87)(-1,2.73)
\psline(-1,2.73)(-2,2.73)
\psline(-1,2.73)(-0.5,1.87)
\psline(-0.5,1.87)(0.5,1.87)
\psline(0.5,1.87)(1,2.73)
\psline(1,2.73)(0.5,3.6)
\psline(0.5,3.6)(-0.5,3.6)
\psline(-0.5,3.6)(-1,2.73)
\psline(0.5,1.87)(-0.5,1.87)
\psline(-0.5,1.87)(-1,1)
\psline(-1,1)(-0.5,0.13)
\psline(-0.5,0.13)(0.5,0.13)
\psline(0.5,0.13)(1,1)
\psline(1,1)(0.5,1.87)
\psline(0.5,1.87)(1,1)
\psline(1,1)(2,1)
\psline(2,2.73)(1,2.73)
\psline(1,2.73)(0.5,1.87)
\psline(5.5,0.13)(7,1)
\psline(7,1)(8.5,0.13)
\psline(8.5,1.87)(8.5,0.13)
\psline(7,1)(8.5,1.87)
\psline(8.5,1.87)(7,2.73)
\psline(7,2.73)(7,1)
\psline(7,1)(5.5,1.87)
\psline(5.5,1.87)(5.5,0.13)
\psline(5.5,1.87)(7,2.73)
\psline(7,2.73)(5.5,3.6)
\psline(5.5,3.6)(5.5,1.87)
\psline(7,2.73)(8.5,3.6)
\psline(8.5,3.6)(8.5,1.87)
\psline[ArrowInside=->, linecolor=zzzzzz](-1.5,0.13)(0,1)
\psline[ArrowInside=->, linecolor=zzzzzz](1.5,0.13)(0,1)
\psline[ArrowInside=->, linecolor=zzzzzz](1.5,1.87)(1.5,0.13)
\psline[ArrowInside=->, linecolor=zzzzzz](0,1)(1.5,1.87)
\psline[ArrowInside=->, linecolor=zzzzzz](1.5,1.87)(0,2.73)
\psline[ArrowInside=->, linecolor=zzzzzz](0,2.73)(0,1)
\psline[ArrowInside=->, linecolor=zzzzzz](0,1)(-1.5,1.87)
\psline[ArrowInside=->, linecolor=zzzzzz](-1.5,1.87)(-1.5,0.13)
\psline[ArrowInside=->, linecolor=zzzzzz](-1.5,1.87)(0,2.73)
\psline[ArrowInside=->, linecolor=zzzzzz](0,2.73)(-1.5,3.6)
\psline[ArrowInside=->, linecolor=zzzzzz](-1.5,3.6)(-1.5,1.87)
\psline[ArrowInside=->, linecolor=zzzzzz](0,2.73)(1.5,3.6)
\psline[ArrowInside=->, linecolor=zzzzzz](1.5,3.6)(1.5,1.87)
\psline[linecolor=zzzzzz](5,1)(6,1)
\psline[linecolor=zzzzzz](6,1)(6.5,1.87)
\psline[linecolor=zzzzzz](6.5,1.87)(6,2.73)
\psline[linecolor=zzzzzz](6,2.73)(5,2.73)
\psline[linecolor=zzzzzz](6,2.73)(6.5,1.87)
\psline[linecolor=zzzzzz](6.5,1.87)(7.5,1.87)
\psline[linecolor=zzzzzz](7.5,1.87)(8,2.73)
\psline[linecolor=zzzzzz](8,2.73)(7.5,3.6)
\psline[linecolor=zzzzzz](7.5,3.6)(6.5,3.6)
\psline[linecolor=zzzzzz](6.5,3.6)(6,2.73)
\psline[linecolor=zzzzzz](7.5,1.87)(6.5,1.87)
\psline[linecolor=zzzzzz](6.5,1.87)(6,1)
\psline[linecolor=zzzzzz](6,1)(6.5,0.13)
\psline[linecolor=zzzzzz](6.5,0.13)(7.5,0.13)
\psline[linecolor=zzzzzz](7.5,0.13)(8,1)
\psline[linecolor=zzzzzz](8,1)(7.5,1.87)
\psline[linecolor=zzzzzz](7.5,1.87)(8,1)
\psline[linecolor=zzzzzz](8,1)(9,1)
\psline[linecolor=zzzzzz](9,2.73)(8,2.73)
\psline[linecolor=zzzzzz](8,2.73)(7.5,1.87)
\psline[ArrowInside=->](5.5,0.13)(7,1)
\psline[ArrowInside=->](8.5,0.13)(7,1)
\psline[ArrowInside=->](8.5,1.87)(8.5,0.13)
\psline[ArrowInside=->](7,1)(8.5,1.87)
\psline[ArrowInside=->](8.5,1.87)(7,2.73)
\psline[ArrowInside=->](7,2.73)(7,1)
\psline[ArrowInside=->](7,1)(5.5,1.87)
\psline[ArrowInside=->](5.5,1.87)(5.5,0.13)
\psline[ArrowInside=->](5.5,1.87)(7,2.73)
\psline[ArrowInside=->](7,2.73)(5.5,3.6)
\psline[ArrowInside=->](5.5,3.6)(5.5,1.87)
\psline[ArrowInside=->](7,2.73)(8.5,3.6)
\psline[ArrowInside=->](8.5,3.6)(8.5,1.87)
\psline[linestyle=dashed,dash=3pt 3pt](-0.26,2.29)(-0.26,0.56)
\psline[linestyle=dashed,dash=3pt 3pt](-0.26,0.56)(1.24,1.43)
\psline[linestyle=dashed,dash=3pt 3pt](1.24,1.43)(1.24,3.16)
\psline[linestyle=dashed,dash=3pt 3pt](1.24,3.16)(-0.26,2.29)
\psline[linestyle=dashed,dash=3pt 3pt](6.74,2.29)(6.74,0.56)
\psline[linestyle=dashed,dash=3pt 3pt](6.74,0.56)(8.24,1.43)
\psline[linestyle=dashed,dash=3pt 3pt](8.24,1.43)(8.24,3.16)
\psline[linestyle=dashed,dash=3pt 3pt](8.24,3.16)(6.74,2.29)
\psdots[dotstyle=o](-2,1)
\psdots[dotstyle=o](-0.5,1.87)
\psdots(-1,2.73)
\psdots[dotstyle=o](-2,2.73)
\psdots(0.5,1.87)
\psdots[dotstyle=o](1,2.73)
\psdots(0.5,3.6)
\psdots[dotstyle=o](-0.5,3.6)
\psdots(-1,1)
\psdots[dotstyle=o](-0.5,0.13)
\psdots(0.5,0.13)
\psdots[dotstyle=o](1,1)
\psdots(2,1)
\psdots(2,2.73)
\psdots[dotstyle=o](1,2.73)
\psdots[dotsize=2pt 0,linecolor=zzzzzz](-1.5,1.87)
\psdots[dotsize=2pt 0,linecolor=zzzzzz](0,1)
\psdots[dotsize=2pt 0](5.5,1.87)
\psdots[dotsize=2pt 0](7,1)
\psdots[dotsize=2pt 0](7,2.73)
\psdots[dotsize=2pt 0](8.5,1.87)
\psdots[dotsize=2pt 0](5.5,3.6)
\psdots[dotsize=2pt 0](5.5,0.13)
\psdots[dotsize=2pt 0](8.5,0.13)
\psdots[dotsize=2pt 0](8.5,3.6)
\psdots[dotsize=2pt 0,linecolor=zzzzzz](-1.5,0.13)
\psdots[dotsize=2pt 0,linecolor=zzzzzz](0,1)
\psdots[dotsize=2pt 0,linecolor=zzzzzz](0,1)
\psdots[dotsize=2pt 0,linecolor=zzzzzz](1.5,0.13)
\psdots[dotsize=2pt 0,linecolor=zzzzzz](1.5,1.87)
\psdots[dotsize=2pt 0,linecolor=zzzzzz](1.5,0.13)
\psdots[dotsize=2pt 0,linecolor=zzzzzz](0,1)
\psdots[dotsize=2pt 0,linecolor=zzzzzz](1.5,1.87)
\psdots[dotsize=2pt 0,linecolor=zzzzzz](1.5,1.87)
\psdots[dotsize=2pt 0,linecolor=zzzzzz](0,2.73)
\psdots[dotsize=2pt 0,linecolor=zzzzzz](0,2.73)
\psdots[dotsize=2pt 0,linecolor=zzzzzz](0,1)
\psdots[dotsize=2pt 0,linecolor=zzzzzz](0,1)
\psdots[dotsize=2pt 0,linecolor=zzzzzz](-1.5,1.87)
\psdots[dotsize=2pt 0,linecolor=zzzzzz](-1.5,1.87)
\psdots[dotsize=2pt 0,linecolor=zzzzzz](-1.5,0.13)
\psdots[dotsize=2pt 0,linecolor=zzzzzz](-1.5,1.87)
\psdots[dotsize=2pt 0,linecolor=zzzzzz](0,2.73)
\psdots[dotsize=2pt 0,linecolor=zzzzzz](0,2.73)
\psdots[dotsize=2pt 0,linecolor=zzzzzz](-1.5,3.6)
\psdots[dotsize=2pt 0,linecolor=zzzzzz](-1.5,3.6)
\psdots[dotsize=2pt 0,linecolor=zzzzzz](-1.5,1.87)
\psdots[dotsize=2pt 0,linecolor=zzzzzz](0,2.73)
\psdots[dotsize=2pt 0,linecolor=zzzzzz](1.5,3.6)
\psdots[dotsize=2pt 0,linecolor=zzzzzz](1.5,3.6)
\psdots[dotsize=2pt 0,linecolor=zzzzzz](1.5,1.87)
\psdots[dotstyle=o,linecolor=zzzzzz](5,1)
\psdots[linecolor=zzzzzz](6,1)
\psdots[dotstyle=o,linecolor=zzzzzz](6.5,1.87)
\psdots[linecolor=zzzzzz](6,2.73)
\psdots[dotstyle=o,linecolor=zzzzzz](5,2.73)
\psdots[linecolor=zzzzzz](6,2.73)
\psdots[dotstyle=o,linecolor=zzzzzz](6.5,1.87)
\psdots[linecolor=zzzzzz](7.5,1.87)
\psdots[dotstyle=o,linecolor=zzzzzz](8,2.73)
\psdots[linecolor=zzzzzz](7.5,3.6)
\psdots[dotstyle=o,linecolor=zzzzzz](6.5,3.6)
\psdots[linecolor=zzzzzz](7.5,1.87)
\psdots[dotstyle=o,linecolor=zzzzzz](6.5,1.87)
\psdots[linecolor=zzzzzz](6,1)
\psdots[dotstyle=o,linecolor=zzzzzz](6.5,0.13)
\psdots[linecolor=zzzzzz](7.5,0.13)
\psdots[dotstyle=o,linecolor=zzzzzz](8,1)
\psdots[linecolor=zzzzzz](7.5,1.87)
\psdots[dotstyle=o,linecolor=zzzzzz](8,1)
\psdots[linecolor=zzzzzz](9,1)
\psdots[linecolor=zzzzzz](9,2.73)
\psdots[dotstyle=o,linecolor=zzzzzz](8,2.73)
\psdots[dotsize=2pt 0](5.5,0.13)
\psdots[dotsize=2pt 0](7,1)
\psdots[dotsize=2pt 0](7,1)
\psdots[dotsize=2pt 0](8.5,0.13)
\psdots[dotsize=2pt 0](8.5,1.87)
\psdots[dotsize=2pt 0](8.5,0.13)
\psdots[dotsize=2pt 0](7,1)
\psdots[dotsize=2pt 0](8.5,1.87)
\psdots[dotsize=2pt 0](8.5,1.87)
\psdots[dotsize=2pt 0](7,2.73)
\psdots[dotsize=2pt 0](5.5,1.87)
\psdots[dotsize=2pt 0](5.5,0.13)
\psdots[dotsize=2pt 0](5.5,1.87)
\psdots[dotsize=2pt 0](7,2.73)
\psdots[dotsize=2pt 0](7,2.73)
\psdots[dotsize=2pt 0](5.5,3.6)
\psdots[dotsize=2pt 0](5.5,3.6)
\psdots[dotsize=2pt 0](5.5,1.87)
\psdots[dotsize=2pt 0](7,2.73)
\psdots[dotsize=2pt 0](8.5,3.6)
\psdots[dotsize=2pt 0](8.5,3.6)
\psdots[dotsize=2pt 0](8.5,1.87)
\rput[bl](7.6,1.5){$a$}
\rput[bl](7.76,2.38){$b$}
\rput[bl](7.1,1.76){$c$}
\rput[bl](7.6,3.23){$a$}
\rput[bl](8.58,2.62){$c$}
\end{pspicture*}
\end{center}
Observe that the bipartite tiling has one face, three edges and two vertices, and dually, the quiver has one vertex, three arrows and two faces. Since the quiver has one vertex, this is the head and tail of each of the arrows. Therefore the path algebra $\C Q$ is actually the free algebra $\C \langle a, b, c \rangle$ with three generators corresponding to the three arrows. We see that the cyclic elements corresponding to the boundaries of the black and white faces, are $(abc)$ and $(acb)$ respectively. Thus the superpotential is:
$$
W = (abc) - (acb)
$$  and differentiating this with respect to each arrow, we obtain the three relations:
$$\begin{matrix}
     \frac{\del W}{\del c}=&ab-ba &=0 \cr
      \frac{\del W}{\del b}=&ca-ac &=0 \cr
     \frac{\del W}{\del a}=& bc-cb &=0 \cr
\end{matrix}
 $$
Therefore, the ideal $I_W$ is generated by the commutation relations between all the generators of $\C Q$, and the algebra
$$A= \C \langle a, b, c \rangle/ I_W = \C [a,b,c]$$
is the polynomial ring in three variables.
\end{example}

\begin{example} \label{conifold}
The tiling of the torus by squares has two faces and is therefore slightly more complicated. We show the bipartite tiling and the quiver below.
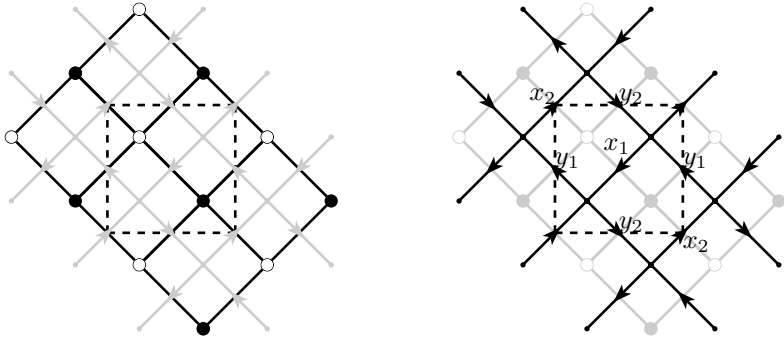
\begin{center}
\newrgbcolor{zzzzzz}{0.8 0.8 0.8}
\newrgbcolor{zzttqq}{0.8 0.8 0.8}
\psset{xunit=0.85cm,yunit=0.85cm,algebraic=true,dotstyle=*,dotsize=5pt 0,linewidth=1.0pt,arrowsize=3pt 2,arrowinset=0.25}
\begin{pspicture*}(-4,0)(10,7)
\psline(-2,3)(-1,2)
\psline(-1,2)(0,3)
\psline(0,3)(-1,4)
\psline(-1,4)(-2,3)
\psline(0,3)(-1,2)
\psline(-1,2)(0,1)
\psline(0,1)(1,2)
\psline[linecolor=zzzzzz](1,2)(0,3)
\psline(-1,4)(0,3)
\psline(0,3)(1,4)
\psline(1,4)(0,5)
\psline(0,5)(-1,4)
\psline(0,3)(1,2)
\psline(1,2)(2,3)
\psline(2,3)(1,4)
\psline(1,4)(0,3)
\psline(-2,3)(-1,4)
\psline(-1,4)(-2,5)
\psline(-2,5)(-3,4)
\psline(-3,4)(-2,3)
\psline(-1,4)(0,5)
\psline(0,5)(-1,6)
\psline(-1,6)(-2,5)
\psline(-2,5)(-1,4)
\psline[linestyle=dashed,dash=3pt 3pt](-1.5,4.5)(0.5,4.5)
\psline[linestyle=dashed,dash=3pt 3pt](0.5,4.5)(0.5,2.5)
\psline[linestyle=dashed,dash=3pt 3pt](0.5,2.5)(-1.5,2.5)
\psline[linestyle=dashed,dash=3pt 3pt](-1.5,2.5)(-1.5,4.5)
\psline[ArrowInside=->, linecolor=zzzzzz](-3,5)(-2,4)
\psline[ArrowInside=->, linecolor=zzzzzz](-2,4)(-1,5)
\psline[ArrowInside=->, linecolor=zzzzzz](-1,5)(0,4)
\psline[ArrowInside=->, linecolor=zzzzzz](0,4)(-1,3)
\psline[ArrowInside=->, linecolor=zzzzzz](-1,3)(-2,4)
\psline[ArrowInside=->, linecolor=zzzzzz](-1,5)(-2,6)
\psline[ArrowInside=->, linecolor=zzzzzz](0,6)(-1,5)
\psline[ArrowInside=->, linecolor=zzzzzz](0,4)(1,5)
\psline[ArrowInside=->, linecolor=zzzzzz](1,3)(0,4)
\psline[ArrowInside=->, linecolor=zzzzzz](2,4)(1,3)
\psline[ArrowInside=->, linecolor=zzzzzz](1,3)(2,2)
\psline[ArrowInside=->, linecolor=zzzzzz](0,2)(1,3)
\psline[ArrowInside=->, linecolor=zzzzzz](1,1)(0,2)
\psline[ArrowInside=->, linecolor=zzzzzz](0,2)(-1,1)
\psline[ArrowInside=->, linecolor=zzzzzz](-2,2)(-1,3)
\psline[ArrowInside=->, linecolor=zzzzzz](-1,3)(0,2)
\psline[ArrowInside=->, linecolor=zzzzzz](-2,4)(-3,3)
\psline[linecolor=zzttqq](5,3)(6,2)
\psline[linecolor=zzttqq](6,2)(7,3)
\psline[linecolor=zzttqq](7,3)(6,4)
\psline[linecolor=zzttqq](6,4)(5,3)
\psline[linecolor=zzttqq](7,3)(6,2)
\psline[linecolor=zzttqq](6,2)(7,1)
\psline[linecolor=zzttqq](7,1)(8,2)
\psline[linecolor=zzttqq](8,2)(7,3)
\psline[linecolor=zzttqq](6,4)(7,3)
\psline[linecolor=zzttqq](7,3)(8,4)
\psline[linecolor=zzttqq](8,4)(7,5)
\psline[linecolor=zzttqq](7,5)(6,4)
\psline[linecolor=zzttqq](7,3)(8,2)
\psline[linecolor=zzttqq](8,2)(9,3)
\psline[linecolor=zzttqq](9,3)(8,4)
\psline[linecolor=zzttqq](8,4)(7,3)
\psline[linecolor=zzttqq](5,3)(6,4)
\psline[linecolor=zzttqq](6,4)(5,5)
\psline[linecolor=zzttqq](5,5)(4,4)
\psline[linecolor=zzttqq](4,4)(5,3)
\psline[linecolor=zzttqq](6,4)(7,5)
\psline[linecolor=zzttqq](7,5)(6,6)
\psline[linecolor=zzttqq](6,6)(5,5)
\psline[linecolor=zzttqq](5,5)(6,4)
\psline[linestyle=dashed,dash=3pt 3pt](5.5,4.5)(7.5,4.5)
\psline[linestyle=dashed,dash=3pt 3pt](7.5,4.5)(7.5,2.5)
\psline[linestyle=dashed,dash=3pt 3pt](7.5,2.5)(5.5,2.5)
\psline[linestyle=dashed,dash=3pt 3pt](5.5,2.5)(5.5,4.5)
\psline[ArrowInside=->](4,5)(5,4)
\psline[ArrowInside=->](5,4)(6,5)
\psline[ArrowInside=->](6,5)(7,4)
\psline[ArrowInside=->](7,4)(6,3)
\psline[ArrowInside=->](6,3)(5,4)
\psline[ArrowInside=->](6,5)(5,6)
\psline[ArrowInside=->](7,6)(6,5)
\psline[ArrowInside=->](7,4)(8,5)
\psline[ArrowInside=->](8,3)(7,4)
\psline[ArrowInside=->](9,4)(8,3)
\psline[ArrowInside=->](8,3)(9,2)
\psline[ArrowInside=->](7,2)(8,3)
\psline[ArrowInside=->](8,1)(7,2)
\psline[ArrowInside=->](7,2)(6,1)
\psline[ArrowInside=->](5,2)(6,3)
\psline[ArrowInside=->](6,3)(7,2)
\psline[ArrowInside=->](5,4)(4,3)
\psdots[dotstyle=o](-1,2)
\psdots(-2,3)
\psdots(0,3)
\psdots[dotstyle=o](-1,4)
\psdots(0,1)
\psdots(0,1)
\psdots[dotstyle=o](1,2)
\psdots[dotstyle=o](1,4)
\psdots(0,5)
\psdots(2,3)
\psdots[dotstyle=o](1,4)
\psdots[dotstyle=o](-3,4)
\psdots(-2,5)
\psdots[dotstyle=o](-3,4)
\psdots[dotstyle=o](-1,6)
\psdots(-2,5)
\psdots[dotsize=2pt 0,linecolor=zzzzzz](-1,5)
\psdots[dotsize=2pt 0,linecolor=zzzzzz](-2,4)
\psdots[dotsize=2pt 0,linecolor=zzzzzz](0,4)
\psdots[dotsize=2pt 0,linecolor=zzzzzz](-1,3)
\psdots[dotsize=2pt 0,linecolor=zzzzzz](0,2)
\psdots[dotsize=2pt 0,linecolor=zzzzzz](1,3)
\psdots[dotsize=2pt 0,linecolor=zzzzzz](-3,5)
\psdots[dotsize=2pt 0,linecolor=zzzzzz](-2,6)
\psdots[dotsize=2pt 0,linecolor=zzzzzz](0,6)
\psdots[dotsize=2pt 0,linecolor=zzzzzz](1,5)
\psdots[dotsize=2pt 0,linecolor=zzzzzz](2,4)
\psdots[dotsize=2pt 0,linecolor=zzzzzz](2,2)
\psdots[dotsize=2pt 0,linecolor=zzzzzz](1,1)
\psdots[dotsize=2pt 0,linecolor=zzzzzz](-1,1)
\psdots[dotsize=2pt 0,linecolor=zzzzzz](-2,2)
\psdots[dotsize=2pt 0,linecolor=zzzzzz](-3,3)
\psdots[dotsize=2pt 0,linecolor=zzzzzz](4,5)
\psdots[linecolor=zzzzzz](5,3)
\psdots[dotstyle=o,linecolor=zzzzzz](6,2)
\psdots[linecolor=zzzzzz](7,3)
\psdots[dotstyle=o,linecolor=zzzzzz](6,4)
\psdots[linecolor=zzzzzz](7,1)
\psdots[linecolor=zzzzzz](7,3)
\psdots[dotstyle=o,linecolor=zzzzzz](6,2)
\psdots[linecolor=zzzzzz](7,1)
\psdots[dotstyle=o,linecolor=zzzzzz](8,2)
\psdots[dotstyle=o,linecolor=zzzzzz](6,4)
\psdots[linecolor=zzzzzz](7,3)
\psdots[dotstyle=o,linecolor=zzzzzz](8,4)
\psdots[linecolor=zzzzzz](7,5)
\psdots[linecolor=zzzzzz](7,3)
\psdots[dotstyle=o,linecolor=zzzzzz](8,2)
\psdots[linecolor=zzzzzz](9,3)
\psdots[dotstyle=o,linecolor=zzzzzz](8,4)
\psdots[dotstyle=o,linecolor=zzzzzz](4,4)
\psdots[linecolor=zzzzzz](5,3)
\psdots[dotstyle=o,linecolor=zzzzzz](6,4)
\psdots[linecolor=zzzzzz](5,5)
\psdots[dotstyle=o,linecolor=zzzzzz](4,4)
\psdots[dotstyle=o,linecolor=zzzzzz](6,4)
\psdots[linecolor=zzzzzz](7,5)
\psdots[dotstyle=o,linecolor=zzzzzz](6,6)
\psdots[linecolor=zzzzzz](5,5)
\psdots[dotsize=2pt 0](4,5)
\psdots[dotsize=2pt 0](5,4)
\psdots[dotsize=2pt 0](5,4)
\psdots[dotsize=2pt 0](6,5)
\psdots[dotsize=2pt 0](6,5)
\psdots[dotsize=2pt 0](7,4)
\rput[bl](6.5,4.5){$y_2$}
\psdots[dotsize=2pt 0](7,4)
\psdots[dotsize=2pt 0](6,3)
\rput[bl](6.26,3.74){$x_1$}
\psdots[dotsize=2pt 0](6,3)
\psdots[dotsize=2pt 0](5,4)
\rput[bl](5.5,3.5){$y_1$}
\psdots[dotsize=2pt 0](6,5)
\psdots[dotsize=2pt 0](5,6)
\psdots[dotsize=2pt 0](7,6)
\psdots[dotsize=2pt 0](6,5)
\psdots[dotsize=2pt 0](7,4)
\psdots[dotsize=2pt 0](8,5)
\psdots[dotsize=2pt 0](8,3)
\psdots[dotsize=2pt 0](7,4)
\rput[bl](7.5,3.5){$y_1$}
\psdots[dotsize=2pt 0](9,4)
\psdots[dotsize=2pt 0](8,3)
\psdots[dotsize=2pt 0](8,3)
\psdots[dotsize=2pt 0](9,2)
\psdots[dotsize=2pt 0](7,2)
\psdots[dotsize=2pt 0](8,3)
\psdots[dotsize=2pt 0](8,1)
\psdots[dotsize=2pt 0](7,2)
\psdots[dotsize=2pt 0](7,2)
\psdots[dotsize=2pt 0](6,1)
\psdots[dotsize=2pt 0](5,2)
\psdots[dotsize=2pt 0](6,3)
\psdots[dotsize=2pt 0](6,3)
\psdots[dotsize=2pt 0](7,2)
\rput[bl](6.5,2.5){$y_2$}
\psdots[dotsize=2pt 0](5,4)
\psdots[dotsize=2pt 0](4,3)
\rput[bl](7.5,2.2){$x_2$}
\rput[bl](5.1,4.52){$x_2$}
\end{pspicture*}
\end{center}
Observe that the quiver has two vertices, with two arrows in each direction between them:
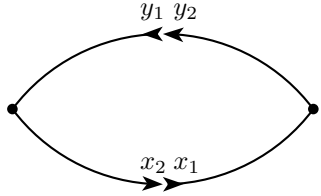
\begin{center}
\psset{xunit=1.0cm,yunit=1.0cm}
\begin{pspicture*}(0,-1.5)(6,1.5)
\psset{xunit=1.0cm,yunit=1.0cm,algebraic=true,dotstyle=*,dotsize=4pt 0,linewidth=0.8pt,arrowsize=4pt 2,arrowinset=0.25}
\psarc{>-}(3,-1.5){2.5}{90}{143.13}
\psarc{->}(3,-1.5){2.5}{36.87}{90}
\psarc{>-}(3,1.5){2.5}{270}{323.13}
\psarc{->}(3,1.5){2.5}{216.87}{270}
\psdots(1,0)
\psdots(5,0)
\rput[bl](2.7,1.18){$y_1$}
\rput[bl](3.14,1.18){$y_2$}
\rput[bl](3.14,-0.85){$x_1$}
\rput[bl](2.7,-0.85){$x_2$}
\end{pspicture*}
\end{center}
 Thus, the path algebra $\C Q$ is generated by two idempotents and elements corresponding to the four arrows. As in the previous example, the quiver has two faces so the superpotential has two terms. The cyclic elements corresponding to the boundaries of the black and white faces, are $(x_1 y_2 x_2 y_1)$ and $(x_1 y_1 x_2 y_2)$ respectively, and so the superpotential is:
$$
W = (x_1 y_2 x_2 y_1) - (x_1 y_1 x_2 y_2)
$$
Applying the cyclic derivative with respect to each arrow, we obtain four relations:
$$\begin{matrix}
     \frac{\del W}{\del x_1}=& y_2 x_2 y_1 - y_1 x_2 y_2 &=0 \cr
      \frac{\del W}{\del x_2}=&y_1 x_1 y_2 - y_2 x_1 y_1 &=0 \cr
     \frac{\del W}{\del y_1}=& x_1 y_2 x_2 - x_2 y_2 x_1 &=0 \cr
      \frac{\del W}{\del y_2}=&x_2 y_1 x_1 - x_1 x_1 x_2 &=0 \cr
\end{matrix}
 $$
We note that unlike in the hexagonal tiling example above, the algebra $A= \C Q / I_W$ is a non-commutative algebra. We will see later that its centre is isomorphic to the coordinate ring of the threefold ordinary double point, or conifold singularity. It is known that $A$ is a non-commutative crepant resolution of this singularity (see Proposition~7.3 of \cite{VdB}, \cite{Balazs}), and for this reason the example is usually referred to as the `non-commutative conifold' in the physics literature \cite{Kennaway, HananyKen, HananyVegh}.
\end{example}

\begin{example} \label{nonmin}
We now give an example which has the same algebra $A$, as the conifold example above.

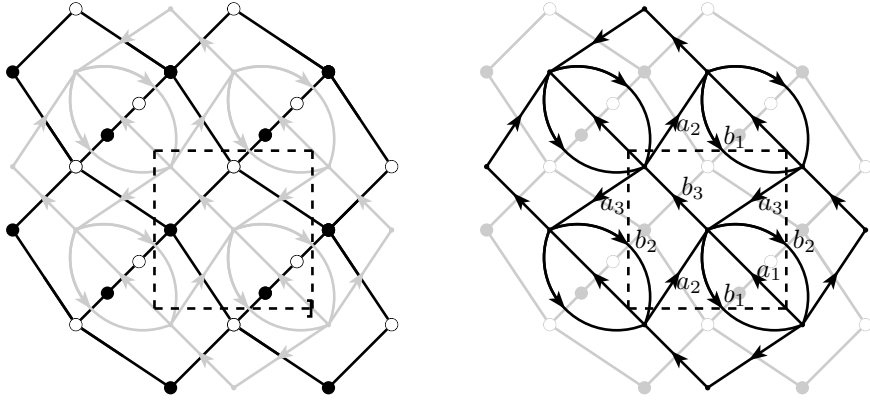
\begin{center}
\newrgbcolor{zzqqtt}{0 0 0}
\newrgbcolor{qqcccc}{0.8 0.8 0.8}
\newrgbcolor{qqccqq}{0.8 0.8 0.8}
\psset{xunit=0.42cm,yunit=0.42cm,runit=0.42cm,algebraic=true,dotstyle=*,dotsize=2pt 0,linewidth=1.0pt,arrowsize=3pt 2,arrowinset=0.25}
\begin{pspicture*}(-0.5,-3)(27.5,11)
\psline[linecolor=zzqqtt](2,0)(3,1)
\psline[linecolor=zzqqtt](3,1)(4,2)
\psline[linecolor=zzqqtt](4,2)(5,3)
\psline[linecolor=zzqqtt](5,3)(6,4)
\psline[linecolor=zzqqtt](2,0)(1,1.5)
\psline[linecolor=zzqqtt](2,0)(3.5,-1)
\psline[linecolor=zzqqtt](3.5,4)(5,3)
\psline[linecolor=zzqqtt](5,3)(6,1.5)
\psline[linecolor=zzqqtt](7,0)(8,1)
\psline[linecolor=zzqqtt](8,1)(9,2)
\psline[linecolor=zzqqtt](9,2)(10,3)
\psline[linecolor=zzqqtt](10,3)(11,4)
\psline[linecolor=zzqqtt](7,0)(6,1.5)
\psline[linecolor=zzqqtt](7,0)(8.5,-1)
\psline[linecolor=zzqqtt](8.5,4)(10,3)
\psline[linecolor=zzqqtt](10,3)(11,1.5)
\psline[linecolor=zzqqtt](2,5)(3,6)
\psline[linecolor=zzqqtt](3,6)(4,7)
\psline[linecolor=zzqqtt](4,7)(5,8)
\psline[linecolor=zzqqtt](2,5)(1,6.5)
\psline[linecolor=zzqqtt](2,5)(3.5,4)
\psline[linecolor=zzqqtt](3.5,9)(5,8)
\psline[linecolor=zzqqtt](5,8)(6,6.5)
\psline[linecolor=zzqqtt](6,4)(7,5)
\psline[linecolor=zzqqtt](7,5)(8,6)
\psline[linecolor=zzqqtt](8,6)(9,7)
\psline[linecolor=zzqqtt](9,7)(10,8)
\psline[linecolor=zzqqtt](7,5)(6,6.5)
\psline[linecolor=zzqqtt](7,5)(8.5,4)
\psline[linecolor=zzqqtt](8.5,9)(10,8)
\psline[linecolor=zzqqtt](10,8)(11,6.5)
\psline[linecolor=zzqqtt](2,5)(0,8)
\psline[linecolor=zzqqtt](0,8)(2,10)
\psline[linecolor=zzqqtt](2,10)(5,8)
\psline[linecolor=zzqqtt](5,8)(7,10)
\psline[linecolor=zzqqtt](7,10)(10,8)
\psline[linecolor=zzqqtt](2,0)(0,3)
\psline[linecolor=zzqqtt](0,3)(2,5)
\psline[linecolor=zzqqtt](10,-2)(7,0)
\psline[linecolor=zzqqtt](10,-2)(12,0)
\psline[linecolor=zzqqtt](12,0)(10,3)
\psline[linecolor=zzqqtt](12,5)(10,8)
\psline[linecolor=zzqqtt](12,5)(10,3)
\psline[linecolor=zzqqtt](5,-2)(7,0)
\psline[linecolor=zzqqtt](5,-2)(2,0)
\psline[ArrowInside=->, linecolor=qqccqq](5,5)(2,8)
\psline[ArrowInside=->, linecolor=qqccqq](5,5)(7,8)
\psline[ArrowInside=->, linecolor=qqccqq](10,5)(7,8)
\psline[ArrowInside=->, linecolor=qqccqq](7,3)(5,5)
\psline[ArrowInside=->, linecolor=qqccqq](10,5)(7,3)
\psline[ArrowInside=->, linecolor=qqccqq](10,0)(7,3)
\psline[ArrowInside=->, linecolor=qqccqq](5,0)(7,3)
\psline[ArrowInside=->, linecolor=qqccqq](5,0)(2,3)
\psline[ArrowInside=->, linecolor=qqccqq](5,5)(2,3)
\psline[ArrowInside=->, linecolor=qqccqq](2,3)(0,5)
\psline[ArrowInside=->, linecolor=qqccqq](7,8)(5,10)
\psline[ArrowInside=->, linecolor=qqccqq](7,-2)(5,0)
\psline[ArrowInside=->, linecolor=qqccqq](10,0)(7,-2)
\psline[ArrowInside=->, linecolor=qqccqq](12,3)(10,5)
\psline[ArrowInside=->, linecolor=qqccqq](10,0)(12,3)
\psline[ArrowInside=->, linecolor=qqccqq](5,10)(2,8)
\psline[ArrowInside=->, linecolor=qqccqq](0,5)(2,8)
\psarc[ArrowInside=->, linecolor=qqccqq](2.88,5.88){2.3}{337.62}{112.38}
\psarc[ArrowInside=->, linecolor=qqccqq](4.12,7.12){2.3}{157.62}{292.38}
\psarc[ArrowInside=->, linecolor=qqccqq](7.88,5.88){2.3}{337.62}{112.38}
\psarc[ArrowInside=->, linecolor=qqccqq](9.12,7.12){2.3}{157.62}{292.38}
\psarc[ArrowInside=->, linecolor=qqccqq](2.88,0.88){2.3}{337.62}{112.38}
\psarc[ArrowInside=->, linecolor=qqccqq](4.12,2.12){2.3}{157.62}{292.38}
\psarc[ArrowInside=->, linecolor=qqccqq](7.88,0.88){2.3}{337.62}{112.38}
\psarc[ArrowInside=->, linecolor=qqccqq](9.12,2.12){2.3}{157.62}{292.38}
\psarc[ArrowInside=->, linecolor=qqccqq]{<-}(2.88,5.88){2.3}{45}{112.38}
\psarc[ArrowInside=->, linecolor=qqccqq]{->}(4.12,7.12){2.3}{157.62}{225}
\psarc[ArrowInside=->, linecolor=qqccqq]{<-}(7.88,5.88){2.3}{45}{112.38}
\psarc[ArrowInside=->, linecolor=qqccqq]{->}(9.12,7.12){2.3}{157.62}{225}
\psarc[ArrowInside=->, linecolor=qqccqq]{<-}(2.88,0.88){2.3}{45}{112.38}
\psarc[ArrowInside=->, linecolor=qqccqq]{->}(4.12,2.12){2.3}{157.62}{225}
\psarc[ArrowInside=->, linecolor=qqccqq]{<-}(7.88,0.88){2.3}{45}{112.38}
\psarc[ArrowInside=->, linecolor=qqccqq]{->}(9.12,2.12){2.3}{157.62}{225}
\psline[linecolor=qqcccc](17,0)(18,1)
\psline[linecolor=qqcccc](18,1)(19,2)
\psline[linecolor=qqcccc](19,2)(20,3)
\psline[linecolor=qqcccc](20,3)(21,4)
\psline[linecolor=qqcccc](18.5,4)(20,3)
\psline[linecolor=qqcccc](20,3)(21,1.5)
\psline[linecolor=qqcccc](22,0)(23,1)
\psline[linecolor=qqcccc](23,1)(24,2)
\psline[linecolor=qqcccc](24,2)(25,3)
\psline[linecolor=qqcccc](25,3)(26,4)
\psline[linecolor=qqcccc](22,0)(21,1.5)
\psline[linecolor=qqcccc](23.5,4)(25,3)
\psline[linecolor=qqcccc](25,3)(26,1.5)
\psline[linecolor=qqcccc](17,5)(18,6)
\psline[linecolor=qqcccc](18,6)(19,7)
\psline[linecolor=qqcccc](19,7)(20,8)
\psline[linecolor=qqcccc](17,5)(18.5,4)
\psline[linecolor=qqcccc](20,8)(21,6.5)
\psline[linecolor=qqcccc](21,4)(22,5)
\psline[linecolor=qqcccc](22,5)(23,6)
\psline[linecolor=qqcccc](23,6)(24,7)
\psline[linecolor=qqcccc](24,7)(25,8)
\psline[linecolor=qqcccc](22,5)(21,6.5)
\psline[linecolor=qqcccc](22,5)(23.5,4)
\psline[linecolor=qqcccc](17,5)(15,8)
\psline[linecolor=qqcccc](15,8)(17,10)
\psline[linecolor=qqcccc](17,10)(20,8)
\psline[linecolor=qqcccc](20,8)(22,10)
\psline[linecolor=qqcccc](22,10)(25,8)
\psline[linecolor=qqcccc](17,0)(15,3)
\psline[linecolor=qqcccc](15,3)(17,5)
\psline[linecolor=qqcccc](25,-2)(22,0)
\psline[linecolor=qqcccc](25,-2)(27,0)
\psline[linecolor=qqcccc](27,0)(25,3)
\psline[linecolor=qqcccc](27,5)(25,8)
\psline[linecolor=qqcccc](27,5)(25,3)
\psline[linecolor=qqcccc](20,-2)(22,0)
\psline[linecolor=qqcccc](20,-2)(17,0)
\psline[ArrowInside=->](20,5)(17,8)
\psline[ArrowInside=->](20,5)(22,8)
\psline[ArrowInside=->](25,5)(22,8)
\psline[ArrowInside=->](22,3)(20,5)
\psline[ArrowInside=->](25,5)(22,3)
\psline[ArrowInside=->](25,0)(22,3)
\psline[ArrowInside=->](20,0)(22,3)
\psline[ArrowInside=->](20,0)(17,3)
\psline[ArrowInside=->](20,5)(17,3)
\psline[ArrowInside=->](17,3)(15,5)
\psline[ArrowInside=->](22,8)(20,10)
\psline[ArrowInside=->](22,-2)(20,0)
\psline[ArrowInside=->](25,0)(22,-2)
\psline[ArrowInside=->](27,3)(25,5)
\psline[ArrowInside=->](25,0)(27,3)
\psline[ArrowInside=->](20,10)(17,8)
\psline[ArrowInside=->](15,5)(17,8)
\psarc[ArrowInside=->](17.88,5.88){2.3}{337.62}{112.38}
\psarc[ArrowInside=->](19.12,7.12){2.3}{157.62}{292.38}
\psarc[ArrowInside=->](22.88,5.88){2.3}{337.62}{112.38}
\psarc[ArrowInside=->](24.12,7.12){2.3}{157.62}{292.38}
\psarc[ArrowInside=->](17.88,0.88){2.3}{337.62}{112.38}
\psarc[ArrowInside=->](19.12,2.12){2.3}{157.62}{292.38}
\psarc[ArrowInside=->](22.88,0.88){2.3}{337.62}{112.38}
\psarc[ArrowInside=->](24.12,2.12){2.3}{157.62}{292.38}
\psarc[ArrowInside=->]{<-}(17.88,5.88){2.3}{45}{112.38}
\psarc[ArrowInside=->]{->}(19.12,7.12){2.3}{157.62}{225}
\psarc[ArrowInside=->]{<-}(22.88,5.88){2.3}{45}{112.38}
\psarc[ArrowInside=->]{->}(24.12,7.12){2.3}{157.62}{225}
\psarc[ArrowInside=->]{<-}(17.88,0.88){2.3}{45}{112.38}
\psarc[ArrowInside=->]{->}(19.12,2.12){2.3}{157.62}{225}
\psarc[ArrowInside=->]{<-}(22.88,0.88){2.3}{45}{112.38}
\psarc[ArrowInside=->]{->}(24.12,2.12){2.3}{157.62}{225}
\psline[linestyle=dashed,dash=3pt 3pt](4.49,5.51)(9.49,5.51)
\psline[linestyle=dashed,dash=3pt 3pt](9.49,5.51)(9.49,0.51)
\psline[linestyle=dashed,dash=3pt 3pt]](9.49,0.51)(4.49,0.51)
\psline[linestyle=dashed,dash=3pt 3pt](4.49,0.51)(4.49,5.51)
\psline[linestyle=dashed,dash=3pt 3pt](19.49,5.51)(24.49,5.51)
\psline[linestyle=dashed,dash=3pt 3pt](24.49,5.51)(24.49,0.51)
\psline[linestyle=dashed,dash=3pt 3pt](24.49,0.51)(19.49,0.51)
\psline[linestyle=dashed,dash=3pt 3pt](19.49,0.51)(19.49,5.51)
\psdots[dotsize=5pt 0](5,3)
\psdots[dotsize=5pt 0,dotstyle=o](2,0)
\psdots[dotsize=5pt 0](3,1)
\psdots[dotsize=5pt 0,dotstyle=o](4,2)
\psdots[dotsize=5pt 0,dotstyle=o](7,0)
\psdots[dotsize=5pt 0,dotstyle=o](7,0)
\psdots[dotsize=5pt 0,linecolor=black](8,1)
\psdots[dotsize=5pt 0,linecolor=black](8,1)
\psdots[dotsize=5pt 0,dotstyle=o](9,2)
\psdots[dotsize=5pt 0,dotstyle=o](9,2)
\psdots[dotsize=5pt 0,linecolor=black](10,3)
\psdots[dotsize=5pt 0,linecolor=black](10,3)
\psdots[dotsize=5pt 0,dotstyle=o](7,0)
\psdots[dotsize=5pt 0,dotstyle=o](7,0)
\psdots[dotsize=5pt 0,linecolor=black](10,3)
\psdots[dotsize=5pt 0,linecolor=black](10,3)
\psdots[dotsize=5pt 0,dotstyle=o](2,5)
\psdots[dotsize=5pt 0,dotstyle=o](2,5)
\psdots[dotsize=5pt 0,linecolor=black](3,6)
\psdots[dotsize=5pt 0,linecolor=black](3,6)
\psdots[dotsize=5pt 0,dotstyle=o](4,7)
\psdots[dotsize=5pt 0,dotstyle=o](4,7)
\psdots[dotsize=5pt 0,linecolor=black](5,8)
\psdots[dotsize=5pt 0,linecolor=black](5,8)
\psdots[dotsize=5pt 0,dotstyle=o](2,5)
\psdots[dotsize=5pt 0,dotstyle=o](2,5)
\psdots[dotsize=5pt 0,linecolor=black](5,8)
\psdots[dotsize=5pt 0,linecolor=black](5,8)
\psdots[dotsize=5pt 0,dotstyle=o](7,5)
\psdots[dotsize=5pt 0,dotstyle=o](7,5)
\psdots[dotsize=5pt 0,linecolor=black](8,6)
\psdots[dotsize=5pt 0,linecolor=black](8,6)
\psdots[dotsize=5pt 0,dotstyle=o](9,7)
\psdots[dotsize=5pt 0,dotstyle=o](9,7)
\psdots[dotsize=5pt 0,linecolor=black](10,8)
\psdots[dotsize=5pt 0,linecolor=black](10,8)
\psdots[dotsize=5pt 0,dotstyle=o](7,5)
\psdots[dotsize=5pt 0,dotstyle=o](7,5)
\psdots[dotsize=5pt 0,linecolor=black](10,8)
\psdots[dotsize=5pt 0,linecolor=black](10,8)
\psdots[dotsize=5pt 0](0,3)
\psdots[dotsize=5pt 0,dotstyle=o](2,10)
\psdots[dotsize=5pt 0](0,8)
\psdots[dotsize=5pt 0,dotstyle=o](7,10)
\psdots[dotsize=5pt 0,dotstyle=o](12,0)
\psdots[dotsize=5pt 0](10,-2)
\psdots[dotsize=5pt 0,dotstyle=o](12,5)
\psdots[dotsize=5pt 0,linecolor=black](10,8)
\psdots[dotsize=5pt 0](5,-2)
\psdots[linecolor=qqcccc](2,8)
\psdots[linecolor=qqcccc](5,5)
\psdots[linecolor=qqcccc](10,5)
\psdots[linecolor=qqcccc](7,3)
\psdots[linecolor=qqcccc](10,0)
\psdots[linecolor=qqcccc](5,0)
\psdots[linecolor=qqcccc](2,3)
\psdots[linecolor=qqcccc](7,8)
\psdots[linecolor=qqcccc](0,5)
\psdots[linecolor=qqcccc](5,10)
\psdots[linecolor=qqcccc](7,-2)
\psdots[linecolor=qqcccc](12,3)
\psdots[dotsize=5pt 0,linecolor=qqccqq](15,8)
\psdots[dotstyle=o,dotsize=5pt 0,linecolor=qqccqq](17,0)
\psdots[dotsize=5pt 0,linecolor=qqccqq](18,1)
\psdots[dotstyle=o,dotsize=5pt 0,linecolor=qqccqq](19,2)
\psdots[dotsize=5pt 0,linecolor=qqccqq](20,3)
\psdots[dotstyle=o,dotsize=5pt 0,linecolor=qqccqq](22,0)
\psdots[dotsize=5pt 0,linecolor=qqccqq](23,1)
\psdots[dotstyle=o,dotsize=5pt 0,linecolor=qqccqq](24,2)
\psdots[dotsize=5pt 0,linecolor=qqccqq](25,3)
\psdots[dotstyle=o,dotsize=5pt 0,linecolor=qqccqq](17,5)
\psdots[dotsize=5pt 0,linecolor=qqccqq](18,6)
\psdots[dotstyle=o,dotsize=5pt 0,linecolor=qqccqq](19,7)
\psdots[dotsize=5pt 0,linecolor=qqccqq](20,8)
\psdots[dotstyle=o,dotsize=5pt 0,linecolor=qqccqq](22,5)
\psdots[dotsize=5pt 0,linecolor=qqccqq](23,6)
\psdots[dotstyle=o,dotsize=5pt 0,linecolor=qqccqq](24,7)
\psdots[dotsize=5pt 0,linecolor=qqccqq](25,8)
\psdots[dotstyle=o,dotsize=5pt 0,linecolor=qqccqq](17,10)
\psdots[dotstyle=o,dotsize=5pt 0,linecolor=qqccqq](22,10)
\psdots[dotsize=5pt 0,linecolor=qqccqq](15,3)
\psdots[dotsize=5pt 0,linecolor=qqccqq](25,-2)
\psdots[dotstyle=o,dotsize=5pt 0,linecolor=qqccqq](27,0)
\psdots[dotstyle=o,dotsize=5pt 0,linecolor=qqccqq](27,5)
\psdots[dotsize=5pt 0,linecolor=qqccqq](20,-2)
\psdots[linecolor=black](20,5)
\psdots[linecolor=black](17,8)
\psdots[linecolor=black](22,8)
\psdots[linecolor=black](25,5)
\psdots[linecolor=black](22,3)
\psdots[linecolor=black](25,0)
\psdots[linecolor=black](20,0)
\psdots[linecolor=black](17,3)
\psdots[linecolor=black](15,5)
\psdots[linecolor=black](20,10)
\psdots[linecolor=black](22,-2)
\psdots[linecolor=black](27,3)
\rput[bl](22.5,5.6){$b_1$}
\rput[bl](19.7,2.3){$b_2$}
\rput[bl](24.7,2.3){$b_2$}
\rput[bl](22.5,0.6){$b_1$}
\rput[bl](21.0,6.0){$a_2$}
\rput[bl](21.17,3.98){$b_3$}
\rput[bl](23.6,3.5){$a_3$}
\rput[bl](23.55,1.4){$a_1$}
\rput[bl](21.0,1.0){$a_2$}
\rput[bl](18.6,3.5){$a_3$}
\end{pspicture*}
\end{center}

This quiver also has two vertices, there are three arrows in each direction between them. Thus, the path algebra $\C Q$ is generated by two idempotents and elements corresponding to the six arrows. The quiver has four faces so the superpotential has four terms:
$$
W = (a_3 b_1)-(a_3 b_3)+(b_3 a_2 b_2 a_1) - (b_1 a_1 b_2 a_2)
$$
Applying the cyclic derivative with respect to each arrow, we now obtain six relations:
$$\begin{matrix}
      \frac{\del W}{\del a_1}=& b_3 a_2 b_2 - b_2 a_2 b_1 &=0 \cr
      \frac{\del W}{\del a_2}=&b_2 a_1 b_3 - b_1 a_1 b_2 &=0 \cr
      \frac{\del W}{\del a_3}=& b_1-b_3 &=0 \cr
      \frac{\del W}{\del b_1}=&a_3 - a_1 b_2 a_2 &=0 \cr
      \frac{\del W}{\del b_2}=& a_1 b_3 a_2 - a_2 b_1 a_1 &=0 \cr
      \frac{\del W}{\del b_3}=&a_2 b_2 a_1 - a_3 &=0 \cr
\end{matrix}
 $$
We can use the relations $\frac{\del W}{\del a_3}$ and $\frac{\del W}{\del b_3}$ respectively, to write $b_3$ and $a_3$ in terms of the other generators, namely, $b_3 = b_1$ and $a_3 = a_2 b_2 a_1$. The path algebra subject to these two relations, is patently isomorphic to the path algebra of the quiver from the conifold example. Furthermore, substituting $b_3$ and $a_3$ into the relations, we obtain four relations:
$$\begin{matrix}
     \frac{\del W}{\del a_1}=& b_1 a_2 b_2 - b_2 a_2 b_1 &=0 \cr
      \frac{\del W}{\del a_2}=& b_2 a_1 b_1- b_1 a_1 b_2 &=0 \cr
     \frac{\del W}{\del b_1}=&  a_2 b_2 a_1- a_1 b_2 a_2 &=0 \cr
      \frac{\del W}{\del b_2}=& a_1 b_1 a_2 - a_2 b_1 a_1&=0 \cr
\end{matrix}
 $$
which are the relations in the conifold example. Therefore we have demonstrated a dimer model which outputs the same algebra $A$.
\end{example}

\subsection{Minimality} \label{minimality}
As we have just seen, it is possible for two distinct dimer models to have the same superpotential algebra $A$. In fact, if we choose any vertex of a dimer model, we can `split' this into two vertices of the same colour, connected together via a bivalent vertex of the other colour:

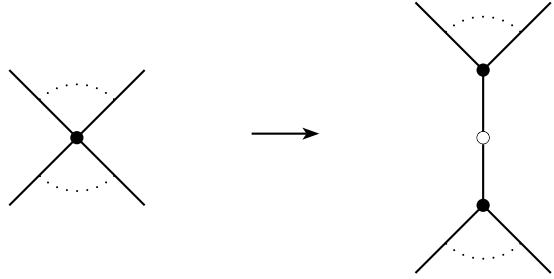
\begin{center}
\psset{xunit=0.9cm,yunit=0.9cm,algebraic=true,dotstyle=*,dotsize=5pt 0,linewidth=0.8pt,arrowsize=3pt 2,arrowinset=0.25}
\begin{pspicture*}(-3,-1)(7,5)
\psline(-1,2)(-2,1)
\psline(-1,2)(0,1)
\psarc[linestyle=dotted](-1,2){0.71}{225}{315}
\psline(5,3)(4,4)
\psline(6,4)(5,3)
\psline(4,0)(5,1)
\psline(5,1)(6,0)
\psarc[linestyle=dotted](5,3){0.71}{45}{135}
\psarc[linestyle=dotted](5,1){0.71}{225}{315}
\psline{->}(1.58,2.06)(2.58,2.06)
\psline(5,3)(5,2)
\psline(5,2)(5,1)
\psline(-1,2)(-2,3)
\psline(-1,2)(0,3)
\psarc[linestyle=dotted](-1,2){0.71}{45}{135}
\psdots(-1,2)
\psdots(5,3)
\psdots(5,1)
\psdots[dotstyle=o](5,2)
\end{pspicture*}
\end{center}

The resulting dimer model has two additional edges and so the quiver has two additional arrows. However the relations dual to these arrows equate the new arrows to paths which previously existed, and the resulting superpotential algebras for the two dimer models are in fact the same. We call a dimer model `non-minimal' if it can be obtained from a dimer model with fewer edges in this way. Note that Example~\ref{nonmin} is non-minimal as it can be obtained from Example~\ref{conifold}.

If a dimer model has a bivalent vertex which is connected to two distinct vertices, then it is possible to do the converse of the above process, i.e. remove the bivalent vertex and contract its two neighbours to a single vertex. 
For an example where it is not possible to remove the bivalent vertices, see (\ref{balwnopm}).

\section{Symmetries} \label{symmetries}

A \emph{global} symmetry is a one-parameter subgroup $\rho\colon \C^*\to \Aut(A)$
that arises from an action on the arrow fields
\[
\rho(t)\colon x_a \mapsto t^{v_a} x_a,
\]
for some $v\in \Z^{Q_1}$, which we may think of as a 1-cochain in the complex \eqref{eq:cochain}.
Then its coboundary $dv\in\Z^{Q_2}$ gives precisely the weights of the $\rho$-action on the terms in the superpotential $W$.

Thus, $\rho$ is a well-defined map to $\Aut(A)$ when it acts homogeneously on all terms in the superpotential $W$,
in other words, when
\begin{equation}
\label{eq:homog}
d v =  \lambda\const{1}
\end{equation}
for some constant $\lambda\in\Z$, which we will also call the \emph{degree} of $\rho$.

Using intentionally toric notation,
we shall write
\begin{equation}
\label{eq:Ndefs}
\Ng = d^{-1}(\Z\const{1})\subs \Z^{Q_1}
\quad\text{and}\quad
\Ng^+ = \Ng \cap \N^{Q_1}
\end{equation}
Then $\Ng$ is the one-parameter subgroup lattice of a complex torus
$\Tg\subg\Aut(A)$ containing all global symmetries.

The other differential in the cochain complex \eqref{eq:cochain}
also has a natural interpretation in this context.
The lattice $\Nb=\Z^{Q_0}$ 
is the one-parameter subgroup lattice
of a complex torus $\Tb$ of invertible elements of $A$,
namely
\[
\Tb = \bigl\{ \sum_{i\in Q_0} t_i e_i : t_i\in\Cx \bigr\},
\]
where $e_i\in A$ are the idempotents corresponding to the vertices of
$Q$.
Then the lattice map $d\colon \Z^{Q_0} \to \Ng$ corresponds
to the map $\Tb \to \Tg\subg\Aut(A)$ giving the action on $A$ by inner
automorphisms, i.e. by conjugation.
Then the cokernel of $d\colon \Z^{Q_0} \to \Ng$ is the lattice of one-parameter subgroups
of outer automorphisms arising from global symmetries.
In other words, we have an exact sequence of complex tori
\[
1 \to \Cx \to \Tb \to \Tg \to \Tm \to 1,
\]
with corresponding exact sequence of one parameter subgroup
lattices
\begin{equation}\label{Nseq}
0 \to \Z \to \Nb \to \Ng \to \Nm \to 0
\end{equation}
In the physics literature (e.g. \cite{Kennaway}), elements of $\Nb$ are usually referred to as \emph{baryonic} symmetries, and elements of $\Nm$ as \emph{mesonic} symmetries.

Finally in this section we define the notion of an \emph{R-symmetry}, whose name comes from physics, but which is important mathematically as it makes $A$ into a graded algebra with finite dimensional graded pieces. 
\begin{definition} \label{RSym}
An R-symmetry is a global symmetry that acts with strictly positive weights (or `charges') on all the arrows.
\end{definition}

The R-symmetries are the `interior lattice points' of the cone $\Ng^+$. In the physics literature (e.g. \cite{Kennaway}) it is traditional to normalise the R-symmetries so they are of degree 2, (i.e. they act homogeneously with weight 2 on the superpotential) but also to extend the definition to allow the weights to be real, i.e. in $\R^{Q_1}$. By this definition, R-symmetries are real one-parameter subgroups whose weights lie in the interior of the degree 2 slice of the real cone corresponding to $\Ng^+$. Since this is a rational polyhedral cone, the interior is non-empty if and only if it contains rational points and hence if and only if $\Ng^+$ itself contains integral points 
with all weights strictly positive. Therefore for the purposes of this article we consider R-symmetries to be integral, as we defined above, and we do not impose the degree 2 normalisation.

Note finally that the existence of an R-symmetry is equivalent to the
fact that $\Ng^+$ is a `full' cone, i.e. it spans $\Ng$.

\section{Perfect matchings} \label{Perfmatchsec}

A \emph{perfect matching} on a bipartite graph is a collection of edges such that each vertex is the end point of precisely one edge (see for example \cite{Kenyonintro}). The edges in a perfect matching are sometimes also referred to as dimers and the perfect matching as a dimer configuration. 

Using the notation from \eqref{eq:cochain}, we take the essentially equivalent
dual point of view; we consider a perfect matching to be a 1-cochain $\pf\in \Z^{Q_1}$, with all values
in $\{0,1\}$, such that $d\pf=\const{1}$.
This is equivalent to requiring that $\pf\in\N^{Q_1}$ and $d\pf=\const{1}$, and so perfect matchings are the degree 1 elements of $\Ng^+$.

Not every dimer model has a perfect matching. This is obvious for dimer models which are not balanced, but even those with equal numbers of black and white vertices need not have a perfect matching. For example, the following case drawn on the torus:

\begin{equation}\label{balwnopm}
\psset{xunit=1.3cm,yunit=1.3cm,runit=1.3cm,algebraic=true,dotstyle=*,dotsize=5pt 0,linewidth=0.8pt,arrowsize=3pt 2,arrowinset=0.25}
\begin{pspicture*}(-2,0.5)(3,4.3)  
\psline[linestyle=dashed](-1,4)(2,4)
\psline[linestyle=dashed](2,4)(2,1)
\psline[linestyle=dashed](2,1)(-1,1)
\psline[linestyle=dashed](-1,1)(-1,4)
\psline(0,1)(-0.5,2.5)
\psline(0.18,2.5)(0,1)
\psline(0,4)(0.18,2.5)
\psline(-0.5,2.5)(0,4)
\psline(1,4)(0.82,2.5)
\psline(0.82,2.5)(1,1)
\psline(1,1)(1.5,2.5)
\psline(1.5,2.5)(1,4)
\psarc(0.5,0.52){0.69}{43.83}{136.17}
\psarc(-0.86,5.03){1.34}{263.8}{309.6}
\psarc(1.86,5.03){1.34}{230.4}{276.2}
\psdots(0,4)
\psdots[dotstyle=o](1,4)
\psdots(0,1)
\psdots[dotstyle=o](1,1)
\psdots[dotstyle=o](-0.5,2.5)
\psdots[dotstyle=o](0.18,2.5)
\psdots(0.82,2.5)
\psdots(1.5,2.5)
\end{pspicture*}
\end{equation}

The condition that a perfect matching does exist is provided by Hall's (Marriage) Theorem \cite{Hall35}:
\begin{lemma} \label{Hall}
A bipartite graph admits a perfect matching if and only if
it has the same number of black and white vertices and
every subset of black vertices is connected to at least as
many white vertices.
\end{lemma}
In the example in figure \ref{balwnopm} above, it can be see that the two black vertices in the interior of the fundamental domain are connected to just one vertex and so there are no perfect matchings.
\begin{remark}

The above example contains bivalent vertices which can not be removed in the way explained in Section~\ref{minimality}. However these are not the source of the `problem', in fact by doubling one of the edges ending at each bivalent vertex:

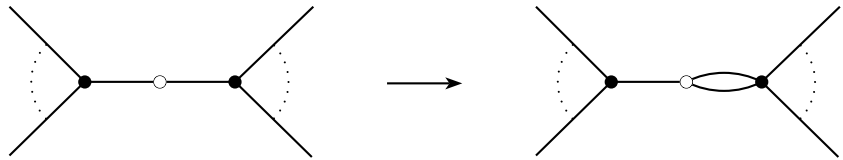
\begin{center}
\psset{xunit=1.0cm,yunit=1.0cm,algebraic=true,dotstyle=*,dotsize=5pt 0,linewidth=0.8pt,arrowsize=3pt 2,arrowinset=0.25}
\begin{pspicture*}(-5,0)(8,4)
\psline(-1,2)(0.04,3)
\psline(0.02,1)(-1,2)
\psline(-4,3)(-3,2)
\psline(-3,2)(-4,1.02)
\psarc[linestyle=dotted](-1,2){0.71}{315.57}{43.88}
\psarc[linestyle=dotted](-3,2){0.71}{135}{224.42}
\psline{->}(1.02,1.98)(2.02,1.98)
\psline(-1,2)(-2,2)
\psline(-2,2)(-3,2)
\psline(6,2)(7.04,3)
\psline(7.02,1)(6,2)
\psline(3,3)(4,2)
\psline(4,2)(3,1.02)
\psarc[linestyle=dotted](6,2){0.71}{315.57}{43.88}
\psarc[linestyle=dotted](4,2){0.71}{135}{224.42}
\psline(5,2)(4,2)
\psarc(5.5,3){1.12}{243.43}{296.57}
\psarc(5.5,1){1.12}{63.43}{116.57}
\psdots(-1,2)
\psdots(-3,2)
\psdots[dotstyle=o](-2,2)
\psdots(6,2)
\psdots(4,2)
\psdots[dotstyle=o](5,2)
\end{pspicture*}
\end{center}
we obtain a new dimer model which, by Lemma~\ref{Hall}, admits a perfect matching if and only if the original dimer model did. The resulting dimer model doesn't have bivalent vertices, but it does have {\digon}s. 
We can in turn replace each {\digon} as follows:
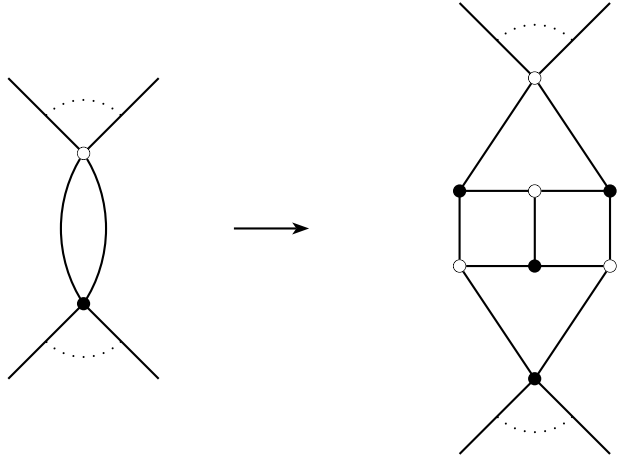
\begin{center}
\psset{xunit=1.0cm,yunit=1.0cm,algebraic=true,dotstyle=*,dotsize=5pt 0,linewidth=0.8pt,arrowsize=3pt 2,arrowinset=0.25}
\begin{pspicture*}(-3,-1)(7,7)
\psarc(0.52,3){1.82}{146.6}{213.4}
\psarc(-2.52,3){1.82}{326.6}{33.4}
\psline(-1,2)(-2,1)
\psline(-1,2)(0,1)
\psline(-1,4)(-2,5)
\psline(-1,4)(0,5)
\psarc[linestyle=dotted](-1,4){0.71}{45}{135}
\psarc[linestyle=dotted](-1,2){0.71}{225}{315}
\psline(5,5)(4,3.5)
\psline(4,3.5)(5,3.5)
\psline(5,3.5)(6,3.5)
\psline(6,3.5)(6,2.5)
\psline(6,2.5)(5,2.5)
\psline(5,3.5)(5,2.5)
\psline(4,3.5)(4,2.5)
\psline(4,2.5)(5,2.5)
\psline(6,2.5)(5,1)
\psline(4,2.5)(5,1)
\psline(5,5)(6,3.5)
\psline(5,5)(4,6)
\psline(6,6)(5,5)
\psline(4,0)(5,1)
\psline(5,1)(6,0)
\psarc[linestyle=dotted](5,5){0.71}{45}{135}
\psarc[linestyle=dotted](5,1){0.71}{225}{315}
\psline{->}(1,3)(2,3)
\psdots[dotstyle=o](-1,4)
\psdots(-1,2)
\psdots(4,3.5)
\psdots[dotstyle=o](4,2.5)
\psdots[dotstyle=o](5,3.5)
\psdots(5,2.5)
\psdots[dotstyle=o](6,2.5)
\psdots(6,3.5)
\psdots[dotstyle=o](5,5)
\psdots(5,1)
\end{pspicture*}
\end{center}

It is simple to check (again using Lemma~\ref{Hall}) that the altered dimer model has a perfect matching if and only if the original did.
Therefore, in the example, if we replaced each bivalent vertex and then each {\digon} as above, we would obtain a dimer model without bivalent vertices or {\digon}s but which still has no perfect matchings. The moral of this is that for simplicity we can leave the bivalent vertices alone (and let {\digon}s be {\digon}s)!
\end{remark}

We will be particularly interested in dimer models which satisfy a slightly stronger condition.
\begin{definition}\label{nondegdefn}
We call a dimer model \emph{non-degenerate} when every edge in the bipartite graph is contained in some perfect matching.
\end{definition}
\noindent 
Given a non-degenerate dimer model, the sum all perfect matchings (as an element of $\Ng^+$) is strictly positive on every arrow. Therefore it defines an R-symmetry. 


In fact, the existence of an R-symmetry and the non-degeneracy condition are equivalent. This is a straightforward consequence of the following integral version of the famous Birkhoff-von~Neumann Theorem for doubly stochastic matrices \cite{Birkhoff}.

\begin{lemma}
\label{thm:BvN}
The cone $\Ng^+$ is integrally generated by the perfect matchings,
all of which are extremal elements.
\end{lemma}
\begin{proof}
We adapt the standard argument \cite{planetmath} to the integral case. Every perfect matching is an element of $\Ng^+$ and therefore the cone generated by the perfect matchings is contained in $\Ng^+$. Conversely, choose any element $v \in \Ng^+$, and suppose $\deg{v} >0$. We construct a graph $G_v$ whose vertex set is the same as the dimer model (bipartite graph) and whose edges are the subset of edges $e$ of the dimer model, such that $v$ evaluated on the dual arrow $a_e$ is non-zero.

We claim that $G_v$ satisfies the conditions of Lemma~\ref{Hall}. To see this, let $A$ be any subset of vertices of one colour (black or white) and denote by $N(A)$ the set of neighbours of $A$, i.e. the vertices (of the other colour) which have an edge connecting them to some element of $A$. In an abuse of notation we shall also consider $A,N(A)\subseteq Q_2$ as sets of faces of the dual quiver $Q$. We note that
$$\deg{v}.|A|= \sum_{f \in A} \langle d(v), f \rangle 
= \sum_{\genfrac{}{}{0pt}{}{f \in A}{g \in N(A)}} \langle v, \del f \cap \del g \rangle $$
where $\del f \cap \del g$ denotes the class in $\Z_{Q_1}$ corresponding to the sum of the arrows which are in the boundary of faces $f$ and $g$.
If $B= N(A)$, then since $A \subseteq N(B)$,
$$|B|= \frac{1}{\deg{v}} \sum_{\genfrac{}{}{0pt}{}{f \in B}{g \in N(B)}} \langle v, \del f \cap \del g \rangle
\geq \frac{1}{\deg{v}} \sum_{\genfrac{}{}{0pt}{}{f \in B}{g \in A}} \langle v, \del f \cap \del g \rangle = |A|$$
Thus $|N(A)| \geq |A|$ as required.

Applying Lemma~\ref{Hall}, we see that $G_v$ has a perfect matching which extends by zero to a perfect matching $\pf \in \Z^{Q_1}$ of the dimer model. Since by construction, $v$ takes strictly positive integral values on all the arrows on which $\pf$ is non-zero, we see that $v-\pf \in \Ng^+$, and has degree $\deg{v}-1$. We proceed inductively and, using the fact that 0 is the only degree zero element in $\Ng^+$, we see that $v$ is equal to a sum of $\deg{v}$ perfect matchings.

To show that all perfect matchings are extremal elements it is sufficient to prove that no perfect matching is a non-trivial convex sum of distinct perfect matchings. However if $\sum_{s=1}^n \kappa_s \pf_s$ is a non-trivial convex sum, i.e. $\kappa_s>0$ for all $s=1, \dots ,n$ and $\sum_{s=1}^n \kappa_s =1 $, then this sum evaluates to a number in the closed interval $[0,1]$ on every arrow in $Q$. Furthermore we see that the values $\{0,1\}$ are attained if and only if all of the perfect matchings evaluate to the same number on that arrow. Therefore if this convex sum is a perfect matching, i.e. a $\{0,1\}$-valued function, then $\pf_s$ evaluate to the same number on every arrow for all $s=1, \dots ,n$, so the perfect matchings are not distinct.
\end{proof}
%
\begin{remark}
Another straightforward corollary of Hall's theorem states that a dimer model is non-degenerate if and only if the bipartite graph has an equal numbers of black and white vertices and every proper subset of the black vertices of size $n$ is connected to at least $n+1$ white vertices. We shall refer to this condition as the `strong marriage' condition. Using this, it is easy to construct examples of dimer models which have a perfect matching but do not satisfy the non-degeneracy condition. In the following example the two white vertices in the interior of the fundamental domain are connected to two black vertices, so it is degenerate. It can also be checked directly that the edge marked in grey must be contained in every perfect matching. Therefore the other edges which share an end vertex with this edge are not contained in any perfect matching.

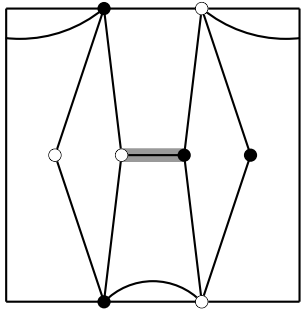
\begin{center}
\newrgbcolor{zzzzzz}{0.6 0.6 0.6}
\psset{xunit=1.3cm,yunit=1.3cm,runit=1.3cm,algebraic=true,dotstyle=*,dotsize=5pt 0,linewidth=0.8pt,arrowsize=3pt 2,arrowinset=0.25}
\begin{pspicture*}(-2,0.2)(3,5)
\psline(-1,4)(2,4)
\psline(2,4)(2,1)
\psline(2,1)(-1,1)
\psline(-1,1)(-1,4)
\psline(0,1)(-0.5,2.5)
\psline(0.18,2.5)(0,1)
\psline(0,4)(0.18,2.5)
\psline(1,4)(0.82,2.5)
\psline(0.82,2.5)(1,1)
\psline(1,1)(1.5,2.5)
\psarc(0.5,0.52){0.69}{43.83}{136.17}
\psarc(-0.86,5.03){1.34}{263.8}{309.6}
\psarc(1.86,5.03){1.34}{230.4}{276.2}
\psline[linewidth=5.2pt,linecolor=zzzzzz](0.18,2.5)(0.82,2.5)
\psline(0.18,2.5)(0.82,2.5)
\psline(1,4)(1.5,2.5)
\psline(0,4)(-0.5,2.5)
\psdots(0,4)
\psdots[dotstyle=o](1,4)
\psdots(0,1)
\psdots[dotstyle=o](1,1)
\psdots[dotstyle=o](-0.5,2.5)
\psdots[dotstyle=o](0.18,2.5)
\psdots(0.82,2.5)
\psdots(1.5,2.5)
\end{pspicture*}
\end{center}
\end{remark}

Let $\Nm^+$ be the saturation of the projection of the cone $\Ng^+\subs \Ng$
into the rank $2g +1$ lattice $\Nm$, where $g$ is the genus of $\RS$. In other words $\Nm^+$ is the intersection of $\Nm$ with the real cone generated by the image of $\Ng^+$ in $\Nm \otimes_{\Z} \R$.
Because of Theorem~\ref{thm:BvN}, it is natural to use the perfect matchings to describe $\Nm^+$.
From its construction there is a short exact sequence
\begin{equation}
\label{eq:Nmes}
0 \to H^{1}(\RS;\Z) \lra{} \Nm \lra{deg} \Z \to 0
\end{equation}
and, since every perfect matching has degree 1, their images in $\Nm$
span a lattice polytope in a rank 2$g$ affine sublattice such that $\Nm^+$ is the
cone on this polytope. By choosing some fixed reference matching $\pf_0$,
this polytope may be translated into $H^{1}(\RS;\Z)$ and described more
directly as follows:
for any perfect matching $\pf$, the difference $\pf-\pf_0$ is a
cocycle and hence has a well-defined cohomology class. We call this the relative cohomology class of $\pf$. The lattice polytope described above, is the convex hull of all relative cohomology classes of perfect matchings. We note that there is usually not a 1-1 correspondence between perfect matchings and lattice points in the polytope.
\begin{definition} \label{multipdefn}
The \emph{multiplicity} of a lattice point in the polytope is defined to be the number of perfect matchings whose relative cohomology class is that point.
\end{definition}

In the cases which will be of most interest, when the Riemann surface is a torus, then $\Nm$ is a rank 3 lattice, and the images of the perfect matchings span a polygon in a rank 2 affine sublattice.
\begin{definition} \label{externaldef}
A lattice point is called \emph{external} if it lies on a facet of the polygon, and \emph{extremal} if it lies at a vertex of the polygon. A perfect matching is external (extremal) if it corresponds to an external (extremal) lattice point.
\end{definition}
The translated polygon in $H^{1}(T;\Z) \cong \Z^{2}$ may be computed by various explicit methods,
e.g using the Kastelyn determinant as in \cite{HananyKen}.

From the point of view of toric geometry, it is natural to think of
$\Ng^+$ and $\Nm^+$ as describing two (normal) affine toric varieties
$\Xg$ and $\Xm$, such that $\Xm=\Xg / \Tb$ where the $\Tb$ action
is determined by the map $\Nb\to\Ng$.
Furthermore, $\Xg$ and $\Xm$ both have Gorenstein singularities, since the
corresponding cones are generated by hyperplane sections.

\chapter{Consistency} \label{consistchap}
In Chapter~\ref{Introto} we saw some non-degeneracy conditions which we can impose on dimer models. For example we saw what it means for a dimer
model to be balanced (\ref{balanced}) and non-degenerate (\ref{nondegdefn}). We now come to the most important and strongest of
these conditions which are called \emph{consistency} conditions. We describe two types of consistency which appear in the physics literature
and state how they relate to each other. Then in Section~\ref{stiengul} we briefly describe a construction due to Gulotta \cite{Gulotta}, which allows us
to produce examples of dimer models which satisfy these consistency conditions. We will see in later chapters that some kind of consistency
condition is needed in order to prove properties we are interested in, such as the Calabi-Yau property. 
While many of the ideas in this chapter come from the physics literature, as in the previous chapter, the way we present things here
will often be different in order to fit better within the overarching mathematical framework that we are constructing. Any proclaimed results
which are not otherwise referenced are new.
\section{A further condition on the R-symmetry}
We defined, in Section \ref{symmetries}, an \emph{R-symmetry} to be a global symmetry that acts with strictly positive weights on all the arrows. 
We recall that the existence of an R-symmetry is equivalent to non-degeneracy of a dimer model. The first definition of consistency is a strengthening of this, and states that a dimer model is consistent if there exists an `anomaly-free' R-symmetry. 

We recall that in the physics literature (e.g. in \cite{Kennaway}) it is traditional to allow real R-symmetries, but to normalise so they are of degree 2,  
i.e. a real R-symmetry $\Rsym \in \R^{Q_1}$ which acts on each arrow $a \in Q_1$ with weight $R_a$, satisfies
\begin{equation}\label{normnondegen}
\sum_{a \in \del f} R_a =2 \qquad \forall f \in Q_2
\end{equation}
Of these real R-symmetries, physicists are particularly interested in ones which have no `anomalies'.
Formulated mathematically, these are R-symmetries which satisfy the following `anomaly-vanishing' condition at each vertex of the quiver
\begin{equation}\label{anomfree}
\sum_{a \in H_v \cup T_v} (1-R_a) =2 \qquad \forall v \in Q_0
\end{equation}
where $H_v:= \{ b \in Q_1 \mid hb=v \}$ and $T_v:= \{ b \in Q_1 \mid tb=v \}$.

As stated before, we usually work with integral R-symmetries, without any normalisation, and let $\deg(\Rsym)$ be the degree of an R-symmetry $\Rsym \in \Z^{Q_1}$. Therefore, if $\Rsym$ acts on each arrow $a \in Q_1$ with weight $R_a \in \Z$, it satisfies
\begin{equation}\label{nondegen}
\sum_{a \in \del f} R_a =\deg(\Rsym) \qquad \forall f \in Q_2
\end{equation}
The corresponding un-normalised `anomaly-vanishing' condition at each vertex of the quiver is given by
\begin{equation}\label{anomfree1}
\sum_{ a \in H_v \cup T_v} R_a =
           \deg(\Rsym)(|H_v|-1)  \qquad \forall v \in Q_0
\end{equation}
We note that since $Q$ is dual to a bipartite tiling, the arrows around any given vertex $v$ alternate between outgoing and incoming arrows, so $|H_v|=|T_v|$.
\begin{remark} \label{realrat}
Since the conditions (\ref{anomfree1}) are rational linear equations, the intersection of their zero locus and the cone $\Ng^+$ is a rational cone. Thus, using a similar argument to that in Section~\ref{symmetries}, we see that there exists an `anomaly-free' real R-symmetry if and only if there exists an `anomaly free' integral R-symmetry.
\end{remark}

\begin{definition}\label{consistency}
A dimer model is called \emph{consistent} if there exists an R-symmetry satisfying the condition~(\ref{anomfree1}).
\end{definition}
We note that in particular a consistent dimer model has an R-symmetry and so it is non-degenerate (Definition~\ref{nondegdefn}). Up to this point we have considered dimer models on an arbitrary Riemann surface $\RS$. However the following argument, given by Kennaway in Section~3.1 of \cite{Kennaway}, shows that consistency forces this surface to be a torus.

Consider an R-symmetry which satisfies conditions~(\ref{nondegen}) and (\ref{anomfree1}). Summing these equations over all the quiver faces $Q_2$ and the quiver vertices $Q_0$ respectively and using the fact that each arrow is in exactly two faces and has two ends we observe that:
\begin{equation}\label{toruspf}
\deg(\Rsym) |Q_2| = 2 \sum_{a \in Q_1}R_a = \deg(\Rsym) (|Q_1|-|Q_0|)
\end{equation}
Hence $|Q_0| - |Q_1| + |Q_2| =0$. Since the quiver gives a cell decomposition of the surface $\RS$, this implies that $\RS$ has Euler characteristic zero and must be a 2-torus.

\section{Rhombus tilings} \label{Rhombtiling}
We now define another consistency condition which we call `geometric consistency'. This will imply the consistency condition (Definition~\ref{consistency}) and when it holds, it gives a geometric interpretation of the conditions (\ref{normnondegen}) and (\ref{anomfree}). This was first understood by Hanany and Vegh in \cite{HananyVegh}. In practice it is not easy to see directly if a dimer model is geometrically consistent, however we will explain an equivalent characterisation in terms of `train tracks' on the `quad graph' due to Kenyon and Schlenker which will be easier to check.

Given a dimer model on a torus, we construct the `quad graph' 
associated to it. This is a tiling of the torus whose set of vertices is the union of the vertices of the bipartite tiling and its dual quiver $Q$. The edges of the quad graph connect a dimer vertex $\face$ to a quiver vertex $v$ if and only if the face dual to $\face $ has vertex $v$ in its boundary. The faces of this new tiling, which we call `quads' to avoid confusion, are by construction quadrilaterals and are in 1-1 correspondence with the arrows in the quiver $Q$; given any arrow $a$ in the quiver we have remarked previously (Remark~\ref{Ftermremark}) that it lies in the boundary of exactly two quiver faces $\face_+$ and $\face_-$ of different colours. Therefore, there are edges in the quad graph between each of $\face_\pm$ and both $ha$ and $ta$. These four edges form the boundary of a quad, and every quad is of this form. In particular, the corresponding arrow and dual edge in the dimer are the two diagonals of the quad.

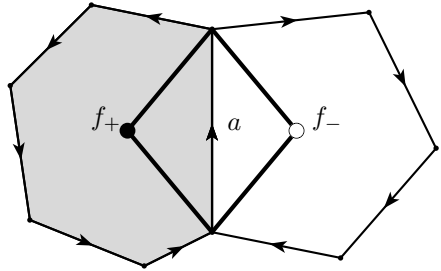
\begin{center}
\psset{xunit=1.0cm,yunit=1.0cm}
\newrgbcolor{zzzzzz}{0.85 0.85 0.85}
\begin{pspicture*}(-2,-2.5)(5.5,2.5)
\psset{xunit=0.45cm,yunit=0.45cm,algebraic=true,dotstyle=*,dotsize=3pt 0,linewidth=0.8pt,arrowsize=3pt 2,arrowinset=0.25}
\pspolygon[fillcolor=zzzzzz,fillstyle=solid](3,3)(-0.56,3.72)(-2.96,1.33)(-2.38,-2.64)(1,-4)(3,-3)
\psline[linewidth=1.6pt](0.5,0)(3,-3)
\psline[linewidth=1.6pt](3,-3)(5.5,0)
\psline[linewidth=1.6pt](5.5,0)(3,3)
\psline[linewidth=1.6pt](3,3)(0.5,0)
\psline[ArrowInside=->](3,-3)(3,3)
\psline[ArrowInside=->](3,3)(-0.56,3.72)
\psline[ArrowInside=->](-0.56,3.72)(-2.96,1.33)
\psline[ArrowInside=->](-2.96,1.33)(-2.38,-2.64)
\psline[ArrowInside=->](-2.38,-2.64)(1,-4)
\psline[ArrowInside=->](1,-4)(3,-3)
\psline[ArrowInside=->](3,-3)(3,3)
\psline[ArrowInside=->](3,3)(7.64,3.5)
\psline[ArrowInside=->](7.64,3.5)(9.62,-0.52)
\psline[ArrowInside=->](9.62,-0.52)(6.8,-3.76)
\psline[ArrowInside=->](6.8,-3.76)(3,-3)
\psline[ArrowInside=->](3,-3)(3,3)
\rput[tl](-0.58,0.7){$f_+$}
\rput[tl](5.95,0.7){$f_-$}
\psdots[dotsize=6pt 0](0.5,0)
\psdots[dotsize=2pt 0](3,3)
\psdots[dotsize=2pt 0](3,-3)
\psdots[dotsize=6pt 0,dotstyle=o](5.5,0)
\rput[bl](3.46,0.01){$a$}
\psdots[dotsize=2pt 0](-0.56,3.72)
\psdots[dotsize=2pt 0](-2.96,1.33)
\psdots[dotsize=2pt 0](-2.38,-2.64)
\psdots[dotsize=2pt 0](1,-4)
\psdots[dotsize=2pt 0](9.62,-0.52)
\psdots[dotsize=2pt 0](7.64,3.5)
\psdots[dotsize=2pt 0](6.8,-3.76)
\end{pspicture*}
\end{center}

We draw the bipartite graph, the quiver and the quad graph of the hexagonal tiling below as an example.

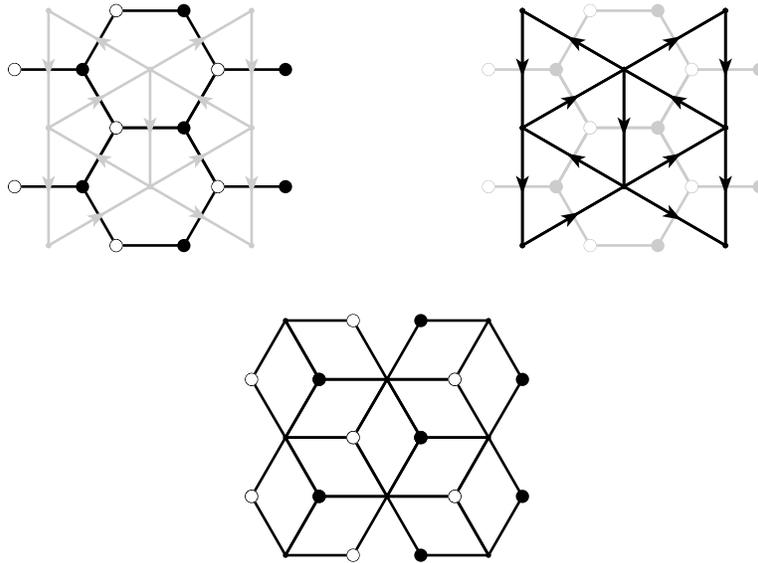
\begin{center}
\newrgbcolor{zzzzzz}{0.8 0.8 0.8}
\psset{xunit=0.9cm,yunit=0.9cm,algebraic=true,dotstyle=*,dotsize=5pt 0,linewidth=1.0pt,arrowsize=3pt 2,arrowinset=0.25}
\begin{pspicture*}(-3,-5)(10,4)
\psline(-2,1)(-1,1)
\psline(-1,1)(-0.5,1.87)
\psline(-0.5,1.87)(-1,2.73)
\psline(-1,2.73)(-2,2.73)
\psline(-1,2.73)(-0.5,1.87)
\psline(-0.5,1.87)(0.5,1.87)
\psline(0.5,1.87)(1,2.73)
\psline(1,2.73)(0.5,3.6)
\psline(0.5,3.6)(-0.5,3.6)
\psline(-0.5,3.6)(-1,2.73)
\psline(0.5,1.87)(-0.5,1.87)
\psline(-0.5,1.87)(-1,1)
\psline(-1,1)(-0.5,0.13)
\psline(-0.5,0.13)(0.5,0.13)
\psline(0.5,0.13)(1,1)
\psline(1,1)(0.5,1.87)
\psline(0.5,1.87)(1,1)
\psline(1,1)(2,1)
\psline(2,2.73)(1,2.73)
\psline(1,2.73)(0.5,1.87)
\psline(5.5,0.13)(7,1)
\psline(7,1)(8.5,0.13)
\psline(8.5,1.87)(8.5,0.13)
\psline(7,1)(8.5,1.87)
\psline(8.5,1.87)(7,2.73)
\psline(7,2.73)(7,1)
\psline(7,1)(5.5,1.87)
\psline(5.5,1.87)(5.5,0.13)
\psline(5.5,1.87)(7,2.73)
\psline(7,2.73)(5.5,3.6)
\psline(5.5,3.6)(5.5,1.87)
\psline(7,2.73)(8.5,3.6)
\psline(8.5,3.6)(8.5,1.87)
\psline(3.5,-3.58)(3,-4.45)
\psline(3,-4.45)(2,-4.45)
\psline(2,-4.45)(2.5,-3.58)
\psline(2.5,-3.58)(3.5,-3.58)
\psline(3.5,-3.58)(4,-4.45)
\psline(4,-4.45)(5,-4.45)
\psline(5,-4.45)(4.5,-3.58)
\psline(4.5,-3.58)(3.5,-3.58)
\psline(3.5,-3.58)(3,-2.71)
\psline(3,-2.71)(3.5,-1.85)
\psline(3.5,-1.85)(4,-2.71)
\psline(4,-2.71)(3.5,-3.58)
\psline(4,-2.71)(5,-2.71)
\psline(5,-2.71)(4.5,-1.85)
\psline(4.5,-1.85)(3.5,-1.85)
\psline(3.5,-1.85)(4,-2.71)
\psline(5,-2.71)(4.5,-3.58)
\psline(4.5,-3.58)(3.5,-3.58)
\psline(3.5,-3.58)(4,-2.71)
\psline(4,-2.71)(5,-2.71)
\psline(5,-2.71)(5.5,-3.58)
\psline(5.5,-3.58)(5,-4.45)
\psline(5,-4.45)(4.5,-3.58)
\psline(4.5,-3.58)(5,-2.71)
\psline(2,-2.71)(3,-2.71)
\psline(3,-2.71)(3.5,-1.85)
\psline(3.5,-1.85)(2.5,-1.85)
\psline(2.5,-1.85)(2,-2.71)
\psline(2,-2.71)(2.5,-3.58)
\psline(2.5,-3.58)(3.5,-3.58)
\psline(3.5,-3.58)(3,-2.71)
\psline(3,-2.71)(2,-2.71)
\psline(2,-2.71)(1.5,-1.85)
\psline(1.5,-1.85)(2,-0.98)
\psline(2,-0.98)(2.5,-1.85)
\psline(2.5,-1.85)(2,-2.71)
\psline(2,-2.71)(1.5,-3.58)
\psline(1.5,-3.58)(2,-4.45)
\psline(2,-4.45)(2.5,-3.58)
\psline(2.5,-3.58)(2,-2.71)
\psline(3.5,-1.85)(3,-0.98)
\psline(3,-0.98)(2,-0.98)
\psline(2,-0.98)(2.5,-1.85)
\psline(2.5,-1.85)(3.5,-1.85)
\psline(3.5,-1.85)(4,-0.98)
\psline(4,-0.98)(5,-0.98)
\psline(5,-0.98)(4.5,-1.85)
\psline(4.5,-1.85)(3.5,-1.85)
\psline(5,-2.71)(5.5,-1.85)
\psline(5.5,-1.85)(5,-0.98)
\psline[ArrowInside=->, linecolor=zzzzzz](-1.5,0.13)(0,1)
\psline[ArrowInside=->, linecolor=zzzzzz](0,1)(1.5,0.13)
\psline[ArrowInside=->, linecolor=zzzzzz](1.5,1.87)(1.5,0.13)
\psline[ArrowInside=->, linecolor=zzzzzz](0,1)(1.5,1.87)
\psline[ArrowInside=->, linecolor=zzzzzz](1.5,1.87)(0,2.73)
\psline[ArrowInside=->, linecolor=zzzzzz](0,2.73)(0,1)
\psline[ArrowInside=->, linecolor=zzzzzz](0,1)(-1.5,1.87)
\psline[ArrowInside=->, linecolor=zzzzzz](-1.5,1.87)(-1.5,0.13)
\psline[ArrowInside=->, linecolor=zzzzzz](-1.5,1.87)(0,2.73)
\psline[ArrowInside=->, linecolor=zzzzzz](0,2.73)(-1.5,3.6)
\psline[ArrowInside=->, linecolor=zzzzzz](-1.5,3.6)(-1.5,1.87)
\psline[ArrowInside=->, linecolor=zzzzzz](0,2.73)(1.5,3.6)
\psline[ArrowInside=->, linecolor=zzzzzz](1.5,3.6)(1.5,1.87)
\psline[linecolor=zzzzzz](5,1)(6,1)
\psline[linecolor=zzzzzz](6,1)(6.5,1.87)
\psline[linecolor=zzzzzz](6.5,1.87)(6,2.73)
\psline[linecolor=zzzzzz](6,2.73)(5,2.73)
\psline[linecolor=zzzzzz](6,2.73)(6.5,1.87)
\psline[linecolor=zzzzzz](6.5,1.87)(7.5,1.87)
\psline[linecolor=zzzzzz](7.5,1.87)(8,2.73)
\psline[linecolor=zzzzzz](8,2.73)(7.5,3.6)
\psline[linecolor=zzzzzz](7.5,3.6)(6.5,3.6)
\psline[linecolor=zzzzzz](6.5,3.6)(6,2.73)
\psline[linecolor=zzzzzz](7.5,1.87)(6.5,1.87)
\psline[linecolor=zzzzzz](6.5,1.87)(6,1)
\psline[linecolor=zzzzzz](6,1)(6.5,0.13)
\psline[linecolor=zzzzzz](6.5,0.13)(7.5,0.13)
\psline[linecolor=zzzzzz](7.5,0.13)(8,1)
\psline[linecolor=zzzzzz](8,1)(7.5,1.87)
\psline[linecolor=zzzzzz](7.5,1.87)(8,1)
\psline[linecolor=zzzzzz](8,1)(9,1)
\psline[linecolor=zzzzzz](9,2.73)(8,2.73)
\psline[linecolor=zzzzzz](8,2.73)(7.5,1.87)
\psline[ArrowInside=->](5.5,0.13)(7,1)
\psline[ArrowInside=->](7,1)(8.5,0.13)
\psline[ArrowInside=->](8.5,1.87)(8.5,0.13)
\psline[ArrowInside=->](7,1)(8.5,1.87)
\psline[ArrowInside=->](8.5,1.87)(7,2.73)
\psline[ArrowInside=->](7,2.73)(7,1)
\psline[ArrowInside=->](7,1)(5.5,1.87)
\psline[ArrowInside=->](5.5,1.87)(5.5,0.13)
\psline[ArrowInside=->](5.5,1.87)(7,2.73)
\psline[ArrowInside=->](7,2.73)(5.5,3.6)
\psline[ArrowInside=->](5.5,3.6)(5.5,1.87)
\psline[ArrowInside=->](7,2.73)(8.5,3.6)
\psline[ArrowInside=->](8.5,3.6)(8.5,1.87)
\psdots[dotstyle=o](-2,1)
\psdots[dotstyle=o](-0.5,1.87)
\psdots(-1,2.73)
\psdots[dotstyle=o](-2,2.73)
\psdots(0.5,1.87)
\psdots[dotstyle=o](1,2.73)
\psdots(0.5,3.6)
\psdots[dotstyle=o](-0.5,3.6)
\psdots(-1,1)
\psdots[dotstyle=o](-0.5,0.13)
\psdots(0.5,0.13)
\psdots[dotstyle=o](1,1)
\psdots(2,1)
\psdots(2,2.73)
\psdots[dotstyle=o](1,2.73)
\psdots[dotsize=2pt 0,linecolor=zzzzzz](-1.5,1.87)
\psdots[dotsize=2pt 0,linecolor=zzzzzz](0,1)
\psdots[dotsize=2pt 0](5.5,1.87)
\psdots[dotsize=2pt 0](7,1)
\psdots[dotsize=2pt 0](7,2.73)
\psdots[dotsize=2pt 0](8.5,1.87)
\psdots[dotsize=2pt 0](5.5,3.6)
\psdots[dotsize=2pt 0](5.5,0.13)
\psdots[dotsize=2pt 0](8.5,0.13)
\psdots[dotsize=2pt 0](8.5,3.6)
\psdots[dotstyle=o](1.5,-3.58)
\psdots[dotstyle=o](1.5,-3.58)
\psdots(2.5,-3.58)
\psdots[dotstyle=o](3,-2.71)
\psdots(2.5,-1.85)
\psdots[dotstyle=o](1.5,-1.85)
\psdots(2.5,-1.85)
\psdots[dotstyle=o](3,-2.71)
\psdots(4,-2.71)
\psdots[dotstyle=o](4.5,-1.85)
\psdots(4,-0.98)
\psdots[dotstyle=o](3,-0.98)
\psdots(4,-2.71)
\psdots[dotstyle=o](3,-2.71)
\psdots(2.5,-3.58)
\psdots[dotstyle=o](3,-4.45)
\psdots(4,-4.45)
\psdots[dotstyle=o](4.5,-3.58)
\psdots(4,-2.71)
\psdots[dotstyle=o](4.5,-3.58)
\psdots(5.5,-3.58)
\psdots(5.5,-1.85)
\psdots[dotstyle=o](4.5,-1.85)
\psdots[dotsize=2pt 0](2,-2.71)
\psdots[dotsize=2pt 0](3.5,-3.58)
\psdots[dotsize=2pt 0](3.5,-1.85)
\psdots[dotsize=2pt 0](5,-2.71)
\psdots[dotsize=2pt 0](2,-0.98)
\psdots[dotsize=2pt 0](2,-4.45)
\psdots[dotsize=2pt 0](5,-4.45)
\psdots[dotsize=2pt 0](5,-0.98)
\psdots[dotsize=2pt 0,linecolor=zzzzzz](-1.5,0.13)
\psdots[dotsize=2pt 0,linecolor=zzzzzz](0,1)
\psdots[dotsize=2pt 0,linecolor=zzzzzz](0,1)
\psdots[dotsize=2pt 0,linecolor=zzzzzz](1.5,0.13)
\psdots[dotsize=2pt 0,linecolor=zzzzzz](1.5,1.87)
\psdots[dotsize=2pt 0,linecolor=zzzzzz](1.5,0.13)
\psdots[dotsize=2pt 0,linecolor=zzzzzz](0,1)
\psdots[dotsize=2pt 0,linecolor=zzzzzz](1.5,1.87)
\psdots[dotsize=2pt 0,linecolor=zzzzzz](1.5,1.87)
\psdots[dotsize=2pt 0,linecolor=zzzzzz](0,2.73)
\psdots[dotsize=2pt 0,linecolor=zzzzzz](0,2.73)
\psdots[dotsize=2pt 0,linecolor=zzzzzz](0,1)
\psdots[dotsize=2pt 0,linecolor=zzzzzz](0,1)
\psdots[dotsize=2pt 0,linecolor=zzzzzz](-1.5,1.87)
\psdots[dotsize=2pt 0,linecolor=zzzzzz](-1.5,1.87)
\psdots[dotsize=2pt 0,linecolor=zzzzzz](-1.5,0.13)
\psdots[dotsize=2pt 0,linecolor=zzzzzz](-1.5,1.87)
\psdots[dotsize=2pt 0,linecolor=zzzzzz](0,2.73)
\psdots[dotsize=2pt 0,linecolor=zzzzzz](0,2.73)
\psdots[dotsize=2pt 0,linecolor=zzzzzz](-1.5,3.6)
\psdots[dotsize=2pt 0,linecolor=zzzzzz](-1.5,3.6)
\psdots[dotsize=2pt 0,linecolor=zzzzzz](-1.5,1.87)
\psdots[dotsize=2pt 0,linecolor=zzzzzz](0,2.73)
\psdots[dotsize=2pt 0,linecolor=zzzzzz](1.5,3.6)
\psdots[dotsize=2pt 0,linecolor=zzzzzz](1.5,3.6)
\psdots[dotsize=2pt 0,linecolor=zzzzzz](1.5,1.87)
\psdots[dotstyle=o,linecolor=zzzzzz](5,1)
\psdots[linecolor=zzzzzz](6,1)
\psdots[dotstyle=o,linecolor=zzzzzz](6.5,1.87)
\psdots[linecolor=zzzzzz](6,2.73)
\psdots[dotstyle=o,linecolor=zzzzzz](5,2.73)
\psdots[linecolor=zzzzzz](6,2.73)
\psdots[dotstyle=o,linecolor=zzzzzz](6.5,1.87)
\psdots[linecolor=zzzzzz](7.5,1.87)
\psdots[dotstyle=o,linecolor=zzzzzz](8,2.73)
\psdots[linecolor=zzzzzz](7.5,3.6)
\psdots[dotstyle=o,linecolor=zzzzzz](6.5,3.6)
\psdots[linecolor=zzzzzz](7.5,1.87)
\psdots[dotstyle=o,linecolor=zzzzzz](6.5,1.87)
\psdots[linecolor=zzzzzz](6,1)
\psdots[dotstyle=o,linecolor=zzzzzz](6.5,0.13)
\psdots[linecolor=zzzzzz](7.5,0.13)
\psdots[dotstyle=o,linecolor=zzzzzz](8,1)
\psdots[linecolor=zzzzzz](7.5,1.87)
\psdots[dotstyle=o,linecolor=zzzzzz](8,1)
\psdots[linecolor=zzzzzz](9,1)
\psdots[linecolor=zzzzzz](9,2.73)
\psdots[dotstyle=o,linecolor=zzzzzz](8,2.73)
\psdots[dotsize=2pt 0](5.5,0.13)
\psdots[dotsize=2pt 0](7,1)
\psdots[dotsize=2pt 0](7,1)
\psdots[dotsize=2pt 0](8.5,0.13)
\psdots[dotsize=2pt 0](8.5,1.87)
\psdots[dotsize=2pt 0](8.5,0.13)
\psdots[dotsize=2pt 0](7,1)
\psdots[dotsize=2pt 0](8.5,1.87)
\psdots[dotsize=2pt 0](8.5,1.87)
\psdots[dotsize=2pt 0](7,2.73)
\psdots[dotsize=2pt 0](5.5,1.87)
\psdots[dotsize=2pt 0](5.5,0.13)
\psdots[dotsize=2pt 0](5.5,1.87)
\psdots[dotsize=2pt 0](7,2.73)
\psdots[dotsize=2pt 0](7,2.73)
\psdots[dotsize=2pt 0](5.5,3.6)
\psdots[dotsize=2pt 0](5.5,3.6)
\psdots[dotsize=2pt 0](5.5,1.87)
\psdots[dotsize=2pt 0](7,2.73)
\psdots[dotsize=2pt 0](8.5,3.6)
\psdots[dotsize=2pt 0](8.5,3.6)
\psdots[dotsize=2pt 0](8.5,1.87)
\end{pspicture*}
\end{center}


\begin{definition}\label{rhombtiling} A dimer model on a torus is called \emph{geometrically consistent} if there exists a rhombic embedding of the quad graph associated to it, i.e. an embedding in the torus such that all edges are line segments with the same length.
\end{definition}
\noindent There exists a flat metric on the torus, so line segments are well defined, and each quad in a rhombic embedding is a rhombus. We note that the hexagonal tiling example above, is a geometrically consistent dimer model and the quad graph drawn is a rhombic embedding.

Suppose we have a geometrically consistent dimer model and consider a single rhombus in the rhombus embedding. 
The interior angles at opposite corners are the same, and the total of the internal angles is $2\pi$, so its shape is determined by one angle. To each rhombus, and so to each arrow $a$ of the quiver, we associate $\Rsym_a$, the interior angle of the rhombus at the dimer vertices divided by $\pi$.

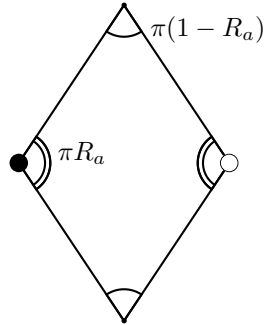
\begin{center}
\psset{xunit=0.7cm,yunit=0.7cm}
\begin{pspicture*}(0,-4)(6,4)
\psset{xunit=0.7cm,yunit=0.7cm,runit=0.7cm,algebraic=true,dotstyle=*,dotsize=3pt 0,linewidth=0.8pt,arrowsize=3pt 2,arrowinset=0.25}
\psline(1,0)(3,-3)
\psline(3,-3)(5,0)
\psline(5,0)(3,3)
\psline(3,3)(1,0)
\pscustom{\parametricplot{-0.982793723247329}{0.982793723247329}{0.6*cos(t)+1|0.6*sin(t)+0}\lineto(1,0)\closepath}
\parametricplot{-0.982793723247329}{0.982793723247329}{0.6*cos(t)+1|0.6*sin(t)+0}
\parametricplot{-0.982793723247329}{0.982793723247329}{0.5*cos(t)+1|0.5*sin(t)+0}
\pscustom{\parametricplot{0.982793723247329}{2.1587989303424644}{0.6*cos(t)+3|0.6*sin(t)+-3}\lineto(3,-3)\closepath}
\pscustom{\parametricplot{-2.1587989303424644}{-0.9827937232473292}{0.6*cos(t)+3|0.6*sin(t)+3}\lineto(3,3)\closepath}
\pscustom{\parametricplot{2.1587989303424644}{4.124386376837123}{0.6*cos(t)+5|0.6*sin(t)+0}\lineto(5,0)\closepath}
\parametricplot{2.1587989303424644}{4.124386376837123}{0.6*cos(t)+5|0.6*sin(t)+0}
\parametricplot{2.1587989303424644}{4.124386376837123}{0.5*cos(t)+5|0.5*sin(t)+0}
\rput[tl](3.5,2.8){$\pi(1-R_a)$}
\rput[tl](1.74,0.4){$\pi R_a$}
\psdots[dotsize=7pt 0](1,0)
\psdots[dotsize=2pt 0](3,3)
\psdots[dotsize=2pt 0](3,-3)
\psdots[dotsize=7pt 0,dotstyle=o](5,0)
\end{pspicture*}
\end{center}

The conditions that the rhombi fit together around each dimer vertex and each quiver vertex, are exactly the conditions (\ref{normnondegen}) and (\ref{anomfree}) for an `anomaly free' real R-charge, normalised with $\deg(\Rsym) = 2$. Thus, recalling Remark~\ref{realrat}, every geometrically consistent dimer model is consistent.

The converse however is not true. Given some normalised anomaly-free R-symmetry, we certainly require the additional condition that $\Rsym_a < 1$ for all $a \in Q_1$ in order to be able to produce a genuine rhombus embedding with angles in $(0,\pi)$. This does not hold in all cases; for example, the following dimer model is consistent but not geometrically consistent, as we shall see shortly. 

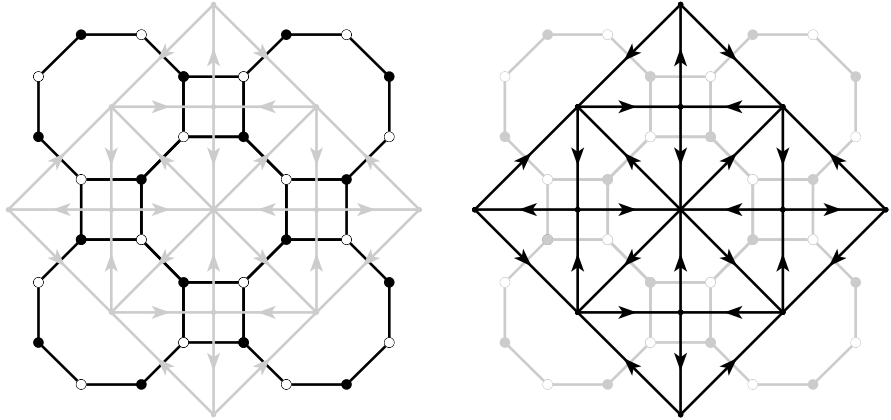
\begin{example} \label{examplestp} \end{example}
\begin{center}
\newrgbcolor{zzzzzz}{0.8 0.8 0.8}
\psset{xunit=0.8cm,yunit=0.8cm,algebraic=true,dotstyle=*,dotsize=4pt 0,linewidth=1.0pt,arrowsize=3pt 2,arrowinset=0.25}
\begin{pspicture*}(-4,-3.5)(11.5,4.5)
\psline(1,0)(2,0)
\psline(2,0)(2,1)
\psline(2,1)(1,1)
\psline(1,1)(1,0)
\psline(1,1)(2,1)
\psline(2,1)(2.71,1.71)
\psline(2.71,1.71)(2.71,2.71)
\psline(2.71,2.71)(2,3.41)
\psline(2,3.41)(1,3.41)
\psline(1,3.41)(0.29,2.71)
\psline(0.29,2.71)(0.29,1.71)
\psline(0.29,1.71)(1,1)
\psline(2,0)(1,0)
\psline(1,0)(0.29,-0.71)
\psline(0.29,-0.71)(0.29,-1.71)
\psline(0.29,-1.71)(1,-2.41)
\psline(1,-2.41)(2,-2.41)
\psline(2,-2.41)(2.71,-1.71)
\psline(2.71,-1.71)(2.71,-0.71)
\psline(2.71,-0.71)(2,0)
\psline(1,0)(1,1)
\psline(1,1)(0.29,1.71)
\psline(0.29,1.71)(-0.71,1.71)
\psline(-0.71,1.71)(-1.41,1)
\psline(-1.41,1)(-1.41,0)
\psline(-1.41,0)(-0.71,-0.71)
\psline(-0.71,-0.71)(0.29,-0.71)
\psline(0.29,-0.71)(1,0)
\psline(-0.71,1.71)(0.29,1.71)
\psline(0.29,1.71)(0.29,2.71)
\psline(0.29,2.71)(-0.71,2.71)
\psline(-0.71,2.71)(-0.71,1.71)
\psline(0.29,-0.71)(-0.71,-0.71)
\psline(-0.71,-0.71)(-0.71,-1.71)
\psline(-0.71,-1.71)(0.29,-1.71)
\psline(0.29,-1.71)(0.29,-0.71)
\psline(-1.41,1)(-0.71,1.71)
\psline(-0.71,1.71)(-0.71,2.71)
\psline(-0.71,2.71)(-1.41,3.41)
\psline(-1.41,3.41)(-2.41,3.41)
\psline(-2.41,3.41)(-3.12,2.71)
\psline(-3.12,2.71)(-3.12,1.71)
\psline(-3.12,1.71)(-2.41,1)
\psline(-2.41,1)(-1.41,1)
\psline(-0.71,-0.71)(-1.41,0)
\psline(-1.41,0)(-2.41,0)
\psline(-2.41,0)(-3.12,-0.71)
\psline(-3.12,-0.71)(-3.12,-1.71)
\psline(-3.12,-1.71)(-2.41,-2.41)
\psline(-2.41,-2.41)(-1.41,-2.41)
\psline(-1.41,-2.41)(-0.71,-1.71)
\psline(-0.71,-1.71)(-0.71,-0.71)
\psline(-1.41,1)(-2.41,1)
\psline(-2.41,1)(-2.41,0)
\psline(-2.41,0)(-1.41,0)
\psline(-1.41,0)(-1.41,1)
\psline[ArrowInside=->,linecolor=zzzzzz](1.5,-1.21)(1.5,0.5)
\psline[ArrowInside=->,linecolor=zzzzzz](1.5,2.21)(1.5,0.5)
\psline[ArrowInside=->,linecolor=zzzzzz](1.5,0.5)(-0.21,0.5)
\psline[ArrowInside=->,linecolor=zzzzzz](-0.21,0.5)(1.5,2.21)
\psline[ArrowInside=->,linecolor=zzzzzz](-0.21,0.5)(1.5,-1.21)
\psline[ArrowInside=->,linecolor=zzzzzz](1.5,-1.21)(-0.21,-1.21)
\psline[ArrowInside=->,linecolor=zzzzzz](-0.21,-1.21)(-0.21,0.5)
\psline[ArrowInside=->,linecolor=zzzzzz](-0.21,0.5)(-1.91,-1.21)
\psline[ArrowInside=->,linecolor=zzzzzz](-1.91,-1.21)(-0.21,-1.21)
\psline[ArrowInside=->,linecolor=zzzzzz](-1.91,-1.21)(-1.91,0.5)
\psline[ArrowInside=->,linecolor=zzzzzz](-1.91,0.5)(-0.21,0.5)
\psline[ArrowInside=->,linecolor=zzzzzz](-0.21,0.5)(-1.91,2.21)
\psline[ArrowInside=->,linecolor=zzzzzz](-1.91,2.21)(-1.91,0.5)
\psline[ArrowInside=->,linecolor=zzzzzz](-1.91,2.21)(-0.21,2.21)
\psline[ArrowInside=->,linecolor=zzzzzz](-0.21,2.21)(-0.21,0.5)
\psline[ArrowInside=->,linecolor=zzzzzz](1.5,2.21)(-0.21,2.21)
\psline[ArrowInside=->,linecolor=zzzzzz](-0.21,2.21)(-0.21,3.91)
\psline[ArrowInside=->,linecolor=zzzzzz](-0.21,3.91)(-1.91,2.21)
\psline[ArrowInside=->,linecolor=zzzzzz](-0.21,3.91)(1.5,2.21)
\psline[ArrowInside=->,linecolor=zzzzzz](1.5,0.5)(3.21,0.5)
\psline[ArrowInside=->,linecolor=zzzzzz](3.21,0.5)(1.5,2.21)
\psline[ArrowInside=->,linecolor=zzzzzz](3.21,0.5)(1.5,-1.21)
\psline[ArrowInside=->,linecolor=zzzzzz](-0.21,-1.21)(-0.21,-2.91)
\psline[ArrowInside=->,linecolor=zzzzzz](-0.21,-2.91)(1.5,-1.21)
\psline[ArrowInside=->,linecolor=zzzzzz](-0.21,-2.91)(-1.91,-1.21)
\psline[ArrowInside=->,linecolor=zzzzzz](-1.91,0.5)(-3.62,0.5)
\psline[ArrowInside=->,linecolor=zzzzzz](-3.62,0.5)(-1.91,-1.21)
\psline[ArrowInside=->,linecolor=zzzzzz](-3.62,0.5)(-1.91,2.21)
\psline[linecolor=zzzzzz](8.75,0)(9.75,0)
\psline[linecolor=zzzzzz](9.75,0)(9.75,1)
\psline[linecolor=zzzzzz](9.75,1)(8.75,1)
\psline[linecolor=zzzzzz](8.75,1)(8.75,0)
\psline[linecolor=zzzzzz](8.75,1)(9.75,1)
\psline[linecolor=zzzzzz](9.75,1)(10.46,1.71)
\psline[linecolor=zzzzzz](10.46,1.71)(10.46,2.71)
\psline[linecolor=zzzzzz](10.46,2.71)(9.75,3.41)
\psline[linecolor=zzzzzz](9.75,3.41)(8.75,3.41)
\psline[linecolor=zzzzzz](8.75,3.41)(8.05,2.71)
\psline[linecolor=zzzzzz](8.05,2.71)(8.05,1.71)
\psline[linecolor=zzzzzz](8.05,1.71)(8.75,1)
\psline[linecolor=zzzzzz](9.75,0)(8.75,0)
\psline[linecolor=zzzzzz](8.75,0)(8.05,-0.71)
\psline[linecolor=zzzzzz](8.05,-0.71)(8.05,-1.71)
\psline[linecolor=zzzzzz](8.05,-1.71)(8.75,-2.41)
\psline[linecolor=zzzzzz](8.75,-2.41)(9.75,-2.41)
\psline[linecolor=zzzzzz](9.75,-2.41)(10.46,-1.71)
\psline[linecolor=zzzzzz](10.46,-1.71)(10.46,-0.71)
\psline[linecolor=zzzzzz](10.46,-0.71)(9.75,0)
\psline[linecolor=zzzzzz](8.75,0)(8.75,1)
\psline[linecolor=zzzzzz](8.75,1)(8.05,1.71)
\psline[linecolor=zzzzzz](8.05,1.71)(7.05,1.71)
\psline[linecolor=zzzzzz](7.05,1.71)(6.34,1)
\psline[linecolor=zzzzzz](6.34,1)(6.34,0)
\psline[linecolor=zzzzzz](6.34,0)(7.05,-0.71)
\psline[linecolor=zzzzzz](7.05,-0.71)(8.05,-0.71)
\psline[linecolor=zzzzzz](8.05,-0.71)(8.75,0)
\psline[linecolor=zzzzzz](7.05,1.71)(8.05,1.71)
\psline[linecolor=zzzzzz](8.05,1.71)(8.05,2.71)
\psline[linecolor=zzzzzz](8.05,2.71)(7.05,2.71)
\psline[linecolor=zzzzzz](7.05,2.71)(7.05,1.71)
\psline[linecolor=zzzzzz](8.05,-0.71)(7.05,-0.71)
\psline[linecolor=zzzzzz](7.05,-0.71)(7.05,-1.71)
\psline[linecolor=zzzzzz](7.05,-1.71)(8.05,-1.71)
\psline[linecolor=zzzzzz](8.05,-1.71)(8.05,-0.71)
\psline[linecolor=zzzzzz](6.34,1)(7.05,1.71)
\psline[linecolor=zzzzzz](7.05,1.71)(7.05,2.71)
\psline[linecolor=zzzzzz](7.05,2.71)(6.34,3.41)
\psline[linecolor=zzzzzz](6.34,3.41)(5.34,3.41)
\psline[linecolor=zzzzzz](5.34,3.41)(4.63,2.71)
\psline[linecolor=zzzzzz](4.63,2.71)(4.63,1.71)
\psline[linecolor=zzzzzz](4.63,1.71)(5.34,1)
\psline[linecolor=zzzzzz](5.34,1)(6.34,1)
\psline[linecolor=zzzzzz](7.05,-0.71)(6.34,0)
\psline[linecolor=zzzzzz](6.34,0)(5.34,0)
\psline[linecolor=zzzzzz](5.34,0)(4.63,-0.71)
\psline[linecolor=zzzzzz](4.63,-0.71)(4.63,-1.71)
\psline[linecolor=zzzzzz](4.63,-1.71)(5.34,-2.41)
\psline[linecolor=zzzzzz](5.34,-2.41)(6.34,-2.41)
\psline[linecolor=zzzzzz](6.34,-2.41)(7.05,-1.71)
\psline[linecolor=zzzzzz](7.05,-1.71)(7.05,-0.71)
\psline[linecolor=zzzzzz](6.34,1)(5.34,1)
\psline[linecolor=zzzzzz](5.34,1)(5.34,0)
\psline[linecolor=zzzzzz](5.34,0)(6.34,0)
\psline[linecolor=zzzzzz](6.34,0)(6.34,1)
\psline[ArrowInside=->](9.25,-1.21)(9.25,0.5)
\psline[ArrowInside=->](9.25,2.21)(9.25,0.5)
\psline[ArrowInside=->](9.25,0.5)(7.55,0.5)
\psline[ArrowInside=->](7.55,0.5)(9.25,2.21)
\psline[ArrowInside=->](7.55,0.5)(9.25,-1.21)
\psline[ArrowInside=->](9.25,-1.21)(7.55,-1.21)
\psline[ArrowInside=->](7.55,-1.21)(7.55,0.5)
\psline[ArrowInside=->](7.55,0.5)(5.84,-1.21)
\psline[ArrowInside=->](5.84,-1.21)(7.55,-1.21)
\psline[ArrowInside=->](5.84,-1.21)(5.84,0.5)
\psline[ArrowInside=->](5.84,0.5)(7.55,0.5)
\psline[ArrowInside=->](7.55,0.5)(5.84,2.21)
\psline[ArrowInside=->](5.84,2.21)(5.84,0.5)
\psline[ArrowInside=->](5.84,2.21)(7.55,2.21)
\psline[ArrowInside=->](7.55,2.21)(7.55,0.5)
\psline[ArrowInside=->](9.25,2.21)(7.55,2.21)
\psline[ArrowInside=->](7.55,2.21)(7.55,3.91)
\psline[ArrowInside=->](7.55,3.91)(5.84,2.21)
\psline[ArrowInside=->](7.55,3.91)(9.25,2.21)
\psline[ArrowInside=->](9.25,0.5)(10.96,0.5)
\psline[ArrowInside=->](10.96,0.5)(9.25,2.21)
\psline[ArrowInside=->](10.96,0.5)(9.25,-1.21)
\psline[ArrowInside=->](7.55,-1.21)(7.55,-2.91)
\psline[ArrowInside=->](7.55,-2.91)(9.25,-1.21)
\psline[ArrowInside=->](7.55,-2.91)(5.84,-1.21)
\psline[ArrowInside=->](5.84,0.5)(4.13,0.5)
\psline[ArrowInside=->](4.13,0.5)(5.84,-1.21)
\psline[ArrowInside=->](4.13,0.5)(5.84,2.21)
\psdots(1,0)
\psdots[dotstyle=o](2,0)
\psdots(2,1)
\psdots[dotstyle=o](1,1)
\psdots[dotstyle=o](2.71,1.71)
\psdots(2.71,2.71)
\psdots[dotstyle=o](2,3.41)
\psdots(1,3.41)
\psdots[dotstyle=o](0.29,2.71)
\psdots[dotsize=2pt 0](0.29,1.71)
\psdots[dotstyle=o](0.29,-0.71)
\psdots(0.29,-1.71)
\psdots[dotstyle=o](1,-2.41)
\psdots(2,-2.41)
\psdots[dotstyle=o](2.71,-1.71)
\psdots(2.71,-0.71)
\psdots(0.29,1.71)
\psdots[dotstyle=o](-0.71,1.71)
\psdots(-1.41,1)
\psdots[dotstyle=o](-1.41,0)
\psdots(-0.71,-0.71)
\psdots[dotstyle=o](0.29,-0.71)
\psdots[dotstyle=o](0.29,2.71)
\psdots(-0.71,2.71)
\psdots[dotstyle=o](-0.71,-1.71)
\psdots(0.29,-1.71)
\psdots(-0.71,2.71)
\psdots[dotstyle=o](-1.41,3.41)
\psdots(-2.41,3.41)
\psdots[dotstyle=o](-3.12,2.71)
\psdots(-3.12,1.71)
\psdots[dotstyle=o](-2.41,1)
\psdots(-2.41,0)
\psdots[dotstyle=o](-3.12,-0.71)
\psdots(-3.12,-1.71)
\psdots[dotstyle=o](-2.41,-2.41)
\psdots(-1.41,-2.41)
\psdots[dotstyle=o](-0.71,-1.71)
\psdots(-2.41,0)
\psdots[dotstyle=o](-1.41,0)
\psdots(5.34,0)
\psdots[dotsize=2pt 0,linecolor=zzzzzz](-1.91,2.21)
\psdots[dotsize=2pt 0,linecolor=zzzzzz](-0.21,2.21)
\psdots[dotsize=2pt 0,linecolor=zzzzzz](1.5,2.21)
\psdots[dotsize=2pt 0,linecolor=zzzzzz](1.5,0.5)
\psdots[dotsize=2pt 0,linecolor=zzzzzz](-0.21,-1.21)
\psdots[dotsize=2pt 0,linecolor=zzzzzz](1.5,-1.21)
\psdots[dotsize=2pt 0,linecolor=zzzzzz](-0.21,0.5)
\psdots[dotsize=2pt 0,linecolor=zzzzzz](-1.91,-1.21)
\psdots[dotsize=2pt 0,linecolor=zzzzzz](-1.91,0.5)
\psdots[dotsize=2pt 0,linecolor=zzzzzz](-0.21,3.91)
\psdots[dotsize=2pt 0,linecolor=zzzzzz](3.21,0.5)
\psdots[dotsize=2pt 0,linecolor=zzzzzz](-0.21,-2.91)
\psdots[dotsize=2pt 0,linecolor=zzzzzz](-3.62,0.5)
\psdots[linecolor=zzzzzz](8.75,0)
\psdots[dotstyle=o, linecolor=zzzzzz](9.75,0)
\psdots[linecolor=zzzzzz](9.75,1)
\psdots[dotstyle=o, linecolor=zzzzzz](8.75,1)
\psdots[dotstyle=o, linecolor=zzzzzz](8.75,1)
\psdots[linecolor=zzzzzz](9.75,1)
\psdots[dotstyle=o, linecolor=zzzzzz](10.46,1.71)
\psdots[linecolor=zzzzzz](10.46,2.71)
\psdots[dotstyle=o, linecolor=zzzzzz](9.75,3.41)
\psdots[linecolor=zzzzzz](8.75,3.41)
\psdots[dotstyle=o, linecolor=zzzzzz](8.05,2.71)
\psdots[linecolor=zzzzzz](8.05,1.71)
\psdots[dotstyle=o, linecolor=zzzzzz](9.75,0)
\psdots[linecolor=zzzzzz](8.75,0)
\psdots[dotstyle=o, linecolor=zzzzzz](8.05,-0.71)
\psdots[linecolor=zzzzzz](8.05,-1.71)
\psdots[dotstyle=o, linecolor=zzzzzz](8.75,-2.41)
\psdots[linecolor=zzzzzz](9.75,-2.41)
\psdots[dotstyle=o, linecolor=zzzzzz](10.46,-1.71)
\psdots[linecolor=zzzzzz](10.46,-0.71)
\psdots[linecolor=zzzzzz](8.75,0)
\psdots[dotstyle=o, linecolor=zzzzzz](8.75,1)
\psdots[linecolor=zzzzzz](8.05,1.71)
\psdots[dotstyle=o, linecolor=zzzzzz](7.05,1.71)
\psdots[linecolor=zzzzzz](6.34,1)
\psdots[dotstyle=o, linecolor=zzzzzz](6.34,0)
\psdots[linecolor=zzzzzz](7.05,-0.71)
\psdots[dotstyle=o, linecolor=zzzzzz](8.05,-0.71)
\psdots[dotstyle=o, linecolor=zzzzzz](7.05,1.71)
\psdots[linecolor=zzzzzz](8.05,1.71)
\psdots[dotstyle=o, linecolor=zzzzzz](8.05,2.71)
\psdots[linecolor=zzzzzz](7.05,2.71)
\psdots[dotstyle=o, linecolor=zzzzzz](8.05,-0.71)
\psdots[linecolor=zzzzzz](7.05,-0.71)
\psdots[dotstyle=o, linecolor=zzzzzz](7.05,-1.71)
\psdots[linecolor=zzzzzz](8.05,-1.71)
\psdots[linecolor=zzzzzz](6.34,1)
\psdots[dotstyle=o, linecolor=zzzzzz](7.05,1.71)
\psdots[linecolor=zzzzzz](7.05,2.71)
\psdots[dotstyle=o, linecolor=zzzzzz](6.34,3.41)
\psdots[linecolor=zzzzzz](5.34,3.41)
\psdots[dotstyle=o, linecolor=zzzzzz](4.63,2.71)
\psdots[linecolor=zzzzzz](4.63,1.71)
\psdots[dotstyle=o, linecolor=zzzzzz](5.34,1)
\psdots[linecolor=zzzzzz](7.05,-0.71)
\psdots[dotstyle=o, linecolor=zzzzzz](6.34,0)
\psdots[linecolor=zzzzzz](5.34,0)
\psdots[dotstyle=o, linecolor=zzzzzz](4.63,-0.71)
\psdots[linecolor=zzzzzz](4.63,-1.71)
\psdots[dotstyle=o, linecolor=zzzzzz](5.34,-2.41)
\psdots[linecolor=zzzzzz](6.34,-2.41)
\psdots[dotstyle=o, linecolor=zzzzzz](7.05,-1.71)
\psdots[linecolor=zzzzzz](6.34,1)
\psdots[dotstyle=o, linecolor=zzzzzz](5.34,1)
\psdots[linecolor=zzzzzz](5.34,0)
\psdots[dotstyle=o, linecolor=zzzzzz](6.34,0)
\psdots[dotsize=2pt 0](9.25,-1.21)
\psdots[dotsize=2pt 0](9.25,0.5)
\psdots[dotsize=2pt 0](9.25,2.21)
\psdots[dotsize=2pt 0](9.25,0.5)
\psdots[dotsize=2pt 0](9.25,0.5)
\psdots[dotsize=2pt 0](7.55,0.5)
\psdots[dotsize=2pt 0](7.55,0.5)
\psdots[dotsize=2pt 0](9.25,2.21)
\psdots[dotsize=2pt 0](7.55,0.5)
\psdots[dotsize=2pt 0](9.25,-1.21)
\psdots[dotsize=2pt 0](9.25,-1.21)
\psdots[dotsize=2pt 0](7.55,-1.21)
\psdots[dotsize=2pt 0](7.55,-1.21)
\psdots[dotsize=2pt 0](7.55,0.5)
\psdots[dotsize=2pt 0](7.55,0.5)
\psdots[dotsize=2pt 0](5.84,-1.21)
\psdots[dotsize=2pt 0](5.84,-1.21)
\psdots[dotsize=2pt 0](7.55,-1.21)
\psdots[dotsize=2pt 0](5.84,-1.21)
\psdots[dotsize=2pt 0](5.84,0.5)
\psdots[dotsize=2pt 0](5.84,0.5)
\psdots[dotsize=2pt 0](7.55,0.5)
\psdots[dotsize=2pt 0](7.55,0.5)
\psdots[dotsize=2pt 0](5.84,2.21)
\psdots[dotsize=2pt 0](5.84,2.21)
\psdots[dotsize=2pt 0](5.84,0.5)
\psdots[dotsize=2pt 0](5.84,2.21)
\psdots[dotsize=2pt 0](7.55,2.21)
\psdots[dotsize=2pt 0](7.55,2.21)
\psdots[dotsize=2pt 0](7.55,0.5)
\psdots[dotsize=2pt 0](9.25,2.21)
\psdots[dotsize=2pt 0](7.55,2.21)
\psdots[dotsize=2pt 0](7.55,2.21)
\psdots[dotsize=2pt 0](7.55,3.91)
\psdots[dotsize=2pt 0](7.55,3.91)
\psdots[dotsize=2pt 0](5.84,2.21)
\psdots[dotsize=2pt 0](7.55,3.91)
\psdots[dotsize=2pt 0](9.25,2.21)
\psdots[dotsize=2pt 0](9.25,0.5)
\psdots[dotsize=2pt 0](10.96,0.5)
\psdots[dotsize=2pt 0](10.96,0.5)
\psdots[dotsize=2pt 0](9.25,2.21)
\psdots[dotsize=2pt 0](10.96,0.5)
\psdots[dotsize=2pt 0](9.25,-1.21)
\psdots[dotsize=2pt 0](7.55,-1.21)
\psdots[dotsize=2pt 0](7.55,-2.91)
\psdots[dotsize=2pt 0](7.55,-2.91)
\psdots[dotsize=2pt 0](9.25,-1.21)
\psdots[dotsize=2pt 0](7.55,-2.91)
\psdots[dotsize=2pt 0](5.84,-1.21)
\psdots[dotsize=2pt 0](5.84,0.5)
\psdots[dotsize=2pt 0](4.13,0.5)
\psdots[dotsize=2pt 0](4.13,0.5)
\psdots[dotsize=2pt 0](5.84,-1.21)
\psdots[dotsize=2pt 0](4.13,0.5)
\psdots[dotsize=2pt 0](5.84,2.21)
\end{pspicture*}
\end{center}


In \cite{Kenyon}, Kenyon and Schlenker prove a necessary and sufficient condition for the existence of a rhombus embedding, and thus in the dimer model case, for geometric consistency. They define a `train track' to be a path of quads (each quad being adjacent along an edge to the previous quad) which does not turn, i.e. for each quad in the train track, its shared edges with the previous and subsequent quads are opposite each other. Train tracks are assumed to extend in both directions as far as possible. The shaded quads in the example below, form part of a train track.
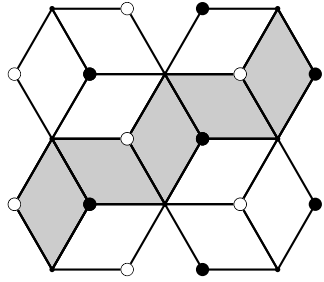
\begin{center}
\psset{xunit=1.0cm,yunit=1.0cm,algebraic=true,dotstyle=*,dotsize=5pt 0,linewidth=0.8pt,arrowsize=3pt 2,arrowinset=0.25}
\newrgbcolor{zzzzzz}{0.8 0.8 0.8}
\begin{pspicture*}(1,-5)(6,-0.3)
\pspolygon[fillcolor=zzzzzz,fillstyle=solid](3.5,-3.58)(3,-2.71)(3.5,-1.85)(4,-2.71)
\pspolygon[fillcolor=zzzzzz,fillstyle=solid](4,-2.71)(5,-2.71)(4.5,-1.85)(3.5,-1.85)
\pspolygon[fillcolor=zzzzzz,fillstyle=solid](2,-2.71)(2.5,-3.58)(3.5,-3.58)(3,-2.71)
\pspolygon[fillcolor=zzzzzz,fillstyle=solid](2,-2.71)(1.5,-3.58)(2,-4.45)(2.5,-3.58)
\pspolygon[fillcolor=zzzzzz,fillstyle=solid](5,-0.98)(4.5,-1.85)(5,-2.71)(5.5,-1.85)
\psline(3.5,-3.58)(3,-4.45)
\psline(3,-4.45)(2,-4.45)
\psline(2,-4.45)(2.5,-3.58)
\psline(2.5,-3.58)(3.5,-3.58)
\psline(3.5,-3.58)(4,-4.45)
\psline(4,-4.45)(5,-4.45)
\psline(5,-4.45)(4.5,-3.58)
\psline(4.5,-3.58)(3.5,-3.58)
\psline(3.5,-3.58)(3,-2.71)
\psline(3,-2.71)(3.5,-1.85)
\psline(3.5,-1.85)(4,-2.71)
\psline(4,-2.71)(3.5,-3.58)
\psline(4,-2.71)(5,-2.71)
\psline(5,-2.71)(4.5,-1.85)
\psline(4.5,-1.85)(3.5,-1.85)
\psline(3.5,-1.85)(4,-2.71)
\psline(5,-2.71)(4.5,-3.58)
\psline(4.5,-3.58)(3.5,-3.58)
\psline(3.5,-3.58)(4,-2.71)
\psline(4,-2.71)(5,-2.71)
\psline(5,-2.71)(5.5,-3.58)
\psline(5.5,-3.58)(5,-4.45)
\psline(5,-4.45)(4.5,-3.58)
\psline(4.5,-3.58)(5,-2.71)
\psline(2,-2.71)(3,-2.71)
\psline(3,-2.71)(3.5,-1.85)
\psline(3.5,-1.85)(2.5,-1.85)
\psline(2.5,-1.85)(2,-2.71)
\psline(2,-2.71)(2.5,-3.58)
\psline(2.5,-3.58)(3.5,-3.58)
\psline(3.5,-3.58)(3,-2.71)
\psline(3,-2.71)(2,-2.71)
\psline(2,-2.71)(1.5,-1.85)
\psline(1.5,-1.85)(2,-0.98)
\psline(2,-0.98)(2.5,-1.85)
\psline(2.5,-1.85)(2,-2.71)
\psline(2,-2.71)(1.5,-3.58)
\psline(1.5,-3.58)(2,-4.45)
\psline(2,-4.45)(2.5,-3.58)
\psline(2.5,-3.58)(2,-2.71)
\psline(3.5,-1.85)(3,-0.98)
\psline(3,-0.98)(2,-0.98)
\psline(2,-0.98)(2.5,-1.85)
\psline(2.5,-1.85)(3.5,-1.85)
\psline(3.5,-1.85)(4,-0.98)
\psline(4,-0.98)(5,-0.98)
\psline(5,-0.98)(4.5,-1.85)
\psline(4.5,-1.85)(3.5,-1.85)
\psline(5,-2.71)(5.5,-1.85)
\psline(5.5,-1.85)(5,-0.98)
\psline(2,-2.71)(1.5,-3.58)
\psline(1.5,-3.58)(2,-4.45)
\psline(2,-4.45)(2.5,-3.58)
\psline(2.5,-3.58)(2,-2.71)
\psline(5,-0.98)(4.5,-1.85)
\psline(4.5,-1.85)(5,-2.71)
\psline(5,-2.71)(5.5,-1.85)
\psline(5.5,-1.85)(5,-0.98)
\psdots[dotstyle=o](1.5,-3.58)
\psdots[dotstyle=o](1.5,-3.58)
\psdots(2.5,-3.58)
\psdots[dotstyle=o](3,-2.71)
\psdots(2.5,-1.85)
\psdots[dotstyle=o](1.5,-1.85)
\psdots(2.5,-1.85)
\psdots[dotstyle=o](3,-2.71)
\psdots(4,-2.71)
\psdots[dotstyle=o](4.5,-1.85)
\psdots(4,-0.98)
\psdots[dotstyle=o](3,-0.98)
\psdots(4,-2.71)
\psdots[dotstyle=o](3,-2.71)
\psdots(2.5,-3.58)
\psdots[dotstyle=o](3,-4.45)
\psdots(4,-4.45)
\psdots[dotstyle=o](4.5,-3.58)
\psdots(4,-2.71)
\psdots[dotstyle=o](4.5,-3.58)
\psdots(5.5,-3.58)
\psdots(5.5,-1.85)
\psdots[dotstyle=o](4.5,-1.85)
\psdots[dotsize=2pt 0](2,-2.71)
\psdots[dotsize=2pt 0](3.5,-3.58)
\psdots[dotsize=2pt 0](3.5,-1.85)
\psdots[dotsize=2pt 0](5,-2.71)
\psdots[dotsize=2pt 0](2,-0.98)
\psdots[dotsize=2pt 0](2,-4.45)
\psdots[dotsize=2pt 0](5,-4.45)
\psdots[dotsize=2pt 0](5,-0.98)
\end{pspicture*}

\end{center}

\begin{theorem}\label{Kenyon} \emph{(Theorem 5.1, \cite{Kenyon})} Suppose $G$ is a quad graph on a torus. Then $G$ has a rhombic embedding on a torus if and only if the following two conditions are satisfied:
\begin{enumerate}
\item Each train track is a simple closed curve.
\item The lift of two train tracks to the universal cover intersect at most once.
\end{enumerate}
\end{theorem}
The knowledge of a quad, and a pair of opposite edges is enough to completely determine a train track. Thus there are at most two train tracks containing any given quad, and since there are a finite number of quads on the torus, there are a finite number of train tracks. Therefore Theorem~\ref{Kenyon} gives a practical way of checking if a dimer model is geometrically consistent.

It will sometimes be more convenient to consider (and draw) paths, rather than paths of quads. For this reason we define the \emph{spine} of a train track $t$ to be the path which, on each quad of $t$, connects the mid-points of the opposite edges which are in adjacent quads in $t$. The diagram below shows part of a train track, with its spine dawn in grey:
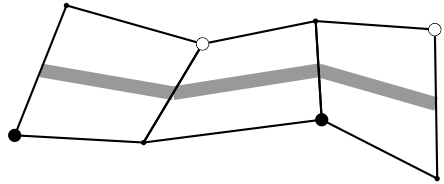
\begin{center}
\newrgbcolor{zzzzzz}{0.6 0.6 0.6}
\psset{xunit=0.8cm,yunit=0.8cm,algebraic=true,dotstyle=*,dotsize=5pt 0,linewidth=0.8pt,arrowsize=3pt 2,arrowinset=0.25}
\begin{pspicture*}(-2,0)(6.5,5.3)
\psline[linewidth=5.2pt,linecolor=zzzzzz](-1.03,3.1)(1.17,2.72)
\psline[linewidth=5.2pt,linecolor=zzzzzz](1.17,2.72)(3.59,3.1)
\psline[linewidth=5.2pt,linecolor=zzzzzz](3.59,3.1)(5.54,2.54)
\psline(-1.46,2.02)(-0.6,4.18)
\psline(-0.6,4.18)(1.66,3.54)
\psline(1.66,3.54)(0.68,1.9)
\psline(0.68,1.9)(-1.46,2.02)
\psline(1.66,3.54)(3.54,3.92)
\psline(3.54,3.92)(3.64,2.28)
\psline(3.64,2.28)(0.68,1.9)
\psline(0.68,1.9)(1.66,3.54)
\psline(3.64,2.28)(5.56,1.3)
\psline(5.56,1.3)(5.52,3.78)
\psline(5.52,3.78)(3.54,3.92)
\psline(3.54,3.92)(3.64,2.28)
\psdots[dotsize=2pt 0](-0.6,4.18)
\psdots(-1.46,2.02)
\psdots[dotsize=2pt 0](0.68,1.9)
\psdots[dotstyle=o](1.66,3.54)
\psdots[dotsize=2pt 0](3.54,3.92)
\psdots(3.64,2.28)
\psdots[dotstyle=o](5.52,3.78)
\psdots[dotsize=2pt 0](5.56,1.3)
\end{pspicture*}
\end{center}
The spine of a train track is a closed curve on the torus. We note that two train tracks intersect in a quad if and only if their spines intersect, and the intersection of two spines may be assumed to be transversal.

Returning to Example~\ref{examplestp} above, we draw part of the universal cover of the quad graph. The grey paths are lifts of the spines of two train tracks which can be seen to intersect more than once. Using Theorem~\ref{Kenyon}, this shows that the example is not geometrically consistent.
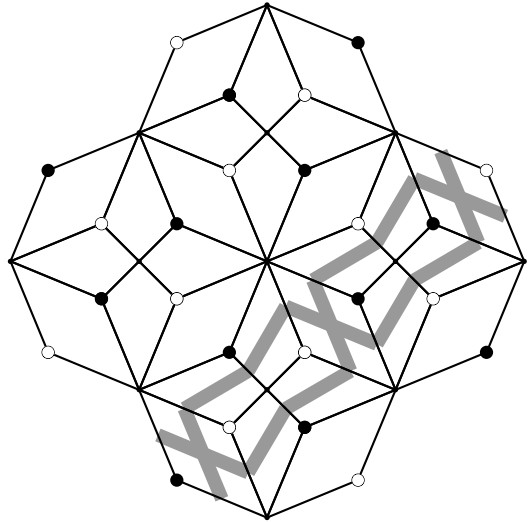
\begin{center}
\newrgbcolor{zzzzzz}{0.6 0.6 0.6}
\psset{xunit=1.0cm,yunit=1.0cm,algebraic=true,dotstyle=*,dotsize=5pt 0,linewidth=0.8pt,arrowsize=3pt 2,arrowinset=0.25}
\begin{pspicture*}(-4.5,-3.5)(4,4.5)
\psline[linewidth=5.2pt,linecolor=zzzzzz](-0.81,-2.66)(-1.31,-1.46)
\psline[linewidth=5.2pt,linecolor=zzzzzz](-1.31,-1.46)(-0.46,-0.96)
\psline[linewidth=5.2pt,linecolor=zzzzzz](-0.46,-0.96)(0.04,-0.1)
\psline[linewidth=5.2pt,linecolor=zzzzzz](0.04,-0.1)(1.25,-0.6)
\psline[linewidth=5.2pt,linecolor=zzzzzz](1.25,-0.6)(1.75,0.25)
\psline[linewidth=5.2pt,linecolor=zzzzzz](1.75,0.25)(2.6,0.75)
\psline[linewidth=5.2pt,linecolor=zzzzzz](2.6,0.75)(2.1,1.96)
\psline[linewidth=5.2pt,linecolor=zzzzzz](-1.66,-1.81)(-0.46,-2.31)
\psline[linewidth=5.2pt,linecolor=zzzzzz](-0.46,-2.31)(0.04,-1.46)
\psline[linewidth=5.2pt,linecolor=zzzzzz](0.04,-1.46)(0.9,-0.96)
\psline[linewidth=5.2pt,linecolor=zzzzzz](0.9,-0.96)(0.4,0.25)
\psline[linewidth=5.2pt,linecolor=zzzzzz](0.4,0.25)(1.25,0.75)
\psline[linewidth=5.2pt,linecolor=zzzzzz](1.25,0.75)(1.75,1.6)
\psline[linewidth=5.2pt,linecolor=zzzzzz](1.75,1.6)(2.96,1.1)
\psline(-0.21,0.5)(-0.71,1.71)
\psline(-0.71,1.71)(-0.21,2.21)
\psline(-0.21,2.21)(0.29,1.71)
\psline(0.29,1.71)(-0.21,0.5)
\psline(0.29,1.71)(1.5,2.21)
\psline(1.5,2.21)(1,1)
\psline(1,1)(-0.21,0.5)
\psline(-0.21,0.5)(0.29,1.71)
\psline(-0.21,0.5)(1,1)
\psline(1,1)(1.5,0.5)
\psline(1.5,0.5)(1,0)
\psline(1,0)(-0.21,0.5)
\psline(-0.21,0.5)(1,0)
\psline(1,0)(1.5,-1.21)
\psline(1.5,-1.21)(0.29,-0.71)
\psline(0.29,-0.71)(-0.21,0.5)
\psline(-0.21,0.5)(0.29,-0.71)
\psline(0.29,-0.71)(-0.21,-1.21)
\psline(-0.21,-1.21)(-0.71,-0.71)
\psline(-0.71,-0.71)(-0.21,0.5)
\psline(-0.71,-0.71)(-1.91,-1.21)
\psline(-1.91,-1.21)(-1.41,0)
\psline(-1.41,0)(-0.21,0.5)
\psline(-0.21,0.5)(-0.71,-0.71)
\psline(-1.41,0)(-1.91,0.5)
\psline(-1.91,0.5)(-1.41,1)
\psline(-1.41,1)(-0.21,0.5)
\psline(-0.21,0.5)(-1.41,0)
\psline(-1.41,1)(-1.91,2.21)
\psline(-1.91,2.21)(-0.71,1.71)
\psline(-0.71,1.71)(-0.21,0.5)
\psline(-0.21,0.5)(-1.41,1)
\psline(-0.21,2.21)(0.29,2.71)
\psline(0.29,2.71)(1.5,2.21)
\psline(1.5,2.21)(0.29,1.71)
\psline(0.29,1.71)(-0.21,2.21)
\psline(0.29,2.71)(-0.21,3.91)
\psline(-0.21,3.91)(1,3.41)
\psline(1,3.41)(1.5,2.21)
\psline(1.5,2.21)(0.29,2.71)
\psline(1.5,2.21)(2,1)
\psline(2,1)(1.5,0.5)
\psline(1.5,0.5)(1,1)
\psline(1,1)(1.5,2.21)
\psline(1.5,2.21)(2.71,1.71)
\psline(2.71,1.71)(3.21,0.5)
\psline(3.21,0.5)(2,1)
\psline(2,1)(1.5,2.21)
\psline(1.5,0.5)(2,0)
\psline(2,0)(3.21,0.5)
\psline(3.21,0.5)(2,1)
\psline(2,1)(1.5,0.5)
\psline(2,0)(1.5,-1.21)
\psline(1.5,-1.21)(1,0)
\psline(1,0)(1.5,0.5)
\psline(1.5,0.5)(2,0)
\psline(1.5,-1.21)(2.71,-0.71)
\psline(2.71,-0.71)(3.21,0.5)
\psline(3.21,0.5)(2,0)
\psline(2,0)(1.5,-1.21)
\psline(1.5,-1.21)(1,-2.41)
\psline(1,-2.41)(-0.21,-2.91)
\psline(-0.21,-2.91)(0.29,-1.71)
\psline(0.29,-1.71)(1.5,-1.21)
\psline(0.29,-1.71)(-0.21,-1.21)
\psline(-0.21,-1.21)(0.29,-0.71)
\psline(0.29,-0.71)(1.5,-1.21)
\psline(1.5,-1.21)(0.29,-1.71)
\psline(-0.21,-1.21)(-0.71,-1.71)
\psline(-0.71,-1.71)(-0.21,-2.91)
\psline(-0.21,-2.91)(0.29,-1.71)
\psline(0.29,-1.71)(-0.21,-1.21)
\psline(-0.71,-0.71)(-0.21,-1.21)
\psline(-0.21,-1.21)(-0.71,-1.71)
\psline(-0.71,-1.71)(-1.91,-1.21)
\psline(-1.91,-1.21)(-0.71,-0.71)
\psline(-2.41,0)(-1.91,0.5)
\psline(-1.91,0.5)(-1.41,0)
\psline(-1.41,0)(-1.91,-1.21)
\psline(-1.91,-1.21)(-2.41,0)
\psline(-1.91,0.5)(-2.41,1)
\psline(-2.41,1)(-3.62,0.5)
\psline(-3.62,0.5)(-2.41,0)
\psline(-2.41,0)(-1.91,0.5)
\psline(-1.91,-1.21)(-3.12,-0.71)
\psline(-3.12,-0.71)(-3.62,0.5)
\psline(-3.62,0.5)(-2.41,0)
\psline(-2.41,0)(-1.91,-1.21)
\psline(-1.91,-1.21)(-1.41,-2.41)
\psline(-1.41,-2.41)(-0.21,-2.91)
\psline(-0.21,-2.91)(-0.71,-1.71)
\psline(-0.71,-1.71)(-1.91,-1.21)
\psline(-2.41,1)(-1.91,2.21)
\psline(-1.91,2.21)(-1.41,1)
\psline(-1.41,1)(-1.91,0.5)
\psline(-1.91,0.5)(-2.41,1)
\psline(-3.62,0.5)(-3.12,1.71)
\psline(-3.12,1.71)(-1.91,2.21)
\psline(-1.91,2.21)(-2.41,1)
\psline(-2.41,1)(-3.62,0.5)
\psline(-1.91,2.21)(-0.71,2.71)
\psline(-0.71,2.71)(-0.21,2.21)
\psline(-0.21,2.21)(-0.71,1.71)
\psline(-0.71,1.71)(-1.91,2.21)
\psline(-0.71,2.71)(-0.21,3.91)
\psline(-0.21,3.91)(0.29,2.71)
\psline(0.29,2.71)(-0.21,2.21)
\psline(-0.21,2.21)(-0.71,2.71)
\psline(-1.41,3.41)(-1.91,2.21)
\psline(-1.91,2.21)(-0.71,2.71)
\psline(-0.71,2.71)(-0.21,3.91)
\psline(-0.21,3.91)(-1.41,3.41)
\psdots(1,0)
\psdots[dotstyle=o](2,0)
\psdots(2,1)
\psdots[dotstyle=o](1,1)
\psdots[dotstyle=o](2.71,1.71)
\psdots(1,3.41)
\psdots[dotstyle=o](0.29,2.71)
\psdots[dotsize=2pt 0](0.29,1.71)
\psdots[dotstyle=o](0.29,-0.71)
\psdots(0.29,-1.71)
\psdots[dotstyle=o](1,-2.41)
\psdots(2.71,-0.71)
\psdots(0.29,1.71)
\psdots[dotstyle=o](-0.71,1.71)
\psdots(-1.41,1)
\psdots[dotstyle=o](-1.41,0)
\psdots(-0.71,-0.71)
\psdots[dotstyle=o](0.29,-0.71)
\psdots[dotstyle=o](0.29,2.71)
\psdots(-0.71,2.71)
\psdots[dotstyle=o](-0.71,-1.71)
\psdots(0.29,-1.71)
\psdots(-0.71,2.71)
\psdots[dotstyle=o](-1.41,3.41)
\psdots(-3.12,1.71)
\psdots[dotstyle=o](-2.41,1)
\psdots(-2.41,0)
\psdots[dotstyle=o](-3.12,-0.71)
\psdots(-1.41,-2.41)
\psdots[dotstyle=o](-0.71,-1.71)
\psdots(-2.41,0)
\psdots[dotstyle=o](-1.41,0)
\psdots[dotsize=2pt 0](-1.91,2.21)
\psdots[dotsize=2pt 0](-0.21,2.21)
\psdots[dotsize=2pt 0](1.5,2.21)
\psdots[dotsize=2pt 0](1.5,0.5)
\psdots[dotsize=2pt 0](-0.21,-1.21)
\psdots[dotsize=2pt 0](1.5,-1.21)
\psdots[dotsize=2pt 0](-0.21,0.5)
\psdots[dotsize=2pt 0](-1.91,-1.21)
\psdots[dotsize=2pt 0](-1.91,0.5)
\psdots[dotsize=2pt 0](-0.21,3.91)
\psdots[dotsize=2pt 0](3.21,0.5)
\psdots[dotsize=2pt 0](-0.21,-2.91)
\psdots[dotsize=2pt 0](-3.62,0.5)
\end{pspicture*}
\end{center}

\section{Zig-zag flows} \label{ZZFsec}
Although we now have a way of checking if a given dimer model is geometrically consistent, this is done on the quad graph and we would prefer to return to the language of quivers which we used in previous chapters. We consider how the properties of train tracks transfer to the quiver. 

We saw in the previous section that there is a 1-1 correspondence between quads and arrows, for example the arrow corresponding to a quad is the one in the boundary of the quiver faces dual to the dimer vertices of the quad. 
Therefore, if two quads are adjacent along an edge $\{v,\face\}$, the corresponding arrows $a,b$ have $ha = v =tb$, and are both in the boundary of $\face$. Furthermore, since the dimer vertices on opposite edges of a quad, are different colours, the pairs of arrows corresponding to adjacent quads in a train track alternate between being in the boundary of black and white faces.

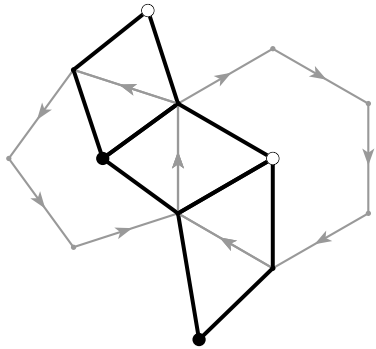
\begin{center}
\newrgbcolor{zzzzzz}{0.6 0.6 0.6}
\psset{xunit=1.0cm,yunit=1.0cm}
\begin{pspicture*}(-2,-1)(3,5)
\psset{xunit=1.0cm,yunit=1.0cm,algebraic=true,dotstyle=*,dotsize=3pt 0,linewidth=0.8pt,arrowsize=3pt 2,arrowinset=0.25}
\psline[ArrowInside=->, linecolor=zzzzzz](0.38,1.62)(0.38,3.08)
\psline[ArrowInside=->, linecolor=zzzzzz](0.38,3.08)(-1.01,3.53)
\psline[ArrowInside=->, linecolor=zzzzzz](-1.01,3.53)(-1.87,2.35)
\psline[ArrowInside=->, linecolor=zzzzzz](-1.87,2.35)(-1.01,1.17)
\psline[ArrowInside=->, linecolor=zzzzzz](-1.01,1.17)(0.38,1.62)
\psline[ArrowInside=->, linecolor=zzzzzz](0.38,1.62)(0.38,3.08)
\psline[ArrowInside=->, linecolor=zzzzzz](1.64,0.89)(0.38,1.62)
\psline[ArrowInside=->, linecolor=zzzzzz](2.91,1.62)(1.64,0.89)
\psline[ArrowInside=->, linecolor=zzzzzz](2.91,3.08)(2.91,1.62)
\psline[ArrowInside=->, linecolor=zzzzzz](1.64,3.81)(2.91,3.08)
\psline[ArrowInside=->, linecolor=zzzzzz](0.38,3.08)(1.64,3.81)
\psline[linewidth=1.5pt](-0.62,2.35)(0.38,1.62)
\psline[linewidth=1.5pt](0.38,1.62)(1.64,2.35)
\psline[linewidth=1.5pt](1.64,2.35)(0.38,3.08)
\psline[linewidth=1.5pt](0.38,3.08)(-0.62,2.35)
\psline[linewidth=1.5pt](1.64,2.35)(1.64,0.89)
\psline[linewidth=1.5pt](1.64,0.89)(0.66,-0.06)
\psline[linewidth=1.5pt](0.66,-0.06)(0.38,1.62)
\psline[linewidth=1.5pt](0.38,1.62)(1.64,2.35)
\psline[linewidth=1.5pt](-0.62,2.35)(-1.01,3.53)
\psline[linewidth=1.5pt](-1.01,3.53)(-0.02,4.32)
\psline[linewidth=1.5pt](-0.02,4.32)(0.38,3.08)
\psline[linewidth=1.5pt](0.38,3.08)(-0.62,2.35)
\psline[ArrowInside=->, linecolor=zzzzzz](1.64,0.89)(0.38,1.62)
\psline[ArrowInside=->, linecolor=zzzzzz](0.38,1.62)(0.38,3.08)
\psline[ArrowInside=->, linecolor=zzzzzz](0.38,3.08)(-1.01,3.53)
\psdots[dotsize=2pt 0](0.38,1.62)
\psdots[dotsize=2pt 0](0.38,3.08)
\psdots[dotsize=2pt 0](-1.01,3.53)
\psdots[dotsize=2pt 0,linecolor=zzzzzz](-1.87,2.35)
\psdots[dotsize=2pt 0,linecolor=zzzzzz](-1.01,1.17)
\psdots[dotsize=2pt 0](1.64,0.89)
\psdots[dotsize=2pt 0,linecolor=zzzzzz](2.91,1.62)
\psdots[dotsize=2pt 0,linecolor=zzzzzz](2.91,3.08)
\psdots[dotsize=2pt 0,linecolor=zzzzzz](1.64,3.81)
\psdots[dotsize=5pt 0](-0.62,2.35)
\psdots[dotsize=5pt 0,dotstyle=o](1.64,2.35)
\psdots[dotsize=5pt 0](0.66,-0.06)
\psdots[dotsize=5pt 0,dotstyle=o](-0.02,4.32)
\end{pspicture*}
\end{center}


\begin{definition}
A \emph{zig-zag path} $\eta$ is a map $\eta : \Z \lra{} Q_1$ such that,
\begin{compactenum}[i)]
\item $h\eta_n = t\eta_{n+1}$ for each $n \in \Z$.
\item $\eta_{2n}$ and $ \eta_{2n+1}$ are both in the boundary of the same black face and, $\eta_{2n-1}$ and $ \eta_{2n}$ are both in the boundary of the same white face.
\end{compactenum}
\end{definition}
We observe that shifting the indexing by an even integer, generates a different zig-zag path with the same image. We call this a reparametrisation of the path, and we will always consider paths up to reparametrisation.
\begin{definition}\label{zig} An arrow $a$ in a zig-zag path $\eta$ is called a zig (respectively a zag) of $\eta$ if it is the image of an even (respectively odd) integer.
\end{definition}
\noindent This is independent of the choice of parametrisation.

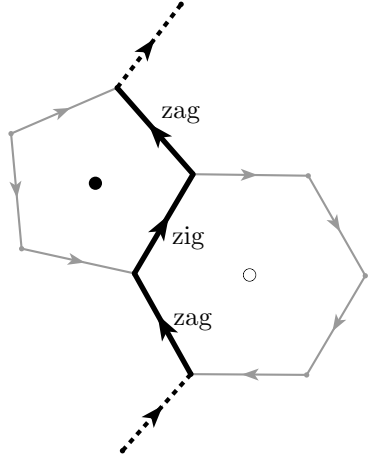
\begin{center}
\newrgbcolor{zzzzzz}{0.6 0.6 0.6}
\psset{xunit=1.0cm,yunit=1.0cm}
\begin{pspicture*}(-2,-1)(4,5.5)
\psset{xunit=1.0cm,yunit=1.0cm,algebraic=true,dotstyle=*,dotsize=5pt 0,linewidth=0.8pt,arrowsize=3pt 2,arrowinset=0.25}
\psline[ArrowInside=->,linecolor=zzzzzz](-1.26,3.48)(0.15,4.09)
\psline[ArrowInside=->,linecolor=zzzzzz](-1.26,3.48)(-1.12,1.95)
\psline[ArrowInside=->,linecolor=zzzzzz](-1.12,1.95)(0.38,1.62)
\psline(0.38,1.62)(1.13,0.28)
\psline[ArrowInside=->,linecolor=zzzzzz](2.67,0.27)(1.13,0.28)
\psline[ArrowInside=->,linecolor=zzzzzz](3.45,1.59)(2.67,0.27)
\psline[ArrowInside=->,linecolor=zzzzzz](2.69,2.92)(3.45,1.59)
\psline[ArrowInside=->,linecolor=zzzzzz](1.16,2.94)(2.69,2.92)
\psline[ArrowInside=->,linewidth=2pt](1.13,0.28)(0.38,1.62)
\psline[ArrowInside=->,linewidth=2pt](0.38,1.62)(1.16,2.94)
\psline[ArrowInside=->,linewidth=2pt](1.16,2.94)(0.15,4.09)
\psline[ArrowInside=->,linewidth=1.6pt,linestyle=dashed,dash=2pt 2pt](0.15,4.09)(1,5.2)
\psline[ArrowInside=->,linewidth=1.6pt,linestyle=dashed,dash=2pt 2pt](0.22,-0.74)(1.13,0.28)
\rput[tl](0.9,1.08){zag}
\rput[tl](0.88,2.24){zig}
\rput[tl](0.74,3.84){zag}
\psdots[dotsize=2pt 0](0.38,1.62)
\psdots[dotsize=2pt 0](1.16,2.94)
\psdots[dotsize=2pt 0](0.15,4.09)
\psdots[dotsize=2pt 0,linecolor=zzzzzz](-1.26,3.48)
\psdots[dotsize=2pt 0,linecolor=zzzzzz](-1.12,1.95)
\psdots[dotsize=2pt 0](1.13,0.28)
\psdots[dotsize=2pt 0,linecolor=zzzzzz](2.67,0.27)
\psdots[dotsize=2pt 0,linecolor=zzzzzz](3.45,1.59)
\psdots[dotsize=2pt 0,linecolor=zzzzzz](2.69,2.92)
\psdots[dotsize=5pt 0](-0.14,2.82)
\psdots[dotsize=5pt 0,dotstyle=o](1.91,1.6)
\psdots[dotsize=2pt 0](1,5.2)
\psdots[dotsize=2pt 0](0.22,-0.74)
\end{pspicture*}
\end{center}

We saw in the previous section that the knowledge of a quad, and a pair of opposite edges is enough information to completely determine a train track. Let $a$ be any arrow in $Q$ and suppose we decide that it is a zig (respectively zag). Then similarly,
we see that this is enough to uniquely determine a zig-zag path (up to reparametrisation). Thus every arrow is in at most two zig-zag paths. Furthermore since there are a finite number of arrows in $Q$, we see that all zig-zag paths are periodic.


We now turn our attention to the universal cover $\Qcov$ of $Q$, and maps $\zzf : \Z \lra{} \Qcov_1$ into this, which satisfy the same `zig-zag' property.

\begin{definition}\label{zzflo}
A \emph{zig-zag flow} $\zzf$ is a map $\zzf : \Z \lra{} \Qcov_1$ such that,
\begin{compactenum}[i)]
\item $h\zzf_n = t\zzf_{n+1}$ for each $n \in \Z$.
\item $\zzf_{2n}$ and $ \zzf_{2n+1}$ are both in the boundary of the same black face and, $\zzf_{2n-1}$ and $ \zzf_{2n}$ are both in the boundary of the same white face.
\end{compactenum}
\end{definition}
We define reparametrisation, zigs and zags in the same way as above, and observe that if we decide that an arrow $a \in \Qcov_1$ is a zig (respectively zag), then this is enough to uniquely determine a zig-zag flow (up to reparametrisation).
\begin{remark}
We use the terms zig-zag path and zig-zag flow in order to distinguish between objects on the quiver $Q$, and on its universal cover $\Qcov$. We note that composing a zig-zag flow with the projection from the universal cover $\Qcov$ to the quiver $Q$ produces a zig-zag path. The `zig-zag' property may also be characterised by saying that the path turns `maximally left' at a vertex, then `maximally right' at the following vertex, and then left again and so on. This is how it is defined in the physics literature, for example in \cite{Kenyonintro}. In this language, knowing that $a$ is a zig or a zag of a zig-zag path or flow, is equivalent to knowing whether the zig-zag path or flow turns left or right at the head of $a$.
\end{remark}


Since there exists a unique zig-zag path or flow containing any given arrow $a$ as a zig (respectively zag), we will refer to this as the zig-zag path or flow generated by the zig (respectively zag) $a$.
%
In particular, the lift of a single zig or zag of a zig-zag path $\eta$ to $\Qcov$, (remembering that it is a zig or zag), generates a zig-zag flow $\zzf$ which projects down to $\eta$.

Let $\zzf$ be a zig-zag flow and denote by $\eta$, the zig-zag path obtained by projecting this down to $Q$. Since $\eta$ is periodic, there is a well defined element $(\eta) \in \Z_{Q_1}$ which is the sum of the arrows in a single period. This is obviously closed and has a homology class $[\eta] \in H_1(Q)$ i.e. in the homology of the torus. Thus each zig-zag flow $\zzf$ (and zig-zag path $\eta$), has a corresponding homology class.

\begin{remark}
The 2-torus is the quotient of the plane by the fundamental group $\pi_1(T)$ which is isomorphic to $H_1(T)$ as it is abelian. The action is by deck transformations. Given a point $x$ on the plane and a homology class $\lambda$ we find a curve on the torus with this homology class which passes through the projection of $x$. We lift this curve to a path in the plane starting at $x$ and define $\lambda.x$ to be the end point. This depends only on the homology class, and not on the choice of curve. We note that in particular, the action of $[\eta] \in H_1(Q)$ on an arrow $\zzf_n$ in a representative zig-zag flow $\zzf$, is the arrow $\zzf_{n+ \varpi}$, where $\varpi$ is the length of one period of $\zzf$.
\end{remark}

We note that there is an obvious homotopy between the spine of a train track and the zig-zag path corresponding to that train track.
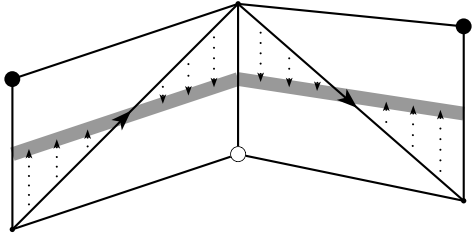
\begin{center}
\newrgbcolor{zzzzzz}{0.6 0.6 0.6}
\psset{xunit=1.0cm,yunit=1.0cm,algebraic=true,dotstyle=*,dotsize=6pt 0,linewidth=0.8pt,arrowsize=1pt 2,arrowinset=0.25}
\begin{pspicture*}(-3,-1)(5,4)
\psline[linewidth=5.2pt,linecolor=zzzzzz](-2,1)(1,2)
\psline[linewidth=5.2pt,linecolor=zzzzzz](1,2)(4,1.54)
\psline(-2,2)(-2,0)
\psline(-2,0)(1,1)
\psline(1,1)(1,3)
\psline(1,3)(-2,2)
\psline(1,1)(4,0.38)
\psline(4,0.38)(4,2.7)
\psline(4,2.7)(1,3)
\psline[ArrowInside=->,arrowsize=4pt 2](-2,0)(1,3)
\psline[ArrowInside=->,arrowsize=4pt 2](1,3)(4,0.38)
\psline[linestyle=dotted]{->}(-1.78,0.22)(-1.78,1.07)
\psline[linestyle=dotted]{->}(-1.41,0.59)(-1.41,1.2)
\psline[linestyle=dotted]{->}(-1,1)(-1,1.33)
\psline[linestyle=dotted]{->}(0,2)(-0,1.67)
\psline[linestyle=dotted]{->}(0.34,2.34)(0.34,1.78)
\psline[linestyle=dotted]{->}(0.68,2.68)(0.68,1.89)
\psline[linestyle=dotted]{->}(1.3,2.74)(1.3,1.96)
\psline[linestyle=dotted]{->}(1.68,2.4)(1.68,1.89)
\psline[linestyle=dotted]{->}(2.06,2.07)(2.05,1.84)
\psline[linestyle=dotted]{->}(2.97,1.28)(2.97,1.7)
\psline[linestyle=dotted]{->}(3.32,0.98)(3.33,1.65)
\psline[linestyle=dotted]{->}(3.69,0.65)(3.69,1.59)
\psdots(-2,2)
\psdots[dotsize=2pt 0](-2,0)
\psdots[dotsize=2pt 0](1,3)
\psdots[dotstyle=o](1,1)
\psdots[dotsize=2pt 0](4,0.38)
\psdots(4,2.7)
\end{pspicture*}
\end{center}
Therefore the spine of a train track and the corresponding zig-zag path have the same class in the homology of the torus.
\begin{remark} \label{intno}
We will often talk about the intersection between two zig-zag paths or flows. We have defined a zig-zag path or flow as a doubly infinite sequence of arrows, and as such, we will consider them to intersect if and only if they have an arrow in common. Thus there may exist zig-zag paths which share a vertex but not intersect. This is consistent with the train tracks definitions where two paths of quads intersect if and only if they have a quad in common which happens if and only if the spines of the train tracks intersect. Furthermore since the spines of zig-zag paths intersect each other transversally, each intersection of spines counts as $\pm 1$ in the intersection product of the corresponding homology classes. Thus each arrow in the intersection between two zig-zag paths contributes $\pm 1$ to the intersection number, and the intersection number is the sum of these contributions. In particular if $a$ is a zag of $[\eta]$ and a zig of $[\eta^\prime]$, then it contributes $+1$ to $[\eta] \wedge [\eta^\prime]$.
\end{remark}

We now give a proposition that describes necessary and sufficient conditions for a dimer model to be geometrically consistent in terms of zig-zag flows. We note that unlike Theorem~\ref{Kenyon}, all conditions are formulated on the universal cover.
\begin{proposition}\label{geomcons} A dimer model is geometrically consistent if and only if the following conditions hold.
\begin{compactenum} [(a)]
\item No zig-zag flow $\zzf$ in $\Qcov$ intersects itself, i.e. $\zzf$ is an injective map.
\item If $\zzf{{}}$ and $\zzf^{\prime}{{}}$ are zig-zag flows and $[\eta{{}}], [\eta^{\prime}{{}}] \in H_1{(Q)}$ are linearly independent, then they intersect in precisely one arrow.
\item If $\zzf{{}}$ and $\zzf^{\prime}{{}}$ are zig-zag flows and $[\eta{{}}], [\eta^{\prime}{{}}] \in H_1{(Q)}$ are linearly dependent, then they do not intersect.
\end{compactenum}
\end{proposition}
\begin{proof}
First we suppose conditions (a)-(c) hold, and prove that these imply conditions (1) and (2) of Theorem~\ref{Kenyon}. As there are a finite number of quads on the torus, all train tracks are closed. Therefore to prove condition (1), we need to show that they are simple. Suppose there is a train track which is not, and so intersects itself in a quad. This quad corresponds to an arrow in the quiver $Q$, which is contained in the corresponding zig-zag path as both a zig and a zag. Lifting this arrow to an arrow $a$ in the universal cover, we consider the flows generated by $a$ considered as a zig and as a zag. We either obtain one zig-zag flow which intersects itself, contradicting condition (a), or we obtain two zig-zag flows which intersect in the arrow $a$. By construction these project down to the same zig-zag path $\eta$, and therefore have the same corresponding homology class, but this contradicts condition (c). Therefore we have shown that condition (1) of Theorem~\ref{Kenyon} holds. Since the lift of a train track corresponds to a zig-zag flow, we observe that condition (2) of Theorem~\ref{Kenyon} follows directly from (b) and (c).

Now conversely, assuming that we have a geometrically consistent dimer model, so conditions (1) and (2) of Theorem~\ref{Kenyon} hold, we show that (a)-(c) hold.
First we consider the case where $\zzf{{}}$ and $\zzf^{\prime}{{}}$ are distinct zig-zag flows which project down to the same zig-zag path $\eta$, and show that they  don't intersect. If they did, then there is an arrow $a \in \Qcov$ which is (without loss of generality) a zig of $\zzf{{}}$ and a zag of $\zzf^{\prime}{{}}$. This projects down to an arrow which is both a zig and a zag of $\eta$. Thus $\eta$ intersects itself and the corresponding train track is not a simple closed curve. This contradicts (1), and so $\zzf{{}}$ and $\zzf^{\prime}{{}}$ do not intersect. This proves part of condition (c).

Now we prove that (a) holds. A zig-zag flow can intersect itself in two ways. If an arrow $a$ occurs as a zig and a zag in $\zzf$, then $\zzf$ projects down to a zig-zag path $\eta$ which contains the image of $a$ as a zig and a zag. Thus $\eta$ intersects itself and the corresponding train track is not a simple closed curve. This contradicts (1).

On the other hand, suppose an arrow $a$ occurs more than once as a zig (or zag) in $\zzf$. Then $\zzf$ is a periodic sequence of arrows, whose support is a finite closed path in $\Qcov$. In particular, the corresponding homology class $[\eta] = 0$.
Let $\zzf^{\prime}$ be the zig-zag flow generated by $a$ considered as a zag (or zig). We note that if $\zzf$ and $\zzf^{\prime}$ are the same then they contain $a$ as a zig and a zag which we showed could not happen in the case above. Therefore they must be distinct. Since $[\eta] = 0$, we see that $[\eta]$ and $[{\eta}^{\prime}]$ have zero intersection number. We know from Remark \ref{intno} that each arrow in the intersection of $\eta$ and ${\eta}^{\prime}$ contributes either $\pm 1$ to the intersection number, so they must intersect in an even number of arrows. They intersect in at least one, namely the image of $a$, and so they must intersect at least twice which contradicts (2).

Finally we complete the proof of (b) and (c). We may suppose $\zzf{{}}$ and $\zzf^{\prime}{{}}$ are zig-zag flows which project down to distinct zig-zag paths $\eta{{}}$ and $\eta^{\prime}{{}}$ in $Q$, as the other case of (c) was already done above. Then $\zzf{{}}$ and $\zzf^{\prime}{{}}$ correspond to the lifts of two distinct train tracks, and so by condition (2) they intersect at most once.

To prove (b) we suppose $\zzf{{}}$ and $\zzf^{\prime}{{}}$ are zig-zag flows with the property that $[\eta{{}}]$ and $[\eta^{\prime}{{}}]$ are linearly independent. If $v$ is any vertex of $\zzf{{}}$, then repeatedly applying the deck transformations corresponding to the homology classes $\pm[\eta{{}}]$ we obtain a sequence of vertices of $\zzf{{}}$ which lie on a line $\ell$ in the plane which has gradient given by $[\eta{{}}]$. In particular $\zzf$, and therefore the spine of the corresponding train track, lie within a bounded region of $\ell$. Similarly, the spine of the train track corresponding to $\zzf^{\prime}{{}}$ lies within a bounded region of a line $\ell^{\prime}$ which has gradient given by $[\eta^{\prime}{{}}]$. Since $[\eta{{}}]$ and $[\eta^{\prime}{{}}]$ are linearly independent, the lines $\ell$ and $\ell^{\prime}$ are not parallel and so intersect. Using the boundedness we see that the two spines must also intersect at least once. 
Thus $\zzf{{}}$ and $\zzf^{\prime}{{}}$ intersect at least once and therefore exactly once. This proves (b).

Finally suppose 
$[\eta{{}}]$ and $[\eta^{\prime}{{}}]$ are linearly dependent, so there exists a non-zero homology class $c=k[\eta{{}}]=k^\prime[\eta^{\prime}{{}}]$ which is a common multiple of $[\eta{{}}]$ and $[\eta^{\prime}{{}}]$. If $\zzf{{}}$ and $\zzf^{\prime}{{}}$ intersect in an arrow $a$, we consider applying the deck transformation corresponding to $c$ to this arrow. The image is an arrow which is both $k$ periods along $\zzf{{}}$ from $a$, and $k^\prime$ periods along $\zzf^\prime{{}}$ from $a$. Thus the image of $a$ is a second arrow which is contained in both $\zzf{{}}$ and $\zzf^{\prime}{{}}$. This is a contradiction, so $\zzf{{}}$ and $\zzf^{\prime}{{}}$ do not intersect.
This completes the proof of condition (c).

\end{proof}

\begin{remark} \label{primitiverem}
We note here that the property that a zig-zag flow $\zzf$ doesn't intersect itself (condition (a) of Proposition~\ref{geomcons}) implies that the homology class $[\eta]$ is non-zero.
Also, condition (1) of Theorem~\ref{Kenyon} implies that the homology class corresponding to any zig-zag flow is primitive (or zero). Thus the homology class corresponding to any zig-zag flow in a geometrically consistent dimer model is non-zero and primitive. This observation will be useful in the next chapter.
\end{remark}

\begin{remark} \label{leftrightzz}
Suppose we have a geometrically consistent dimer model and $\zzf$ and $\zzf^\prime$ are such that $[\eta]\wedge[\eta^\prime]>0$. Considering the lines $\ell$ and $\ell^{\prime}$ as in the proof of part (b) in the proposition above, we note that $\ell$ crosses $\ell^{\prime}$ from left to right. Since $\zzf$ and $\zzf^\prime$ lie within a bounded region of $\ell$ and $\ell^{\prime}$ respectively and $\zzf$ and $\zzf^\prime$ intersect exactly once, we see that $\zzf$ crosses $\zzf^\prime$ from left to right. Therefore this intersection is a zag of $\zzf$ and a zig of $\zzf^\prime$.
\end{remark}

\section{Constructing dimer models} \label{stiengul}
Given a dimer model, we saw in Section~\ref{Perfmatchsec} that we can associate to
it a lattice polygon, the ``perfect matching polygon", which is the
degree 1 slice of the cone $\Nm^+$ generated by perfect matchings. This
gives
the toric data for a Gorenstein affine toric threefold. It is natural to
ask whether it is possible to go in the opposite direction.
Given a Gorenstein affine toric threefold, or equivalently a lattice
polgon $V$, can we construct a dimer model for which $V$ is the perfect
matching polygon. Furthermore, can we construct it in such a way that
the dimer model satisfies some consistency conditions? In this section
we describe a method due to Gulotta which produces a geometrically
consistent dimer model for any given lattice polygon. 



\subsection{Gulotta's construction} \label{Gulottascon}
We give here an outline of the construction by Gulotta from \cite{Gulotta}, while referring the  
interested reader to either Gulotta's original paper or, for a slightly more detailed mathematical description, to  
Stienstra's paper \cite{Stienstra2}. In that paper Stienstra reformulates the construction as a (programmable) algorithm 
in terms of the adjacency matrices of the graphs. 

We have seen that the intersection properties of zig-zag flows on a dimer model play an important role and can be used to characterise geometric consistency. A key observation is that, given a configuration of oriented curves on the torus
satisfying certain properties, it is possible to write down a dimer model whose zig-zag paths (or the spines of the train tracks) intersect in the same way as the configuration of curves. In particular if the lifts of the curves to the universal cover also satisfy the `geometric consistency' intersection properties (cf. Proposition~\ref{geomcons}) then the constructed dimer model will be geometrically consistent.
For a full list of the desired properties of a `good' configurations of curves used in the algorithm, see the definition of ``a good pattern of zig-zags" in \cite{Stienstra2}. For the purpose of this overview we note that, as a consequence of these, a configuration must satisfy the following conditions together with the intersection properties needed for geometric consistency.

\begin{compactenum}
\item It induces a cell decomposition of the torus.
\item Each curve has a finite number of intersections with other curves and these intersections are transversal.
\item No three curves intersect in the same place.
\item Travelling along any given curve, its intersections with other curves occur with
alternating orientations, i.e. it is crossed from right to left and then
left to right etc..
\end{compactenum}

As a result we see that each point of intersection (0-cell) is in the boundary of exactly four
2-cells. Furthermore, it can be seen that the cell
decomposition has precisely three types of 2-cell. Those where the
boundary is consistently oriented
clockwise, those where the boundary is consistently oriented
anti-clockwise and those where the orientations around the boundary
alternate. Each 0-cell is in the boundary of two oriented
faces and two unoriented ones. In fact it can be checked that the dual cell
decomposition is a quad graph of the form discussed in Section~\ref{Rhombtiling}, where the clockwise, anti-clockwise and unoriented faces
correspond to white, black and quiver vertices respectively. From this we can easily write down the corresponding dimer model.
Therefore we have changed the problem to that of finding a `good' configuration of curves associated to each lattice polygon. 

\begin{remark} 
Given a lattice polygon in $H^1(T,\Z)$, Gulotta's algorithm actually gives a method of writing down a `good' configuration of curves with the property that the edges of the polygon are normal to the directions of the curves (as classes in $H_1(T,\Z)$)  and the number of curves in a given direction is given by the length of the corresponding edge. Given any dimer model we can write down this polygon whose edges are normal to the directions of the zig-zag paths. In general it is not true that it is the same as the perfect matching polygon. However we will see in Chapter~\ref{zzfpmchap} that this is the case for geometrically consistent dimer models. Since `good' configurations of curves lead to geometrically consistent dimer models, the distinction doesn't cause a problem here.
\end{remark}

The algorithm works as follows. First we observe that if the lattice polygon is a square of arbitary size, then it is not hard to write down a `good' configuration of curves. For example, if the polygon is the one shown on the left below, then there should be three curves with each of the classes $(1,0), (-1,0), (0,1), (0, -1)$ and these can be arranged as shown on the right.
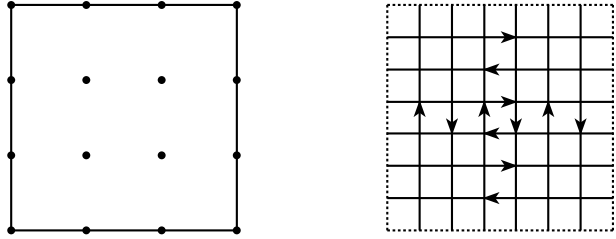
\begin{center}
\psset{xunit=1.0cm,yunit=1.0cm,algebraic=true,dotstyle=*,dotsize=3pt  
0,linewidth=0.8pt,arrowsize=3pt 2,arrowinset=0.25, ArrowInsidePos=0.55}
\begin{pspicture*}(-4,-0.6)(6,3.6)
\psline(-3,3)(0,3)
\psline(0,3)(0,0)
\psline(0,0)(-3,0)
\psline(-3,0)(-3,3)
\psline[linestyle=dashed,dash=1pt 1pt](2,3)(5,3)
\psline[linestyle=dashed,dash=1pt 1pt](5,3)(5,0)
\psline[linestyle=dashed,dash=1pt 1pt](5,0)(2,0)
\psline[linestyle=dashed,dash=1pt 1pt](2,0)(2,3)
\psline[ArrowInside=->](2.43,0)(2.43,3)
\psline[ArrowInside=->](2.86,3)(2.86,0)
\psline[ArrowInside=->](3.29,0)(3.29,3)
\psline[ArrowInside=->](3.71,3)(3.71,0)
\psline[ArrowInside=->](4.14,0)(4.14,3)
\psline[ArrowInside=->](4.57,3)(4.57,0)
\psline[ArrowInside=->](2,2.57)(5,2.57)
\psline[ArrowInside=->](5,2.14)(2,2.14)
\psline[ArrowInside=->](2,1.71)(5,1.71)
\psline[ArrowInside=->](5,1.29)(2,1.29)
\psline[ArrowInside=->](2,0.86)(5,0.86)
\psline[ArrowInside=->](5,0.43)(2,0.43)
\psdots(-3,3)
\psdots(-2,3)
\psdots(-1,3)
\psdots(0,3)
\psdots(-3,0)
\psdots(-3,1)
\psdots(-3,2)
\psdots(-2,1)
\psdots(-2,2)
\psdots(-1,1)
\psdots(-1,2)
\psdots(0,2)
\psdots(-2,0)
\psdots(-1,0)
\psdots(0,1)
\psdots(0,0)
\end{pspicture*}
\end{center}

Gulotta's algorithm transforms a `good' configuration of curves defining a dimer model into another one, in an explicit way that corresponds to removing a triangle from the lattice polygon. This is done in two stages. Firstly one considers two curves corresponding to the edges of the triangle that you wish to remove. These curves intersect in a unique point at which one applies the ``merging move". 
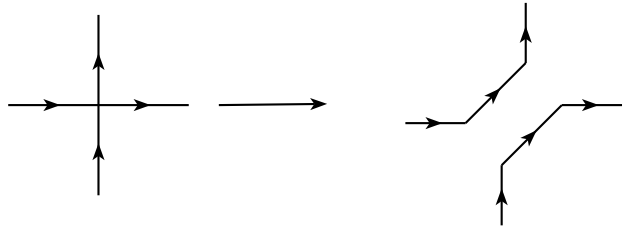
\begin{center}
\psset{xunit=0.8cm,yunit=0.8cm,algebraic=true,dotstyle=o,dotsize=3pt  
0,linewidth=0.8pt,arrowsize=3pt 2,arrowinset=0.25}
\newrgbcolor{zzzzzz}{0.8 0.8 0.8}
\begin{pspicture*}(-2.5,-1.5)(9,4)
\psline{->}(2,1)(3.8,1.02)
\psline[ArrowInside=->](5.1,0.7)(6.1,0.7)
\psline[ArrowInside=->](6.1,0.7)(7.1,1.7)
\psline[ArrowInside=->](7.1,1.7)(7.1,2.7)
\psline[ArrowInside=->](6.7,-1)(6.7,0)
\psline[ArrowInside=->](7.7,1)(8.7,1)
\psline[ArrowInside=->](6.7,0)(7.7,1)
\psline[ArrowInside=->](0,-0.5)(0,1)
\psline[ArrowInside=->](0,1)(0,2.5)
\psline[ArrowInside=->](-1.5,1)(0,1)
\psline[ArrowInside=->](0,1)(1.5,1)
\end{pspicture*}
\end{center}
The normal polygon corresponding to the new configuration of curves is the one obtained by removing the triangle from the previous polygon as required. However the new configuration of curves may not be good. In the second stage one applies a sequence of ``repairing moves", for example the following move which passes one curve through another removing two intersections. 
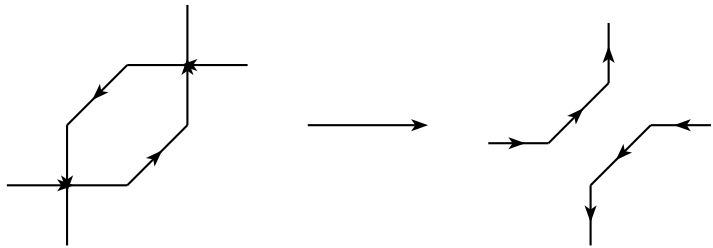
\begin{center}
\psset{xunit=0.8cm,yunit=0.8cm}
\begin{pspicture*}(-2.5,-1.5)(10,3.5)
\psset{xunit=0.8cm,yunit=0.8cm,algebraic=true,dotstyle=o,dotsize=3pt  
0,linewidth=0.8pt,arrowsize=3pt 2,arrowinset=0.25}
\psline[ArrowInside=->](-1,1)(-1,-1)
\psline[ArrowInside=->](0,2)(-1,1)
\psline[ArrowInside=->](2,2)(0,2)
\psline[ArrowInside=->](1,1)(1,3)
\psline[ArrowInside=->](0,0)(1,1)
\psline[ArrowInside=->](-2,0)(0,0)
\psline{->}(3,1)(5,1)
\psline[ArrowInside=->](6,0.7)(7,0.7)
\psline[ArrowInside=->](7,0.7)(8,1.7)
\psline[ArrowInside=->](8,1.7)(8,2.7)
\psline[ArrowInside=->](7.7,0)(7.7,-1)
\psline[ArrowInside=->](9.7,1)(8.7,1)
\psline[ArrowInside=->](8.7,1)(7.7,0)
\end{pspicture*}
\end{center}
This is the simplest of three types of repairing move. These repairing moves are defined in such a way that they do not affect the normal polygon. Moreover it is shown that following a finite sequence of one can obtain a `good' configuration of curves. 
Any lattice polygon can be ``cut out" from a sufficiently large square by removing triangles in this way. Therefore the algorithm allows one to associate a geometrically consistent dimer model to any lattice polygon.

\begin{remark}
In \cite{Ishii}, Ishii and Ueda generalise Gulotta's method and describe the process of changing a dimer model as one alters the polygon by removing a vertex and taking the convex hull of the remaining vertices.
The method is more complicated than the one described above and makes use of the special McKay correspondence. The resulting dimer model is not necessarily geometrically consistent but can be shown to satisfy the following weakend conditions.
\begin{compactenum}[(a)]
\item No zig-zag flow $\zzf$ in $\Qcov$ intersects itself in an arrow.
\item Suppose ${\zzf}$ and ${\zzf}^{\prime}$ are zig-zag flows which intersect more than once, and let $\zzf_{k_1} = {\zzf}^{\prime}_{n_1}$ and $\zzf_{k_2} = {\zzf}^{\prime}_{n_2}$ be any two arrows in their intersection. If $k_1 < k_2$, then $n_1 > n_2$.
\end{compactenum}
A dimer models satisfying these conditions might be considered to be ``marginally geometrically consistent'' and such models have been considered recently in both \cite{Ishii} and \cite{Bocklandt2}. However, since there is a geometrically consistent dimer model associated to each Gorenstein affine toric threefold, for the purposes of this article we can (and will) continue to work with this stronger definition.
\end{remark}

\begin{remark}
In an earlier paper by Stienstra \cite{Stienstra}, he proposed a different algorithm to produce geometrically consistent dimer models. This uses methods developed in the study of hypergeometric systems and dessins d'enfants. The algorithm associates a rhombus tiling and thus a geometrically consistent dimer model to certain rank 2 subgroups of $\Z^N$. (These subgroups can in turn be associated to lattice polygons via the `secondary polytope' construction.) The rhombus tiling is obtained as the projection of what Stienstra calls a `perfect surface'. Unfortunately, there is currently no proof that perfect surfaces exist in all cases and the method used to find such a surface in his examples is a trial and error approach.  
\end{remark}

\section{Some consequences of geometric consistency} \label{conseqgcsec}
We now look at some of the consequences of the geometric consistency condition. In Section~\ref{Quivalgdim} we explained the construction of the algebra $A$ which from an algebraic point of view is the output of a dimer model.
This is the quotient of the path algebra $\C Q$ by the ideal generated by `F-term relations'.
If a dimer model is geometrically consistent, a theorem of Hanany, Herzog and Vegh gives a concrete criterion for two paths to be F-term equivalent. In this section we set up the notation required and then state this result. This will also act as motivation for the definition of `non-commutative toric algebras' which we shall introduce in Chapter~\ref{ch:ToricA}.

Recall that every arrow $a \in Q_1$ occurs in precisely two oppositely oriented faces $\face^+ ,\face^- \in Q_2$. 
Then each F-term relation can be written explicitly as an equality of two paths $p_a^+ = p_a^-$ for some $a \in Q_1$,
where $p_a^\pm$ is the path from $ha$ around the boundary of $\face^\pm$ to $ta$. We say that two paths differ by a single F-term relation if they are of the form $q_1p_a^+q_2$ and $q_1p_a^-q_2$ for some paths $q_1, q_2$ and some $a \in Q_1$. 
The F-term relations generate an equivalence relation on paths in the quiver. 
\begin{definition}\label{Feq} Two paths $p_1$ and $p_2$ in $Q$ are F-term equivalent denoted $p_1 \sim_F p_2$ if there is a finite sequence of paths $p_1 = \sigma_0, \sigma_1, \dots , \sigma_k = p_2$ such that $\sigma_i$ and $\sigma_{i+1}$ differ by a single F-term relation for $i=0,\dots,k-1$.
\end{definition}
\noindent The F-term equivalence classes of paths form a natural basis for $A$.

We recall that the quiver $Q$ gives a cellular decomposition of the Riemann surface $\RS$ of a dimer model which, in consistent examples, is a 2-torus. Therefore there exist corresponding chain and cochain complexes
\begin{equation}
\label{eq:chaincom}
\Z_{Q_2} \lra{\del} \Z_{Q_1} \lra{\del} \Z_{Q_0}
\end{equation}
\begin{equation}
\label{eq:cochaincom}
\Z^{Q_0} \lra{d} \Z^{Q_1} \lra{d} \Z^{Q_2}
\end{equation}
In Section \ref{symmetries} we defined a sublattice $\Ng:= d^{-1}(\Z \const{1}) \subset \Z^{Q_1}$. We now wish to consider the dual lattice to $\Ng$ which we denote by $\Mg$. This is the quotient of $\Z_{Q_1}$ by the sub-lattice $\del(\const{1}^{\perp})$, where $\const{1}^{\perp} := \{ u \in \Z_{Q_2} \mid \langle u , \const{1} \rangle = 0 \}$. Since  $\del(\const{1}^{\perp})$ is contained in the kernel of the boundary map, this boundary map descends to a well defined map $\del : \Mg \rightarrow \Z_{Q_0}$ on the quotient.

Every path $p$ in $Q$ defines an element of $(p) \in \Z_{Q_1}$, obtained by simply summing up the arrows in the path. This in turn defines a class $\Meq{p} \in \Mg$, for which $\del\Meq{p}= hp -tp$. Now consider any F-term relation $p_a^+ = p_a^-$ for some $a \in Q_1$. Since $p_a^\pm$ is the path from $ha$ around the boundary of $\face^\pm$ to $ta$, we note that $(p_a^\pm) + (a) = \del\face^\pm$. In particular $(p_a^+) - (p_a^-) = \del(\face^+ - \face^-)$ in $ \Z_{Q_1}$, so $\Meq{p_a^+} = \Meq{p_a^-}$. Thus we see that it is necessary for F-term equivalent paths to have the same class in $\Mg$. The result due to Hanany, Herzog and Vegh says that for geometrically consistent dimer models, this condition is also sufficient i.e. if two paths have the same class in $\Mg$ then they are F-term equivalent.

Let $\Mm := \del^{-1}(0)$ and note that this is a rank 3 lattice which is dual to $\Nm$. Therefore it sits in the short exact sequence dual to (\ref{eq:Nmes})
\begin{equation} \label{eq:Mmes}
0 \lra{} \Z \lra{} \Mm \lra{H} H_1(T;\Z) \lra{} 0
\end{equation}
The kernel of $H$ is spanned by the class $\square = \del \Meq{\face} $ for any face $\face \in Q_2$. If $\Rsym$ is any R-symmetry then by definition $\langle \Rsym , \square \rangle = \deg{\Rsym} \neq 0$. Therefore, from (\ref{eq:Mmes}), we observe that $p$ and $q$ have the same class in $\Mg$, i.e. $\Meq{p} - \Meq{q}=0$
, precisely when they are \emph{homologous}, i.e. $H(\Meq{p} - \Meq{q}) = 0$, and they have the same weight under $\Rsym$, i.e. $\langle \Rsym , \Meq{p} \rangle = \langle \Rsym , \Meq{q} \rangle$. 

\begin{theorem} \label{Uniqueness}(Hanany, Herzog, Vegh, Lemma~5.3.1 in \cite{HHV}) For a geometrically consistent dimer model, two paths in $Q$ are F-term equivalent (i.e. represent the same element of $A$) if and only if they are homologous and have the same weight under a fixed R-symmetry.
\end{theorem}

\begin{remark}
In the terminology of \cite{HHV} homologous paths are called `homotopic' and paths which evaluate to the same integer on a fixed R-symmetry are said to be of the same `length'. A more detailed proof of this theorem is now given by Ishii and Ueda in \cite{Ishii}.
\end{remark}

\chapter{Zig-zag flows and perfect matchings} \label{zzfpmchap}
In Chapter~\ref{consistchap} we saw how properties of the intersections of zig-zag flows could be used to characterise geometric consistency. In this chapter we assume geometric consistency, and use this to prove some properties about the intersections between zig-zag flows and paths. We will be particularly interested in zig-zag flows which intersect the boundary of a given face. We define `zig-zag fans' which encode information about intersections of these flows, and use some of their properties to explicitly write down all the external and extremal perfect matchings (as defined in Section~\ref{Perfmatchsec}) of a geometrically consistent dimer model. We prove that the extremal perfect matchings have multiplicity one 
and that the multiplicities of the external perfect matchings are binomial coefficients. 

We start by recalling (Definition~\ref{zzflo}) that a zig-zag flow $\zzf$ is a doubly infinite path $\zzf : \Z \lra{} \Qcov_1$ such that,
$\zzf_{2n}$ and $ \zzf_{2n+1}$ are both in the boundary of the same black face and, $\zzf_{2n-1}$ and $ \zzf_{2n}$ are both in the boundary of the same white face.
%
We called an arrow $a$ in a zig-zag flow $\zzf$, a zig (respectively a zag) of $\zzf$ if it is the image of an even (respectively odd) integer.

\begin{remark}
For a geometrically consistent dimer model, a zig-zag flow does not intersect itself in an arrow. Therefore each arrow in a zig-zag flow $\zzf$ is either a zig or a zag of $\zzf$ but can not be both. Furthermore we recall that knowledge that an arrow is a zig (or a zag) in a zig-zag flow, is enough to uniquely determine that zig-zag flow. Thus every arrow is in precisely two zig-zag flows, and is a zig in one and a zag in the other.
\end{remark}

\section{Boundary flows} \label{boundflowsec}
We start by defining the `boundary flows' of a zig-zag flow $\zzf$. 
By definition, for each $n \in \Z$ we see that the `zig-zag pair' $\zzf_{2n}$ and $ \zzf_{2n+1}$ both lie in the boundary of a black face $\face_n$.
The other arrows in the boundary of this face form an oriented path $p^{(n)}$ from $h\zzf_{2n+1}=t\zzf_{2(n+1)}$ to $t\zzf_{2n}= h\zzf_{2n-1}$, i.e. a finite sequence of consecutive arrows.

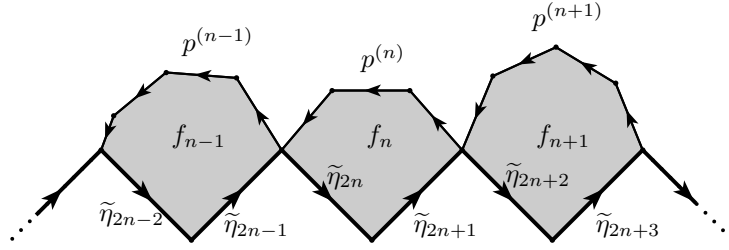
\begin{center}
\psset{xunit=1.2cm,yunit=1.2cm}
\begin{pspicture*}(3.5,-6)(12.5,-2)
\newrgbcolor{zzzzzz}{0.8 0.8 0.8}
\psset{xunit=1.2cm,yunit=1.2cm,algebraic=true,dotstyle=*,dotsize=3pt 0,linewidth=0.8pt,arrowsize=3pt 2,arrowinset=0.25}
\pspolygon[fillcolor=zzzzzz,fillstyle=solid](7,-4)(6.5,-3.2)(5.72,-3.14)(5.14,-3.62)(5,-4)(6,-5)
\pspolygon[fillcolor=zzzzzz,fillstyle=solid](9,-4)(8.42,-3.34)(7.56,-3.34)(7,-4)(8,-5)
\pspolygon[fillcolor=zzzzzz,fillstyle=solid](11,-4)(10.7,-3.26)(10.04,-2.86)(9.34,-3.18)(9,-4)(10,-5)
\psline[ArrowInside=->,linewidth=1.5pt](5,-4)(6,-5)
\psline[ArrowInside=->,linewidth=1.5pt](6,-5)(7,-4)
\psline[ArrowInside=->,linewidth=1.5pt](7,-4)(8,-5)
\psline[ArrowInside=->,linewidth=1.5pt](8,-5)(9,-4)
\psline[ArrowInside=->,linewidth=1.5pt](9,-4)(10,-5)
\psline[ArrowInside=->,linewidth=1.5pt](10,-5)(11,-4)
\psline[ArrowInside=->](7,-4)(6.5,-3.2)
\psline[ArrowInside=->](6.5,-3.2)(5.72,-3.14)
\psline[ArrowInside=->](5.72,-3.14)(5.14,-3.62)
\psline[ArrowInside=->](5.14,-3.62)(5,-4)
\psline[ArrowInside=->](5,-4)(6,-5)
\psline[ArrowInside=->](6,-5)(7,-4)
\psline[ArrowInside=->](9,-4)(8.42,-3.34)
\psline[ArrowInside=->](8.42,-3.34)(7.56,-3.34)
\psline[ArrowInside=->](7.56,-3.34)(7,-4)
\psline[ArrowInside=->](7,-4)(8,-5)
\psline[ArrowInside=->](8,-5)(9,-4)
\psline[ArrowInside=->](11,-4)(10.7,-3.26)
\psline[ArrowInside=->](10.7,-3.26)(10.04,-2.86)
\psline[ArrowInside=->](10.04,-2.86)(9.34,-3.18)
\psline[ArrowInside=->](9.34,-3.18)(9,-4)
\psline[ArrowInside=->](9,-4)(10,-5)
\psline[ArrowInside=->](10,-5)(11,-4)
\psline[ArrowInside=->,linewidth=1.5pt,linestyle=dotted](4,-5)(5,-4)
\psline[ArrowInside=->,linewidth=1.5pt,linestyle=dotted](11,-4)(12,-5)
\psline[linewidth=1.5pt](11,-4)(11.59,-4.59)
\psline[linewidth=1.5pt](4.28,-4.72)(5,-4)
\rput[tl](5.78,-3.72){$f_{n-1}$}
\rput[tl](7.94,-3.72){$f_n$}
\rput[tl](9.82,-3.72){$f_{n+1}$}
\rput[tl](4.98,-4.58){$\zzf_{2n-2}$}
\rput[tl](6.36,-4.7){$\zzf_{2n-1}$}
\rput[tl](7.5,-4.16){$\zzf_{2n}$}
\rput[tl](8.46,-4.7){$\zzf_{2n+1}$}
\rput[tl](9.48,-4.14){$\zzf_{2n+2}$}
\rput[tl](10.48,-4.7){$\zzf_{2n+3}$}
\rput[tl](7.88,-2.82){$p^{(n)}$}
\rput[tl](9.78,-2.36){$p^{(n+1)}$}
\rput[tl](5.9,-2.64){$p^{(n-1)}$}
\psdots[dotsize=2pt 0](5,-4)
\psdots[dotsize=2pt 0](6,-5)
\psdots[dotsize=2pt 0](7,-4)
\psdots[dotsize=2pt 0](8,-5)
\psdots[dotsize=2pt 0](9,-4)
\psdots[dotsize=2pt 0](10,-5)
\psdots[dotsize=2pt 0](11,-4)
\psdots[dotsize=2pt 0](6.5,-3.2)
\psdots[dotsize=2pt 0](5.72,-3.14)
\psdots[dotsize=2pt 0](5.14,-3.62)
\psdots[dotsize=2pt 0](8.42,-3.34)
\psdots[dotsize=2pt 0](7.56,-3.34)
\psdots[dotsize=2pt 0](10.7,-3.26)
\psdots[dotsize=2pt 0](10.04,-2.86)
\psdots[dotsize=2pt 0](9.34,-3.18)
\end{pspicture*} 
\end{center}
We piece these $p^{(n)}$ together in the obvious way, 
to construct a doubly infinite sequence of arrows which we call the black boundary flow of $\zzf$. By considering pairs of arrows $\zzf_{2n-1}$ and $ \zzf_{2n}$ which lie in the boundary of a white face for each $n \in \Z$, (i.e. the `zag-zig pairs') we can similarly construct the white boundary flow of $\zzf$. The two boundary flows of a zig-zag flow $\zzf$, project down to `boundary paths' of the zig-zag path $\eta$. The construction we have just described, is well defined on the quiver $Q$ and so the boundary paths can be constructed directly in $Q$. Since $\eta$ is periodic, it follows that each boundary path is also periodic.


We recall from Section~\ref{ZZFsec} that the image of a single period of $\eta$ is a closed cycle $(\eta) \in \Z_{Q_1}$ which defines a homology class $[\eta] \in H_{1}(Q)$. Similarly the image of a single period of a boundary path is a closed cycle in $\Z_{Q_1}$ and also defines a homology class.
Each `zig-zag pair' of $\eta$ is in the boundary of a unique black face. 
Summing these faces over a single period of $\eta$, defines an element in $\Z_{Q_2}$ whose boundary, by construction contains all the arrows in a period of $\eta$. If we subtract the cycle $(\eta) \in \Z_{Q_1}$ from this boundary, we obtain a cycle which is the sum of all arrows in the boundary of each of the faces except the zig-zag pair of $\eta$. Thus we see that by construction this cycle is the one corresponding to the black boundary flow of $\zzf$. The same argument follows through with white faces and zag-zig pairs.
Thus we have shown:
\begin{lemma} \label{boundhom}
For each zig-zag flow $\zzf$, the black and white boundary flows on $\Qcov$ project to boundary paths on $Q$ whose corresponding homology classes are both $-[\eta] \in H_{1}(Q)$.
\end{lemma}

\section{Some properties of zig-zag flows}
It will now be useful to consider the intersections of zig-zag flows and more general paths.
We consider a path in $\Qcov$ to be a (finite or infinite) sequence of arrows $(a_n)$ 
such that $ha_n = ta_{n+1}$. We will call a path \emph{simple} if it contains no repeated arrows. We will also consider `unoriented' paths, where we do not have to respect the orientation of the arrows. These can be considered as paths in the doubled quiver and where we denote the opposite arrow to $a$ by $a^{-1}$. We call an unoriented path \emph{simple} if it contains no arrow which is repeated with either orientation.
\begin{remark}
We note that the support of a simple path does not have to be a simple curve in the usual sense, and may intersect itself as long as the intersections occur at vertices.
\end{remark}

\begin{lemma} \label{zzclosedloop}
Let $p$ be a (possibly unoriented) finite simple closed path in $\Qcov$ and $\zzf$ be any zig-zag flow. If $p$ and $\zzf$ intersect in an arrow $a$, then they must intersect in at least two arrows.
\end{lemma}
\begin{proof}
Without loss of generality we can consider the case when the support of $p$ is a Jordan curve. Indeed, 
if the support of $p$ is not simple, there is a path $p^{\prime}$ containing $a$ whose support is a simple closed curve which is contained in the support of $p$. This is true because the curve only intersects itself at vertices of the quiver. The path $p^{\prime}$ can be constructed by considering the path starting at $a$, and each time a vertex is repeated, removing all the arrows between the repeated vertices. If $\zzf$ intersects $p^{\prime}$ in at least two arrows, it must intersect $p$ in at least these arrows.

Since the universal cover $\Qcov$ is planar we may apply the Jordan curve theorem which implies that $p$ is the boundary of the union of a finite number of faces. Let $F$ denote the set consisting of these faces. We observe that the arrows in the path $p$ are those which are contained in the boundary of precisely one face in $F$. We say that an arrow is `interior' if it is contained in the boundary of two distinct faces in $F$. In particular since $F$ contains a finite number of faces, there are a finite number of arrows in the interior.

Since $a=\zzf_k$ is an arrow in $p$, it is in the boundary of one face $\face \in F$. Using the zig-zag property, one of the arrows $\zzf_{k-1}$ or $\zzf_{k+1}$ in $\zzf$ must also be in the boundary of $\face$. Either this arrow in the path $p$, and we are done, or it lies in the interior.
Without loss of generality, assume that $\zzf_{k+1}$ lies in the interior. 
There are a finite number of arrows in the interior, and $\zzf$ is an infinite sequence of arrows which, by geometric consistency, never intersects itself. Therefore, 
%
%
there exists a first arrow $\zzf_{k^\prime}$ with $k^{\prime} > k+1$, which is not in the interior.

In particular the arrow $\zzf_{k^\prime-1}$ is in the interior, and therefore is in the boundary of two faces in $F$. Using the zig-zag property, we see that $\zzf_{k^\prime}$ is in the boundary of one of these faces and, because it is not in the interior, it must be an arrow in $p$. Finally, since $a=\zzf_k$ and $\zzf_{k^\prime}$ are both in $\zzf$, which does not intersect itself, they must be distinct.
\end{proof}

Now let $\face$ be a face in the universal cover $\Qcov$, and consider the intersections between zig-zag flows and the boundary of $\face$.
\begin{lemma}\label{singlein}
If $\zzf$ is a zig-zag flow which intersects the boundary of $\face$, then it intersects in exactly two arrows which form a zig-zag pair of $\zzf$ if $\face$ is black and a zag-zig pair if $\face$ is white.
\end{lemma}
\begin{proof}
Let $\zzf$ be a zig-zag flow that intersects the boundary of $\face$. We assume that $\face$ is black, the result for white faces follows by a symmetric argument. Using the zig-zag property we see that each intersection between $\zzf$ and the boundary of $\face$ must occur as a pair of arrows and by definition, this is a zig-zag pair of $\zzf$.

Suppose for a contradiction that some zig-zag flow intersects the boundary of face $\face$ in more than one pair. We choose such a flow $\zzf$ so that the number of arrows around the boundary of $\face$, from one zig-zag pair $(\zzf_0,\zzf_1)$ to another pair $(\zzf_k,\zzf_{k+1})$, where $k \geq 2$, is minimal (it could be zero). We note that minimality ensures that $\zzf$ does not intersect the part of the boundary of $\face$ between $\zzf_1$ and $\zzf_k$. The arrow $\zzf_1$ is a zig of a different zig-zag flow which we call $\zzf^{\prime}$, and without loss of generality $\zzf_1= \zzf^{\prime}_0$. By the zig-zag property, $\zzf^{\prime}_1$ is also in the boundary of $\face$.

If the number of arrows around the boundary of $\face$, between $h\zzf_1 =h\zzf^{\prime}_0$ and $t\zzf_k$ is zero then $\zzf^{\prime}_1 = \zzf_k$ and $\zzf$ and $\zzf^{\prime}$ intersect in at least two arrows, which contradicts geometric consistency.
Otherwise there are two finite oriented simple paths from $h\zzf_1 =h\zzf^{\prime}_0$ to $t\zzf_k$; the path $p$ along the boundary of $\face$ and the path $q$ which is part of $\zzf$. Then $pq^{-1}$ is a (finite) simple closed path which intersects $\zzf^{\prime}$ in the arrow $\zzf^{\prime}_0$.
Simplicity follows since $p$ and $q$ are both simple and $\zzf$ does not intersect $p$. Applying Lemma~\ref{zzclosedloop} we see that $\zzf^{\prime}$ must intersect $pq^{-1}$ in another arrow. It can not intersect $q$ otherwise $\zzf$ and $\zzf^{\prime}$ intersect in at least two arrows, so it must intersect $p$ in a second arrow. We show that this leads to a contradiction.

First note that it can not intersect $p$ in an arrow $\zzf^{\prime}_l$ for $l \geq 2$, as this would contradict the minimality condition above. Suppose it intersects in an arrow $\zzf^{\prime}_l$ for $l < 1$. Without loss of generality assume there are no other intersections of $\zzf^{\prime}$ with the boundary between $\zzf^{\prime}_1$ and $\zzf^{\prime}_l$. Then since $\zzf^{\prime}$ doesn't intersect itself, the path along the boundary from $h\zzf^{\prime}_1$ to $t\zzf^{\prime}_l$ followed by the path along $\zzf^{\prime}$ from $t\zzf^{\prime}_l$ to $h\zzf^{\prime}_1$ is a simple closed path. However this intersects $\zzf$ in the single arrow $\zzf_1= \zzf^{\prime}_0$, which contradicts Lemma~\ref{zzclosedloop}.

\end{proof}

\begin{corollary} \label{zzboundcor}
A zig-zag flow does not intersect its black (or white) boundary flow.
\end{corollary}

\section{Right and left hand sides}
Intuitively, one can see that since zig-zag flows do not intersect themselves in a geometrically consistent dimer model, any given zig-zag flow $\zzf$ splits the universal cover of the quiver into two pieces.
We formalise this idea by defining an equivalence relation on the vertices as follows.

\begin{lemma}\label{rlequiv}
There is an equivalence relation on $\Qcov_0$, where $i,j \in \Qcov_0$ are equivalent if and only if there exists a (possibly unoriented) finite path from $i$ to $j$ in $\Qcov$ which doesn't intersect $\zzf$ in any arrows. There are exactly two equivalence classes.
\end{lemma}
\begin{proof}
It is trivial to check that the equivalence relation is well defined. 
To prove that there are two equivalence classes we consider the following sets of vertices
\begin{align*} R(\zzf)&:= \{ v \in \Qcov_0 \mid v=ha \text{ where } a \text{ is a zig of } \zzf \} \\
L(\zzf)&:= \{ v \in \Qcov_0 \mid v=ta \text{ where } a \text{ is a zig of } \zzf \}
\end{align*}
The black and white boundary flows of $\zzf$ don't intersect $\zzf$ in an arrow by Corollary~\ref{zzboundcor}. We note that the black (white) flow is a path which passes through all the vertices of $L(\zzf)$ (respectively $R(\zzf)$). Thus all the vertices in $L(\zzf)$ (respectively $R(\zzf)$) lie in a single equivalence class. We show that these two classes are distinct. \\
If they were not, then for any zig $a$ of $\zzf$, there would exist a (possibly unoriented) finite path in $\Qcov$ from $ha \in R(\zzf)$ to $ta \in L(\zzf)$ which doesn't intersect $\zzf$ in an arrow. By removing the part of the path between any repeated arrows (including the arrows themselves if they occur with opposite orientations), we may assume the path is simple. We prove that such a path does not exist. If $p$ is any finite simple path from $ha$ to $ta$ which doesn't intersect $\zzf$, then in particular it doesn't contain $a$, so $ap$ is a finite simple closed path. Then as $\zzf$ intersects this in the arrow $a$, by Lemma~\ref{zzclosedloop} it must intersect in another distinct arrow. Thus $\zzf$ must intersect $p$ in an arrow which is a contradiction.\\
Finally we prove that the equivalence classes containing $L(\zzf)$ and $R(\zzf)$ are the only two equivalence classes. Let $v \in \Qcov_0$ be any vertex. The quiver $\Qcov$ is strongly connected so, if $a$ is some zig in $\zzf$, there exists a finite path from $v$ to $ha$. Either the path doesn't intersect $\zzf$ in an arrow, and we are done, or it does. Let $b$ be the first arrow in the intersection. Then by construction the path from $v$ to $tb$ doesn't intersect $\zzf$ and $tb$ is obviously in $L(\zzf)$ or $R(\zzf)$.
\end{proof}

\begin{definition}
We say a vertex $v \in \Qcov_0$ is on the left (right) of a zig-zag flow $\zzf$ if it is in the same equivalence class as an element of $L(\zzf)$ (respectively $R(\zzf)$) under the equivalence relation defined above.
\end{definition}
\begin{corollary}
Any path which passes from one side of a zig-zag flow to the other side, must intersect the flow in an arrow.
\end{corollary}

\begin{remark}
We note here that the above definition is consistent with the usual understanding of right and left, however it also gives a well defined notion of right or left to vertices which actually lie on the zig-zag flow itself. We recall that the lift of the spine of the corresponding train track (see the end of Section~\ref{Rhombtiling}) to the universal cover of the quad graph is a periodic path which doesn't pass through any quiver vertices. This splits $\Qcov_0$ into the same two classes.



\end{remark}

\section{Zig-zag fans}
We have seen that if a zig-zag flow intersects the boundary of a face, then it does so in a single zig-zag pair. We now consider the collection of zig-zag flows which intersect the boundary of a given face and construct fans in the integer homology lattice generated by the classes of these flows. We observe that these fans encode the information about the intersections of zig-zag flows which occur around the boundary of the face.

If $\face$ is some face of $\Qcov$, we define:
$$ \mathcal{X}(\face) = \{ \zzf \mid \zzf \text{ intersects the boundary of }\face \}.$$
For any zig-zag flow $\zzf$, recall (Remark~\ref{primitiverem}) that $[\eta]$ is a non-zero primitive homology class in the integer homology lattice $H_1(Q)$. We consider the collection of rays in $H_1(Q)$ generated by the classes $[\eta]$ where $\zzf \in \mathcal{X}(\face)$.

\begin{lemma}\label{spansconvex}
Given a ray $\gamma$ generated by $[\eta]$ where $\zzf \in \mathcal{X}(\face)$, there exists $\zzf^{\prime} \in \mathcal{X}(\face)$ such that the ray generated by $[\eta^{\prime}]$ is at an angle less than $\pi$ in an anti-clockwise direction from $\gamma$.
\end{lemma}
\begin{proof}
Since $\zzf \in \mathcal{X}(\face)$, by Lemma~\ref{singlein} it intersects the boundary of $\face$ in a zig and a zag. The zag of $\zzf$ is a zig of a uniquely defined zig-zag flow $\zzf^{\prime} \in \mathcal{X}(\face)$ which crosses $\zzf$ from right to left.
This intersection contributes $+1$ to the intersection number $ [\eta] \wedge [\eta^{\prime}]$, and therefore recalling Remark~\ref{leftrightzz} this implies that intersection number is strictly positive. Therefore the ray generated by $[\eta^{\prime}]$ is at an angle less than $\pi$ in an anti-clockwise direction from $\gamma$.
\end{proof}

\begin{definition}
The local zig-zag fan $\xi(\face)$ at the face $\face$ of $\Qcov$, is the complete fan of strongly convex rational polyhedral cones in $H_1(Q)$ whose rays are generated by the homology classes corresponding to zig-zag flows in $\mathcal{X}(\face)$.
\end{definition}

Since they are primitive, we note that if the homology classes corresponding to $\zzf$ and $\zzf^{\prime}$ are linearly dependent, then $[\eta]= \pm [\eta^{\prime}]$. We say $\zzf$ and $\zzf^{\prime}$ are \emph{parallel} if $[\eta]=  [\eta^{\prime}]$, and \emph{anti-parallel} if $[\eta]=  -[\eta^{\prime}]$.

\begin{lemma} \label{uniquerep}
If $\zzf$ and $\zzf^{\prime}$ are distinct parallel zig-zag flows, then at most one of them intersects the boundary of any given face.
\end{lemma}
\begin{proof}

Without loss of generality let $\face $ be a black face. Suppose $\zzf , \zzf^{\prime} \in \mathcal{X}(\face)$ are parallel. Since they don't intersect by Proposition~\ref{geomcons}, all vertices of $\zzf$ lie on the same side of $\zzf^{\prime} $ and vice-versa. Both $\zzf$ and $\zzf^{\prime}$ intersect the boundary of $\face$ in a zig-zag pair, and, by geometric consistency, these pairs of arrows are distinct. By considering paths around the boundary of $\face$ we observe that all vertices of $\zzf$ lie on the left of $\zzf^{\prime} $ and all vertices of $\zzf^{\prime}$ lie on the left of $\zzf$.
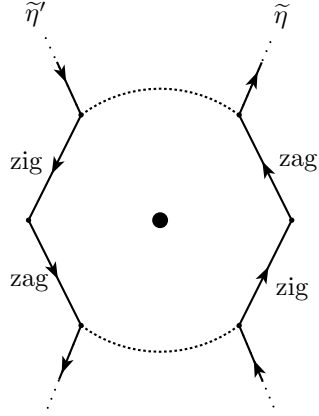
\begin{center}
\psset{xunit=0.7cm,yunit=0.7cm,runit=0.7cm,algebraic=true,dotstyle=*,dotsize=2pt 0,linewidth=0.8pt,arrowsize=3pt 2,arrowinset=0.25}
\begin{pspicture*}(-1,-2)(6,6.5)
\psarc[linestyle=dashed,dash=1pt 1pt](2.5,2){2.5}{53.13}{126.87}
\psarc[linestyle=dashed,dash=1pt 1pt](2.5,2){2.5}{233.13}{306.87}
\psline[ArrowInside=->](1,4)(0,2)
\psline[ArrowInside=->](0,2)(1,0)
\psline[ArrowInside=->](5,2)(4,4)
\psline[ArrowInside=->](4,0)(5,2)
\psline[ArrowInside=->, linestyle=dotted](4,4)(4.6,5.4)
\psline[ArrowInside=->,linestyle=dotted](4.66,-1.56)(4,0)
\psline[ArrowInside=->,linestyle=dotted](1,0)(0.36,-1.58)
\psline[ArrowInside=->,linestyle=dotted](0.32,5.48)(1,4)
\psline(1,4)(0.54,5.01)
\psline(4,4)(4.43,5)
\psline(1,0)(0.58,-1.03)
\psline(4,0)(4.45,-1.07)
\rput[tl](4.66,6.12){$\widetilde{\eta}$}
\rput[tl](-0.06,6.16){$\widetilde{\eta}^{\prime}$}
\rput[tl](4.7,0.92){zig}
\rput[lt](4.76,3.24){\parbox{1.15 cm}{zag}}
\rput[tl](-0.35,3.32){zig}
\rput[tl](-0.35,0.96){zag}
\psdots[dotsize=2pt 0](1,4)
\psdots[dotsize=2pt 0](0,2)
\psdots[dotsize=2pt 0](1,0)
\psdots[dotsize=2pt 0](4,4)
\psdots[dotsize=2pt 0](5,2)
\psdots[dotsize=2pt 0](4,0)
\psdots[dotsize=6pt 0](2.5,2)
\end{pspicture*} 
\end{center}
However, choosing any other zig-zag flow which is not parallel to $\zzf$, this intersects each of $\zzf$ and $\zzf^{\prime}$ in precisely one arrow. Recalling Remark~\ref{leftrightzz}, we see that these arrows are either both zigs of their respective flows, or both zags. Therefore, by considering the path along the zig-zag flow between these two arrows, we see that if all vertices of $\zzf$ lie on the left of $\zzf^{\prime}$ then all vertices of $\zzf^{\prime}$ lie on the right of $\zzf$. This is a contradiction.
\end{proof}

\begin{remark}
We have just shown that for every ray $\gamma$ of $\xi(\face)$, there is a unique representative zig-zag flow which intersects the boundary of $\face$, i.e. a unique flow which intersects the boundary of $\face$ and whose corresponding homology class is the generator of $\gamma$. Together with Lemma~\ref{singlein} this also implies that if a zig-zag \emph{path} intersects a face then it does so in a zig and a zag.
\end{remark}

\begin{proposition}\label{2cones}
Two rays $\gamma^+$ and $\gamma^-$ span a two dimensional cone in $\xi(\face)$ if and only if they have representative zig-zag flows $\zzf^+$ and $\zzf^-$ which intersect each other in the boundary of $\face$.
\end{proposition}
\begin{proof}

Without loss of generality we assume that $\face$ is black, as the general result follows by symmetry. Suppose zig-zag flows $\zzf^+$ and $\zzf^-$ intersect in arrow $a$ contained in the boundary of $\face$. Without loss of generality we assume that $\zzf^+$ crosses $\zzf^-$ from right to left. As in the proof of Lemma~\ref{spansconvex} we see that corresponding ray $\gamma^-$ is at an angle less than $\pi$ in a clockwise direction from $\gamma^+$.
We prove that $\gamma^+$ and $\gamma^-$ span a two dimensional cone in $\xi(\face)$.
If they do not, then there is a zig-zag flow $\zzf$ which intersects the boundary of $\face$ in zig-zag pair $(\zzf_0,\zzf_1)$ and whose ray in the local zig-zag fan lies strictly between $\gamma_1$ and $ \gamma_2$ (so $\zzf$ is not parallel or anti-parallel to $\zzf^+$ or $\zzf^-$.) Therefore $\zzf$ crosses $\zzf^-$ from right to left and crosses $\zzf^+$ from left to right. Let $b$ be the intersection of $\zzf$ and $\zzf^-$ which we note is a zig of $\zzf$.

By considering paths around the boundary of $\face$ we see that vertex $v:=h\zzf_0$ is on the left of both $\zzf^+$ and $\zzf^-$. Therefore it must occur in $\zzf$ before $\zzf$ intersects $\zzf_2$ but after $\zzf$ intersects $\zzf_1$.
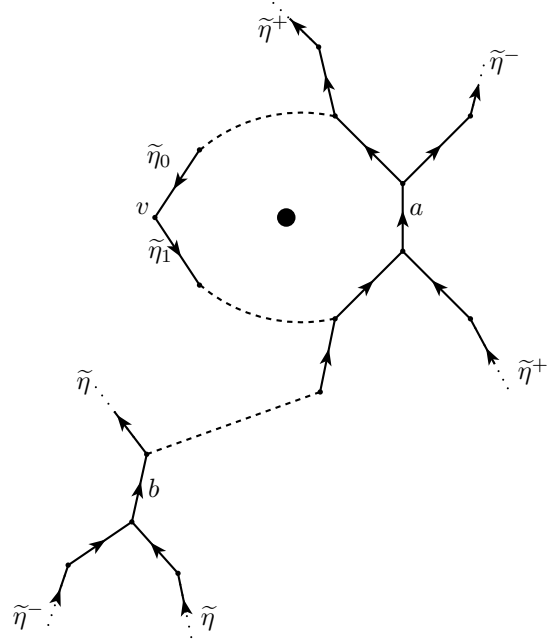
\begin{center}
\psset{xunit=0.9cm,yunit=0.9cm,runit=0.9cm,algebraic=true,dotstyle=*,dotsize=3pt 0,linewidth=0.8pt,arrowsize=3pt 2,arrowinset=0.25}
\begin{pspicture*}(-2,-4)(7,6)
\psline[ArrowInside=->](3,1)(4,2)
\psline[ArrowInside=->](4,2)(4,3)
\psline[ArrowInside=->](4,3)(3,4)
\psline[ArrowInside=->](4,3)(5,4)
\psline[ArrowInside=->](5,1)(4,2)
\psline[ArrowInside=->](1,3.5)(0.34,2.5)
\psline[ArrowInside=->](0.34,2.5)(1,1.5)
\psline[ArrowInside=->](2.78,-0.08)(3,1)
\psline[ArrowInside=->](3,4)(2.76,5.02)
\psarc[linestyle=dashed,dash=2pt 2pt](2.5,1.75){2.3}{77.47}{130.6}
\psarc[linestyle=dashed,dash=2pt 2pt](2.5,3.25){2.3}{229.4}{282.53}
\psline[ArrowInside=->](0,-2)(0.22,-1)
\psline[ArrowInside=->](0.68,-2.76)(0,-2)
\psline[ArrowInside=->](-0.94,-2.64)(0,-2)
\psline[linestyle=dashed,dash=2pt 2pt](0.22,-1)(2.78,-0.08)
\psline[ArrowInside=->,linestyle=dotted](0.22,-1)(-0.54,0)
\psline(0.22,-1)(-0.26,-0.37)
\psline[ArrowInside=->,linestyle=dotted](0.88,-3.72)(0.68,-2.76)
\psline(0.68,-2.76)(0.8,-3.34)
\psline[ArrowInside=->,linestyle=dotted](-1.24,-3.54)(-0.94,-2.64)
\psline(-1.11,-3.14)(-0.94,-2.64)
\psline[ArrowInside=->,linestyle=dotted](5,4)(5.2,4.76)
\psline(5,4)(5.11,4.43)
\psline[ArrowInside=->,linestyle=dotted](5.56,-0.04)(5,1)
\psline(5.36,0.33)(5,1)
\psline[ArrowInside=->,linestyle=dotted](2.76,5.02)(2.08,5.68)
\psline(2.76,5.02)(2.33,5.44)
\rput[tl](5.28,5.){$\widetilde{\eta}^-$}
\rput[tl](-1.76,-3.2){$\widetilde{\eta}^-$}
\rput[tl](4.1,2.7){${a}$}
\rput[tl](0.24,-1.36){${b}$}
\rput[tl](1.86,5.48){$\widetilde{\eta}^+$}
\rput[tl](5.7,0.42){$\widetilde{\eta}^+$}
\rput[tl](1.02,-3.24){$\widetilde{\eta}$}
\rput[tl](-0.82,0.24){$\widetilde{\eta}$}
\rput[tl](0.24,2.18){$\widetilde{\eta}_1$}
\rput[tl](0.22,3.58){$\widetilde{\eta}_0$}
\rput[tl](0.06,2.72){${v}$}
\psdots[dotsize=2pt 0](4,2)
\psdots[dotsize=2pt 0](4,3)
\psdots[dotsize=2pt 0](3,4)
\psdots[dotsize=2pt 0](3,1)
\psdots[dotsize=2pt 0](5,4)
\psdots[dotsize=2pt 0](5,1)
\psdots[dotsize=2pt 0](1,3.5)
\psdots[dotsize=2pt 0](1,1.5)
\psdots[dotsize=2pt 0](0.34,2.5)
\psdots[dotsize=7pt 0](2.28,2.5)
\psdots[dotsize=2pt 0](2.78,-0.08)
\psdots[dotsize=2pt 0](2.76,5.02)
\psdots[dotsize=2pt 0](0,-2)
\psdots[dotsize=2pt 0](0.22,-1)
\psdots[dotsize=2pt 0](0.68,-2.76)
\psdots[dotsize=2pt 0](-0.94,-2.64)
\end{pspicture*} 
\end{center}

In particular $\zzf$ crosses $\zzf^-$ before it crosses $\zzf^+$ and the vertex $hb$ must be on the left of $\zzf^+$. Since $\zzf^-$ crosses $\zzf^+$ from left to right, we see that $b$ occurs before $a$ in $\zzf^-$. The path from $hb$ to $ta$ along $\zzf_1$ doesn't intersect $\zzf$ (since geometric consistency implies that $b$ is the unique arrow where they intersect). Therefore $ta$ is on the right of $\zzf$. However there is a path around the boundary of $\face $ from $hb_1$ to $ta$ which doesn't intersect $\zzf$. This implies that $ta$ is on the left of $\zzf$. Therefore we have a contradiction and so $\gamma_1$ and $\gamma_2$ generate a two-dimensional cone in $\xi(\face)$, where $\gamma_1$ is the negative ray.

\end{proof}
\begin{remark} \label{wellorder} The previous two lemmas together imply that the zig-zag flows which intersect the boundary of a face $\face$, intersect it in pairs of arrows in such a way that the cyclic order of these pairs around the face is (up to orientation) the same as the cyclic order of the corresponding rays around the local zig-zag fan.
In \cite{Gulotta} Gulotta also observes that this is an important property. He calls a dimer model where this is satisfied at each face `properly ordered' and proves that this happens if and only if the number of quiver vertices is equal to twice the area of the polygon whose edges are normal to the directions of the zig-zag paths. He proposes this as an alternative `consistency condition'. 
\end{remark}
Given a two-dimensional cone $\sigma$ in the local zig-zag fan of some face $\face$, we call the unique arrow in the boundary of $\face$, which is the intersection of representative zig-zag flows of the rays of $\sigma$, the arrow corresponding to $\sigma$.


\begin{definition}\label{zzfan} The global zig-zag fan $\Xi$ in $H_1(Q) \otimes_\Z \R$, is the fan whose rays are generated by the homology classes corresponding to all the zig-zag flows on $\Qcov$.
\end{definition}
This is a refinement of each of the local zig-zag fans and is therefore a well defined fan. 
\begin{example} \label{memeg}
We illustrate with an example which we shall return to later in the chapter. Consider the following dimer model and corresponding quiver drawn together as before. The dotted line again marks a fundamental domain. 
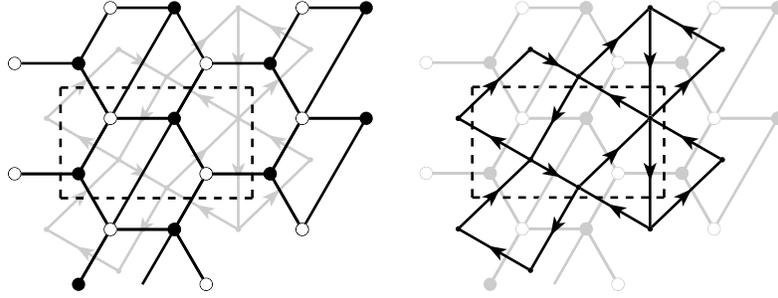
\begin{center}
\newrgbcolor{zzttqq}{0.8 0.8 0.8}
\newrgbcolor{wwwwff}{0 0 0}
\psset{xunit=0.85cm,yunit=0.85cm,algebraic=true,dotstyle=*,dotsize=5pt  0,linewidth=1pt,arrowsize=3pt 2,arrowinset=0.25}
\begin{pspicture*}(-1.5,1)(11,6)
\psline[ArrowInside=->, linecolor=zzttqq](0.43,3.32)(-0.7,3.97)
\psline[ArrowInside=->, linecolor=zzttqq](-0.7,2.23)(0.43,3.32)
\psline[ArrowInside=->, linecolor=zzttqq](-0.7,3.97)(0.43,5.05)
\psline[ArrowInside=->, linecolor=zzttqq](1.18,4.62)(0.43,3.32)
\psline[ArrowInside=->, linecolor=zzttqq](0.43,5.05)(1.18,4.62)
\psline[ArrowInside=->, linecolor=zzttqq](0.43,3.32)(1.18,2.88)
\psline[ArrowInside=->, linecolor=zzttqq](1.18,2.88)(2.3,3.97)
\psline[ArrowInside=->, linecolor=zzttqq](2.3,3.97)(1.18,4.62)
\psline[ArrowInside=->, linecolor=zzttqq](2.3,3.97)(2.3,2.23)
\psline[ArrowInside=->, linecolor=zzttqq](2.3,2.23)(1.18,2.88)
\psline[ArrowInside=->, linecolor=zzttqq](0.43,1.58)(-0.7,2.23)
\psline[ArrowInside=->, linecolor=zzttqq](1.18,2.88)(0.43,1.58)
\psline[ArrowInside=->, linecolor=zzttqq](2.3,5.7)(2.3,3.97)
\psline[ArrowInside=->, linecolor=zzttqq](1.18,4.62)(2.3,5.7)
\psline[ArrowInside=->, linecolor=zzttqq](2.3,3.97)(3.43,5.05)
\psline[ArrowInside=->, linecolor=zzttqq](3.43,5.05)(2.3,5.7)
\psline[ArrowInside=->, linecolor=zzttqq](3.43,3.32)(2.3,3.97)
\psline[ArrowInside=->, linecolor=zzttqq](2.3,2.23)(3.43,3.32)
\psline[linecolor=black](-1.2,3.1)(-0.2,3.1)
\psline[linecolor=black](-0.2,3.1)(0.3,3.97)
\psline[linecolor=black](0.3,3.97)(-0.2,4.83)
\psline[linecolor=black](-0.2,4.83)(-1.2,4.83)
\psline[linecolor=black](0.3,3.97)(-0.2,3.1)
\psline[linecolor=black](-0.2,3.1)(0.3,2.23)
\psline[linecolor=black](0.3,2.23)(1.3,2.23)
\psline[linecolor=black](1.3,2.23)(1.8,3.1)
\psline[linecolor=black](1.8,3.1)(1.3,3.97)
\psline[linecolor=black](1.3,3.97)(0.3,3.97)
\psline[linecolor=black](0.3,3.97)(1.3,3.97)
\psline[linecolor=black](1.3,3.97)(1.8,4.83)
\psline[linecolor=black](1.8,4.83)(1.3,5.7)
\psline[linecolor=black](1.3,5.7)(0.3,5.7)
\psline[linecolor=black](0.3,5.7)(-0.2,4.83)
\psline[linecolor=black](-0.2,4.83)(0.3,3.97)
\psline[linecolor=black](1.8,4.83)(1.3,3.97)
\psline[linecolor=black](1.3,3.97)(1.8,3.1)
\psline[linecolor=black](1.8,3.1)(2.8,3.1)
\psline[linecolor=black](2.8,3.1)(3.3,3.97)
\psline[linecolor=black](3.3,3.97)(2.8,4.83)
\psline[linecolor=black](2.8,4.83)(1.8,4.83)
\psline[linecolor=black](-0.2,3.1)(-1.2,3.1)
\psline[linecolor=black](-0.2,1.37)(0.3,2.23)
\psline[linecolor=black](0.3,2.23)(-0.2,3.1)
\psline[linecolor=black](2.8,3.1)(1.8,3.1)
\psline[linecolor=black](1.8,3.1)(1.3,2.23)
\psline[linecolor=black](1.3,2.23)(1.8,1.37)
\psline[linecolor=black](3.3,2.23)(2.8,3.1)
\psline[linecolor=black](0.3,3.97)(1.3,5.7)
\psline[linecolor=black](0.3,2.23)(1.3,3.97)
\psline[linestyle=dashed,dash=3pt 3pt](-0.48,4.45)(2.52,4.45)
\psline[linestyle=dashed,dash=3pt 3pt](2.52,4.45)(2.52,2.72)
\psline[linestyle=dashed,dash=3pt 3pt](2.52,2.72)(-0.48,2.72)
\psline[linestyle=dashed,dash=3pt 3pt](-0.48,2.72)(-0.48,4.45)
\psline[linecolor=black](1.8,1.37)(1.3,2.23)
\psline[linecolor=black](1.3,2.23)(0.3,2.23)
\psline[linecolor=black](0.3,2.23)(-0.2,1.37)
\psline[linecolor=black](1.3,2.23)(0.8,1.37)
\psline[linecolor=black](2.8,4.83)(3.3,3.97)
\psline[linecolor=black](3.3,3.97)(4.3,3.97)
\psline[linecolor=black](4.3,5.7)(3.3,5.7)
\psline[linecolor=black](3.3,5.7)(2.8,4.83)
\psline[linecolor=black](4.3,3.97)(3.3,3.97)
\psline[linecolor=black](3.3,3.97)(2.8,3.1)
\psline[linecolor=black](2.8,3.1)(3.3,2.23)
\psline[linecolor=black](4.3,5.7)(3.3,3.97)
\psline[linecolor=black](4.3,3.97)(3.3,2.23)
\psline[linecolor=zzttqq](5.24,3.11)(6.24,3.11)
\psline[linecolor=zzttqq](6.24,3.11)(6.74,3.97)
\psline[linecolor=zzttqq](6.74,3.97)(6.24,4.84)
\psline[linecolor=zzttqq](6.24,4.84)(5.24,4.84)
\psline[linecolor=zzttqq](6.74,3.97)(6.24,3.11)
\psline[linecolor=zzttqq](6.24,3.11)(6.74,2.24)
\psline[linecolor=zzttqq](6.74,2.24)(7.74,2.24)
\psline[linecolor=zzttqq](7.74,2.24)(8.24,3.11)
\psline[linecolor=zzttqq](8.24,3.11)(7.74,3.97)
\psline[linecolor=zzttqq](7.74,3.97)(6.74,3.97)
\psline[linecolor=zzttqq](6.74,3.97)(7.74,3.97)
\psline[linecolor=zzttqq](7.74,3.97)(8.24,4.84)
\psline[linecolor=zzttqq](8.24,4.84)(7.74,5.7)
\psline[linecolor=zzttqq](7.74,5.7)(6.74,5.7)
\psline[linecolor=zzttqq](6.74,5.7)(6.24,4.84)
\psline[linecolor=zzttqq](6.24,4.84)(6.74,3.97)
\psline[linecolor=zzttqq](8.24,4.84)(7.74,3.97)
\psline[linecolor=zzttqq](7.74,3.97)(8.24,3.11)
\psline[linecolor=zzttqq](8.24,3.11)(9.24,3.11)
\psline[linecolor=zzttqq](9.24,3.11)(9.74,3.97)
\psline[linecolor=zzttqq](9.74,3.97)(9.24,4.84)
\psline[linecolor=zzttqq](9.24,4.84)(8.24,4.84)
\psline[linecolor=zzttqq](6.24,3.11)(5.24,3.11)
\psline[linecolor=zzttqq](6.24,1.37)(6.74,2.24)
\psline[linecolor=zzttqq](6.74,2.24)(6.24,3.11)
\psline[linecolor=zzttqq](9.24,3.11)(8.24,3.11)
\psline[linecolor=zzttqq](8.24,3.11)(7.74,2.24)
\psline[linecolor=zzttqq](7.74,2.24)(8.24,1.37)
\psline[linecolor=zzttqq](9.74,2.24)(9.24,3.11)
\psline[linestyle=dashed,dash=3pt 3pt](5.96,4.46)(8.96,4.46)
\psline[linestyle=dashed,dash=3pt 3pt](8.96,4.46)(8.96,2.73)
\psline[linestyle=dashed,dash=3pt 3pt](8.96,2.73)(5.96,2.73)
\psline[linestyle=dashed,dash=3pt 3pt](5.96,2.73)(5.96,4.46)
\psline[linecolor=zzttqq](8.24,1.37)(7.74,2.24)
\psline[linecolor=zzttqq](7.74,2.24)(6.74,2.24)
\psline[linecolor=zzttqq](6.74,2.24)(6.24,1.37)
\psline[linecolor=zzttqq](9.24,4.84)(9.74,3.97)
\psline[linecolor=zzttqq](9.74,3.97)(10.74,3.97)
\psline[linecolor=zzttqq](10.74,5.7)(9.74,5.7)
\psline[linecolor=zzttqq](9.74,5.7)(9.24,4.84)
\psline[linecolor=zzttqq](10.74,3.97)(9.74,3.97)
\psline[linecolor=zzttqq](9.74,3.97)(9.24,3.11)
\psline[linecolor=zzttqq](9.24,3.11)(9.74,2.24)
\psline[linecolor=zzttqq](6.74,3.97)(7.74,5.7)
\psline[linecolor=zzttqq](6.74,2.24)(7.74,3.97)
\psline[ArrowInside=->, linecolor=wwwwff](6.87,3.32)(5.74,3.97)
\psline[ArrowInside=->, linecolor=wwwwff](5.74,2.24)(6.87,3.32)
\psline[ArrowInside=->, linecolor=wwwwff](5.74,3.97)(6.87,5.05)
\psline[ArrowInside=->, linecolor=wwwwff](7.62,4.62)(6.87,3.32)
\psline[ArrowInside=->, linecolor=wwwwff](6.87,5.05)(7.62,4.62)
\psline[ArrowInside=->, linecolor=wwwwff](6.87,3.32)(7.62,2.89)
\psline[ArrowInside=->, linecolor=wwwwff](7.62,2.89)(8.74,3.97)
\psline[ArrowInside=->, linecolor=wwwwff](8.74,3.97)(7.62,4.62)
\psline[ArrowInside=->, linecolor=wwwwff](8.74,3.97)(8.74,2.24)
\psline[ArrowInside=->, linecolor=wwwwff](8.74,2.24)(7.62,2.89)
\psline[linecolor=zzttqq](7.74,2.24)(7.24,1.37)
\psline[ArrowInside=->, linecolor=wwwwff](6.87,1.59)(5.74,2.24)
\psline[ArrowInside=->, linecolor=wwwwff](7.62,2.89)(6.87,1.59)
\psline[ArrowInside=->, linecolor=wwwwff](8.74,5.7)(8.74,3.97)
\psline[ArrowInside=->, linecolor=wwwwff](7.62,4.62)(8.74,5.7)
\psline[linecolor=zzttqq](10.74,5.7)(9.74,3.97)
\psline[linecolor=zzttqq](10.74,3.97)(9.74,2.24)
\psline[ArrowInside=->, linecolor=wwwwff](8.74,3.97)(9.87,5.05)
\psline[ArrowInside=->, linecolor=wwwwff](9.87,5.05)(8.74,5.7)
\psline[ArrowInside=->, linecolor=wwwwff](9.87,3.32)(8.74,3.97)
\psline[ArrowInside=->, linecolor=wwwwff](8.74,2.24)(9.87,3.32)
\psdots[dotstyle=o](-1.2,3.1)
\psdots[linecolor=black](-0.2,3.1)
\psdots[dotstyle=o](0.3,3.97)
\psdots[linecolor=black](-0.2,4.83)
\psdots[dotstyle=o](-1.2,4.83)
\psdots[dotstyle=o](0.3,2.23)
\psdots[linecolor=black](1.3,2.23)
\psdots[dotstyle=o](1.8,3.1)
\psdots[linecolor=black](1.3,3.97)
\psdots[dotstyle=o](1.8,4.83)
\psdots[linecolor=black](1.3,5.7)
\psdots[dotstyle=o](0.3,5.7)
\psdots[linecolor=black](-0.2,4.83)
\psdots[dotstyle=o](1.8,3.1)
\psdots[linecolor=black](2.8,3.1)
\psdots[dotstyle=o](3.3,3.97)
\psdots[linecolor=black](2.8,4.83)
\psdots[linecolor=black](-0.2,1.37)
\psdots[dotstyle=o](0.3,2.23)
\psdots[linecolor=black](1.3,2.23)
\psdots[dotstyle=o](1.8,1.37)
\psdots[dotstyle=o](3.3,2.23)
\psdots[dotsize=2pt 0, linecolor=zzttqq](-0.7,3.97)
\psdots[dotsize=2pt 0, linecolor=zzttqq](-0.7,2.23)
\psdots[dotsize=2pt 0, linecolor=zzttqq](2.3,2.23)
\psdots[dotsize=2pt 0, linecolor=zzttqq](2.3,3.97)
\psdots[dotsize=2pt 0, linecolor=zzttqq](0.43,3.32)
\psdots[dotsize=2pt 0, linecolor=zzttqq](1.18,2.88)
\psdots[dotsize=2pt 0, linecolor=zzttqq](0.43,5.05)
\psdots[dotsize=2pt 0, linecolor=zzttqq](1.18,4.62)
\psdots[dotsize=2pt 0, linecolor=zzttqq](0.43,1.58)
\psdots[dotsize=2pt 0, linecolor=zzttqq](2.3,5.7)
\psdots[linecolor=black](4.3,3.97)
\psdots[linecolor=black](4.3,5.7)
\psdots[dotstyle=o](3.3,5.7)
\psdots[linecolor=black](2.8,3.1)
\psdots[dotstyle=o](3.3,2.23)
\psdots[dotsize=2pt 0, linecolor=zzttqq](3.43,5.05)
\psdots[dotsize=2pt 0, linecolor=zzttqq](3.43,3.32)
\psdots[dotstyle=o, linecolor=zzttqq](9.74,2.24)
\psdots[dotsize=2pt 0, linecolor=zzttqq](5.24,3.11)
\psdots[linecolor=zzttqq](6.24,3.11)
\psdots[dotstyle=o, linecolor=zzttqq](6.74,3.97)
\psdots[dotsize=2pt 0, linecolor=zzttqq](6.24,4.84)
\psdots[dotstyle=o, linecolor=zzttqq](5.24,4.84)
\psdots[dotsize=2pt 0, linecolor=zzttqq](6.74,3.97)
\psdots[dotsize=2pt 0, linecolor=zzttqq](6.24,3.11)
\psdots[dotstyle=o, linecolor=zzttqq](6.74,2.24)
\psdots[linecolor=zzttqq](7.74,2.24)
\psdots[dotsize=2pt 0, linecolor=zzttqq](8.24,3.11)
\psdots[linecolor=zzttqq](7.74,3.97)
\psdots[dotsize=2pt 0, linecolor=zzttqq](6.74,3.97)
\psdots[dotsize=2pt 0, linecolor=zzttqq](7.74,3.97)
\psdots[dotsize=2pt 0, linecolor=zzttqq](8.24,4.84)
\psdots[linecolor=zzttqq](7.74,5.7)
\psdots[dotstyle=o, linecolor=zzttqq](6.74,5.7)
\psdots[linecolor=zzttqq](6.24,4.84)
\psdots[dotstyle=o, linecolor=zzttqq](8.24,4.84)
\psdots[dotsize=2pt 0, linecolor=zzttqq](7.74,3.97)
\psdots[dotsize=2pt 0, linecolor=zzttqq](8.24,3.11)
\psdots[linecolor=zzttqq](9.24,3.11)
\psdots[dotsize=2pt 0, linecolor=zzttqq](9.74,3.97)
\psdots[linecolor=zzttqq](9.24,4.84)
\psdots[linecolor=zzttqq](6.24,3.11)
\psdots[dotstyle=o, linecolor=zzttqq](5.24,3.11)
\psdots[linecolor=zzttqq](6.24,1.37)
\psdots[dotsize=2pt 0, linecolor=zzttqq](6.74,2.24)
\psdots[dotsize=2pt 0, linecolor=zzttqq](9.24,3.11)
\psdots[dotstyle=o, linecolor=zzttqq](8.24,3.11)
\psdots[dotsize=2pt 0, linecolor=zzttqq](7.74,2.24)
\psdots[dotsize=2pt 0, linecolor=zzttqq](8.24,1.37)
\psdots[dotstyle=o, linecolor=zzttqq](8.24,1.37)
\psdots[dotsize=2pt 0, linecolor=zzttqq](7.74,2.24)
\psdots[linecolor=zzttqq](9.24,4.84)
\psdots[dotsize=2pt 0, linecolor=zzttqq](9.74,3.97)
\psdots[linecolor=zzttqq](10.74,3.97)
\psdots[linecolor=zzttqq](10.74,5.7)
\psdots[dotstyle=o, linecolor=zzttqq](9.74,5.7)
\psdots[dotsize=2pt 0, linecolor=zzttqq](10.74,3.97)
\psdots[dotsize=2pt 0, linecolor=zzttqq](9.74,3.97)
\psdots[linecolor=zzttqq](9.24,3.11)
\psdots[dotstyle=o, linecolor=zzttqq](6.74,3.97)
\psdots[linecolor=zzttqq](7.74,5.7)
\psdots[dotstyle=o, linecolor=zzttqq](6.74,2.24)
\psdots[linecolor=zzttqq](7.74,3.97)
\psdots[dotsize=2pt 0, linecolor=wwwwff](5.74,3.97)
\psdots[dotsize=2pt 0, linecolor=wwwwff](6.87,3.32)
\psdots[dotsize=2pt 0, linecolor=wwwwff](5.74,2.24)
\psdots[dotsize=2pt 0, linecolor=zzttqq](6.87,5.05)
\psdots[dotsize=2pt 0, linecolor=wwwwff](5.74,3.97)
\psdots[dotsize=2pt 0, linecolor=zzttqq](7.62,4.62)
\psdots[dotsize=2pt 0, linecolor=zzttqq](7.62,4.62)
\psdots[dotsize=2pt 0, linecolor=wwwwff](6.87,5.05)
\psdots[dotsize=2pt 0, linecolor=zzttqq](7.62,2.89)
\psdots[dotsize=2pt 0, linecolor=zzttqq](8.74,3.97)
\psdots[dotsize=2pt 0, linecolor=zzttqq](7.62,2.89)
\psdots[dotsize=2pt 0, linecolor=zzttqq](7.62,4.62)
\psdots[dotsize=2pt 0, linecolor=zzttqq](8.74,3.97)
\psdots[dotsize=2pt 0, linecolor=zzttqq](8.74,2.24)
\psdots[dotsize=2pt 0, linecolor=zzttqq](8.74,3.97)
\psdots[dotsize=2pt 0, linecolor=zzttqq](7.62,2.89)
\psdots[dotsize=2pt 0, linecolor=zzttqq](8.74,2.24)
\psdots[linecolor=zzttqq](7.74,2.24)
\psdots[dotsize=2pt 0, linecolor=zzttqq](6.87,1.59)
\psdots[dotsize=2pt 0, linecolor=wwwwff](6.87,1.59)
\psdots[dotsize=2pt 0, linecolor=wwwwff](7.62,2.89)
\psdots[dotsize=2pt 0, linecolor=zzttqq](8.74,3.97)
\psdots[dotsize=2pt 0, linecolor=zzttqq](8.74,5.7)
\psdots[dotsize=2pt 0, linecolor=zzttqq](8.74,5.7)
\psdots[dotsize=2pt 0, linecolor=wwwwff](7.62,4.62)
\psdots[linecolor=zzttqq](10.74,5.7)
\psdots[dotstyle=o, linecolor=zzttqq](9.74,3.97)
\psdots[linecolor=zzttqq](10.74,3.97)
\psdots[dotsize=2pt 0, linecolor=zzttqq](8.74,3.97)
\psdots[dotsize=2pt 0, linecolor=zzttqq](9.87,5.05)
\psdots[dotsize=2pt 0, linecolor=wwwwff](9.87,5.05)
\psdots[dotsize=2pt 0, linecolor=wwwwff](8.74,5.7)
\psdots[dotsize=2pt 0, linecolor=wwwwff](8.74,3.97)
\psdots[dotsize=2pt 0, linecolor=zzttqq](9.87,3.32)
\psdots[dotsize=2pt 0, linecolor=wwwwff](9.87,3.32)
\psdots[dotsize=2pt 0, linecolor=wwwwff](8.74,2.24)
\end{pspicture*}
\end{center}
In this example there are five zig-zag paths, which we label $\eta_1, \dots, \eta_5$. The following diagrams show zig-zag flows $\widetilde{\eta_1}, \dots, \widetilde{\eta_5}$ which are possible lifts of $\eta_1, \dots, \eta_5$ respectively.
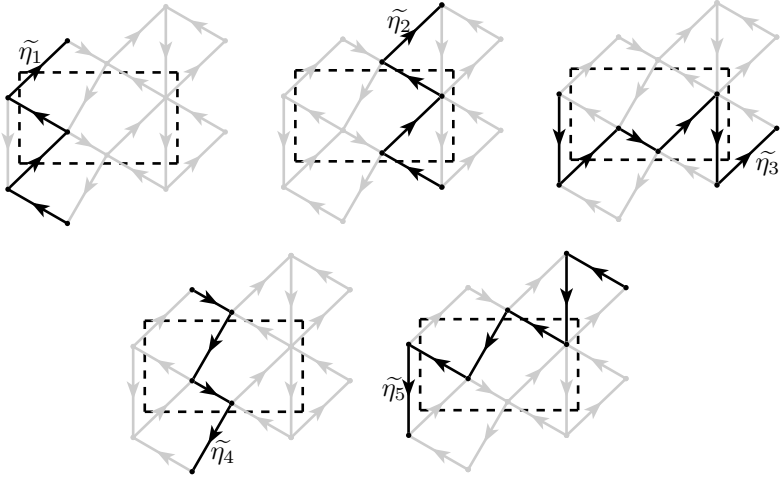
\begin{center}
\newrgbcolor{zzttqq}{0.8 0.8 0.8}
\psset{xunit=0.7cm,yunit=0.7cm,algebraic=true,dotstyle=*,dotsize=3pt  0,linewidth=1pt,arrowsize=3pt 2,arrowinset=0.25}
\begin{pspicture*}(-1,-4)(14.3,6.5)
\psline[linestyle=dashed,dash=3pt 3pt](2.52,4.45)(-0.48,4.45)
\psline[linestyle=dashed,dash=3pt 3pt](2.52,2.72)(2.52,4.45)
\psline[linestyle=dashed,dash=3pt 3pt](2.52,2.72)(-0.48,2.72)
\psline[linestyle=dashed,dash=3pt 3pt](-0.48,2.72)(-0.48,4.45)
\psline[ArrowInside=->, linecolor=black](0.43,3.32)(-0.7,3.97)
\psline[ArrowInside=->, linecolor=black](-0.7,2.23)(0.43,3.32)
\psline[ArrowInside=->, linecolor=black](-0.7,3.97)(0.43,5.05)
\psline[ArrowInside=->, linecolor=zzttqq](1.18,4.62)(0.43,3.32)
\psline[ArrowInside=->, linecolor=zzttqq](0.43,5.05)(1.18,4.62)
\psline[ArrowInside=->, linecolor=zzttqq](0.43,3.32)(1.18,2.88)
\psline[ArrowInside=->, linecolor=zzttqq](1.18,2.88)(2.3,3.97)
\psline[ArrowInside=->, linecolor=zzttqq](2.3,3.97)(1.18,4.62)
\psline[ArrowInside=->, linecolor=zzttqq](2.3,3.97)(2.3,2.23)
\psline[ArrowInside=->, linecolor=zzttqq](2.3,2.23)(1.18,2.88)
\psline[ArrowInside=->, linecolor=black](0.43,1.58)(-0.7,2.23)
\psline[ArrowInside=->, linecolor=zzttqq](1.18,2.88)(0.43,1.58)
\psline[ArrowInside=->, linecolor=zzttqq](2.3,5.7)(2.3,3.97)
\psline[ArrowInside=->, linecolor=zzttqq](1.18,4.62)(2.3,5.7)
\psline[ArrowInside=->, linecolor=zzttqq](2.3,3.97)(3.43,5.05)
\psline[ArrowInside=->, linecolor=zzttqq](3.43,5.05)(2.3,5.7)
\psline[ArrowInside=->, linecolor=zzttqq](3.43,3.32)(2.3,3.97)
\psline[ArrowInside=->, linecolor=zzttqq](2.3,2.23)(3.43,3.32)
\psline[linestyle=dashed,dash=3pt 3pt](4.76,4.49)(7.76,4.49)
\psline[linestyle=dashed,dash=3pt 3pt](7.76,4.49)(7.76,2.76)
\psline[linestyle=dashed,dash=3pt 3pt](7.76,2.76)(4.76,2.76)
\psline[linestyle=dashed,dash=3pt 3pt](4.76,2.76)(4.76,4.49)
\psline[ArrowInside=->, linecolor=zzttqq](5.66,3.35)(4.54,4)
\psline[ArrowInside=->, linecolor=zzttqq](4.54,2.27)(5.66,3.35)
\psline[ArrowInside=->, linecolor=zzttqq](4.54,4)(5.66,5.08)
\psline[ArrowInside=->, linecolor=zzttqq](6.41,4.65)(5.66,3.35)
\psline[ArrowInside=->, linecolor=zzttqq](5.66,5.08)(6.41,4.65)
\psline[ArrowInside=->, linecolor=zzttqq](5.66,3.35)(6.41,2.92)
\psline[ArrowInside=->, linecolor=black](6.41,2.92)(7.54,4)
\psline[ArrowInside=->, linecolor=black](7.54,4)(6.41,4.65)
\psline[ArrowInside=->, linecolor=zzttqq](7.54,4)(7.54,2.27)
\psline[ArrowInside=->, linecolor=black](7.54,2.27)(6.41,2.92)
\psline[ArrowInside=->, linecolor=zzttqq](5.66,1.62)(4.54,2.27)
\psline[ArrowInside=->, linecolor=zzttqq](6.41,2.92)(5.66,1.62)
\psline[ArrowInside=->, linecolor=zzttqq](7.54,5.73)(7.54,4)
\psline[ArrowInside=->, linecolor=black](6.41,4.65)(7.54,5.73)
\psline[ArrowInside=->, linecolor=zzttqq](7.54,4)(8.66,5.08)
\psline[ArrowInside=->, linecolor=zzttqq](8.66,5.08)(7.54,5.73)
\psline[ArrowInside=->, linecolor=zzttqq](8.66,3.35)(7.54,4)
\psline[ArrowInside=->, linecolor=zzttqq](7.54,2.27)(8.66,3.35)
\psline[linestyle=dashed,dash=3pt 3pt](9.99,4.52)(12.99,4.52)
\psline[linestyle=dashed,dash=3pt 3pt](12.99,4.52)(12.99,2.79)
\psline[linestyle=dashed,dash=3pt 3pt](12.99,2.79)(9.99,2.79)
\psline[linestyle=dashed,dash=3pt 3pt](9.99,2.79)(9.99,4.52)
\psline[ArrowInside=->, linecolor=zzttqq](10.9,3.39)(9.77,4.04)
\psline[ArrowInside=->, linecolor=black](9.77,2.31)(10.9,3.39)
\psline[ArrowInside=->, linecolor=zzttqq](9.77,4.04)(10.9,5.12)
\psline[ArrowInside=->, linecolor=zzttqq](11.65,4.69)(10.9,3.39)
\psline[ArrowInside=->, linecolor=zzttqq](10.9,5.12)(11.65,4.69)
\psline[ArrowInside=->, linecolor=black](10.9,3.39)(11.65,2.95)
\psline[ArrowInside=->, linecolor=black](11.65,2.95)(12.77,4.04)
\psline[ArrowInside=->, linecolor=zzttqq](12.77,4.04)(11.65,4.69)
\psline[ArrowInside=->, linecolor=black](12.77,4.04)(12.77,2.31)
\psline[ArrowInside=->, linecolor=zzttqq](12.77,2.31)(11.65,2.95)
\psline[ArrowInside=->, linecolor=zzttqq](10.9,1.66)(9.77,2.31)
\psline[ArrowInside=->, linecolor=zzttqq](11.65,2.95)(10.9,1.66)
\psline[ArrowInside=->, linecolor=zzttqq](12.77,5.77)(12.77,4.04)
\psline[ArrowInside=->, linecolor=zzttqq](11.65,4.69)(12.77,5.77)
\psline[ArrowInside=->, linecolor=zzttqq](12.77,4.04)(13.9,5.12)
\psline[ArrowInside=->, linecolor=zzttqq](13.9,5.12)(12.77,5.77)
\psline[ArrowInside=->, linecolor=zzttqq](13.9,3.39)(12.77,4.04)
\psline[ArrowInside=->, linecolor=black](12.77,2.31)(13.9,3.39)
\psline[linestyle=dashed,dash=3pt 3pt](1.9,-0.27)(4.9,-0.27)
\psline[linestyle=dashed,dash=3pt 3pt](4.9,-0.27)(4.9,-2)
\psline[linestyle=dashed,dash=3pt 3pt](4.9,-2)(1.9,-2)
\psline[linestyle=dashed,dash=3pt 3pt](1.9,-2)(1.9,-0.27)
\psline[linestyle=dashed,dash=3pt 3pt](7.13,-0.24)(10.13,-0.24)
\psline[linestyle=dashed,dash=3pt 3pt](10.13,-0.24)(10.13,-1.97)
\psline[linestyle=dashed,dash=3pt 3pt](10.13,-1.97)(7.13,-1.97)
\psline[linestyle=dashed,dash=3pt 3pt](7.13,-1.97)(7.13,-0.24)
\psline[ArrowInside=->, linecolor=zzttqq](2.8,-1.41)(1.68,-0.76)
\psline[ArrowInside=->, linecolor=zzttqq](1.68,-2.49)(2.8,-1.41)
\psline[ArrowInside=->, linecolor=zzttqq](1.68,-0.76)(2.8,0.32)
\psline[ArrowInside=->, linecolor=black](3.55,-0.11)(2.8,-1.41)
\psline[ArrowInside=->, linecolor=black](2.8,0.32)(3.55,-0.11)
\psline[ArrowInside=->, linecolor=black](2.8,-1.41)(3.55,-1.84)
\psline[ArrowInside=->, linecolor=zzttqq](3.55,-1.84)(4.68,-0.76)
\psline[ArrowInside=->, linecolor=zzttqq](4.68,-0.76)(3.55,-0.11)
\psline[ArrowInside=->, linecolor=zzttqq](4.68,-0.76)(4.68,-2.49)
\psline[ArrowInside=->, linecolor=zzttqq](4.68,-2.49)(3.55,-1.84)
\psline[ArrowInside=->, linecolor=zzttqq](2.8,-3.14)(1.68,-2.49)
\psline[ArrowInside=->, linecolor=black](3.55,-1.84)(2.8,-3.14)
\psline[ArrowInside=->, linecolor=zzttqq](4.68,0.97)(4.68,-0.76)
\psline[ArrowInside=->, linecolor=zzttqq](3.55,-0.11)(4.68,0.97)
\psline[ArrowInside=->, linecolor=zzttqq](4.68,-0.76)(5.8,0.32)
\psline[ArrowInside=->, linecolor=zzttqq](5.8,0.32)(4.68,0.97)
\psline[ArrowInside=->, linecolor=zzttqq](5.8,-1.41)(4.68,-0.76)
\psline[ArrowInside=->, linecolor=zzttqq](4.68,-2.49)(5.8,-1.41)
\psline[ArrowInside=->, linecolor=black](8.04,-1.37)(6.91,-0.72)
\psline[ArrowInside=->, linecolor=zzttqq](6.91,-2.45)(8.04,-1.37)
\psline[ArrowInside=->, linecolor=zzttqq](6.91,-0.72)(8.04,0.36)
\psline[ArrowInside=->, linecolor=black](8.79,-0.07)(8.04,-1.37)
\psline[ArrowInside=->, linecolor=zzttqq](8.04,0.36)(8.79,-0.07)
\psline[ArrowInside=->, linecolor=zzttqq](8.04,-1.37)(8.79,-1.81)
\psline[ArrowInside=->, linecolor=zzttqq](8.79,-1.81)(9.91,-0.72)
\psline[ArrowInside=->, linecolor=black](9.91,-0.72)(8.79,-0.07)
\psline[ArrowInside=->, linecolor=zzttqq](9.91,-0.72)(9.91,-2.45)
\psline[ArrowInside=->, linecolor=zzttqq](9.91,-2.45)(8.79,-1.81)
\psline[ArrowInside=->, linecolor=zzttqq](8.04,-3.1)(6.91,-2.45)
\psline[ArrowInside=->, linecolor=zzttqq](8.79,-1.81)(8.04,-3.1)
\psline[ArrowInside=->, linecolor=black](9.91,1.01)(9.91,-0.72)
\psline[ArrowInside=->, linecolor=zzttqq](8.79,-0.07)(9.91,1.01)
\psline[ArrowInside=->, linecolor=zzttqq](9.91,-0.72)(11.04,0.36)
\psline[ArrowInside=->, linecolor=black](11.04,0.36)(9.91,1.01)
\psline[ArrowInside=->, linecolor=zzttqq](11.04,-1.37)(9.91,-0.72)
\psline[ArrowInside=->, linecolor=zzttqq](9.91,-2.45)(11.04,-1.37)

\psline[ArrowInside=->, linecolor=zzttqq](-0.7,3.97)(-0.7,2.23)
\psline[ArrowInside=->, linecolor=zzttqq](4.54,4)(4.54,2.27)
\psline[ArrowInside=->, linecolor=black](9.77,4.04)(9.77,2.31)
\psline[ArrowInside=->, linecolor=zzttqq](1.68,-0.76)(1.68,-2.49)
\psline[ArrowInside=->, linecolor=black](6.91,-0.72)(6.91,-2.45)
\psdots[dotsize=2pt 0, linecolor=black](-0.7,3.97)
\psdots[dotsize=2pt 0, linecolor=black](-0.7,2.23)
\psdots[dotsize=2pt 0, linecolor=zzttqq](2.3,2.23)
\psdots[dotsize=2pt 0, linecolor=zzttqq](2.3,3.97)
\psdots[dotsize=2pt 0, linecolor=black](0.43,3.32)
\psdots[dotsize=2pt 0, linecolor=zzttqq](1.18,2.88)
\psdots[dotsize=2pt 0, linecolor=black](0.43,5.05)
\psdots[dotsize=2pt 0, linecolor=zzttqq](1.18,4.62)
\psdots[dotsize=2pt 0, linecolor=black](0.43,1.58)
\psdots[dotsize=2pt 0, linecolor=zzttqq](2.3,5.7)
\psdots[dotsize=2pt 0, linecolor=zzttqq](3.43,5.05)
\psdots[dotsize=2pt 0, linecolor=zzttqq](3.43,3.32)
\psdots[dotsize=2pt 0, linecolor=zzttqq](4.54,4)
\psdots[dotsize=2pt 0, linecolor=zzttqq](5.66,3.35)
\psdots[dotsize=2pt 0, linecolor=zzttqq](5.66,3.35)
\psdots[dotsize=2pt 0, linecolor=zzttqq](4.54,2.27)
\psdots[dotsize=2pt 0, linecolor=zzttqq](5.66,5.08)
\psdots[dotsize=2pt 0, linecolor=zzttqq](4.54,4)
\psdots[dotsize=2pt 0, linecolor=zzttqq](5.66,3.35)
\psdots[dotsize=2pt 0, linecolor=black](6.41,4.65)
\psdots[dotsize=2pt 0, linecolor=zzttqq](6.41,4.65)
\psdots[dotsize=2pt 0, linecolor=zzttqq](5.66,5.08)
\psdots[dotsize=2pt 0, linecolor=black](6.41,2.92)
\psdots[dotsize=2pt 0, linecolor=zzttqq](5.66,3.35)
\psdots[dotsize=2pt 0, linecolor=black](7.54,4)
\psdots[dotsize=2pt 0, linecolor=zzttqq](6.41,2.92)
\psdots[dotsize=2pt 0, linecolor=zzttqq](6.41,4.65)
\psdots[dotsize=2pt 0, linecolor=zzttqq](7.54,4)
\psdots[dotsize=2pt 0, linecolor=black](7.54,2.27)
\psdots[dotsize=2pt 0, linecolor=zzttqq](7.54,4)
\psdots[dotsize=2pt 0, linecolor=zzttqq](6.41,2.92)
\psdots[dotsize=2pt 0, linecolor=zzttqq](7.54,2.27)
\psdots[dotsize=2pt 0, linecolor=zzttqq](4.54,2.27)
\psdots[dotsize=2pt 0, linecolor=zzttqq](5.66,1.62)
\psdots[dotsize=2pt 0, linecolor=zzttqq](5.66,1.62)
\psdots[dotsize=2pt 0, linecolor=black](6.41,2.92)
\psdots[dotsize=2pt 0, linecolor=zzttqq](7.54,4)
\psdots[dotsize=2pt 0, linecolor=black](7.54,5.73)
\psdots[dotsize=2pt 0, linecolor=zzttqq](7.54,5.73)
\psdots[dotsize=2pt 0, linecolor=black](6.41,4.65)
\psdots[dotsize=2pt 0, linecolor=zzttqq](7.54,4)
\psdots[dotsize=2pt 0, linecolor=zzttqq](8.66,5.08)
\psdots[dotsize=2pt 0, linecolor=zzttqq](8.66,5.08)
\psdots[dotsize=2pt 0, linecolor=black](7.54,5.73)
\psdots[dotsize=2pt 0, linecolor=black](7.54,4)
\psdots[dotsize=2pt 0, linecolor=zzttqq](8.66,3.35)
\psdots[dotsize=2pt 0, linecolor=zzttqq](8.66,3.35)
\psdots[dotsize=2pt 0, linecolor=black](7.54,2.27)
\psdots[dotsize=2pt 0, linecolor=zzttqq](10.9,1.66)
\psdots[dotsize=2pt 0, linecolor=zzttqq](9.77,4.04)
\psdots[dotsize=2pt 0, linecolor=zzttqq](10.9,3.39)
\psdots[dotsize=2pt 0, linecolor=zzttqq](10.9,3.39)
\psdots[dotsize=2pt 0, linecolor=zzttqq](9.77,2.31)
\psdots[dotsize=2pt 0, linecolor=zzttqq](10.9,5.12)
\psdots[dotsize=2pt 0, linecolor=black](9.77,4.04)
\psdots[dotsize=2pt 0, linecolor=zzttqq](10.9,3.39)
\psdots[dotsize=2pt 0, linecolor=zzttqq](11.65,4.69)
\psdots[dotsize=2pt 0, linecolor=zzttqq](11.65,4.69)
\psdots[dotsize=2pt 0, linecolor=zzttqq](10.9,5.12)
\psdots[dotsize=2pt 0, linecolor=zzttqq](11.65,2.95)
\psdots[dotsize=2pt 0, linecolor=black](10.9,3.39)
\psdots[dotsize=2pt 0, linecolor=zzttqq](12.77,4.04)
\psdots[dotsize=2pt 0, linecolor=zzttqq](11.65,2.95)
\psdots[dotsize=2pt 0, linecolor=zzttqq](11.65,4.69)
\psdots[dotsize=2pt 0, linecolor=zzttqq](12.77,4.04)
\psdots[dotsize=2pt 0, linecolor=zzttqq](12.77,2.31)
\psdots[dotsize=2pt 0, linecolor=zzttqq](12.77,4.04)
\psdots[dotsize=2pt 0, linecolor=zzttqq](11.65,2.95)
\psdots[dotsize=2pt 0, linecolor=zzttqq](12.77,2.31)
\psdots[dotsize=2pt 0, linecolor=black](9.77,2.31)
\psdots[dotsize=2pt 0, linecolor=zzttqq](10.9,1.66)
\psdots[dotsize=2pt 0, linecolor=zzttqq](10.9,1.66)
\psdots[dotsize=2pt 0, linecolor=black](11.65,2.95)
\psdots[dotsize=2pt 0, linecolor=zzttqq](12.77,4.04)
\psdots[dotsize=2pt 0, linecolor=zzttqq](12.77,5.77)
\psdots[dotsize=2pt 0, linecolor=zzttqq](12.77,5.77)
\psdots[dotsize=2pt 0, linecolor=zzttqq](11.65,4.69)
\psdots[dotsize=2pt 0, linecolor=zzttqq](12.77,4.04)
\psdots[dotsize=2pt 0, linecolor=zzttqq](13.9,5.12)
\psdots[dotsize=2pt 0, linecolor=zzttqq](13.9,5.12)
\psdots[dotsize=2pt 0, linecolor=zzttqq](12.77,5.77)
\psdots[dotsize=2pt 0, linecolor=black](12.77,4.04)
\psdots[dotsize=2pt 0, linecolor=zzttqq](13.9,3.39)
\psdots[dotsize=2pt 0, linecolor=black](13.9,3.39)
\psdots[dotsize=2pt 0, linecolor=black](12.77,2.31)
\psdots[dotsize=2pt 0, linecolor=zzttqq](10.9,1.66)
\psdots[dotsize=2pt 0, linecolor=zzttqq](1.68,-0.76)
\psdots[dotsize=2pt 0, linecolor=zzttqq](2.8,-1.41)
\psdots[dotsize=2pt 0, linecolor=zzttqq](2.8,-1.41)
\psdots[dotsize=2pt 0, linecolor=zzttqq](1.68,-2.49)
\psdots[dotsize=2pt 0, linecolor=zzttqq](2.8,0.32)
\psdots[dotsize=2pt 0, linecolor=zzttqq](1.68,-0.76)
\psdots[dotsize=2pt 0, linecolor=zzttqq](2.8,-1.41)
\psdots[dotsize=2pt 0, linecolor=zzttqq](3.55,-0.11)
\psdots[dotsize=2pt 0, linecolor=zzttqq](3.55,-0.11)
\psdots[dotsize=2pt 0, linecolor=black](2.8,0.32)
\psdots[dotsize=2pt 0, linecolor=zzttqq](3.55,-1.84)
\psdots[dotsize=2pt 0, linecolor=black](2.8,-1.41)
\psdots[dotsize=2pt 0, linecolor=zzttqq](4.68,-0.76)
\psdots[dotsize=2pt 0, linecolor=zzttqq](3.55,-1.84)
\psdots[dotsize=2pt 0, linecolor=zzttqq](3.55,-0.11)
\psdots[dotsize=2pt 0, linecolor=zzttqq](4.68,-0.76)
\psdots[dotsize=2pt 0, linecolor=zzttqq](4.68,-2.49)
\psdots[dotsize=2pt 0, linecolor=zzttqq](4.68,-0.76)
\psdots[dotsize=2pt 0, linecolor=zzttqq](3.55,-1.84)
\psdots[dotsize=2pt 0, linecolor=zzttqq](4.68,-2.49)
\psdots[dotsize=2pt 0, linecolor=zzttqq](1.68,-2.49)
\psdots[dotsize=2pt 0, linecolor=zzttqq](2.8,-3.14)
\psdots[dotsize=2pt 0, linecolor=black](2.8,-3.14)
\psdots[dotsize=2pt 0, linecolor=black](3.55,-1.84)
\psdots[dotsize=2pt 0, linecolor=zzttqq](4.68,-0.76)
\psdots[dotsize=2pt 0, linecolor=zzttqq](4.68,0.97)
\psdots[dotsize=2pt 0, linecolor=zzttqq](4.68,0.97)
\psdots[dotsize=2pt 0, linecolor=black](3.55,-0.11)
\psdots[dotsize=2pt 0, linecolor=zzttqq](4.68,-0.76)
\psdots[dotsize=2pt 0, linecolor=zzttqq](5.8,0.32)
\psdots[dotsize=2pt 0, linecolor=zzttqq](5.8,0.32)
\psdots[dotsize=2pt 0, linecolor=zzttqq](4.68,0.97)
\psdots[dotsize=2pt 0, linecolor=zzttqq](4.68,-0.76)
\psdots[dotsize=2pt 0, linecolor=zzttqq](5.8,-1.41)
\psdots[dotsize=2pt 0, linecolor=zzttqq](5.8,-1.41)
\psdots[dotsize=2pt 0, linecolor=zzttqq](4.68,-2.49)
\psdots[dotsize=2pt 0, linecolor=zzttqq](6.91,-0.72)
\psdots[dotsize=2pt 0, linecolor=zzttqq](8.04,-1.37)
\psdots[dotsize=2pt 0, linecolor=zzttqq](8.04,-1.37)
\psdots[dotsize=2pt 0, linecolor=zzttqq](6.91,-2.45)
\psdots[dotsize=2pt 0, linecolor=zzttqq](8.04,0.36)
\psdots[dotsize=2pt 0, linecolor=black](6.91,-0.72)
\psdots[dotsize=2pt 0, linecolor=zzttqq](8.04,-1.37)
\psdots[dotsize=2pt 0, linecolor=zzttqq](8.79,-0.07)
\psdots[dotsize=2pt 0, linecolor=zzttqq](8.79,-0.07)
\psdots[dotsize=2pt 0, linecolor=zzttqq](8.04,0.36)
\psdots[dotsize=2pt 0, linecolor=zzttqq](8.79,-1.81)
\psdots[dotsize=2pt 0, linecolor=black](8.04,-1.37)
\psdots[dotsize=2pt 0, linecolor=zzttqq](9.91,-0.72)
\psdots[dotsize=2pt 0, linecolor=zzttqq](8.79,-1.81)
\psdots[dotsize=2pt 0, linecolor=zzttqq](8.79,-0.07)
\psdots[dotsize=2pt 0, linecolor=zzttqq](9.91,-0.72)
\psdots[dotsize=2pt 0, linecolor=zzttqq](9.91,-2.45)
\psdots[dotsize=2pt 0, linecolor=zzttqq](9.91,-0.72)
\psdots[dotsize=2pt 0, linecolor=zzttqq](8.79,-1.81)
\psdots[dotsize=2pt 0, linecolor=zzttqq](9.91,-2.45)
\psdots[dotsize=2pt 0, linecolor=black](6.91,-2.45)
\psdots[dotsize=2pt 0, linecolor=zzttqq](8.04,-3.1)
\psdots[dotsize=2pt 0, linecolor=zzttqq](8.04,-3.1)
\psdots[dotsize=2pt 0, linecolor=zzttqq](8.79,-1.81)
\psdots[dotsize=2pt 0, linecolor=zzttqq](9.91,-0.72)
\psdots[dotsize=2pt 0, linecolor=zzttqq](9.91,1.01)
\psdots[dotsize=2pt 0, linecolor=zzttqq](9.91,1.01)
\psdots[dotsize=2pt 0, linecolor=black](8.79,-0.07)
\psdots[dotsize=2pt 0, linecolor=zzttqq](9.91,-0.72)
\psdots[dotsize=2pt 0, linecolor=zzttqq](11.04,0.36)
\psdots[dotsize=2pt 0, linecolor=black](11.04,0.36)
\psdots[dotsize=2pt 0, linecolor=black](9.91,1.01)
\psdots[dotsize=2pt 0, linecolor=black](9.91,-0.72)
\psdots[dotsize=2pt 0, linecolor=zzttqq](11.04,-1.37)
\psdots[dotsize=2pt 0, linecolor=zzttqq](11.04,-1.37)
\psdots[dotsize=2pt 0, linecolor=zzttqq](9.91,-2.45)
\psdots[dotsize=2pt 0, linecolor=zzttqq](8.04,-3.1)
\rput[bl](-0.5,4.6){$\widetilde{\eta_1}$}
\rput[bl](6.5,5.2){$\widetilde{\eta_2}$}
\rput[bl](13.5,2.5){$\widetilde{\eta_3}$}
\rput[bl](3.15,-3){$\widetilde{\eta_4}$}
\rput[bl](6.4,-1.8){$\widetilde{\eta_5}$}
\end{pspicture*}
\end{center}
Using the basis for $H_1(Q)$ suggested by the indicated fundamental domain, we see that the zig-zag paths $\eta_1, \dots, \eta_5$ have homology classes $(0,1), (0,1), (1,0), (0, -1)$ and $(-1,-1)$ respectively. Therefore the global zig-zag fan $\Xi$ is:
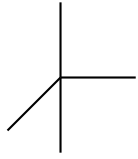
\begin{center}
\psset{xunit=0.5cm,yunit=0.5cm,algebraic=true,dotstyle=*,dotsize=3pt  
0,linewidth=0.8pt,arrowsize=3pt 2,arrowinset=0.25}
\begin{pspicture*}(0,-1.5)(4.5,3.5)
\psline(2,1)(2,3)
\psline(2,1)(4,1)
\psline(2,1)(2,-1)
\psline(2,1)(0.59,-0.41)
\end{pspicture*}
\end{center}
Note that any zig-zag flows $\widetilde{\eta_1}, \widetilde{\eta_2}$ which project down to zig-zag paths $\eta_1, \eta_2$ give an example of parallel zig-zag flows. Similarly, any lifts of $\eta_1$ and $\eta_3$ give an example of anti-parallel zig-zag flows.

The quiver has two faces which are quadrilaterals and two which are triangles. We see that the boundaries of any lift $f$ of either quadrilateral, intersects four zig-zag flows and has local zig-zag fan $\xi(f)$ of type 1 in the diagram below. The boundary of any lift of either triangle intersects three zig-zag flows, and has local zig-zag fan of type 2 in the diagram below.
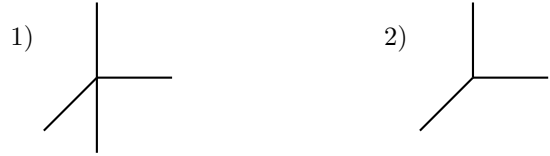
\begin{center}
\psset{xunit=0.5cm,yunit=0.5cm,algebraic=true,dotstyle=*,dotsize=3pt 0,linewidth=0.8pt,arrowsize=3pt 2,arrowinset=0.25}
\begin{pspicture*}(-1,-1.5)(15,3.5)
\psline(2,1)(2,3)
\psline(2,1)(4,1)
\psline(2,1)(2,-1)
\psline(2,1)(0.59,-0.41)
\psline(12,1)(12,3)
\psline(12,1)(14,1)
\psline(12,1)(10.59,-0.41)
\rput[tl](-0.3,2.4){$1)$}
\rput[tl](9.64,2.4){$2)$}
\end{pspicture*}
\end{center}
\end{example}


\section{Constructing some perfect matchings} \label{CSPM}
In this section we use the zig-zag fans we have just introduced to construct a particular collection of perfect matchings that will play a key role in the rest of this article. The construction is a local one in the sense that the restriction of the perfect matching to the arrows in the boundary of a quiver face is determined by the local zig-zag fan at that face.

We start by noting that the identity map on $H_1(Q)$ induces a map of fans $\mof{\face}:\Xi \rightarrow \xi(\face)$ for each $\face \in Q_2$.
If $\sigma $ is a two dimensional cone in $\Xi$, its image in $\xi(\face)$ is contained in a unique two dimensional cone which we shall denote $\sigma_{\face}$.

The cone $\sigma_{\face}$ corresponds to a unique arrow in the boundary of $\face$. We define $P_{\face}(\sigma) \in \Z^{Q_1}$ to be the function which evaluates to 1 on this arrow, and zero on all other arrows in $Q$. Finally we define the following function:
$$P(\sigma) := \frac{1}{2} \sum_{\face \in Q_2} P_{\face}(\sigma)$$

\begin{lemma}\label{pmatch} The function $P(\sigma)$ is a perfect matching.
\end{lemma}
\begin{proof}
First we show that $P(\sigma)\in \N^{Q_1}$. Any arrow $a$ is in the boundary of precisely two faces, one black and one white, which we denote $\face_b$ and $\face_w$ respectively. Since $P_{\face}(\sigma)$ is zero on all arrows which don't lie in the boundary of $\face$, we observe that $P(\sigma)(a) = \frac{1}{2} \left ( P_{\face_b}(\sigma)(a)+ P_{\face_w}(\sigma)(a)\right )$.
If $\zzf^+$ and $\zzf^-$ are the zig-zag flows containing $a$ then, by Lemma~\ref{2cones}, the cones $\sigma_b$ in $\xi(\face_b)$ and $\sigma_w$ in $\xi(\face_w)$ dual to $a$, are both spanned by the rays generated by $[\eta^+]$ and $[\eta^-]$. Therefore the inverse image of $\sigma_b$ and $\sigma_w$ under the respective maps of fans is the same collection of cones in the global zig-zag fan $\Xi$. Thus for any cone $\sigma$ in $\Xi$, 
the functions $P_{\face_b}(\sigma)$ and $ P_{\face_w}(\sigma)$ evaluate to the same value on $a$.
We conclude that
$$P(\sigma) = \frac{1}{2} \sum_{\face \in Q_2} P_{\face}(\sigma)= \sum_{\genfrac{}{}{0pt}{}{\face \in Q_2}{\face \text{ is black}}} P_{\face}(\sigma)$$
and evaluates to zero or one on every arrow in $Q$. By construction $P_{\face}(\sigma)$ is non-zero on a single arrow in the boundary of $\face$, and we have just seen that $P(\sigma)$ is non-zero on an arrow in the boundary of $\face$ if and only if  $P_{\face}(\sigma)$ is non-zero. Therefore $P(\sigma)$ evaluates to one on a single arrow in the boundary of each face. Recalling that the coboundary map $d$ simply sums any function of the edges around each face we see that $d(P(\sigma)) = \const{1}$.
\end{proof}

We have identified a perfect matching $P(\sigma) \in \N^{Q_1}$ for each two dimensional cone $\sigma$ in the global zig-zag fan $\Xi$. We now consider some properties of perfect matchings of this form.

\begin{definition} \label{zigsfunction}
If $\eta$ is a zig-zag path, we define $Zig(\eta) \in \Z^{Q_1}$ (respectively $Zag(\eta) \in \Z^{Q_1}$) to be the function which evaluates to one on all zigs (zags) of $\eta$ and zero on all other arrows. Similarly if $\gamma$ is a ray in the global zig-zag fan $\Xi$, then $Zig(\gamma) \in \Z^{Q_1}$ (respectively $Zag(\gamma) \in \Z^{Q_1}$) is the function which evaluates to one on all zigs (zags) of every representative zig-zag path of $\gamma$, and is zero on all other arrows.
\end{definition}
\begin{remark} \label{raypath}
We see that
$$Zig(\gamma)= \sum_{\langle [\eta] \rangle = \gamma} Zig(\eta).$$
\end{remark}
\noindent The functions defined in Definition~\ref{zigsfunction}, like perfect matchings, are functions which evaluate to one on each arrow in their support. Therefore they can be thought of as sets of arrows. We now show that $P(\sigma)$ is non-zero on the zigs or zags of certain zig-zag paths, in other words that these zigs or zags are contained in the set of arrows on which $P(\sigma)$ is supported. In the language of functions, this is equivalent to the following lemma.
\begin{lemma} \label{resonateall}
Suppose $\sigma$ is a cone in the global zig-zag fan $\Xi$ spanned by rays $\gamma^+$ and $\gamma^-$ (see diagram below). Then the functions $P(\sigma) - Zig(\gamma^+)$ and $P(\sigma) - Zag(\gamma^-)$ are elements of $\N^{Q_1}$.
\end{lemma}
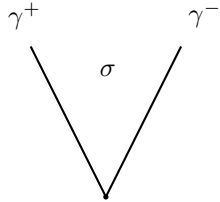
\begin{center}
\begin{pspicture*}(1,-1)(5,2)
\psline(3,-0.8)(2,1.2)
\psline(3,-0.8)(4,1.2)
\rput[tl](2.92,0.96){$\sigma$}
\rput[tl](1.7,1.78){$\gamma^+$}
\rput[tl](4.1,1.8){$\gamma^-$}
\psdots[dotsize=2pt 0](3,-0.8)
\end{pspicture*} 
\end{center}
\begin{proof}
Let $\face \in Q_2$ be any face. If $\gamma^+$ is not a ray in the local zig-zag fan $\xi(\face)$, then by definition, no representative zig-zag flow intersects the boundary of $\face$. Thus $Zig(\gamma^+)$ is zero on all the arrows in the boundary, and so $P(\sigma) - Zig(\gamma^+)$ is non-negative on these arrows.

If $\gamma^+$ is a ray in the local zig-zag fan $\xi(\face)$, we consider the cone $\tau$ in $\xi(\face)$ which is generated by $\gamma^+$ and the next ray around the fan in a clockwise direction (this is $\gamma^-$ if and only if $\gamma^-$ is a ray in $\xi(\face)$). Since $\sigma$ in the global zig-zag fan is generated by $\gamma^+$ and $\gamma^-$ which is the next ray around in the clockwise direction, the image of $\sigma$ under the map of fans lies in $\tau$. Then by definition $P_{\face}(\sigma)$, and therefore $P(\sigma)$, evaluates to 1 on the arrow in the boundary of $\face$ corresponding to the cone $\tau$. 
This arrow is the intersection of the two zig-zag flows in $\mathcal{X}(\face)$ corresponding the rays of $\tau$. Since the anti-clockwise ray is $\gamma^+$, this arrow is the unique (by Lemmas~\ref{singlein} and \ref{uniquerep}) zig of the representative of $\gamma^+$ in the boundary of $\face$. Therefore $P(\sigma) - Zig(\gamma^+)$ is non-negative on all the arrows in the boundary of $\face$. The statement holds on the arrows in the boundary of all faces, and therefore in general. The result follows similarly for $P(\sigma) - Zag(\zzf_-)$.
\end{proof}

\begin{corollary} \label{contzgs}
Suppose $\sigma$ is a cone in $\Xi$ spanned by rays $\gamma^+$ and $\gamma^-$. For any representative zig-zag paths $\eta^+$ of $\gamma^+$ and $\eta^-$ of $\gamma^-$, the functions $P(\sigma) - Zig(\eta^+)$ and $P(\sigma) - Zag(\eta^-)$ are elements of $\N^{Q_1}$.
\end{corollary}
\begin{proof} This follows from the Lemma~\ref{resonateall} together with Remark~\ref{raypath}. \end{proof}
\begin{example}
We return to Example~\ref{memeg} and construct a perfect matching $P(\sigma)$, where $\sigma$ is the cone in the global zig-zag fan shown in the diagram below. 
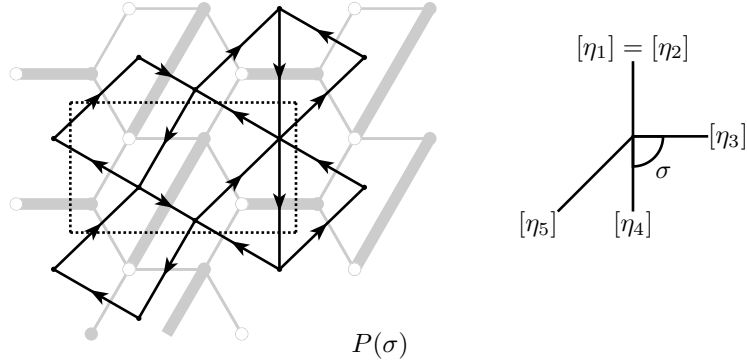
\begin{center}
\newrgbcolor{ttfftt}{0.8 0.8 0.8}
\psset{xunit=1.0cm,yunit=1.0cm,algebraic=true,dotstyle=*,dotsize=5pt 0,linewidth=1pt,arrowsize=3pt 2,arrowinset=0.25}
\begin{pspicture*}(-2,0.5)(8.5,6.5)
\psline[linewidth=4.4pt,linecolor=ttfftt](-1.2,3.1)(-0.2,3.1)
\psline[linecolor=ttfftt](-0.2,3.1)(0.3,3.97)
\psline[linecolor=ttfftt](0.3,3.97)(-0.2,4.83)
\psline[linewidth=4.4pt,linecolor=ttfftt](-0.2,4.83)(-1.2,4.83)
\psline[linecolor=ttfftt](0.3,3.97)(-0.2,3.1)
\psline[linecolor=ttfftt](-0.2,3.1)(0.3,2.23)
\psline[linecolor=ttfftt](0.3,2.23)(1.3,2.23)
\psline[linecolor=ttfftt](1.3,2.23)(1.8,3.1)
\psline[linecolor=ttfftt](1.8,3.1)(1.3,3.97)
\psline[linecolor=ttfftt](1.3,3.97)(0.3,3.97)
\psline[linecolor=ttfftt](0.3,3.97)(1.3,3.97)
\psline[linecolor=ttfftt](1.3,3.97)(1.8,4.83)
\psline[linecolor=ttfftt](1.8,4.83)(1.3,5.7)
\psline[linecolor=ttfftt](1.3,5.7)(0.3,5.7)
\psline[linecolor=ttfftt](0.3,5.7)(-0.2,4.83)
\psline[linecolor=ttfftt](-0.2,4.83)(0.3,3.97)
\psline[linecolor=ttfftt](1.8,4.83)(1.3,3.97)
\psline[linecolor=ttfftt](1.3,3.97)(1.8,3.1)
\psline[linewidth=4.4pt,linecolor=ttfftt](1.8,3.1)(2.8,3.1)
\psline[linecolor=ttfftt](2.8,3.1)(3.3,3.97)
\psline[linecolor=ttfftt](3.3,3.97)(2.8,4.83)
\psline[linewidth=4.4pt,linecolor=ttfftt](2.8,4.83)(1.8,4.83)
\psline[linecolor=ttfftt](-0.2,3.1)(-1.2,3.1)
\psline[linecolor=ttfftt](-0.2,1.37)(0.3,2.23)
\psline[linecolor=ttfftt](0.3,2.23)(-0.2,3.1)
\psline[linecolor=ttfftt](2.8,3.1)(1.8,3.1)
\psline[linecolor=ttfftt](1.8,3.1)(1.3,2.23)
\psline[linecolor=ttfftt](1.3,2.23)(1.8,1.37)
\psline[linecolor=ttfftt](3.3,2.23)(2.8,3.1)
\psline[linewidth=4.4pt,linecolor=ttfftt](0.3,3.97)(1.3,5.7)
\psline[linewidth=4.4pt,linecolor=ttfftt](0.3,2.23)(1.3,3.97)
\psline[linestyle=dashed,dash=1pt 1pt](-0.48,4.45)(2.52,4.45)
\psline[linestyle=dashed,dash=1pt 1pt](2.52,4.45)(2.52,2.72)
\psline[linestyle=dashed,dash=1pt 1pt](2.52,2.72)(-0.48,2.72)
\psline[linestyle=dashed,dash=1pt 1pt](-0.48,2.72)(-0.48,4.45)
\psline[linecolor=ttfftt](1.8,1.37)(1.3,2.23)
\psline[linecolor=ttfftt](1.3,2.23)(0.3,2.23)
\psline[linecolor=ttfftt](0.3,2.23)(-0.2,1.37)
\psline[linewidth=4.4pt,linecolor=ttfftt](1.3,2.23)(0.8,1.37)
\psline[linecolor=ttfftt](2.8,4.83)(3.3,3.97)
\psline[linecolor=ttfftt](3.3,3.97)(4.3,3.97)
\psline[linecolor=ttfftt](4.3,5.7)(3.3,5.7)
\psline[linecolor=ttfftt](3.3,5.7)(2.8,4.83)
\psline[linecolor=ttfftt](4.3,3.97)(3.3,3.97)
\psline[linecolor=ttfftt](3.3,3.97)(2.8,3.1)
\psline[linecolor=ttfftt](2.8,3.1)(3.3,2.23)
\psline[linewidth=4.4pt,linecolor=ttfftt](4.3,5.7)(3.3,3.97)
\psline[linewidth=4.4pt,linecolor=ttfftt](4.3,3.97)(3.3,2.23)
\psline[ArrowInside=->](0.43,5.05)(1.18,4.62)
\psline[ArrowInside=->](1.18,4.62)(0.43,3.32)
\psline[ArrowInside=->](0.43,3.32)(1.18,2.88)
\psline[ArrowInside=->](1.18,2.88)(0.43,1.58)
\psline[ArrowInside=->](0.43,1.58)(-0.7,2.23)
\psline[ArrowInside=->](-0.7,2.23)(0.43,3.32)
\psline[ArrowInside=->](1.18,2.88)(2.3,3.97)
\psline[ArrowInside=->](2.3,3.97)(1.18,4.62)
\psline[ArrowInside=->](0.43,3.32)(-0.7,3.97)
\psline[ArrowInside=->](-0.7,3.97)(0.43,5.05)
\psline[ArrowInside=->](1.18,4.62)(2.3,5.7)
\psline[ArrowInside=->](2.3,5.7)(2.3,3.97)
\psline[ArrowInside=->](2.3,3.97)(3.43,5.05)
\psline[ArrowInside=->](3.43,5.05)(2.3,5.7)
\psline[ArrowInside=->](3.43,3.32)(2.3,3.97)
\psline[ArrowInside=->](2.3,3.97)(2.3,2.23)
\psline[ArrowInside=->](2.3,2.23)(3.43,3.32)
\psline[ArrowInside=->](2.3,2.23)(1.18,2.88)
\rput[tl](3.26,1.4){$P(\sigma)$}
\rput[b](7,5){$[\eta_1] = [\eta_2]$}
\rput[l](8,4){$[\eta_3]$}
\rput[t](7,3){$[\eta_4]$}
\rput[tr](6,3){$[\eta_5]$}
\psline(7,4)(7,5)
\psline(7,4)(8,4)
\psline(7,4)(7,3)
\psline(7,4)(6,3)
\pscustom{\parametricplot{-1.5707963267948966}{0.0}{0.4*cos(t)+7|0.4*sin(t)+4}\lineto(7,4)\closepath}
\rput[tl](7.3,3.6){$\sigma$}
\psdots[dotstyle=o, linecolor=ttfftt](-1.2,3.1)
\psdots[linecolor=ttfftt](-0.2,3.1)
\psdots[dotstyle=o, linecolor=ttfftt](0.3,3.97)
\psdots[linecolor=ttfftt](-0.2,4.83)
\psdots[dotstyle=o, linecolor=ttfftt](-1.2,4.83)
\psdots[dotstyle=o, linecolor=ttfftt](0.3,2.23)
\psdots[linecolor=ttfftt](1.3,2.23)
\psdots[dotstyle=o, linecolor=ttfftt](1.8,3.1)
\psdots[linecolor=ttfftt](1.3,3.97)
\psdots[dotstyle=o, linecolor=ttfftt](1.8,4.83)
\psdots[linecolor=ttfftt](1.3,5.7)
\psdots[dotstyle=o, linecolor=ttfftt](0.3,5.7)
\psdots[linecolor=ttfftt](-0.2,4.83)
\psdots[dotstyle=o, linecolor=ttfftt](1.8,3.1)
\psdots[linecolor=ttfftt](2.8,3.1)
\psdots[dotstyle=o, linecolor=ttfftt](3.3,3.97)
\psdots[linecolor=ttfftt](2.8,4.83)
\psdots[linecolor=ttfftt](-0.2,1.37)
\psdots[dotstyle=o, linecolor=ttfftt](0.3,2.23)
\psdots[linecolor=ttfftt](1.3,2.23)
\psdots[dotstyle=o, linecolor=ttfftt](1.8,1.37)
\psdots[dotstyle=o, linecolor=ttfftt](3.3,2.23)
\psdots[dotsize=2pt  0](-0.7,3.97)
\psdots[dotsize=2pt  0](-0.7,2.23)
\psdots[dotsize=2pt  0](2.3,2.23)
\psdots[dotsize=2pt  0](2.3,3.97)
\psdots[dotsize=2pt  0](0.43,3.32)
\psdots[dotsize=2pt  0](1.18,2.88)
\psdots[dotsize=2pt  0](0.43,5.05)
\psdots[dotsize=2pt  0](1.18,4.62)
\psdots[dotsize=2pt  0](0.43,1.58)
\psdots[dotsize=2pt  0](2.3,5.7)
\psdots[linecolor=ttfftt](4.3,3.97)
\psdots[linecolor=ttfftt](4.3,5.7)
\psdots[dotstyle=o, linecolor=ttfftt](3.3,5.7)
\psdots[linecolor=ttfftt](2.8,3.1)
\psdots[dotstyle=o, linecolor=ttfftt](3.3,2.23)
\psdots[dotsize=2pt  0](3.43,5.05)
\psdots[dotsize=2pt  0](3.43,3.32)
\end{pspicture*}
\end{center}
For each quadrilateral quiver face, the image of $\sigma$ in the local zig-zag fan is the cone generated by the rays corresponding to zig-zag paths $\eta_3$ and $\eta_4$. Therefore on the boundary of these faces, $P(\sigma)$ evaluates to one on the unique arrow contained in $\eta_3$ and $\eta_4$. In the case of the triangular faces, $\eta_4$ doesn't intersect the boundary and there is no corresponding ray in the local zig-zag fan. The image of $\sigma$ lies in the (larger) cone generated by the rays corresponding to $\eta_3$ and $\eta_5$, and 
$P(\sigma)$ is non-zero on the arrow in the intersection of these paths. In the diagram above, the thickened edges of the dual dimer model correspond to those arrows on which $P(\sigma)$ evaluates to one, and it evaluates to zero on all other arrows. We note that $P(\sigma)$ evaluates to one on the zigs of $\eta_3$ and the zags of $\eta_4$.

\end{example}

\section{The extremal perfect matchings}

We show that the perfect matchings we have constructed are precisely  
the extremal perfect matchings of the dimer model. Recall  
(Definition~\ref{externaldef}) that a perfect matching is called  
external if its image in $\Nm$, lies in a facet of the polygon given  
by the convex  hull of degree one elements in $\Nm^+$. It is called  
extremal if its image is a vertex of this polygon. Equivalently, this  
means a perfect matching is external if and only if its image lies in  
a two dimensional facet of of $\Nm^+$, and is extremal if and only if  
its image lies in a one dimensional facet of $\Nm^+$.

We start by identifying an element of $\Mm^{ef} \subset \Mm^+$  
associated to every ray $\gamma$ of $\Xi$. We will prove that for each  
cone $\sigma$, the perfect matching $P(\sigma)$ evaluates to zero on  
two such elements which are linearly independent. Therefore the image  
of $P(\sigma)$ in $\Nm$ lies in a one dimensional facet of $\Nm^+$. In  
other words it is an extremal perfect matching.

Given a zig-zag flow, we defined its black and white boundary flows in  
Section~\ref{boundflowsec}. Recall that the image of a single period  
of a boundary path is a cycle in $\N_{Q_1}$.
\begin{definition}
The system of boundary paths $\sbf(\gamma) \in \N_{Q_1}$ of a ray  
$\gamma$ of $\Xi$, is the sum of all the cycles corresponding to the  
boundary paths (black and white) of every representative zig-zag path  
of $\gamma$.
\end{definition}
Since $\sbf(\gamma)$ is the sum of cycles it is closed and, as it is  
non-negative, it defines an element $\Meq{\sbf(\gamma)} \in \Mm^{ef}  
\subset \Mm^+$. If $\sigma$ is a two dimensional cone in $\Xi$ spanned  
by rays $\gamma^+$ and $\gamma^-$, then the classes  
$\Meq{\sbf(\gamma)}$ and $\Meq{\sbf(\gamma)}$ are linearly  
independent: if they were not, then they would have linearly dependent  
homology classes (see the exact sequence (\ref{eq:Mmes})), however, by  
Lemma~\ref{boundhom} we know that these classes are $-k^+ [\eta^+]$  
and $-k^- [\eta^-]$ which are linearly independent, where $k^\pm >0$  
is the number of boundary paths of $\gamma^\pm$, and $\eta^\pm$ is a  
representative zig-zag path of $\gamma^\pm$.

\begin{proposition} \label{uniqueext}
Suppose $\sigma$ is a cone in $\Xi$ spanned by rays $\gamma^+$ and  
$\gamma^-$. The perfect matching $P(\sigma) \in \N^{Q_1}$ is the  
unique perfect matching which evaluates to zero on the two systems of  
boundary paths $\sbf(\gamma^+)$ and $\sbf(\gamma^-)$. 
\end{proposition}
\begin{proof}
First we prove that $P(\sigma)$ evaluates to zero on the two systems  
of boundary paths. Let $\eta$ be a representative zig-zag path of  
$\gamma^+$ or $\gamma^-$.
Since $P(\sigma)$ is a perfect matching we know that it is non-zero on  
precisely one arrow in the boundary of every face. By construction,  
the boundary path of a zig-zag path was pieced together from paths  
back around faces which have a zig-zag or zag-zig pair in the  
boundary. By Corollary~\ref{contzgs}, $P(\sigma)$ is non-zero on  
either the zig or zag of $\eta$ in the boundary of each of these faces  
and therefore it is zero on all the arrows in the boundary path. This  
holds for all boundary paths of representative zig-zag paths of  
$\gamma^+$ or $\gamma^-$.

Now we prove that $P(\sigma)$ is the unique perfect matching which  
evaluates to zero on the two systems of boundary paths. Let $\zzf^+$  
and $\zzf^-$ be representative zig-zag flows of $\gamma^+$ and  
$\gamma^-$ respectively. Since they are not parallel or anti-parallel, 
the black boundary flows of $\zzf^+$ and $\zzf^-$ have at least one  
vertex in common. Let $v_{00}\in \Qcov_0$ be such a vertex and let $v$  
be its image in $Q_0$.

Let the finite path $\pth^+_0$ be the lift of a single period of the  
black boundary flow of $\eta^+$ to the universal cover, starting at  
$v_{00}$ and ending at vertex $v_{10}$. Then let $\pth^-_1$ be the  
lift of a single period of the black boundary flow of $\eta^-$  
starting at $v_{10}$. Similarly let $\pth^-_0$ be the lift of a period  
of the black boundary flow of $\eta^-$ starting at $v_{00}$ and ending  
at $v_{01}$, and $\pth^+_1$ be the lift of a period of the black  
boundary flow of $\eta^+$ starting at $v_{01}$. 

The paths $\pth_1:= \pth^+_0 \pth^-_1$ and $\pth_2:= \pth^-_0  
\pth^+_1$ have the same class $\Meq{\pth_1} = \Meq{\pth_2}$ in $\Mg$.  
This follows trivially as $\pth^{\pm}_0$ and $\pth^{\pm}_1$ project  
down to the same paths in $Q$. This implies that both paths end at the  
same vertex; they have the same homology class, and start at the same  
vertex. Also recalling Section~\ref{conseqgcsec} we see that they are  
F-term equivalent.

Furthermore, since $\pth^{\pm}_0$ and $\pth^{\pm}_1$ are shifts of  
each other by deck transformations, i.e. by elements in the period  
lattice, we see that $p_1p_2^{-1}$ forms the boundary of a region  
which certainly contains at least one fundamental domain of $\Qcov$.  
Therefore, for every face in $Q$, there is a lift $\face$ in $\Qcov$  
such that the winding number, $\windf(\pth_1\pth_2^{-1}) \neq 0$.

Let $q_0 = \pth_1, q_1, \dots , q_k = \pth_2$ be a sequence of paths  
such that $q_i$ and $q_{i+1}$ differ by a single $F$-term relation for  
$i=0,\dots,k-1$. Fix a face in $Q$, and let $\face$ be the lift to  
$\Qcov$ satisfying the winding number condition above. Then since
$$\windf(\pth_1\pth_2^{-1}) \neq 0 = \windf(\pth_2\pth_2^{-1})$$
there exists $i \in \{1,\dots,k-1\}$ such that $  
\windf(q_i\pth_2^{-1}) \neq \windf(q_{i+1}\pth_2^{-1})$.
As they differ by a single $F$-term relation we can write $q_i =  
\alpha_1 \epsilon_1 \alpha_2$ and $q_{i+1} = \alpha_1 \epsilon_2  
\alpha_2$, where $\epsilon_1a$ and $\epsilon_2a$ are the boundaries of  
two faces $\face_b, \face_w \in \Qcov_2$ (one black, one white) which  
meet along the arrow $a$. 
Using properties of winding numbers,
$$\windf(\epsilon_1a) - \windf(\epsilon_2a) = \windf(q_i\pth_2^{-1}) -  
\windf(q_{i+1}\pth_2^{-1}) \neq 0$$
The winding number around $\face$ of the boundary of a face, is  
non-zero if and only if the face is $\face$. Therefore either $\face =  
\face_b$ or $\face = \face_w$ so all the arrows in the boundary of  
$\face$ except $a$, lie in either $q_i$ or $q_{i+1}$.
Let $n \in \Ng^+$ and suppose $n$ evaluates to zero on the element of  
$\Mg^+$ corresponding to the path $\pth_1$. Since they are F-term  
equivalent, $\Meq{\pth_1}=\Meq{q_i}=\Meq{q_{i+1}}$, so $n$ evaluates  
to zero on $\Meq{q_i}$ and $\Meq{q_{i+1}}$.
Using the positivity, this implies that $n$ evaluates to zero on every  
arrow in the image in $Q$ of both $q_i$ and $q_{i+1}$.
%
In particular it evaluates to zero on all the arrows in the boundary of the image of $\face$ except $a$. This argument holds for any face, and thus $n$ is non-zero on a unique arrow in the boundary of each face in $Q$.

Suppose there exist two perfect matchings $\pf_1$ and $\pf_2$ which  
evaluate to zero on the both systems of boundary flows. By  
construction $\pf_1+\pf_2 \in \Ng^+$ is zero on the path $\pth_1$ and  
therefore by the above argument, it is non-zero on a unique arrow in  
the boundary of each face. Since perfect matchings are degree 1  
elements of $\Ng^+$ this forces $\pf_1$ and $\pf_2$ to evaluate to the  
same value on all the arrows in the boundary of each face. We conclude  
that $\pf_1=\pf_2$ and we have proved uniqueness.
\end{proof}

\begin{corollary} \label{extremultone}
For each $\sigma$ in the global zig-zag fan, the perfect matching  
$P(\sigma)$ is extremal. Furthermore, the corresponding vertex of the  
degree one polygon in $\Nm^+$ has multiplicity one.
\end{corollary}
\begin{proof} By Proposition~\ref{uniqueext}, the image of $P(\sigma)$  
in $\Nm^+$ evaluates to zero on two linearly independent classes in  
$\Mm^+$. Therefore the image lies in a one dimensional facet of  
$\Nm^+$ and so is extremal. The uniqueness part of the proposition  
shows that no other perfect matching maps to the same one dimensional  
facet, and thus the corresponding vertex has multiplicity one.
\end{proof}

We now prove that every extremal vertex is the image of $P(\sigma)$  
for some $\sigma$.
Recall from Section~\ref{Perfmatchsec} that the cone $\Nm^+$ is the  
intersection of $\Nm$ with the real cone $(\Nm^+)_{\R}$ generated by  
the image of $\Ng^+$ in $\Nm \otimes_{\Z} \R$. 

\begin{lemma} \label{genallcone}
The real cone $(\Nm^+)_{\R}$ is generated over $\R^+$ by the images in  
$\Nm^+$ of $P(\sigma)$ for all two dimensional cones $\sigma$ of $\Xi$.
\end{lemma}
\begin{proof}
Let $\mathcal{C}_N$ be the real cone in $\Nm \otimes_\Z \R$ generated  
by the images of $P(\sigma)$ for all two dimensional cones $\sigma$ of  
$\Xi$, and let $\mathcal{C}_M$ be the real cone in $\Mm \otimes_\Z \R$  
generated by $\Meq{\sbf(\gamma)}$ for all rays $\gamma$ of $\Xi$.
We claim that $ \mathcal{C}_M^\vee = \mathcal{C}_N$. Since  
$[P(\sigma)] \in \Nm^+$, every element of $\Mm^+$ evaluates to a  
non-negative number on this class. In particular $[P(\sigma)] \in  
\mathcal{C}_M^\vee$ for every two dimensional cone $\sigma$ in $\Xi$,  
so $ \mathcal{C}_N \subseteq \mathcal{C}_M^\vee$. For the converse, we  
restrict to the degree 1 plane in $\Nm$. The functions  
$\Meq{\sbf(\gamma)}$ restrict to affine linear functions on this plane  
which evaluate to zero on neighbouring points $[P(\sigma^+)]$ and  
$[P(\sigma^-)]$, where $\sigma^\pm$ are the two, 2-dimensional cones  
in $\Xi$ containing ray $\gamma$, and are non-negative on all other  
points $[P(\sigma)]$. These functions define a collection of half  
planes whose intersection is the convex hull of the points
\footnote{This `obvious' fact is the justification for the `gift  
wrapping algorithm' in two dimensions \cite{Jarvis}.}.
Thus the cone over this intersection, namely $ \mathcal{C}_M^\vee$  
equals the cone over the convex hull, namely $\mathcal{C}_N$.

Finally, using the fact that $\mathcal{C}_M \subseteq (\Mm^+)_{\R}$  
where $(\Mm^+)_{\R}$ is the real cone generated by $\Mm^+$,
$$(\Nm^+)_{\R}= (\Mm^+)_{\R}^\vee \subseteq \mathcal{C}_M^\vee =  
\mathcal{C}_N $$
since we also know that $\mathcal{C}_N \subseteq (\Nm^+)_{\R}$, the  
result follows.
\end{proof}
\begin{remark}
Thus, the real cone $(\Mm^+)_{\R} = (\Nm^+)_{\R}^\vee$ is dual to the  
real cone in $\Nm \otimes_\Z \R$ generated by the images of perfect  
matchings of the form $P(\sigma)$ for $\sigma$ of $\Xi$.
We have also shown that the real cone $(\Nm^+)_{\R}$ is dual to the  
real cone in $\Mm \otimes_\Z \R$ generated by $\Meq{\sbf(\gamma)}$ for  
all rays $\gamma$ of $\Xi$, so $(\Mm^+)_{\R}$ is generated by the  
classes $\Meq{\sbf(\gamma)}$.
\end{remark}
\begin{corollary}
The extremal perfect matchings are precisely the perfect matchings of  
the form $P(\sigma)$ for all two dimensional cones $\sigma$ of $\Xi$.
\end{corollary}
\begin{proof}
By Lemma~\ref{genallcone} the images of $P(\sigma)$ in $\Nm^+$, for  
all two dimensional cones $\sigma$ of $\Xi$, generate $(\Nm^+)_{\R}$.  
Therefore perfect matchings of the form $P(\sigma)$ map to generators  
of all the extremal rays of $\Nm^+$, i.e. to all the extremal vertices  
of the degree 1 polygon in $\Nm^+$. Finally, by  
Corollary~\ref{extremultone}, each of these vertices has multiplicity  
one.
\end{proof}

\section{The external perfect matchings}
We next see that given an extremal perfect matching constructed in the previous section, we can alter it along a zig-zag path in such a way that we still have a perfect matching. Furthermore the resulting perfect matchings still evaluate to zero on one of the two systems of boundary flows and are thus external.
\begin{lemma}
If $\sigma$ is a cone in the global zig-zag fan spanned by the rays $\gamma^+$ and $\gamma^-$, and $\eta^+$ and $\eta^-$ are representative zig-zag paths, then $P(\sigma) - Zig(\eta ^+) + Zag(\eta^+)$ and $P(\sigma) - Zag(\eta ^-) + Zig(\eta^-)$ are perfect matchings.
\end{lemma}
\begin{proof}
By Lemma~\ref{contzgs} we know $P(\sigma) - Zig(\eta^+)$ and $P(\sigma) - Zag(\eta^-)$ are elements of $\N^{Q_1}$. Given any face $\face \in Q_2$, its boundary either does not intersect $\eta^+$ or it intersects it in a zig and a zag (see Lemma~\ref{singlein} and Lemma~\ref{uniquerep}). Recalling that the coboundary map adds up the function on the edges in the boundary of each face, we see that in either case
$d(Zig(\eta^+) - Zag(\eta^+)) = 0$. 
Consequently we see that:
$$d(P(\sigma) - Zig(\eta^+) + Zag(\eta^+)) = d(P(\sigma)) = \const{1}$$
so $P(\sigma) - Zig(\zzf ^+) + Zag(\zzf ^+)$ is a perfect matching. Similarly we can see that $P(\sigma) - Zag(\zzf ^-) + Zig(\zzf ^-)$ is a perfect matching.
\end{proof}
We will refer to such perfect matchings as those obtained by `resonating' $P(\sigma)$ along $\eta ^+$ and $\eta^-$ respectively. By resonating along all representative zig-zag paths of a ray, we can go from one extremal perfect matching to another.
\begin{lemma} \label{vertexlink}
Let $\gamma$ be any ray in the global zig-zag fan $\Xi$, and let $\sigma^+$ and $\sigma^-$ be the two 2-dimensional cones containing $\gamma$ (see diagram below). Then $P(\sigma^-) - Zig(\gamma) = P(\sigma^+) - Zag(\gamma)$.
\end{lemma}
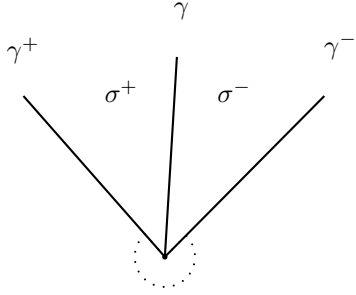
\begin{center}
\psset{xunit=1.0cm,yunit=1.0cm,algebraic=true,dotstyle=*,dotsize=3pt 0,linewidth=0.8pt,arrowsize=3pt 2,arrowinset=0.25}
\begin{pspicture*}(0,-1.85)(6,2.5)
\psline(2.88,-1.14)(1,1)
\psline(2.88,-1.14)(5,1)
\pscustom[linestyle=dotted]{\parametricplot{2.2916078297676727}{7.073278271764529}{0.4*cos(t)+2.88|0.4*sin(t)+-1.14}\lineto(2.88,-1.14)\closepath}
\rput[tl](0.78,1.78){$\gamma^+$}
\rput[tl](5,1.84){$\gamma^-$}
\psline(2.88,-1.14)(3.04,1.52)
\rput[tl](2.08,1.2){$\sigma^+$}
\rput[tl](3.58,1.2){$\sigma^-$}
\rput[tl](3,2.24){$\gamma$}
\psdots[dotsize=2pt 0](2.88,-1.14)
\end{pspicture*} 
\end{center}
\begin{proof}
Let $\face \in Q_2$ be any face. There are two cases which we consider separately. In the first case $\gamma$ is a ray in the local zig-zag fan $\xi(\face)$, i.e. there exists a representative zig-zag flow of $\gamma$ intersecting the boundary of $\face$. This intersection consists of a zig and a zag and therefore $d(Zig(\gamma))_\face = d(Zag(\gamma))_\face = 1 $. Since $P(\sigma^-)$ and $P(\sigma^+)$ are perfect matchings, so are of degree 1, we note that
$$d(P(\sigma^-) - Zig(\gamma))_\face = d(P(\sigma^+) - Zag(\gamma))_\face = 0$$
and together with the fact that $P(\sigma^-) - Zig(\gamma), P(\sigma^+) - Zag(\gamma) \in \N^{Q_1}$, this implies that $P(\sigma^-) - Zig(\gamma)$ and $P(\sigma^+)- Zag(\gamma)$ are both zero on all arrows in the boundary of $\face$.

In the second case, $\gamma$ is not a ray in the local zig-zag fan $\xi(\face)$. Then no representative zig-zag flow of $\gamma$ intersects the boundary of $\face$ and so $Zig(\gamma)$ and $Zag(\gamma)$ both evaluate to zero on all arrows in the boundary. Furthermore, in this case we observe that the images of $\sigma^+$ and $\sigma^-$ lie in the same cone in the local zig-zag fan $\xi(\face)$, so $P_f(\sigma^-)= P_f(\sigma^+)$. Thus $P(\sigma^-) - Zig(\gamma)$ and $P(\sigma^+)- Zag(\gamma)$ are equal on all arrows in the boundary of $\face$. The functions are equal on the boundary of all faces, and therefore are equal.
\end{proof}
\begin{corollary}
The perfect matching $P(\sigma^-) = P(\sigma^+) - Zag(\gamma) + Zig(\gamma)$, i.e. $P(\sigma^-)$ can be obtained from $P(\sigma^+)$ by `resonating' along all representative zig-zag flows of $\gamma$.
\end{corollary}

We have seen that given an extremal perfect matching we can resonate along certain zig-zag paths. We now show that it is always possible to resonate along zig-zag paths in an external perfect matching.

\begin{lemma} \label{canresonate}
Let $\pf$ be any perfect matching which satisfies $ \langle \pf ,\sbf(\gamma) \rangle = 0$ and let $\eta$ be any representative zig-zag flow of $\gamma$. Then either $\pf - Zig(\eta) \in \N^{Q_1}$ or $\pf - Zag(\eta) \in \N^{Q_1}$.
\end{lemma}
\begin{proof}
Using the fact that $\pf \in \N^{Q_1}$ and $\sbf(\gamma) \in \N_{Q_1}$, we observe that $ \langle \pf ,\sbf(\gamma) \rangle = 0$ if and only if $ \langle \pf ,a \rangle = 0$ for every arrow $a$ in the black or white boundary path of any representative of $\gamma$.
Let $\face_k$ be the black or white face with $\eta_k$ and $\eta_{k+1}$ in its boundary. By construction, ever arrow in the boundary of this face except $\eta_k$ and $\eta_{k+1}$, is in the boundary path of $\eta$. Therefore, the perfect matching $\pf$ (which is of degree 1) must evaluate to 1 on either $\eta_k$ or $\eta_{k+1}$, and to zero on all the other arrows in the boundary of $\face_k$. Because this holds for each $k \in \Z$, we note that this forces $\pf$ to evaluate to 1 either on all the zigs, or all the zags of $\zzf$.
\end{proof}

We now show that every external perfect matching, i.e. every perfect matching which evaluates to zero on some boundary flow can be obtained by resonating the extremal perfect matchings. 
\begin{proposition}
Every perfect matching $\pf$ which satisfies $ \langle \pf ,\sbf(\gamma) \rangle = 0$ is of the form:
$$\pf = P(\sigma^+) - \sum_{\eta \in Z} (Zag(\eta) - Zig(\eta))$$
where $Z$ is some set of zig-zag paths which are representatives of $\gamma$.
\end{proposition}
\begin{proof}
Let $\pf$ be such a perfect matching. Let $Z$ be the set of representative zig-zag paths $\eta$ of $\gamma$ for which $\pf - Zig(\eta) \in \N^{Q_1}$.
We note that $\pf^{\prime} := \pf + \sum_{\eta \in Z} (Zag(\eta) - Zig(\eta))$ is a perfect matching which satisfies $ \langle \pf^{\prime} ,\sbf(\gamma) \rangle = 0$ and by Lemma~\ref{canresonate} we see that $\pf^{\prime} - Zag(\gamma) \in \N^{Q_1}$. Since
$$\pf = \pf^{\prime} - \sum_{\eta \in Z} (Zag(\eta) - Zig(\eta))$$
it is sufficient to prove that any perfect matching which satisfies, $ \langle \pf^{\prime} ,\sbf(\gamma) \rangle = 0$ and $\pf^{\prime} - Zag(\gamma) \in \N^{Q_1}$ is equal to $P(\sigma^+)$.
We will be show that if $\pf^{\prime}$ is such a perfect matching, then $ \langle \pf^{\prime} ,\sbf(\gamma^+) \rangle = 0$. The result follows using the uniqueness part of Proposition~\ref{uniqueext}.

Consider the elements $[\pf^{\prime} - P(\sigma^+)]$ and $[P(\sigma^-) - P(\sigma^+)] \neq 0$ in $H^{1}(Q)$, embedded as the rank 2 sublattice of $\Nm$ consisting of degree zero elements (see equation~\ref{eq:Nmes}). We note that the class $\Meq{\sbf(\gamma)} \in \Mm$ restricts to a linear map on $H^{1}(Q)$ which has $[\pf^{\prime} - P(\sigma^+)]$ and $[P(\sigma^-) - P(\sigma^+)]$ in its kernel. The kernel is rank 1, so there exists $k \in \Q$ such that  
$$ [\pf^{\prime} - P(\sigma^+)] = k[P(\sigma^-) - P(\sigma^+)]  $$
We observe that $k \geq 0$ since 
$ \langle P(\sigma^-) ,\sbf(\gamma^+) \rangle > 0$ and
\begin{align*}
 0 \leq  \langle \pf^{\prime} ,\sbf(\gamma^+) \rangle &=  \langle [\pf^{\prime} - P(\sigma^+)] ,\Meq{\sbf(\gamma^+)} \rangle \\ &= k \langle [P(\sigma^-) - P(\sigma^+)] ,\Meq{\sbf(\gamma^+)} \rangle \\ &= k \langle P(\sigma^-) ,\sbf(\gamma^+) \rangle
\end{align*}
Then using Lemma~\ref{vertexlink} and the fact that $\pf^{\prime} - Zag(\gamma) \in \N^{Q_1}$ we see that
\begin{align*}
 0 \geq  -k \langle P(\sigma^+) ,\sbf(\gamma^-) \rangle &= k \langle [P(\sigma^-) - P(\sigma^+)] ,\Meq{\sbf(\gamma^-)} \rangle \\ &=  \langle [\pf^{\prime} - P(\sigma^+)] ,\Meq{\sbf(\gamma^-)} \rangle \\ &=  \langle [\pf^{\prime} - P(\sigma^-) - Zag(\gamma) +Zig(\gamma)] ,\Meq{\sbf(\gamma^-)} \rangle \\ &=  \langle \pf^{\prime} - Zag(\gamma) +Zig(\gamma) ,\sbf(\gamma^-) \rangle \geq 0
\end{align*}
so $k=0$ and thus $ \langle \pf^{\prime} ,\sbf(\gamma^+) \rangle = 0$.


\end{proof}

Finally we see that resonating a perfect matching along different parallel zig-zag paths changes the class in $\Nm$ in the same way.
\begin{lemma}
Let $\eta, \eta^{\prime}$ be representative zig-zag paths of $\gamma$. Then $Zag(\eta) - Zig(\eta) \in \Ng$ and $Zag(\eta^{\prime}) - Zig(\eta^{\prime}) \in \Ng$ project to the same class in $\Nm$.
\end{lemma}
\begin{proof}
First note that if $\eta$ is any zig-zag path, then $Zag(\eta) - Zig(\eta) \in \Ng$ and has degree zero: given any face, $\zzf$ either does not intersect the boundary in which case the function $Zag(\eta) - Zig(\eta)$ is non-zero on all arrows in the boundary, or it intersects in a pair of arrows one of which is a zig, and one a zag. In either case, the sum of $Zag(\eta) - Zig(\eta)$ on the arrows in the boundary is zero. Thus $Zag(\eta) - Zig(\eta)$ defines a class in $H^1(Q)$ considered as the degree zero sublattice of $\Nm$. 

Let $\eta, \eta^{\prime}$ be distinct representative zig-zag paths of $\gamma$. If $\eta^{\prime \prime}$ is any zig-zag path, then we observe that evaluating $Zag(\eta) - Zig(\eta)$ on $\eta^{\prime \prime}$ counts the number of intersections of $\eta$ and $\eta^{\prime \prime}$, with signs depending on whether the intersecting arrow is a zig or zag of $\eta^{\prime \prime}$. Recalling Remark~\ref{intno} we see that this is precisely the intersection number $[\eta^{\prime \prime}] \wedge [\eta]$. Therefore, since $[\eta^\prime]=[\eta]$, the functions $Zag(\eta) - Zig(\eta)$ and $Zag(\eta^{\prime}) - Zig(\eta^{\prime})$ evaluate to the same value on each zig-zag path. The classes of zig-zag paths span a full sublattice of $H_1(Q)$, and therefore $[Zag(\eta) - Zig(\eta)]= [Zag(\eta^{\prime}) - Zig(\eta^{\prime})]$ in $H^1(Q)$ considered as the degree zero sublattice of $\Nm$.
\end{proof}

Therefore the multiplicities of the external perfect matchings along the edge dual to $\Meq{S(\gamma)}$ are binomial coefficients $\binom{r}{n}$ given by choosing $n$ zig-zag paths to resonate out of $r$, the total number of representative zig-zag paths of $\gamma$.

Although he doesn't use the technology of zig-zag fans introduced here, Gulotta in \cite{Gulotta} describes the extremal and external perfect matchings for a `properly ordered' dimer model in essentially the same way that we do. He also gives some justification for their multiplicities. The arguments rely on the assumption that his dimer models are properly ordered. Recall (Remark~\ref{wellorder}) that we proved that a geometrically consistent dimer model is properly ordered. In \cite{Stienstra} Stienstra also produces a correspondence between cones in a fan and perfect matchings. At first sight this looks different to the construction in Section~\ref{CSPM} but it is presumably closely related.

\chapter{Toric algebras and algebraic consistency} \label{ch:ToricA}
In this chapter we introduce the concept of (non-commutative, affine, normal) toric algebras. We prove some general properties, and relate the definition back to dimer models by showing that there is a toric algebra naturally associated to every dimer model. This leads to the definition of another consistency condition which we call `algebraic consistency'. This will play a key role in the rest of this article. Finally we give some examples of algebraically consistent, and non algebraically consistent dimer models.
\section{Toric algebras} \label{toralgsec}
We start the story with two objects, a finite set $Q_0$, and a lattice $N$. Suppose this pair is equipped with two further pieces of information: Firstly a map of lattices
\[ \begin{CD}
\Z^{Q_0}  @>{d}>>  N  \\
\end{CD} \]
with the property that its kernel is $\Z \const{1}$, and secondly a subset $N^+ \subset N$ which is the intersection of a strongly convex rational polyhedral cone in $N \otimes_\Z \R$ with the lattice $N$. 

We consider the corresponding dual map of dual lattices
\[ \begin{CD}
\Z_{Q_0}  @<{\del}<<  M  \\
\end{CD} \]
where $M= N^{\vee}:= \Hom(N,\Z)$ 
contains the saturated monoid $M^+$ corresponding to the dual cone of $N^+$,
$$M^+= (N^+)^{\vee}:= \{ u \in M \mid \langle u , v \rangle \geq 0 \quad \forall v \in N^+ \}$$
Let $M_{ij}:= \del^{-1}(j-i)$ for $i,j \in Q_0$, and denote the intersection with saturated monoid $M^+$ by $M_{ij}^+:=M_{ij} \cap M^+$. We define $\Moth$ to be the submonoid of $M^+$ generated by the elements of $M_{ij}^+$ for all $i,j \in Q_0$.
\begin{definition} \label{nctd}
We call $(N,Q_0, d, N^+)$ \emph{non-commutative toric data} if it satisfies the following two conditions:
\begin{compactenum}
\item $M_{ij}^+$ is non-empty for all $i,j \in Q_0$
\item $N^+= (\Moth)^{\vee} := \{ v \in N \mid \langle u , v \rangle \geq 0 \quad \forall u \in \Moth \}$
\end{compactenum}
\noindent Given some non-commutative toric data, we can associate a \emph{non-commutative toric algebra}
\begin{equation*}
B = \C[\underline{M}^+] = \bigoplus_{i,j\in Q_0} \C[M^+_{ij}]
\end{equation*}
where $\C[M^+_{ij}]$ is the vector space with a basis of monomials of the form $x^m$ for $m\in M^+_{ij}$. The algebra may be considered as a formal matrix algebra
\begin{equation*}
B = \begin{pmatrix}\C[M_{11}^+] & \dots & \C[M_{1n}^+] \\ \vdots &   & \vdots \\ \C[M_{n1}^+] & \dots & \C[M_{nn}^+] \end{pmatrix}
\end{equation*}
with product
\[
\C[M^+_{ij}]\otimes\C[M^+_{jk}]\to \C[M^+_{ik}]
\colon x^{m_1}\otimes x^{m_2} \mapsto x^{m_1+m_2}
\]
\end{definition}
\begin{remark}
We note that $M_{ii}= \del^{-1}(0)$ is obviously independent of the vertex $i \in Q_0$. We denote it by $M_{o}:= \del^{-1}(0)$, and let $M_{o}^+:=M_{o} \cap M^+$. Then $\C[M_o^+]$ is the coordinate ring of a commutative (affine normal) toric variety where $\C[M_o]$ is the coordinate ring of its torus.
\end{remark}

\section{Some examples}
\begin{example}
If $Q_0=\{\bullet\}$ is the set with one element and $N$ is any lattice, then the zero map
\[ \begin{CD}
\Z^{Q_0}  @>{0}>>  N  \\
\end{CD} \]
has kernel $\Z\const{1}$. If $N^+$ is the intersection of a cone with $N$ then $(N,Q_0, d, N^+)$ satisfies all the conditions for non-commutative toric data. In particular $M_{\bullet \bullet}=M $ so $M_{\bullet \bullet}^+ = M^+ = (N^+)^{\vee}$, and
$$B = \C[M^+] $$
is a commutative ring.
Thus we get the usual toric construction of the coordinate ring of the (normal) affine toric variety, defined by the cone $N^+$ in $N$.
\end{example}

\begin{example}
Let $Q_0=\{i,j\}$ and let $N \cong \Z^{2}$ be a rank 2 lattice with some chosen basis. The map
\[ \begin{CD}
\Z^{Q_0}  @>{\left( \begin{smallmatrix}
1 & -1 \\
-1 & 1 \end{smallmatrix} \right)}>>  N  \\
\end{CD} \]
has kernel $\Z\const{1}$. We consider the family of saturated monoids $N_k^+$, for $k \in \N$, corresponding to cones with rays which have primitive generators $(1,0)$ and $(-k,k+1)$.

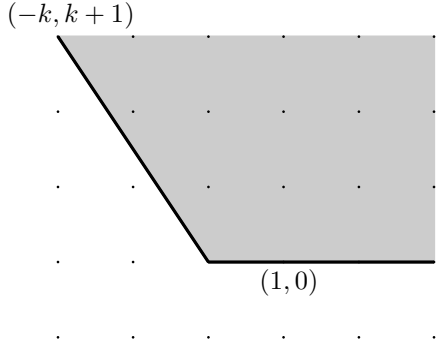
\begin{center}
\psset{xunit=1.0cm,yunit=1.0cm,algebraic=true,dotstyle=*,dotsize=3pt 0,linewidth=0.8pt,arrowsize=3pt 2,arrowinset=0.25}
\newrgbcolor{cccccc}{0.8 0.8 0.8} 
\begin{pspicture*}(-3,-2)(4,4)
\pspolygon[linecolor=cccccc,fillcolor=cccccc,fillstyle=solid](-2,3)(0,0)(3,0)(3,3)
\psline(0,0)(3,0)
\psline(-2,3)(0,0)
\psline[linewidth=1.2pt](-2,3)(0,0)
\psline[linewidth=1.2pt](0,0)(3,0)
\psline(8,0)(14,4)
\psline(8,0)(8,5)
\psline(14,4)(8,0)
\psline(8,0)(8,5)
\psline[linestyle=dotted](6,-2)(14,6)
\psline[linestyle=dotted](15,6)(7,-2)
\psline[linestyle=dotted](5,-2)(13,6)
\rput[tl](0.68,-0.1){$(1,0)$}
\rput[tl](-2.68,3.46){$(-k, k+1)$}
\rput[tl](12,6.3){$(M_k)_{ij}$}
\rput[tl](14.76,6.06){$(M_k)_{ji}$}
\rput[tl](13.8,6.7){$(M_k)_o$}
\psdots[dotsize=1pt 0](-2,3)
\psdots[dotsize=1pt 0](-1,3)
\psdots[dotsize=1pt 0](0,3)
\psdots[dotsize=1pt 0](1,3)
\psdots[dotsize=1pt 0](2,3)
\psdots[dotsize=1pt 0](3,3)
\psdots[dotsize=1pt 0](-2,2)
\psdots[dotsize=1pt 0](-1,2)
\psdots[dotsize=1pt 0](0,2)
\psdots[dotsize=1pt 0](1,2)
\psdots[dotsize=1pt 0](2,2)
\psdots[dotsize=1pt 0](3,2)
\psdots[dotsize=1pt 0](-2,1)
\psdots[dotsize=1pt 0](-1,1)
\psdots[dotsize=1pt 0](0,1)
\psdots[dotsize=1pt 0](1,1)
\psdots[dotsize=1pt 0](2,1)
\psdots[dotsize=1pt 0](3,1)
\psdots[dotsize=1pt 0](-2,0)
\psdots[dotsize=1pt 0](-1,0)
\psdots[dotsize=1pt 0](0,0)
\psdots[dotsize=1pt 0](1,0)
\psdots[dotsize=1pt 0](2,0)
\psdots[dotsize=1pt 0](3,0)
\psdots[dotsize=1pt 0](-2,-1)
\psdots[dotsize=1pt 0](-1,-1)
\psdots[dotsize=1pt 0](0,-1)
\psdots[dotsize=1pt 0](1,-1)
\psdots[dotsize=1pt 0](2,-1)
\psdots[dotsize=1pt 0](3,-1)
\psdots[dotsize=1pt 0](6,3)
\psdots[dotsize=1pt 0](6,2)
\psdots[dotsize=1pt 0](6,1)
\psdots[dotsize=1pt 0](6,0)
\psdots[dotsize=1pt 0](6,-1)
\psdots[dotsize=1pt 0](7,-1)
\psdots[dotsize=1pt 0](7,0)
\psdots[dotsize=1pt 0](7,1)
\psdots[dotsize=1pt 0](7,2)
\psdots[dotsize=1pt 0](7,3)
\psdots[dotsize=1pt 0](8,3)
\psdots[dotsize=1pt 0](8,2)
\psdots[dotsize=1pt 0](8,1)
\psdots[dotsize=1pt 0](8,0)
\psdots[dotsize=1pt 0](8,-1)
\psdots[dotsize=1pt 0](9,3)
\psdots[dotsize=1pt 0](9,2)
\psdots[dotsize=1pt 0](9,1)
\psdots[dotsize=1pt 0](9,0)
\psdots[dotsize=1pt 0](9,-1)
\psdots[dotsize=1pt 0](10,3)
\psdots[dotsize=1pt 0](11,3)
\psdots[dotsize=1pt 0](10,2)
\psdots[dotsize=1pt 0](11,2)
\psdots[dotsize=1pt 0](10,1)
\psdots[dotsize=1pt 0](11,1)
\psdots[dotsize=1pt 0](10,0)
\psdots[dotsize=1pt 0](11,0)
\psdots[dotsize=1pt 0](10,-1)
\psdots[dotsize=1pt 0](11,-1)
\psdots[dotsize=1pt 0](6,4)
\psdots[dotsize=1pt 0](7,4)
\psdots[dotsize=1pt 0](8,4)
\psdots[dotsize=1pt 0](9,4)
\psdots[dotsize=1pt 0](10,4)
\psdots[dotsize=1pt 0](11,4)
\psdots[dotsize=1pt 0](12,4)
\psdots[dotsize=1pt 0](12,3)
\psdots[dotsize=1pt 0](12,2)
\psdots[dotsize=1pt 0](12,1)
\psdots[dotsize=1pt 0](12,0)
\psdots[dotsize=1pt 0](12,-1)
\psdots[dotsize=1pt 0](6,5)
\psdots[dotsize=1pt 0](7,5)
\psdots[dotsize=1pt 0](8,5)
\psdots[dotsize=1pt 0](9,5)
\psdots[dotsize=1pt 0](10,5)
\psdots[dotsize=1pt 0](11,5)
\psdots[dotsize=1pt 0](12,5)
\psdots[dotsize=1pt 0](13,5)
\psdots[dotsize=1pt 0](13,4)
\psdots[dotsize=1pt 0](13,3)
\psdots[dotsize=1pt 0](13,2)
\psdots[dotsize=1pt 0](13,1)
\psdots[dotsize=1pt 0](13,0)
\psdots[dotsize=1pt 0](13,-1)
\psdots[dotsize=1pt 0](14,5)
\psdots[dotsize=1pt 0](14,4)
\psdots[dotsize=1pt 0](14,3)
\psdots[dotsize=1pt 0](14,2)
\psdots[dotsize=1pt 0](14,1)
\psdots[dotsize=1pt 0](14,0)
\psdots[dotsize=1pt 0](14,-1)
\end{pspicture*} 
\end{center}
Then the dual monoid $(M_k)^+$ in the dual lattice $M_k$ is shown below. We can see that the slices $(M_k)_{o}^+$, $(M_k)_{ij}^+$ and $(M_k)_{ji}^+$ are all non-empty, and also that $(\Moth)^\vee = N^+$.
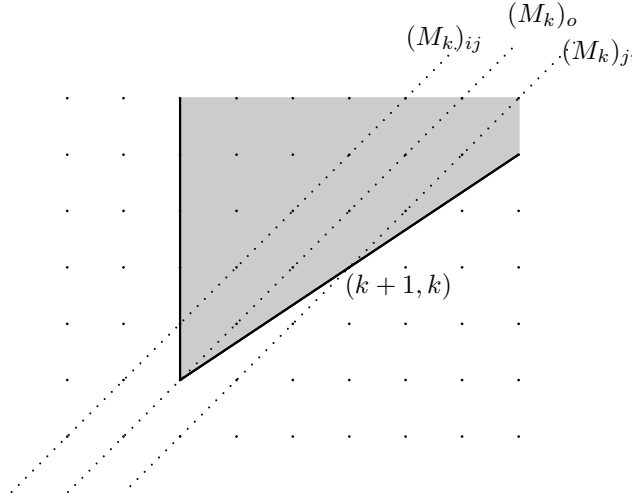
\begin{center}
\psset{xunit=0.75cm,yunit=0.75cm,algebraic=true,dotstyle=*,dotsize=3pt 0,linewidth=0.8pt,arrowsize=3pt 2,arrowinset=0.25}
\newrgbcolor{cccccc}{0.8 0.8 0.8} 
\begin{pspicture*}(4.5,-2.5)(16,7)
\pspolygon[linecolor=cccccc,fillcolor=cccccc,fillstyle=solid](-2,3)(0,0)(3,0)(3,3)
\pspolygon[linecolor=cccccc,fillcolor=cccccc,fillstyle=solid](8,5)(14,5)(14,4)(8,0)
\psline(0,0)(3,0)
\psline(-2,3)(0,0)
\psline[linewidth=1.2pt](-2,3)(0,0)
\psline[linewidth=1.2pt](0,0)(3,0)
\psline(8,0)(14,4)
\psline(8,0)(8,5)
\psline(14,4)(8,0)
\psline(8,0)(8,5)
\psline[linestyle=dotted](6,-2)(14,6)
\psline[linestyle=dotted](15,6)(7,-2)
\psline[linestyle=dotted](5,-2)(13,6)
\rput[tl](0.68,0.1){$(1,0)$}
\rput[tl](-2.68,3.66){$(-k, k+1)$}
\rput[tl](10.92,1.9){$(k+1,k)$}
\rput[tl](12,6.3){$(M_k)_{ij}$}
\rput[tl](14.76,6.06){$(M_k)_{ji}$}
\rput[tl](13.8,6.7){$(M_k)_o$}
\psdots[dotsize=1pt 0](-2,3)
\psdots[dotsize=1pt 0](-1,3)
\psdots[dotsize=1pt 0](0,3)
\psdots[dotsize=1pt 0](1,3)
\psdots[dotsize=1pt 0](2,3)
\psdots[dotsize=1pt 0](3,3)
\psdots[dotsize=1pt 0](-2,2)
\psdots[dotsize=1pt 0](-1,2)
\psdots[dotsize=1pt 0](0,2)
\psdots[dotsize=1pt 0](1,2)
\psdots[dotsize=1pt 0](2,2)
\psdots[dotsize=1pt 0](3,2)
\psdots[dotsize=1pt 0](-2,1)
\psdots[dotsize=1pt 0](-1,1)
\psdots[dotsize=1pt 0](0,1)
\psdots[dotsize=1pt 0](1,1)
\psdots[dotsize=1pt 0](2,1)
\psdots[dotsize=1pt 0](3,1)
\psdots[dotsize=1pt 0](-2,0)
\psdots[dotsize=1pt 0](-1,0)
\psdots[dotsize=1pt 0](0,0)
\psdots[dotsize=1pt 0](1,0)
\psdots[dotsize=1pt 0](2,0)
\psdots[dotsize=1pt 0](3,0)
\psdots[dotsize=1pt 0](-2,-1)
\psdots[dotsize=1pt 0](-1,-1)
\psdots[dotsize=1pt 0](0,-1)
\psdots[dotsize=1pt 0](1,-1)
\psdots[dotsize=1pt 0](2,-1)
\psdots[dotsize=1pt 0](3,-1)
\psdots[dotsize=1pt 0](6,3)
\psdots[dotsize=1pt 0](6,2)
\psdots[dotsize=1pt 0](6,1)
\psdots[dotsize=1pt 0](6,0)
\psdots[dotsize=1pt 0](6,-1)
\psdots[dotsize=1pt 0](7,-1)
\psdots[dotsize=1pt 0](7,0)
\psdots[dotsize=1pt 0](7,1)
\psdots[dotsize=1pt 0](7,2)
\psdots[dotsize=1pt 0](7,3)
\psdots[dotsize=1pt 0](8,3)
\psdots[dotsize=1pt 0](8,2)
\psdots[dotsize=1pt 0](8,1)
\psdots[dotsize=1pt 0](8,0)
\psdots[dotsize=1pt 0](8,-1)
\psdots[dotsize=1pt 0](9,3)
\psdots[dotsize=1pt 0](9,2)
\psdots[dotsize=1pt 0](9,1)
\psdots[dotsize=1pt 0](9,0)
\psdots[dotsize=1pt 0](9,-1)
\psdots[dotsize=1pt 0](10,3)
\psdots[dotsize=1pt 0](11,3)
\psdots[dotsize=1pt 0](10,2)
\psdots[dotsize=1pt 0](11,2)
\psdots[dotsize=1pt 0](10,1)
\psdots[dotsize=1pt 0](11,1)
\psdots[dotsize=1pt 0](10,0)
\psdots[dotsize=1pt 0](11,0)
\psdots[dotsize=1pt 0](10,-1)
\psdots[dotsize=1pt 0](11,-1)
\psdots[dotsize=1pt 0](6,4)
\psdots[dotsize=1pt 0](7,4)
\psdots[dotsize=1pt 0](8,4)
\psdots[dotsize=1pt 0](9,4)
\psdots[dotsize=1pt 0](10,4)
\psdots[dotsize=1pt 0](11,4)
\psdots[dotsize=1pt 0](12,4)
\psdots[dotsize=1pt 0](12,3)
\psdots[dotsize=1pt 0](12,2)
\psdots[dotsize=1pt 0](12,1)
\psdots[dotsize=1pt 0](12,0)
\psdots[dotsize=1pt 0](12,-1)
\psdots[dotsize=1pt 0](6,5)
\psdots[dotsize=1pt 0](7,5)
\psdots[dotsize=1pt 0](8,5)
\psdots[dotsize=1pt 0](9,5)
\psdots[dotsize=1pt 0](10,5)
\psdots[dotsize=1pt 0](11,5)
\psdots[dotsize=1pt 0](12,5)
\psdots[dotsize=1pt 0](13,5)
\psdots[dotsize=1pt 0](13,4)
\psdots[dotsize=1pt 0](13,3)
\psdots[dotsize=1pt 0](13,2)
\psdots[dotsize=1pt 0](13,1)
\psdots[dotsize=1pt 0](13,0)
\psdots[dotsize=1pt 0](13,-1)
\psdots[dotsize=1pt 0](14,5)
\psdots[dotsize=1pt 0](14,4)
\psdots[dotsize=1pt 0](14,3)
\psdots[dotsize=1pt 0](14,2)
\psdots[dotsize=1pt 0](14,1)
\psdots[dotsize=1pt 0](14,0)
\psdots[dotsize=1pt 0](14,-1)
\end{pspicture*} 
\end{center}
Therefore we have well defined non-commutative data, and the corresponding non-commutative toric algebra is  
\begin{equation*}
B = \C[\underline{M}_k^+] = \begin{pmatrix}\C[XY] & Y\C[XY] \\ X^{k+1}Y^k \C[XY] & \C[XY] \end{pmatrix}
\end{equation*}

\end{example}

\begin{example} \label{strongconquiv}
Let $Q=(Q_0,Q_1)$ be a strongly connected quiver and consider any sublattice $L \subset \Z_{Q_1}$ with a torsion free quotient, which is contained in the kernel of the boundary map $\del : \Z_{Q_1} \rightarrow \Z_{Q_0}$. Using this we can construct some non-commutative toric data and the corresponding algebra. We define $M:= \Z_{Q_1}/L$ to be the quotient lattice and let $N:=M^{\vee}$. Since $L$ is in the kernel, the boundary map descends to a map $\del : M \rightarrow \Z_{Q_0}$ which has a dual map $d : \Z^{Q_0} \rightarrow N$. Furthermore, as the quiver is connected, we see that the kernel of $d$ is $\Z \const{1}$ as we required.

Let $M^{ef}$ be the image of $\N_{Q_1}$ in $M$ under the quotient map, and define $N^+ := (M^{ef})^{\vee}$. This is the saturated monoid given by a strongly convex rational polyhedral cone in $N$. 

We now see that $(N,Q_0, d, N^+)$ is non-commutative toric data. First note that $M^{ef}$ is generated by the images of the arrows in $M^+$. The image of an arrow is in $M_{ta,ha}^+$, so $M^{ef} \subseteq \Moth \subseteq M^+$. Dualising this we see that
$$N^+ = (M^+)^\vee \subseteq (\Moth)^\vee \subseteq (M^{ef})^\vee = N^+$$
so condition (2) of Definition~\ref{nctd} is satisfied. Condition (1) is a consequence of the fact that $Q$ is strongly connected. Given any two vertices $i,j \in Q_0$, there exists an (oriented) path from $i$ to $j$. 
This defines an element of $\N_{Q_1}$ and therefore a class in $M^{ef} \subset M^+$. Furthermore since the path is from $i$ to $j$, applying the boundary map to the class we get the element $(j-i) \in \Z_{Q_0}$, and we have therefore constructed an element in $M_{ij}^+$.
\end{example}
\section{Some properties of toric algebras}
We first identify the centre of a toric algebra.
\begin{lemma}\label{cent} The centre $Z(B)$ of a non-commutative toric algebra $B = \C[\underline{M}^+]$ is isomorphic to $R:=\C [M_o^+]$.
\end{lemma}
\begin{proof} First we note that we can consider $B$ as an $R$-algebra, using the map $R \rightarrow R \operatorname{Id} \subset B$, where $\operatorname{Id}$ is the identity element of $B$ (which exists because $M_{o}^+ $ is a cone). This is true because $x^{m}x^{m_{k l}} = x^{m+m_{k l}} = x^{m_{k l}}x^{m}$ for any $m \in M_o^+$ and $m_{k l} \in M_{k l}^+$. In particular $R \operatorname{Id} \subset Z(B)$.

Now suppose $\underline{z}= (z_{ij}) \in Z(B)$. For every $k,l \in Q_0$, by definition $M_{k l}^+$ is non-empty, so we may fix some $m_{k l}\in M_{k l}^+$. Let $x^{m_{k l}} \in B$ be the corresponding element in $B$.
Then
$$ (\underline{z}x^{m_{k k}})_{ij} = 
                                  \begin{cases}
                                  z_{ik}x^{m_{k k}} & j=k \\
                                  0 & \text{otherwise}
                                  \end{cases}
$$
and
$$
(x^{m_{k k}}\underline{z})_{ij} = 
                                  \begin{cases}
                                  x^{m_{k k}}z_{kj} & i=k \\
                                  0 & \text{otherwise}
                                  \end{cases}
$$
Since $\underline{z}$ is central these must be equal and thus for all $i \neq k $ we have $z_{ik}x^{m_{k k}} = 0$. We note that $\C [M_{i k}^+]$ is a torsion free right $\C [M_{k k}^+]$- module and therefore $z_{ik}=0$ for all $i \neq k$. Thus every off-diagonal entry in $\underline{z}$ is zero.

Similarly by considering $(\underline{z}x^{m_{k l}})_{kl}$ and $(x^{m_{k l}}\underline{z})_{kl}$ we observe that $z_{kk}x^{m_{k l}} = x^{m_{k l}}z_{ll}$. Therefore $z_{kk}x^{m_{k l}} = z_{ll}x^{m_{k l}}$ and so $z_{kk}= z_{ll}$ in $R$. Thus $Z(B) \subseteq R \operatorname{Id}$ and we are done.


\end{proof}


We now show that any non-commutative toric algebra is a toric $R$-order. This is defined as follows:
\begin{definition}
Let $R \subset T$ be the coordinate ring of an affine toric variety where $T = \C [M_o]
$ is the coordinate ring of the dense open orbit of the torus action and $M_o$ is the character lattice.
A \emph{toric $R$-order} is an $M_o$-graded $R$-algebra $B$, with a (graded) embedding 
$$ B \hookrightarrow \operatorname{Mat}_{n}(T) $$
as a finitely generated $R$-submodule, such that 
\begin{equation*}  B \otimes_R T \lra{} \operatorname{Mat}_{n}(T) \end{equation*}
is an isomorphism.
\end{definition}
\begin{remark} We note that Bocklandt gives a definition of toric order in \cite{Bocklandt2}. He includes an additional condition (T02) which insists that the standard idempotents are contained in the image of $B$. We do not include this condition in our definition, although we note that it is clearly satisfied by any toric algebra. 
\end{remark}
We first prove the following lemma.
\begin{lemma} \label{posshift} Given any $m \in M$ there exists an element $\zeta \in M_o^+$ such that $ m+ \zeta \in M^+$.\end{lemma}
\begin{proof}
Suppose there exists an element $\zeta^{\prime} \in M_o$ which is contained in the interior of the cone $M^+$. Then $ \langle \zeta^{\prime} , n \rangle$ is strictly positive for all $n \in N^+$ and so by adding sufficiently large positive integer multiples of $\zeta^{\prime}$, it is possible to shift any element of $M$ into $M^+$. 
We prove that such a $\zeta^{\prime} \in M_o$ exists. 

We know that $M_o^+$ is non-empty, so it suffices to prove that $M_o^+$ is not contained in any boundary hyperplane $\mathcal{H} := \{ m \in M \mid \langle m , n \rangle =0 \}$ of $M^+$, for some $n \in N^+$. Suppose for a contradiction that this is the case. First we note that if there exists $m \in M_{ij}^+$ (for any $i,j \in Q_0$) which is not in $\mathcal{H}$, then for any $m^{\prime} \in M_{ji}^+$ (which is non-empty by definition) we have 
$$ \langle m^{\prime} , n \rangle = \langle m^{\prime} + m , n \rangle - \langle m , n \rangle = -\langle m , n \rangle < 0$$
This contradicts the fact that $m^{\prime} \in M^+$ and $n \in N^+$. Therefore we see that $\Moth \subseteq \mathcal{H}$, however this implies that 
$$ \langle \wt{m} , n \rangle = 0 = \langle \wt{m} , -n \rangle  \qquad \forall \; \wt{m} \in \Moth$$
and thus, since $(\Moth)^\vee = N^+$, that $\pm n \in N^+$. This contradicts the fact that $N^+$ comes from a strongly convex cone.
\end{proof}
We are now in a position to prove the toric $R$-order condition.
\begin{lemma} \label{toricorder}
A non-commutative toric algebra $B = \C[\underline{M}^+]$ is a toric $R$-order, where $R:=\C [M_o^+]$.
\end{lemma}
\begin{proof}
We show that there exists an embedding
$$ \C[\underline{M}^+] \hookrightarrow \operatorname{Mat}_{Q_0}(\C [M_o]) $$
such that 
\begin{equation*} \label{localised} \Psi : \C[\underline{M}^+] \otimes_R \C [M_o] \lra{} \operatorname{Mat}_{Q_0}(\C [M_o]) \end{equation*}
is an isomorphism, where $\operatorname{Mat}_{Q_0}(\C [M_o])$ denotes the space of $|Q_0| \times |Q_0|$ matrices over $\C [M_o]$.

First recall that $M_o$ is the kernel of the lattice map $\del : M \lra{} \Z_{Q_0}$. Let $\del(M)$ denote the image of this map, so we have a short exact sequence,
\[ \begin{CD}
0 @>>> M_o @>>> M @>>> \del(M) @>>> 0 \\
\end{CD} \]
Since $\del(M)$ is a sublattice of free lattice $\Z_{Q_0}$, it is also free. Therefore $M \twoheadrightarrow \del(M) $ has a section, and the short exact sequence splits, yielding a projection $\prii: M \lra{} M_o$. Thus $\C[\underline{M}^+]$ becomes an $M_o$-graded $R$-algebra and there is a (graded) $R$-algebra map given by
$$ \C[\underline{M}^+] \hookrightarrow \operatorname{Mat}_{Q_0}(\C [M_o]), \qquad z^m \mapsto (z^{\prii(m)})_{ij} \quad \forall m \in M_{ij}^+ $$
as required. 
We now show that the corresponding map $\Psi$ is an isomorphism.
Using the $\Z_{Q_0}$ grading, which is respected by $\Psi$, it is sufficient to prove that for each $i,j \in Q_0$, the restriction
\begin{equation*} 
\Psi_{ij} : \C[M_{ij}^+] \otimes_R \C [M_o] \lra{} \C [M_o] , \qquad z^m \otimes z^{m^{\prime}} \mapsto z^{\prii(m)+m^{\prime}} 
\end{equation*}
is an isomorphism. First we note that $\Psi_{ij}$ is surjective:
Let $m_0 \in M_o$ be any element. By definition $M_{ij}^+$ is non-empty so let $\wh{m} \in M_{ij}^+$ be an element. Then we see that $z^{\wh{m}} \otimes z^{m_0 - \prii(\wh{m})}$ maps to $z^{m_0}$. 
To prove injectivity, suppose $z^{\prii(m)+m_0} = z^{\prii(\wt{m})+\wt{m_0}}$. We note that this implies that $\wt{m}-m = m_0 - \wt{m_0}$ in $M_o$. Using Lemma~\ref{posshift}, let $ \zeta \in M_o^+$ be an element such that $\wt{m_0} - m_0 + \zeta \in M_o^+$.
Then:
$$z^m \otimes z^{m_0} = z^{\wt{m}} z^{\wt{m_0} - m_0 + \zeta} \otimes z^{m_0-\zeta} = z^{\wt{m}} \otimes z^{\wt{m_0} - m_0 + \zeta} z^{m_0-\zeta} = z^{\wt{m}} \otimes z^{\wt{m_0}} $$
\end{proof}

\section{Algebraic consistency for dimer models} \label{aconssec}
We observe that there is a non-commutative toric algebra naturally associated to every non-degenerate dimer model. 
Given any dimer model, we saw in Section~\ref{quivalgsec} that there is a corresponding quiver. This quiver is strongly connected. This follows from the fact that the quiver is a tiling and given any unoriented path we can construct an oriented one by replacing any arrow $a$ which occurs with the wrong orientation by the path from $ha$ to $ta$ around the boundary of one of the faces containing $a$. Let $L := \del(\const{1}^\perp)$ be the sublattice of $\Z_{Q_1}$ where $\const{1}^{\perp} := \{ u \in \Z_{Q_2} \mid \langle u , \const{1} \rangle = 0 \}$, as defined in Section~\ref{conseqgcsec}. We recall that this is contained in the kernel of the boundary map, and that the quotient $M=\Z_{Q_1}/L$ is torsion free. Then using Section \ref{strongconquiv} we can associate to this a non-commutative toric algebra $B := \C[\underline{M}^+]$. 
\begin{remark}
We observe that the notation has been set up in a consistent way, i.e. the objects $N,Q_0, d, N^+$ and also $M, M^+, M_o$ corresponding to this toric algebra are the objects of the same name, seen in Sections \ref{symmetries} and \ref{conseqgcsec} respectively.
\end{remark}

We are now in a position to introduce a final consistency condition which we shall call `algebraic consistency'. We saw in Section~\ref{conseqgcsec} that every path $p$ from $i$ to $j$ in $Q$ defines a class $\Meq{p} \in \Mg$ which has boundary $j-i$. Since this was obtained by summing the arrows in the path it is clear that this class actually lies in $\Mg^{ef}_{ij} \subset \Mg^+_{ij}$. Furthermore we saw that if two paths are F-term equivalent, then they define the same class.

The F-term equivalence classes of paths in the quiver form a natural basis for the superpotential algebra $A$. By definition the classes in $\Mg^+_{ij}$ for all $i,j \in Q_0$ form a basis for the algebra $B = \C[\underline{M}^+]$. Therefore, since it is defined on a basis, there is a well defined $\C$-linear map from $A$ to $B$. It follows from the definition that if paths $p$ and $q$ are composable then $\Meq{pq}= \Meq{p} + \Meq{q}$ and thus the map respects the multiplication, i.e. we have defined an algebra morphism:
\begin{equation}\label{algmap}
\mathpzc{h}: A \longrightarrow \C [ \underline{M}^+]; \qquad p \mapsto x^{\Meq{p}}
\end{equation}

\begin{definition}\label{acons} A dimer model is called algebraically consistent if the morphism (\ref{algmap}) is an
isomorphism.
\end{definition}
We note that the map is an isomorphism if and only if restricts to an isomorphism on each $i,j$th piece, i.e. $e_i A e_j \lra{\cong} \C[M_{ij}^+]$ for each $i,j \in Q_0$.
\begin{remark}\label{hinject}
Theorem~\ref{Uniqueness} says that for a geometrically consistent dimer model two paths are F-term equivalent if and only if they have the same class in $M$. This is equivalent to saying that for geometrically consistent dimer models the map $\mathpzc{h}$ is injective. 
\end{remark}
\section{Example}
We will construct a large class of examples of dimer models which are algebraically consistent in Chapter 6. For now we look at an example which is not algebraically consistent to see how things can go wrong.
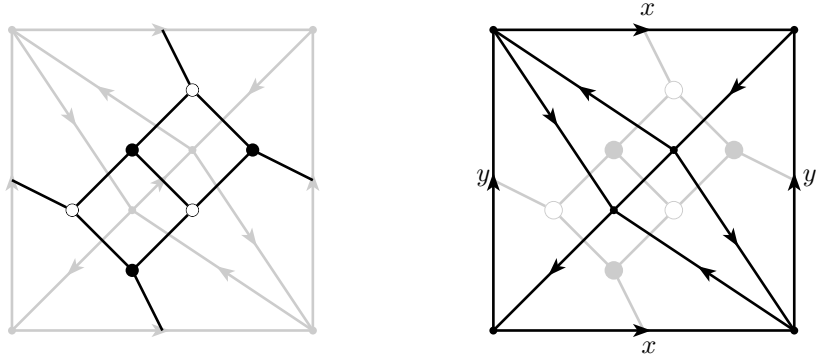
\begin{center}
\newrgbcolor{cccccc}{0.8 0.8 0.8}
\newrgbcolor{qqfftt}{0 0 0}
\psset{xunit=0.8cm,yunit=0.8cm}
\begin{pspicture*}(-1,-2)(14,5)
\psset{xunit=0.8cm,yunit=0.8cm,algebraic=true,dotstyle=*,dotsize=3pt 0,linewidth=1.0pt,arrowsize=3pt 2,arrowinset=0.25}
\psline[ArrowInside=->,linecolor=cccccc](2,1)(0,-1)
\psline[ArrowInside=->,linecolor=cccccc](0,-1)(0,4)
\psline[ArrowInside=->,linecolor=cccccc](0,4)(2,1)
\psline[ArrowInside=->,linecolor=cccccc](2,1)(3,2)
\psline[ArrowInside=->,linecolor=cccccc](3,2)(0,4)
\psline[ArrowInside=->,linecolor=cccccc](0,4)(5,4)
\psline[ArrowInside=->,linecolor=cccccc](5,4)(3,2)
\psline[ArrowInside=->,linecolor=cccccc](3,2)(5,-1)
\psline[ArrowInside=->,linecolor=cccccc](5,-1)(5,4)
\psline[ArrowInside=->,linecolor=cccccc](5,-1)(2,1)
\psline[ArrowInside=->,linecolor=cccccc](0,-1)(5,-1)
\psline[linecolor=cccccc](9,1)(10,0)
\psline[linecolor=cccccc](10,0)(11,1)
\psline[linecolor=cccccc](11,1)(10,2)
\psline[linecolor=cccccc](10,2)(9,1)
\psline[linecolor=cccccc](10,2)(11,3)
\psline[linecolor=cccccc](11,3)(12,2)
\psline[linecolor=cccccc](12,2)(11,1)
\psline[linecolor=cccccc](10,0)(10.5,-1)
\psline[linecolor=cccccc](10.5,4)(11,3)
\psline[linecolor=cccccc](9,1)(8,1.5)
\psline[linecolor=cccccc](12,2)(13,1.5)
\psline[ArrowInside=->,linecolor=qqfftt](10,1)(8,-1)
\psline[ArrowInside=->,linecolor=qqfftt](8,-1)(8,4)
\psline[ArrowInside=->,linecolor=qqfftt](8,4)(10,1)
\psline(10,1)(11,2)
\psline[ArrowInside=->,linecolor=qqfftt](11,2)(8,4)
\psline[ArrowInside=->,linecolor=qqfftt](8,4)(13,4)
\psline[ArrowInside=->,linecolor=qqfftt](13,4)(11,2)
\psline[ArrowInside=->,linecolor=qqfftt](11,2)(13,-1)
\psline[ArrowInside=->,linecolor=qqfftt](13,-1)(13,4)
\psline[ArrowInside=->,linecolor=qqfftt](13,-1)(10,1)
\psline[ArrowInside=->,linecolor=qqfftt](8,-1)(13,-1)
\psline(1,1)(0,1.5)
\psline(1,1)(2,0)
\psline(2,0)(2.5,-1)
\psline(2,0)(3,1)
\psline(3,1)(2,2)
\psline(2,2)(1,1)
\psline(2,2)(3,3)
\psline(3,3)(4,2)
\psline(4,2)(3,1)
\psline(4,2)(5,1.5)
\psline(3,3)(2.5,4)
\rput[tl](13.14,1.68){$y$}
\rput[tl](7.72,1.68){$y$}
\rput[tl](10.44,4.44){$x$}
\rput[tl](10.46,-1.18){$x$}
\psdots[dotsize=5pt 0,dotstyle=o](1,1)
\psdots[dotsize=5pt 0,linecolor=black](2,0)
\psdots[dotsize=5pt 0,linecolor=black](2,2)
\psdots[dotsize=5pt 0,dotstyle=o](3,1)
\psdots[dotsize=5pt 0,dotstyle=o](3,3)
\psdots[dotsize=5pt 0,linecolor=black](4,2)
\psdots[linecolor=cccccc](0,4)
\psdots[linecolor=cccccc](0,-1)
\psdots[linecolor=cccccc](5,-1)
\psdots[linecolor=cccccc](5,4)
\psdots(8,4)
\psdots(8,-1)
\psdots[dotsize=7pt 0,dotstyle=o,linecolor=cccccc](9,1)
\psdots[dotsize=7pt 0,linecolor=cccccc](10,0)
\psdots[dotsize=7pt 0,linecolor=cccccc](10,2)
\psdots[dotsize=7pt 0,dotstyle=o,linecolor=cccccc](11,1)
\psdots[dotsize=7pt 0,dotstyle=o,linecolor=cccccc](11,3)
\psdots[dotsize=7pt 0,linecolor=cccccc](12,2)
\psdots(13,4)
\psdots(13,-1)
\psdots[linecolor=cccccc](2,1)
\psdots[linecolor=cccccc](3,2)
\psdots[linecolor=black](10,1)
\psdots[linecolor=black](11,2)
\end{pspicture*} 
\end{center}
In this example the map $\mathpzc{h}$ is neither injective nor surjective. To see the lack of injectivity we consider the two paths $xy$ and $yx$ around the boundary of the fundamental domain in the diagram above. Neither path can be altered using F-term relations and so they are distinct elements in $A$. However, considered as sum of arrows in $\Z_{Q_1}$, they are the same element, so they have the same class in $M$. 

To see the lack of surjectivity we consider the path $x$ which is a loop at a vertex. Therefore this path maps to an element $m \in M_o^+$. If the map $\mathpzc{h}$ is surjective, then there must also be a path in $e_1Ae_1$ and $e_2Ae_2$ which maps to $m$, where $e_1, e_2$ are the idempotents corresponding to the other two vertices. However, no such paths exist.


\chapter{Geometric consistency implies algebraic consistency}
In this chapter we prove the following main result.
\begin{theorem} \label{GCACTHM}
If a dimer model on a torus is geometrically consistent, then it is algebraically consistent.
\end{theorem}
Our proof will rely on the explicit description of extremal perfect matchings we gave in Chapter~\ref{zzfpmchap}, together with the knowledge of their values on the boundary flows of certain zig-zag paths. The most intricate section of the argument will be the proof of the following proposition which we give in Section~\ref{Proofoflem}. 
\begin{proposition}\label{EXTREM}
If a dimer model on a torus is geometrically consistent, then for all vertices $i,j \in \Qcov_0$ in the universal cover there exists a path from $i$ to $j$ on which some extremal perfect matching evaluates to zero.
\end{proposition}
Thus we see that the extremal perfect matchings play a key role in the theory. We then see in Section~\ref{prop2thm} that this implies the theorem. We start with a few technical lemmas.

\section{Flows which pass between two vertices}
We fix two vertices $i, j$ of the universal cover of the quiver $\Qcov$. In this section we study the zig-zag flows which pass between $i$ and $j$, i.e. they have $i$ and $j$ on opposite sides. These are important as they are the flows that intersect every path from $i$ to $j$. We will see that the classes of these flows generate two convex cones in $H_1(Q) \otimes_\Z \R$. These will be used in the proof of Proposition~\ref{EXTREM} in the next section.

We define the two sets of zig-zag flows which pass between $i$ and $j$ as follows:
$$ \ZC_+:= \{ \zzf \mid \zzf \text{ has } i\text{ on the left and }j \text{ on the right}\}  $$
$$ \ZC_-:= \{ \zzf \mid \zzf \text{ has } i\text{ on the right and }j \text{ on the left}\}  $$
First we show that there always exist flows which pass between distinct vertices:
\begin{lemma} \label{nonemptyC}
If $i,j \in \Qcov_0$ are distinct vertices then $\ZC_+$ and  $\ZC_-$ are non-empty.
\end{lemma}
\begin{proof}
We consider the collection of black faces which have $i$ as a vertex. We are interested in zig-zag flows which intersect the boundary of one or more of these faces and have $i$ on the left. Any zig-zag pair in the boundary of one of these faces defines such a flow if and only if $i$ is not the head of the zig. Let $p$ be any path from $j$ to a vertex $v_0$ contained in one of these flows, say $\zzf$. We may assume that $v_0$ is the first such vertex in $p$, so $p$ doesn't intersect any of the flows. Since $\zzf$ intersects the boundary of one of the black faces, there is a path $p^\prime$ along $\zzf$ from $v_0$ to $v_1$, a vertex of one of the black faces. We may also assume that this is the first such vertex along $p^\prime$, so $p^\prime$ does not intersect the boundary of any of the black faces in an arrow; this implies that $v_1 \neq i$. We note that $v_1$ is the head of a zig $\zzf_0^\prime$ 
of a zig-zag pair in the boundary of one of the black faces. Since $v_1 \neq i$ this defines a zig-zag flow $\zzf^\prime$ which has $i$ on the left.

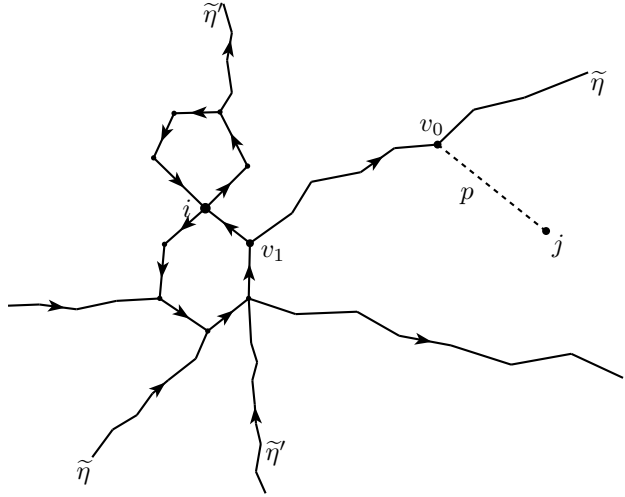
\begin{center}
\psset{xunit=0.8cm,yunit=0.8cm,algebraic=true,dotstyle=*,dotsize=3pt 0,linewidth=0.8pt,arrowsize=3pt 2,arrowinset=0.25}
\begin{pspicture*}(-4,-3)(7,6)
\psline[ArrowInside=->](-0.44,2.08)(0.26,2.78)
\psline[ArrowInside=->](0.26,2.78)(-0.2,3.68)
\psline[ArrowInside=->](-0.2,3.68)(-0.96,3.66)
\psline[ArrowInside=->](-0.96,3.66)(-1.3,2.92)
\psline[ArrowInside=->](-1.3,2.92)(-0.44,2.08)
\psline[ArrowInside=->](-0.44,2.08)(-1.12,1.48)
\psline[ArrowInside=->](-1.12,1.48)(-1.2,0.58)
\psline[ArrowInside=->](-1.2,0.58)(-0.4,0.04)
\psline[ArrowInside=->](-0.4,0.04)(0.28,0.58)
\psline[ArrowInside=->](0.28,0.58)(0.3,1.5)
\psline[ArrowInside=->](0.3,1.5)(-0.44,2.08)
\psline(0.56,-2.66)(0.4,-2.3)
\psline(0.4,-2.3)(0.48,-1.84)
\psline(0.48,-1.84)(0.4,-1.5)
\psline[ArrowInside=->](0.4,-1.5)(0.38,-1.18)
\psline(0.38,-1.18)(0.34,-0.78)
\psline(0.34,-0.78)(0.42,-0.48)
\psline(0.42,-0.48)(0.32,-0.16)
\psline(0.32,-0.16)(0.28,0.58)
\psline(-0.2,3.68)(0,4)
\psline(0,4)(-0.08,4.54)
\psline[ArrowInside=->](-0.08,4.54)(0,5)
\psline(0,5)(-0.14,5.48)
\psline(-2.32,-2.06)(-1.98,-1.6)
\psline(-1.98,-1.6)(-1.58,-1.36)
\psline(-1.58,-1.36)(-1.32,-1)
\psline[ArrowInside=->](-1.32,-1)(-1.06,-0.78)
\psline(-1.06,-0.78)(-0.6,-0.42)
\psline(-0.6,-0.42)(-0.4,0.04)
\psline(0.3,1.5)(1,2)
\psline(1,2)(1.32,2.52)
\psline(1.32,2.52)(2.14,2.68)
\psline[ArrowInside=->](2.14,2.68)(2.7,3.08)
\psline(2.7,3.08)(3.42,3.14)
\psline(3.42,3.14)(4.02,3.72)
\psline(4.02,3.72)(4.86,3.92)
\psline(4.86,3.92)(5.92,4.34)
\psline(-3.72,0.46)(-3.2,0.48)
\psline[ArrowInside=->](-3.2,0.48)(-2.58,0.4)
\psline(-2.58,0.4)(-1.92,0.54)
\psline(-1.92,0.54)(-1.2,0.58)
\psline(0.28,0.58)(1.06,0.32)
\psline(1.06,0.32)(2.08,0.36)
\psline(2.08,0.36)(2.78,-0.04)
\psline[ArrowInside=->](2.78,-0.04)(3.64,-0.2)
\psline(3.64,-0.2)(4.64,-0.52)
\psline(4.64,-0.52)(5.64,-0.34)
\psline(5.64,-0.34)(6.5,-0.74)
\psline[linestyle=dashed,dash=2pt 2pt](5.22,1.7)(3.42,3.14)
\rput[tl](5.96,4.36){$\zzf$}
\rput[tl](-2.58,-2.1){$\zzf$}
\rput[tl](3.12,3.56){$v_0$}
\rput[tl](-0.82,2.22){$i$}
\rput[tl](0.48,1.46){$v_1$}
\rput[tl](5.32,1.64){$j$}
\rput[tl](0.56,-1.78){$\zzf^\prime$}
\rput[tl](-0.48,5.5){$\zzf^\prime$}
\rput[tl](3.8,2.42){$p$}
\psdots[dotsize=4pt 0](-0.44,2.08)
\psdots[dotsize=2pt 0](0.26,2.78)
\psdots[dotsize=2pt 0](-0.2,3.68)
\psdots[dotsize=2pt 0](-0.96,3.66)
\psdots[dotsize=2pt 0](-1.3,2.92)
\psdots[dotsize=2pt 0](-1.12,1.48)
\psdots[dotsize=2pt 0](-1.2,0.58)
\psdots[dotsize=2pt 0](-0.4,0.04)
\psdots[dotsize=2pt 0](0.28,0.58)
\psdots(0.3,1.5)
\psdots(3.42,3.14)
\psdots(5.22,1.7)
\end{pspicture*}

\end{center}

By construction, the flows $\zzf^\prime$ and $\zzf$ intersect in either $\zzf_0^\prime$ or $\zzf_1^\prime$. Since this intersection is unique, $\zzf^\prime$ doesn't intersect $p^\prime$ in an arrow. Since $\zzf^\prime$ doesn't intersect $p$, the path $pp^\prime$ from $j$ to the head of a zig of $\zzf^\prime$ doesn't intersect $\zzf^\prime$. Therefore $j$ is on the right of $\zzf^\prime$.
\end{proof}

\begin{lemma}\label{pmcoincide}
If $\zzf$ and $\zzf^\prime$ are parallel zig-zag flows which pass between $i$ and $j$ then they are both in $\ZC_+$ or both in $\ZC_-$.
If $\zzf$ and $\zzf^\prime$ are anti-parallel zig-zag flows which pass between $i$ and $j$ then one is in $\ZC_+$ and the other is in $\ZC_-$.
\end{lemma}
\begin{proof}
We prove the parallel case, the anti-parallel case follows using a similar argument. Without loss of generality, suppose that $\zzf \in \ZC_+$. Then there is a path from $i$ to a vertex $v$ of $\zzf$ which does not intersect $\zzf$. In particular all vertices of this path are also on the left of $\zzf$. If $\zzf^\prime$ is on the right of $\zzf$ then this path does not intersect $\zzf^\prime$, and so $i$ is on the left of $\zzf^\prime$. 
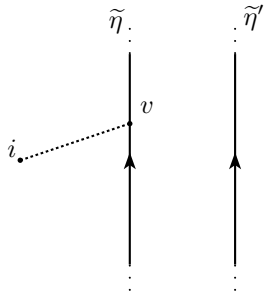
\begin{center}
\psset{xunit=0.7cm,yunit=0.7cm,algebraic=true,dotstyle=*,dotsize=3pt 0,linewidth=0.8pt,arrowsize=3pt 2,arrowinset=0.25}
\begin{pspicture*}(-3.8,0.3)(2,6)
\psline[linestyle=dotted](-1,5.5)(-1,5)
\psline[linestyle=dotted](1,5.5)(1,5)
\psline[linestyle=dotted](1,1)(1,0.5)
\psline[linestyle=dotted](-1,1)(-1,0.5)
\psline[ArrowInside=->](-1,1)(-1,5)
\psline[ArrowInside=->](1,1)(1,5)
\psline[linestyle=dashed,dash=1pt 1pt](-3.08,2.98)(-1,3.68)
\rput[tl](-1.4,5.86){$\zzf$}
\rput[tl](-0.8,4.1){$v$}
\rput[tl](-3.32,3.38){$i$}
\rput[tl](1.16,5.94){$\zzf^\prime$}
\psdots[dotsize=2pt 0](-3.08,2.98)
\psdots[dotsize=2pt 0](-1,3.68)
\end{pspicture*}
\end{center}
Since we assume that $\zzf^\prime$ passes between $i$ and $j$, we conclude that $\zzf^\prime \in \ZC_+$. If $\zzf^\prime$ is on the left of $\zzf$ then a symmetric argument works considering $j$ instead of $i$.
\end{proof}
Consider two zig-zag flows $\zzf, \zzf^\prime$ which are in $\ZC_+$. 
We see by Lemma~\ref{pmcoincide} that they are not anti-parallel, and therefore their homology classes $[\zzf], [\zzf^\prime]$ generate a strongly convex cone in $H_1(Q) \otimes_\Z \R$.  We now prove that no zig-zag flow in $\ZC_-$ has a homology class which lies in this cone.

\begin{lemma}\label{pmintersect}
Let $\zzf, \zzf^\prime \in \ZC_+$  and denote the rays in $H_1(Q) \otimes_\Z \R$ generated by their respective homology classes by $\gamma$ and $\gamma^\prime$. Suppose that $\zzf^{\prime\prime}$ is a zig-zag path which passes between $i$ and $j$. If $[\zzf^{\prime\prime}] \in H_{1}(Q)$ is in the cone spanned by $\gamma$ and $\gamma^\prime$, then $\zzf^{\prime\prime} \in \ZC_+$.
\end{lemma}
\begin{proof}

We first note that if
$[\zzf^{\prime\prime}]$ equals $[\zzf]$ or $[\zzf^\prime]$ then
the result follows from Lemma~\ref{pmcoincide}.

Otherwise suppose that $\gamma$, $\gamma^\prime$ and $\gamma^{\prime\prime}$ are all distinct. Without loss of generality we assume that $\zzf^\prime$ crosses $\zzf$ from right to left. Then since $[\zzf^{\prime\prime}]$ is in the cone spanned by $\gamma$ and $\gamma^\prime$, we note that $\zzf^{\prime\prime}$ crosses $\zzf$ from right to left, and $\zzf^\prime$ from left to right. Therefore $\zzf^{\prime\prime}$ intersects $\zzf$ in a zig $\zzf^{\prime\prime}_k$ and intersects $\zzf^\prime$ in a zag $\zzf^{\prime\prime}_n$. Suppose $k<n$; the result for $k>n$ follows by a symmetric argument. 

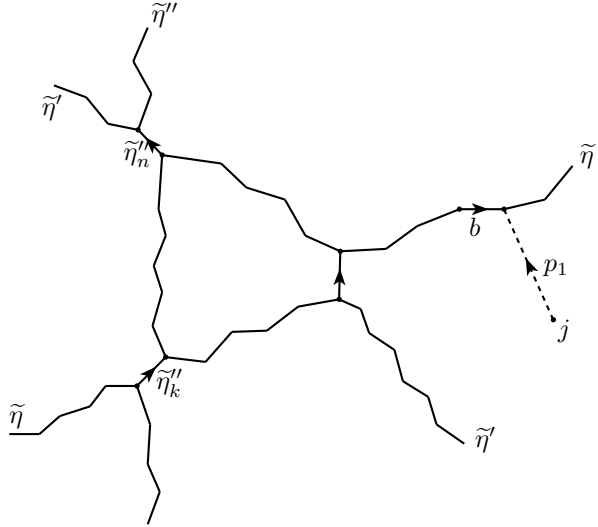
\begin{center}
\newrgbcolor{ccqqqq}{0.8 0 0}
\psset{xunit=0.8cm,yunit=0.8cm,algebraic=true,dotstyle=*,dotsize=3pt 0,linewidth=0.8pt,arrowsize=3pt 2,arrowinset=0.25}
\begin{pspicture*}(-3,-3.5)(7,5.5)
\psline[ArrowInside=->](-0.76,-1.04)(-0.28,-0.56)
\psline[ArrowInside=->](-0.34,2.8)(-0.74,3.22)
\psline(-0.28,-0.56)(-0.5,-0.04)
\psline(-0.5,-0.04)(-0.34,0.52)
\psline(-0.34,0.52)(-0.48,0.96)
\psline(-0.48,0.96)(-0.28,1.46)
\psline(-0.28,1.46)(-0.4,1.9)
\psline(-0.4,1.9)(-0.34,2.8)
\psline(-0.58,-3.34)(-0.38,-2.8)
\psline(-0.38,-2.8)(-0.58,-2.34)
\psline(-0.58,-2.34)(-0.54,-1.68)
\psline(-0.54,-1.68)(-0.76,-1.04)
\psline(-2.88,-1.84)(-2.38,-1.84)
\psline(-2.38,-1.84)(-2.04,-1.54)
\psline(-2.04,-1.54)(-1.54,-1.38)
\psline(-1.54,-1.38)(-1.28,-1.04)
\psline(-1.28,-1.04)(-0.76,-1.04)
\psline(-0.28,-0.56)(0.38,-0.64)
\psline(0.38,-0.64)(0.82,-0.14)
\psline(0.82,-0.14)(1.4,-0.12)
\psline(1.4,-0.12)(1.92,0.28)
\psline(1.92,0.28)(2.6,0.4)
\psline[ArrowInside=->](2.6,0.4)(2.62,1.2)
\psline(2.62,1.2)(3.38,1.24)
\psline(3.38,1.24)(3.9,1.62)
\psline(3.9,1.62)(4.6,1.9)
\psline[ArrowInside=->](4.6,1.9)(5.34,1.9)
\psline(5.34,1.9)(6.02,2.06)
\psline(6.02,2.06)(6.48,2.62)
\psline(4.68,-2)(4.22,-1.68)
\psline(4.22,-1.68)(4.08,-1.22)
\psline(4.08,-1.22)(3.66,-0.96)
\psline(3.66,-0.96)(3.52,-0.44)
\psline(3.52,-0.44)(3.1,-0.16)
\psline(3.1,-0.16)(2.96,0.24)
\psline(2.96,0.24)(2.6,0.4)
\psline(2.62,1.2)(2.04,1.46)
\psline(2.04,1.46)(1.7,2.06)
\psline(1.7,2.06)(1.06,2.26)
\psline(1.06,2.26)(0.64,2.66)
\psline(0.64,2.66)(-0.34,2.8)
\psline(-0.74,3.22)(-0.52,3.82)
\psline(-0.52,3.82)(-0.82,4.36)
\psline(-0.82,4.36)(-0.58,4.92)
\psline(-0.74,3.22)(-1.24,3.32)
\psline(-1.24,3.32)(-1.6,3.76)
\psline(-1.6,3.76)(-2.14,3.96)
\psline[ArrowInside=->,linestyle=dashed,dash=2pt 2pt](6.16,0.06)(5.34,1.9)
\rput[tl](-0.52,5.34){$\zzf^{\prime \prime}$}
\rput[tl](-0.98,3.04){$\zzf^{\prime \prime}_n$}
\rput[tl](-0.44,-0.76){$\zzf^{\prime \prime}_k$}
\rput[tl](4.76,1.76){$b$}
\rput[tl](6.0,1.06){$p_1$}
\rput[tl](6.24,0.06){$j$}
\rput[tl](6.6,2.98){$\zzf$}
\rput[tl](-2.88,-1.34){$\zzf$}
\rput[tl](4.86,-1.72){$\zzf^\prime$}
\rput[tl](-2.34,3.8){$\zzf^\prime$}
\psdots[dotsize=2pt 0](-0.76,-1.04)
\psdots[dotsize=2pt 0](-0.28,-0.56)
\psdots[dotsize=2pt 0](-0.34,2.8)
\psdots[dotsize=2pt 0](-0.74,3.22)
\psdots[dotsize=2pt 0](2.6,0.4)
\psdots[dotsize=2pt 0](2.62,1.2)
\psdots[dotsize=2pt 0](4.6,1.9)
\psdots[dotsize=2pt 0](5.34,1.9)
\psdots[dotsize=2pt 0](6.16,0.06)
\end{pspicture*}
\end{center}

Take any (possibly unoriented) path $p_1$ from $j$ to the head of some zig $b$ of $\zzf$ which intersects neither $\zzf$ nor $\zzf^\prime$. Such a path exists: since $j$ is on the right of $\zzf$ there is a path from $j$ to the head of a zig which does not intersect $\zzf$. If this path intersects $\zzf^\prime$ we consider instead the path up to but not including the first arrow in the intersection, followed by part of the boundary flow of $\zzf^\prime$.

We observe that $p_1$ does not intersect $\zzf^{\prime\prime}$ since all the vertices along it are on the right of both $\zzf$ and $\zzf^{\prime}$, while all the vertices of $\zzf^{\prime\prime}$ are on the left of either $\zzf_1$ or $\zzf_2$ (using the fact that $k<n$, so $\zzf$ crosses $\zzf_1$ from right to left before it crosses $\zzf_2$ from left to right). The arrow $\zzf^{\prime\prime}_k$ occurs before $b$ in $\zzf$, since $h\zzf^{\prime\prime}_k$ 
is on the left and $hb$ is on the right of $\zzf^\prime$.

Therefore the path $p_2$ along $\zzf$ from $h\zzf^{\prime\prime}_k$ to $hb$ does not contain $\zzf^{\prime\prime}_k$ and so does not intersect $\zzf^{\prime\prime}$. Finally we see that the path $p_2p_1^{-1}$ is a path from the head of a zig of $\zzf^{\prime\prime}$ to $j$ which doesn't intersect $\zzf^{\prime\prime}$, so $j$ is on the right of $\zzf$. If $\zzf^{\prime\prime}$ passes between $i$ and $j$ then it also has $i$ on the left, and we are done.
\end{proof}

We now consider the cones generated by homology classes of flows in $\ZC_+$ and $\ZC_-$.
\begin{definition}
Let $C_+$ (resp. $C_-$) be the cone in $H_1(Q) \otimes_\Z \R$ generated by the homology classes $[\zzf]$ of zig-zag flows $\zzf \in \ZC_+$ (respectively $\zzf \in \ZC_-$).
\end{definition}

We note that $C_+$ and $C_-$ are non-empty by Lemma~\ref{nonemptyC}.
If any set of rays corresponding to zig-zag flows $\zzf \in \ZC_+$ is not contained in some half-plane, then every  point in $H_1(Q) \otimes_\Z \R$ is in the cone generated by a pair of these rays. However using Lemma~\ref{pmintersect}, this would then imply that $C_-$ is empty. Therefore $C_+$ is contained in a some  half-plane. Furthermore, by Lemma~\ref{pmcoincide} we see that $C_+$ can not contain both rays in the boundary of this half-plane, as they are anti-parallel. Therefore $C_+$ is a strongly convex cone. Similarly we see that $C_-$ is a strongly convex cone. Furthermore we note that the intersection $C_+ \cap C_- = \{0\}$. There exists a line $\ell$ separating $C_+$ and $C_-$ which must pass through the origin. It can be chosen such that the origin is its only point of intersection with $C_+$ or $C_-$, or in fact with any ray in the global zig-zag fan. 
We choose such an $\ell$, and let $\hp_+ $ and $\hp_- $ be the corresponding closed half spaces containing $C_+$ and $C_-$ respectively.

We consider the 2-dimensional cone $\sigma$ in the global zig-zag fan which intersects the line $\ell$, and lies between the clockwise boundary ray of $C_+$ and the anti-clockwise boundary ray of $C_-$. This cone $\sigma$ may intersect $C_+$ or $C_-$ in a ray, but is not contained in either. The diagram below shows the cones $C_+$ and $C_-$, and a possible choice of $\ell$ and $\sigma$.

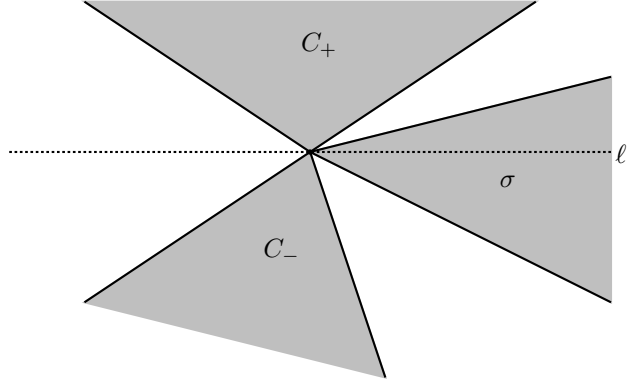
\begin{equation*} \label{Cpicture}
\psset{xunit=1.0cm,yunit=1.0cm}
\begin{pspicture*}(3.5,-6.5)(13,-0.5)
\newrgbcolor{wwwwww}{0.75 0.75 0.75}
\psset{xunit=1.0cm,yunit=1.0cm,algebraic=true,dotstyle=*,dotsize=3pt 0,linewidth=0.8pt,arrowsize=3pt 2,arrowinset=0.25}
\pspolygon[linecolor=wwwwww,fillcolor=wwwwww,fillstyle=solid](5,-1)(8,-3)(11,-1)
\pspolygon[linecolor=wwwwww,fillcolor=wwwwww,fillstyle=solid](8,-3)(5,-5)(9,-6)
\pspolygon[linecolor=wwwwww,fillcolor=wwwwww,fillstyle=solid](12,-2)(8,-3)(12,-5)
\psline[linestyle=dashed,dash=1pt 1pt](4,-3)(12,-3)
\psline(5,-1)(8,-3)
\psline(8,-3)(11,-1)
\psline(8,-3)(5,-5)
\psline(9,-6)(8,-3)
\psline(12,-2)(8,-3)
\psline(8,-3)(12,-5)
\rput[tl](12.06,-2.9){$\ell$}
\rput[tl](10.52,-3.32){$\sigma$}
\rput[tl](7.88,-1.42){$C_+$}
\rput[tl](7.38,-4.16){$C_-$}
\psdots[dotsize=2pt 0](8,-3)
\end{pspicture*}  \end{equation*}
\\
We denote the anti-clockwise and clockwise rays of $\sigma$ by $\sigma^+$ and $\sigma^-$ respectively.
Using polar coordinates, any ray may be described by an angle in the range $(-\pi,\pi]$, where $\ell \cap \sigma$ has angle zero. 
Let $\alpha^+ \in (0,\pi)$ and $\alpha^- \in (-\pi,0)$ be the angles of $\sigma^+$ and $\sigma^-$ respectively.


\section{Proof of Proposition~\ref{EXTREM}} \label{Proofoflem}
\subsection{Overview of proof}
Fix two vertices $i, j$ of the universal cover of the quiver $\Qcov$. We will construct a path from $j$ to $i$ which is zero on the extremal perfect matching $P(\sigma)$ where $\sigma$ is the choice of cone in the global zig-zag fan  made at the end of the previous section. 
First we use an iterative procedure to construct a different path $p$, ending at $i$, 
from pieces of the boundary flows of a sequence zig-zag flows. These zig-zag flows all have vertices $i$ and $j$ on the left, and satisfy a minimality condition which will allow us to prove, using results from Chapter~\ref{zzfpmchap}, that $P(\sigma)$ evaluates to zero on $p$. 

Similarly we construct a path $q$ ending at $i$, from boundary flows of zig-zag flows which have $i$ and $j$ on the right. We will see that $P(\sigma)$ also evaluates to zero on $q$. The paths $p$ and $q$ intersect at some vertex $v \neq i$, and we are interested in the parts $\finitep$ and $\finiteq$ of $p$ and $q$ respectively which go from $v$ to $i$.
We prove that they are F-term equivalent and observe that either $\finitep$ or $\finiteq$ pass through $j$, or the closed path $\finitep(\finiteq)^{-1}$ has a non-zero winding number around $j$. In this case, we show that there exists an F-term equivalent path which passes through $j$. Since it is F-term equivalent, $P(\sigma)$ also evaluates to zero on this path. By taking the piece of this path from $j$ to $i$ we have the required path.
\vspace{0.5cm}
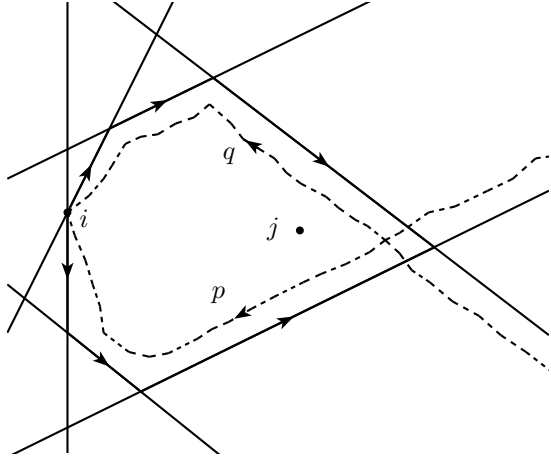
\begin{center}
\psset{xunit=0.8cm,yunit=0.8cm}
\begin{pspicture*}(3,-6)(12,1.5)
\psset{xunit=0.8cm,yunit=0.8cm,algebraic=true,dotstyle=*,dotsize=3pt 0,linewidth=0.8pt,arrowsize=3pt 2,arrowinset=0.25}
\psplot{3}{12}{(-5--1*x)/0.5}
\psplot{3}{12}{(-2.85--0.48*x)/0.98}
\psline(4,-6)(4,1.5)
\psplot{3}{12}{(-0.96-0.98*x)/1.22}
\psplot{3}{12}{(-9.81--0.64*x)/1.3}
\psplot{3}{12}{(--18.18-2.7*x)/3.52}
\psline[linestyle=dashed,dash=2pt 2pt](4,-2)(4.36,-1.72)
\psline[linestyle=dashed,dash=2pt 2pt](4.36,-1.72)(4.6,-1.44)
\psline[linestyle=dashed,dash=2pt 2pt](4.6,-1.44)(4.74,-1.22)
\psline[linestyle=dashed,dash=2pt 2pt](4.74,-1.22)(4.96,-0.86)
\psline[linestyle=dashed,dash=2pt 2pt](4.96,-0.86)(5.22,-0.74)
\psline[linestyle=dashed,dash=2pt 2pt](5.22,-0.74)(5.5,-0.7)
\psline[linestyle=dashed,dash=2pt 2pt](5.5,-0.7)(5.8,-0.52)
\psline[linestyle=dashed,dash=2pt 2pt](5.8,-0.52)(6.1,-0.42)
\psline[linestyle=dashed,dash=2pt 2pt](6.1,-0.42)(6.36,-0.2)
\psline[linestyle=dashed,dash=2pt 2pt](6.36,-0.2)(6.72,-0.52)
\psline[linestyle=dashed,dash=2pt 2pt](6.72,-0.52)(6.92,-0.78)
\psline[linestyle=dashed,dash=2pt 2pt](6.92,-0.78)(7.18,-0.94)
\psline[linestyle=dashed,dash=2pt 2pt](7.18,-0.94)(7.44,-1.1)
\psline[linestyle=dashed,dash=2pt 2pt](7.44,-1.1)(7.64,-1.34)
\psline[linestyle=dashed,dash=2pt 2pt](7.64,-1.34)(7.88,-1.52)
\psline[linestyle=dashed,dash=2pt 2pt](7.88,-1.52)(8.3,-1.74)
\psline[linestyle=dashed,dash=2pt 2pt](8.3,-1.74)(8.5,-1.92)
\psline[linestyle=dashed,dash=2pt 2pt](8.5,-1.92)(8.78,-2.12)
\psline[linestyle=dashed,dash=2pt 2pt](8.78,-2.12)(8.98,-2.22)
\psline[linestyle=dashed,dash=2pt 2pt](8.98,-2.22)(9.28,-2.46)
\psline[linestyle=dashed,dash=2pt 2pt](9.28,-2.46)(9.02,-2.66)
\psline[linestyle=dashed,dash=2pt 2pt](9.02,-2.66)(8.62,-2.86)
\psline[linestyle=dashed,dash=2pt 2pt](8.62,-2.86)(8.22,-3.02)
\psline[linestyle=dashed,dash=2pt 2pt](8.22,-3.02)(7.92,-3.16)
\psline[linestyle=dashed,dash=2pt 2pt](7.92,-3.16)(7.44,-3.4)
\psline[linestyle=dashed,dash=2pt 2pt](7.44,-3.4)(7.06,-3.6)
\psline[ArrowInside=->,linestyle=dashed,dash=2pt 2pt](7.06,-3.6)(6.68,-3.82)
\psline[linestyle=dashed,dash=2pt 2pt](6.68,-3.82)(6.4,-3.94)
\psline[linestyle=dashed,dash=2pt 2pt](6.4,-3.94)(6.16,-4.1)
\psline[linestyle=dashed,dash=2pt 2pt](6.16,-4.1)(5.68,-4.34)
\psline[linestyle=dashed,dash=2pt 2pt](5.68,-4.34)(5.36,-4.4)
\psline[linestyle=dashed,dash=2pt 2pt](5.36,-4.4)(5.02,-4.32)
\psline[linestyle=dashed,dash=2pt 2pt](5.02,-4.32)(4.6,-4)
\psline[linestyle=dashed,dash=2pt 2pt](4.6,-4)(4.54,-3.56)
\psline[linestyle=dashed,dash=2pt 2pt](4.54,-3.56)(4.42,-3.2)
\psline[linestyle=dashed,dash=2pt 2pt](4.42,-3.2)(4.3,-2.8)
\psline[linestyle=dashed,dash=2pt 2pt](4.3,-2.8)(4.12,-2.38)
\psline[linestyle=dashed,dash=2pt 2pt](4.12,-2.38)(4,-2)
\rput[tl](4.2,-1.95){$i$}
\rput[tl](7.32,-2.08){$j$}
\psline[ArrowInside=->](4,-2)(4.7,-0.6)
\psline[ArrowInside=->](4.7,-0.6)(6.42,0.24)
\psline[ArrowInside=->](6.42,0.24)(10.1,-2.58)
\psline[ArrowInside=->](4,-2)(4,-4)
\psline[ArrowInside=->](4,-4)(5.22,-4.98)
\psline[ArrowInside=->](5.22,-4.98)(10.1,-2.58)
\psline[linestyle=dashed,dash=2pt 2pt](9.28,-2.46)(9.52,-2.68)
\psline[linestyle=dashed,dash=2pt 2pt](9.52,-2.68)(9.72,-2.96)
\psline[linestyle=dashed,dash=2pt 2pt](9.72,-2.96)(10.02,-3.16)
\psline[linestyle=dashed,dash=2pt 2pt](10.02,-3.16)(10.4,-3.48)
\psline[linestyle=dashed,dash=2pt 2pt](10.4,-3.48)(10.84,-3.72)
\psline[linestyle=dashed,dash=2pt 2pt](10.84,-3.72)(11.2,-4.04)
\psline[linestyle=dashed,dash=2pt 2pt](11.2,-4.04)(11.62,-4.32)
\psline[linestyle=dashed,dash=2pt 2pt](11.62,-4.32)(12.06,-4.66)
\psline[linestyle=dashed,dash=2pt 2pt](12.06,-4.66)(12.28,-4.9)
\psline[linestyle=dashed,dash=2pt 2pt](12.28,-4.9)(12.74,-5.16)
\psline[linestyle=dashed,dash=2pt 2pt](9.28,-2.46)(9.52,-2.34)
\psline[linestyle=dashed,dash=2pt 2pt](9.52,-2.34)(10,-2)
\psline[linestyle=dashed,dash=2pt 2pt](10,-2)(10.44,-1.9)
\psline[linestyle=dashed,dash=2pt 2pt](10.44,-1.9)(10.98,-1.64)
\psline[linestyle=dashed,dash=2pt 2pt](10.98,-1.64)(11.4,-1.44)
\psline[linestyle=dashed,dash=2pt 2pt](11.4,-1.44)(11.78,-1.1)
\psline[linestyle=dashed,dash=2pt 2pt](11.78,-1.1)(12.36,-1.02)
\psline[linestyle=dashed,dash=2pt 2pt](12.36,-1.02)(12.9,-0.72)
\psline[ArrowInside=->,linestyle=dashed,dash=2pt 2pt](7.18,-0.94)(6.92,-0.78)
\rput[tl](6.4,-3.24){$p$}
\rput[tl](6.58,-0.92){$q$}
\psdots(4,-2)
\rput[bl](6.44,2.6){$a$}
\rput[bl](0.8,-2.26){$b$}
\rput[bl](4.16,2.6){$c$}
\rput[bl](0.8,-1.9){$d$}
\rput[bl](0.8,-6.9){$e$}
\rput[bl](3.5,2.6){$f$}
\psdots(7.86,-2.3)
\end{pspicture*}

\end{center}
\vspace{0.5cm}

\subsection{The starting point} \label{startpt}
In this section we find the first zig-zag flow in our sequence. This must have the property that its black boundary flow passes through $i$. Therefore we start by considering the black faces which have $i$ as a vertex, and look at all the zig-zag flows which intersect the boundary of these faces.
\begin{lemma}\label{existface}
There exists a zig-zag flow which has the following properties:
\begin{compactenum}
\item It has vertices $i$ and $j$ on the left,
\item It intersects the boundary of a black face which has $i$ as a vertex,
\item The corresponding ray in $\Xi$ has angle, measured with respect to $\ell$, in the interval $(-\pi,0]$.
\end{compactenum}
\end{lemma}
\begin{proof}
Consider the set of black faces which have $i$ as a vertex, and let $\mathcal{Y}(i)$ be the set of zig-zag flows which intersect the boundary of one or more of these faces and have $i$ on the left. If $\zzf \in \mathcal{Y}(i)$ then it intersects the boundary of the black face $\face$ in a zig-zag pair $(\zzf_0,\zzf_1)$.
Without loss of generality we may assume that $t\zzf_0 \neq i$; otherwise we can consider instead the zig-zag pair $(\zzf_{-2},\zzf_{-1})$ which is in the boundary of a black face which patently has $i$ as a vertex, and $t\zzf_{-2} \neq i$ since in a geometrically consistent dimer model there are no quiver faces with just two arrows.

Arrow $\zzf_0= \zzf^\prime_1$ is the zag of another zig-zag flow $\zzf^\prime$ which has $i$ on its left. By Lemma~\ref{2cones} we see that the rays $\gamma$ and $\gamma^\prime$ corresponding to $\zzf$ and $\zzf^\prime$ respectively, span a convex cone in the local zig-zag fan $\xi(\face)$ with $\gamma^\prime$ the clockwise ray.
Thus for any ray in $\Xi$ corresponding to a flow in $\mathcal{Y}(i)$, there is a ray of such a flow at an angle less than $\pi$ in a clockwise direction.
In particular, 
there is a ray $\gamma$ corresponding to some zig-zag flow $\zzf \in \mathcal{Y}(i)$, whose angle is in the interval $(-\pi,0)$, so $\gamma$ lies in $\hp_-$. 
We note that $\zzf$ satisfies properties $2$ and $3$, and claim that it satisfies $1$ as well. If $\zzf$ had $j$ on the right then by definition, $\gamma$ would be in $C_+$, however the intersection of $C_+$ with $\hp_-$ is just the origin, which gives a contradiction.
\end{proof}

Finally in this section we define for any face $\face \in Q_2$
$$ \ZF{\face}:= \{ \zzf \mid \zzf \text{ intersects the boundary of } \face \text{ and has } i,j\text{ on left}\}  $$
By Lemma~\ref{existface} there exists a face $\face$, which has $i$ as a vertex, and where $ \ZF{\face}$ is non-empty.
We fix such a face, which we label $f^{(1)}$ and let $v^{(0)}:=i$.

\subsection{A sequence of faces}
We construct a sequence of black faces $\{f^{(n)}\}_{n=1}^\infty$ and vertices $\{v^{(n)}\}_{n=0}^\infty$ with the property that for each $n \in \N^+$, the vertex $v^{(n)}$ is contained in the boundary of $f^{(n)}$ and $f^{(n+1)}$
. We do this inductively as follows:\\
Suppose we have a black face $f^{(n)}$ with a known vertex $v^{(n-1)}$, for which $ \ZF{f^{(n)}}$ is non-empty. Looking at the zig-zag flows in $ \ZF{f^{(n)}}$, let $\zzf^{(n)}$ be the one with the maximum angle $\theta^{(n)}$ in the interval $(-\pi,\alpha^+]$. This zig-zag flow intersects the boundary of $f^{(n)}$ in a zig-zag pair $(\zzf^{(n)}_{2n},\zzf^{(n)}_{2n+1})$. We look at the next zig-zag pair $(\zzf^{(n)}_{2n+2},\zzf^{(n)}_{2n+3})$ which lies in the boundary of a black face which we denote by $f^{(n+1)}$. Note that the vertex $v^{(n)} := h\zzf^{(n)}_{2n+1}$ is contained in the boundaries of both $f^{(n)}$ and $f^{(n+1)}$. Furthermore, $\zzf^{(n)}$ has $i$ and $j$ on the left and intersects the boundary of $f^{(n+1)}$, so $ \ZF{f^{(n+1)}}$ is non-empty.

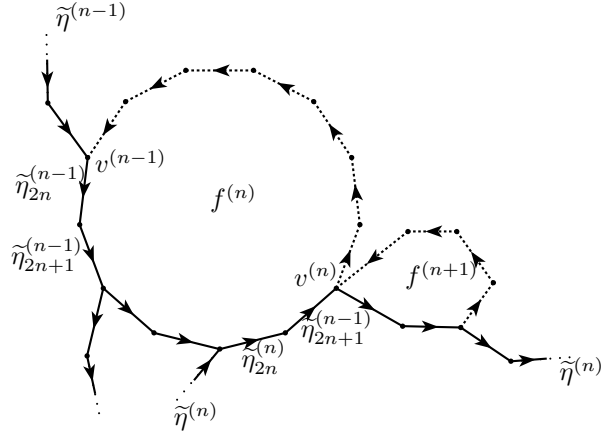
\begin{center}
\psset{xunit=0.9cm,yunit=0.9cm,algebraic=true,dotstyle=*,dotsize=3pt 0,linewidth=0.8pt,arrowsize=3pt 2,arrowinset=0.25}
\begin{pspicture*}(-3,-5)(8,2.5)
\psline[ArrowInside=->,linestyle=dashed,dash=1pt 1pt](1,1)(0,1)
\psline[ArrowInside=->,linestyle=dashed,dash=1pt 1pt](0,1)(-0.89,0.54)
\psline[ArrowInside=->,linestyle=dashed,dash=1pt 1pt](-0.89,0.54)(-1.45,-0.29)
\psline[ArrowInside=->](-1.45,-0.29)(-1.57,-1.28)
\psline[ArrowInside=->](-1.57,-1.28)(-1.22,-2.22)
\psline[ArrowInside=->](-1.22,-2.22)(-0.47,-2.88)
\psline[ArrowInside=->](-0.47,-2.88)(0.5,-3.12)
\psline[ArrowInside=->](0.5,-3.12)(1.47,-2.88)
\psline[ArrowInside=->](1.47,-2.88)(2.22,-2.22)
\psline[ArrowInside=->,linestyle=dashed,dash=1pt 1pt](2.22,-2.22)(2.57,-1.28)
\psline[ArrowInside=->,linestyle=dashed,dash=1pt 1pt](2.57,-1.28)(2.45,-0.29)
\psline[ArrowInside=->,linestyle=dashed,dash=1pt 1pt](2.45,-0.29)(1.89,0.54)
\psline[ArrowInside=->,linestyle=dashed,dash=1pt 1pt](1.89,0.54)(1,1)
\psline[linestyle=dotted](-2.06,1.54)(-2.04,0.52)
\psline[ArrowInside=->](-2.04,0.52)(-1.45,-0.29)
\psline[ArrowInside=->](-1.22,-2.22)(-1.46,-3.22)
\psline[linestyle=dotted](-1.46,-3.22)(-1.28,-4.08)
\psline[linestyle=dotted](-0.14,-3.88)(0.5,-3.12)
\psline[ArrowInside=->](2.22,-2.22)(3.2,-2.78)
\psline[ArrowInside=->](3.2,-2.78)(4.06,-2.8)
\psline[ArrowInside=->,linestyle=dashed,dash=1pt 1pt](4.06,-2.8)(4.54,-2.14)
\psline[ArrowInside=->,linestyle=dashed,dash=1pt 1pt](4.54,-2.14)(4,-1.38)
\psline[ArrowInside=->,linestyle=dashed,dash=1pt 1pt](4,-1.38)(3.28,-1.38)
\psline[ArrowInside=->,linestyle=dashed,dash=1pt 1pt](3.28,-1.38)(2.22,-2.22)
\rput[tl](-1.92,2.0){$\zzf^{(n-1)}$}
\rput[tl](-0.18,-3.9){$\zzf^{(n)}$}
\psline[ArrowInside=->](4.06,-2.8)(4.8,-3.3)
\psline[linestyle=dotted](4.8,-3.3)(5.68,-3.24)
\psline[ArrowInside=->](4.8,-3.3)(5.33,-3.26)
\psline[ArrowInside=->](0.12,-3.58)(0.5,-3.12)
\psline[ArrowInside=->](-2.05,1.16)(-2.04,0.52)
\psline[ArrowInside=->](-1.46,-3.22)(-1.34,-3.78)
\rput[tl](0.36,-0.68){$f^{(n)}$}
\rput[tl](3.24,-1.82){$f^{(n+1)}$}
\rput[tl](1.6,-1.84){$v^{(n)}$}
\rput[tl](-1.34,-0.16){$v^{(n-1)}$}
\rput[tl](5.52,-3.24){$\zzf^{(n)}$}
\rput[tl](-2.52,-0.4){$\zzf^{(n-1)}_{2n}$}
\rput[tl](-2.58,-1.46){$\zzf^{(n-1)}_{2n+1}$}
\rput[tl](0.82,-2.98){$\zzf^{(n)}_{2n}$}
\rput[tl](1.7,-2.54){$\zzf^{(n-1)}_{2n+1}$}
\psdots[dotsize=2pt 0](1,1)
\psdots[dotsize=2pt 0](0,1)
\psdots[dotsize=2pt 0](-0.89,0.54)
\psdots[dotsize=2pt 0](-1.45,-0.29)
\psdots[dotsize=2pt 0](-1.57,-1.28)
\psdots[dotsize=2pt 0](-1.22,-2.22)
\psdots[dotsize=2pt 0](-0.47,-2.88)
\psdots[dotsize=2pt 0](0.5,-3.12)
\psdots[dotsize=2pt 0](1.47,-2.88)
\psdots[dotsize=2pt 0](2.22,-2.22)
\psdots[dotsize=2pt 0](2.57,-1.28)
\psdots[dotsize=2pt 0](2.45,-0.29)
\psdots[dotsize=2pt 0](1.89,0.54)
\psdots[dotsize=2pt 0](-2.04,0.52)
\psdots[dotsize=2pt 0](-1.46,-3.22)
\psdots[dotsize=2pt 0](3.2,-2.78)
\psdots[dotsize=2pt 0](4.06,-2.8)
\psdots[dotsize=2pt 0](4.54,-2.14)
\psdots[dotsize=2pt 0](4,-1.38)
\psdots[dotsize=2pt 0](3.28,-1.38)
\psdots[dotsize=2pt 0](4.8,-3.3)
\end{pspicture*}

\end{center}

The maximality condition and Lemma~\ref{existface} imply that
\begin{equation}\label{angles}
-\pi < \theta^{(n)} \leq \theta^{(n+1)} \leq \alpha^+
\end{equation}
for all $n \in \N^+$. Since by Lemma~\ref{uniquerep} there is at most a unique representative zig-zag flow of any ray which intersects the boundary of a face, we note that zig-zag flows $\zzf^{(n)}=\zzf^{(n+1)}$ if and only if $\theta^{(n)}=\theta^{(n+1)}$.
In particular if $\theta^{(n_0)} = \alpha^+$ then $\zzf^{(n)}=\zzf^{(n_0)}$ for all $n \geq n_0$.
\begin{remark}
We observe that if $n \neq k$ then $f^{(n)} \neq f^{(k)}$, otherwise the local zig-zag fans at $f^{(n)}$ and $ f^{(k)}$ are the same so $\theta^{(n)}=\theta^{(k)}$. This implies that $\zzf^{(n)}=\zzf^{(k)}$, and so $\zzf^{(n)}_{2n}$ and $\zzf^{(n)}_{2k}$ are both arrows in the boundary of $f^{(n)}$. However this would contradict geometric consistency (in particular either Proposition~\ref{geomcons} or Lemma~\ref{singlein}).
\end{remark}

\subsection{The paths $p$ and $\wh{p}$}
We now construct two paths by piecing together paths around the boundaries of these black faces.
For each $n \in \N^+$ define $p^{(n)}$ to be the shortest oriented path around the boundary of $f^{(n)}$ from $v^{(n)}$ to $v^{(n-1)}$ and let $\widehat{p}^{(n)}$ to be the shortest oriented path around the boundary of $f^{(n)}$ from $v^{(n-1)}$ to $v^{(n)}$.

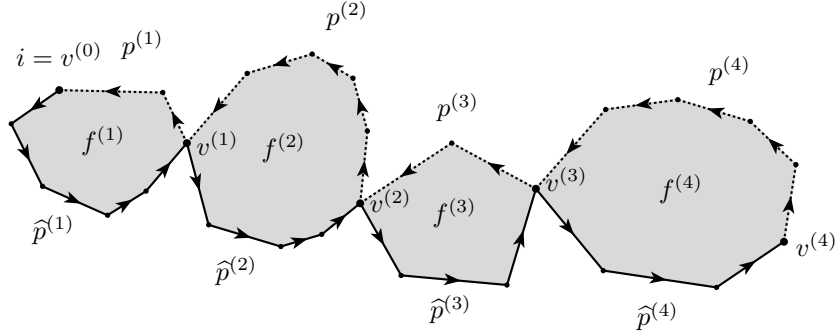
\begin{center}
\newrgbcolor{cccccc}{0.85 0.85 0.85}
\psset{xunit=1.6cm,yunit=1.6cm}
\begin{pspicture*}(-3,-1)(5,3)
\psset{xunit=1.6cm,yunit=1.6cm,algebraic=true,dotstyle=*,dotsize=3pt 0,linewidth=0.8pt,arrowsize=3pt 2,arrowinset=0.25}
\pspolygon[linecolor=cccccc,fillcolor=cccccc,fillstyle=solid](-2,2)(-2.4,1.72)(-2.14,1.2)(-1.6,0.96)(-1.28,1.16)(-0.94,1.56)(-1.14,1.98)
\pspolygon[linecolor=cccccc,fillcolor=cccccc,fillstyle=solid](-0.94,1.56)(-0.76,0.88)(-0.16,0.7)(0.18,0.8)(0.5,1.06)(0.56,1.66)(0.44,2.1)(0.1,2.3)(-0.44,2.14)
\pspolygon[linecolor=cccccc,fillcolor=cccccc,fillstyle=solid](0.5,1.06)(0.84,0.46)(1.72,0.38)(1.96,1.18)(1.26,1.56)
\pspolygon[linecolor=cccccc,fillcolor=cccccc,fillstyle=solid](1.96,1.18)(2.52,0.5)(3.46,0.36)(4.02,0.74)(4.12,1.38)(3.74,1.74)(3.14,1.92)(2.54,1.84)
\psline[ArrowInside=->](-2,2)(-2.4,1.72)
\psline[ArrowInside=->](-2.4,1.72)(-2.14,1.2)
\psline[ArrowInside=->](-2.14,1.2)(-1.6,0.96)
\psline[ArrowInside=->](-1.6,0.96)(-1.28,1.16)
\psline[ArrowInside=->](-1.28,1.16)(-0.94,1.56)
\psline[ArrowInside=->,linestyle=dashed,dash=1pt 1pt](-0.94,1.56)(-1.14,1.98)
\psline[ArrowInside=->,linestyle=dashed,dash=1pt 1pt](-1.14,1.98)(-2,2)
\psline[ArrowInside=->](-0.94,1.56)(-0.76,0.88)
\psline[ArrowInside=->](-0.76,0.88)(-0.16,0.7)
\psline[ArrowInside=->](-0.16,0.7)(0.18,0.8)
\psline[ArrowInside=->](0.18,0.8)(0.5,1.06)
\psline[ArrowInside=->,linestyle=dashed,dash=1pt 1pt](0.5,1.06)(0.56,1.66)
\psline[ArrowInside=->,linestyle=dashed,dash=1pt 1pt](0.56,1.66)(0.44,2.1)
\psline[ArrowInside=->,linestyle=dashed,dash=1pt 1pt](0.44,2.1)(0.1,2.3)
\psline[ArrowInside=->,linestyle=dashed,dash=1pt 1pt](0.1,2.3)(-0.44,2.14)
\psline[ArrowInside=->,linestyle=dashed,dash=1pt 1pt](-0.44,2.14)(-0.94,1.56)
\psline[ArrowInside=->](0.5,1.06)(0.84,0.46)
\psline[ArrowInside=->](0.84,0.46)(1.72,0.38)
\psline[ArrowInside=->](1.72,0.38)(1.96,1.18)
\psline[ArrowInside=->,linestyle=dashed,dash=1pt 1pt](1.96,1.18)(1.26,1.56)
\psline[ArrowInside=->,linestyle=dashed,dash=1pt 1pt](1.26,1.56)(0.5,1.06)
\psline[ArrowInside=->](1.96,1.18)(2.52,0.5)
\psline[ArrowInside=->](2.52,0.5)(3.46,0.36)
\psline[ArrowInside=->](3.46,0.36)(4.02,0.74)
\psline[ArrowInside=->,linestyle=dashed,dash=1pt 1pt](4.02,0.74)(4.12,1.38)
\psline[ArrowInside=->,linestyle=dashed,dash=1pt 1pt](4.12,1.38)(3.74,1.74)
\psline[ArrowInside=->,linestyle=dashed,dash=1pt 1pt](3.74,1.74)(3.14,1.92)
\psline[ArrowInside=->,linestyle=dashed,dash=1pt 1pt](3.14,1.92)(2.54,1.84)
\psline[ArrowInside=->,linestyle=dashed,dash=1pt 1pt](2.54,1.84)(1.96,1.18)
\rput[tl](-1.82,1.7){$f^{(1)}$}
\rput[tl](-0.32,1.64){$f^{(2)}$}
\rput[tl](1.1,1.1){$f^{(3)}$}
\rput[tl](2.98,1.3){$f^{(4)}$}
\rput[tl](-2.36,2.4){$i=v^{(0)}$}
\rput[tl](0.58,1.18){$v^{(2)}$}
\rput[tl](2.04,1.36){$v^{(3)}$}
\rput[tl](4.12,0.82){$v^{(4)}$}
\rput[tl](-0.86,1.66){$v^{(1)}$}
\rput[tl](-1.48,2.5){$p^{(1)}$}
\rput[tl](0.22,2.74){$p^{(2)}$}
\rput[tl](1.14,1.98){$p^{(3)}$}
\rput[tl](3.4,2.28){$p^{(4)}$}
\rput[tl](-2.22,0.96){$\widehat{p}^{(1)}$}
\rput[tl](-0.7,0.62){$\widehat{p}^{(2)}$}
\rput[tl](1.08,0.3){$\widehat{p}^{(3)}$}
\rput[tl](2.8,0.24){$\widehat{p}^{(4)}$}
\psdots(-2,2)
\psdots[dotsize=2pt 0](-2.4,1.72)
\psdots[dotsize=2pt 0](-2.14,1.2)
\psdots[dotsize=2pt 0](-1.6,0.96)
\psdots[dotsize=2pt 0](-1.28,1.16)
\psdots(-0.94,1.56)
\psdots[dotsize=2pt 0](-1.14,1.98)
\psdots[dotsize=2pt 0](-0.76,0.88)
\psdots[dotsize=2pt 0](-0.16,0.7)
\psdots[dotsize=2pt 0](0.18,0.8)
\psdots(0.5,1.06)
\psdots[dotsize=2pt 0](0.56,1.66)
\psdots[dotsize=2pt 0](0.44,2.1)
\psdots[dotsize=2pt 0](0.1,2.3)
\psdots[dotsize=2pt 0](-0.44,2.14)
\psdots[dotsize=2pt 0](0.84,0.46)
\psdots[dotsize=2pt 0](1.72,0.38)
\psdots(1.96,1.18)
\psdots[dotsize=2pt 0](1.26,1.56)
\psdots[dotsize=2pt 0](2.52,0.5)
\psdots[dotsize=2pt 0](3.46,0.36)
\psdots(4.02,0.74)
\psdots[dotsize=2pt 0](4.12,1.38)
\psdots[dotsize=2pt 0](3.74,1.74)
\psdots[dotsize=2pt 0](3.14,1.92)
\psdots[dotsize=2pt 0](2.54,1.84)
\end{pspicture*}

\end{center}

Since $hp^{(n)} = v^{(n-1)} = tp^{(n-1)}$ we can piece the paths $p^{(n)}$ together to form an infinite oriented path $p$ which ends at $v_0 = i$. Similarly we can piece the paths $\widehat{p}^{(n)}$ together to form an infinite oriented path $\widehat{p}$ which starts at $v_0 = i$.
\begin{remark}
Using the property (see Remark~\ref{wellorder}) that the intersections of zig-zag flows with the boundary of face $\face$ occur in the same cyclic order as the rays of the local zig-zag fan $\xi{(\face)}$, we note that the arrows in $\widehat{p}^{(n)}$ are the arrows in the boundary of $f^{(n)}$ which are contained in zig-zag flows with angle in the closed interval $[\theta^{(n-1)},\theta^{(n)}]$.
If $\theta^{(n-1)}=\theta^{(n)}$, then $\widehat{p}^{(n)}$ just contains the zig-zag pair $\zzf^{(n)}_{2n}, \zzf^{(n)}_{2n+1}$ and $p^{(n)}$ is the intersection of the black boundary flow of $\zzf^{(n)}$ with the boundary of $\face^{(n)}$. 
\end{remark}

Now that we have constructed the paths $p$ and $\wh{p}$ we check that they satisfy some properties. Most importantly we need to show that the perfect matching $P(\sigma)$ evaluates to zero on $p$. Recalling the definition of $P(\sigma)$, we are interested in the image of the cone $\sigma$ in the local zig-zag fan at each face. We prove the following lemma:
\begin{lemma}\label{noray}
There are no rays in the local zig-zag fan $\xi(\face^{(n)})$ with angle in the interval $(\theta^{(n)},0]$ for any $n \in \N^+$.
\end{lemma}
\begin{proof}
Suppose to the contrary that this doesn't hold and let $n^{\prime}$ be the least value where such a ray exists. We label this ray by $\gamma$ and suppose the intersection of its representative zig-zag flow with the boundary of $f^{(n^\prime)}$ is the zig-zag pair $(\zzf_0, \zzf_1)$. Denote the angle of ray $\gamma$ by $\theta \in (\theta^{(n^\prime)},0]$. We observe that:
\begin{itemize}
\item The maximality condition in the construction ensures that $\zzf$ can not have both $i$ and $j$ on its left.
\item If $\zzf$ had $i$ on its left and $j$ on its right then $\gamma$ would be in $C_+$. However since $\theta \in (\theta^{(n^\prime)},0]$ we see that $\gamma$ lies in $\hp_-$ which would be a contradiction.
\end{itemize}
Therefore $\zzf$ must have $i$ on its right.
Consider the path $p^{({n^\prime}-1)}p^{({n^\prime}-2)} \dots p^{(1)}$ from $v^{({n^\prime}-1)}$ to $i$. Since ${n^\prime}$ is minimal, $\zzf$ does not intersect the boundary of any of the faces $f^{(1)}, \dots , f^{({n^\prime}-1)}$. In particular it does not intersect this path. Therefore $v_{{n^\prime}-1}$ is also on the right of $\zzf$. Since $f^{(n^\prime)}$ is a black face, all its vertices are on the left of $\zzf$ except for $h\zzf_0=t\zzf_1$, so $v_{{n^\prime}-1}=t\zzf_1$.

Recalling the construction of the sequence of faces $\face^{(n)}$, we see that arrow $\zzf_1$ is a zig of $\zzf^{({n^\prime}-1)}$. Therefore $\zzf$ crosses $\zzf_{{n^\prime}-1}$ from left to right and, by Lemma~\ref{2cones}, the corresponding rays in the local zig-zag fan span a cone where $\gamma$ is the clockwise ray. Then 
$$ -2 \pi < \theta^{({n^\prime}-1)}-\pi < \theta < \theta^{({n^\prime}-1)} \leq \theta^{({n^\prime})}$$
%
However, this contradicts the assumption that $\theta \in (\theta^{(n^\prime)},0]$.
\end{proof}

We now use this lemma to prove the first property we required of the path $p$.
\begin{lemma}
The perfect matching $P(\sigma)$ evaluates to zero on the path $p$.
\end{lemma}
\begin{proof}
The path $p$ was constructed locally of paths $p^{(n)}$ in the boundary of the faces $f^{(n)}$. The perfect matching $P(\sigma)$ was also defined locally at each face. Therefore it is sufficient to prove the statement locally; we must show that $P_{f^{(n)}}(\sigma)$ evaluates to zero on $p^{(n)}$ for each $n \geq 1$.

We recall that $P_{f^{(n)}}(\sigma)$ is non-zero on a single arrow, corresponding to the cone in $\xi(f^{(n)})$ which contains the image of $\sigma$. We split the proof into two cases:

If $\theta^{(n)} \neq \alpha^+$, it follows from Lemma~\ref{noray} that $\gamma ^{(n)} $ is the clockwise ray of the cone containing the image of $\sigma$. Then using Lemma~\ref{2cones}, the unique arrow on which $P_{f^{(n)}}(\sigma)$ is non-zero is the zag $\zzf^{(n)}_{2n+1}$ in the boundary of $f^{(n)} $. By construction $h\zzf^{(n)}_{2n+1} = v^{(n)}$ and thus any oriented path around the boundary of $f^{(n)}$ from $v^{(n)}$, which contains this arrow, must contain a complete cycle. Each $p^{(n)}$ was defined so that this is not the case, so $P_{f_n}(\sigma)$ evaluates to zero on $p_n$.

If $\theta^{(n)} = \alpha^+$ then $\gamma^{(n)}$ is the image of $\sigma^+$ and $p^{(n)}$ is part of the boundary flow of $\zzf^{(n)}$, a representative of $\sigma^+$. By Proposition~\ref{uniqueext} we see that $P(\sigma)$ evaluates to zero on this.

\end{proof}

We now show that $\wh{p}$ is made up of sections of zig-zag flows that have the vertices $i$ and $j$ on their left.
\begin{lemma}
Each arrow in $\wh{p}^{(n)}$ is contained in some zig-zag flow which has $i$ and $j$ on the left.
\end{lemma}
\begin{proof}
If an arrow of $\wh{p}^{(n)}$ is in $\zzf^{(n-1)}$ or $\zzf^{(n)}$, then we are done. Otherwise the arrow is contained in a zig-zag flow $\zzf$ with angle in the open interval $( \theta^{(n-1)},\theta^{(n)})$. We note that $\zzf$ has $v^{(n-1)}$ on its left, otherwise one obtains a contradiction in the same way as the end of the proof of Lemma~\ref{noray}. Using Lemma~\ref{noray} directly, we see that $\zzf$ does not intersect the boundary of $\face^{(r)}$ for any $r \leq n-1$. Therefore considering a path along the boundary of these faces, we see that $\zzf$ has $i$ on its left. Because of its angle, $\zzf$ can not be in $\ZC_+$, and therefore it must have both $i$ and $j$ on the left.
\end{proof}

Finally in this section we see that sufficiently far away from its starting point, the path $\wh{p}$ looks like a representative zig-zag flow of $\sigma^+$.

\begin{lemma}
There exists $n^{\prime} \in \N^+$ such that for all $n \geq n^\prime$ we have $\zzf^{(n)} = \zzf^+$ which is a representative zig-zag flow of $\sigma^+$.
\end{lemma}
\begin{proof}
Since the increasing sequence $\{ \theta^{(n)} \}$ is bounded above, and there are a finite number of rays in $\Xi$, there exists some $n^{\prime} \in \N^+$ such that $\theta^{(n)} = \theta^{(n^\prime)}$, and therefore $\zzf^{(n)}=\zzf^{(n^\prime)}$, for all $n \geq n^\prime$. 
Thus it is sufficient to prove that $\zzf^{(n^\prime)}$ is a representative zig-zag flow of $\sigma^+$.

Suppose it is not. Then $\zzf^{(n^\prime)}$ intersects every representative zig-zag flows of $\sigma^+$ from left to right, and an infinite number of these intersections must occur after $v^{(n^\prime)}$ in  $\zzf^{(n^\prime)}$ (so $v^{(n^\prime)}$ is on their left). The path $p^{(n^\prime)}p^{(n^\prime-1)} \dots p^{(1)}$ from $v^{(n^\prime)}$ to $i$ has a finite number of arrows and therefore intersects at most a finite number of these. Therefore there are an infinite number of representative zig-zag flows of $\sigma^+$ which intersect $\zzf^{(n^\prime)}$ after $v^{(n^\prime)}$, so intersect the boundary of $f^{(n)}$ for some $n \geq n^\prime$, and have $i$ on their left. Only a finite number of zig-zag flows have $i$ on their left and $j$ on their right; such a flow must intersect every path from $i$ to $j$. Thus there exists a representative zig-zag flows of $\sigma^+$ which intersects the boundary of $f^{(n)}$ for some $n \geq n^\prime$, and has $i$ and $j$ on its left. This contradicts the maximality assumption.
\end{proof}


In an analogous way to the construction of $p$ and $\wh{p}$, we construct a paths $q$ and $\wh{q}$ by piecing together sequences of paths $q^{(k)}$ and $\wh{q}^{(k)}$ respectively, around the boundaries of a sequence of white faces $\{g^{(k)}\}$ with a distinguished set of vertices $ \{w^{(k)}\}$. This is constructed by considering zig-zag flows $\zzff^{(k)}$ which have both $i$ and $j$ on the right and whose ray in the global zig-zag fan has minimal angle $\varpi^{(k)}$ which lies in the interval $[-\alpha^-,\pi)$. Certain corresponding properties hold, which can be proved by symmetric arguments:
\begin{compactenum}
\item The path $q$ ends at the vertex $i$.
\item The perfect matching $P(\sigma)$ evaluates to zero on $q$.
\item Each arrow in $\wh{q}^{(n)}$ is contained in some zig-zag flow which has $i$ and $j$ on the right.
\item There exists $k^{\prime} \in \N^+$ such that for all $k \geq k^\prime$ we have $\zzff^{(k)} = \zzff^-$ which is a representative zig-zag flow of $\sigma^-$.
\end{compactenum}

\subsection{$p$ and $q$ intersect}
We now find a vertex (other than $i$) which is in both $p$ and $q$. This will be a vertex of the zig-zag paths $\zzf^+$ and $\zzff^-$. 
\begin{lemma}
Let $n^\prime$ and $k^\prime$ be be the least integers such that $\zzf^{(n^\prime)} = \zzf^+$ and $\zzff^{(k^\prime)} = \zzff^-$. Then $ \zzff^-$ does not intersect the boundary of $f^{(n)}$ for any $n \leq n^\prime$ and $ \zzf^+$ does not intersect the boundary of $g^{(k)}$ for any $k \leq k^\prime$.
\end{lemma}
\begin{proof}
First consider the case when $n < n^\prime$ so $\theta^{(n)} \leq \alpha^-$. If $\theta_n = \alpha^-$ then $\zzf^{(n)}$ and $ \zzff^-$ are parallel zig-zag flows which are distinct; $i$ is on the left of $\zzf^{(n)}$ but on the right of $ \zzff^-$. By definition $\zzf^{(n)}$ intersects the boundary of $f^{(n)}$, so by Lemma~\ref{uniquerep}, $\zzff^-$ does not. If $\theta_n < \alpha^-$ then by Lemma~\ref{noray} there are no rays of $\xi(f^{(n)})$ with angle in the interval $(\theta_n,0)$ and therefore $\zzff^-$ does not intersect the boundary of $f^{(n)}$.

Now consider the case when $n = n^\prime$. We have just shown that $\zzff_-$ does not intersect the path $p^{({n^\prime}-1)}p^{({n^\prime}-2)} \dots p^{(1)}$ from $v_{{n^\prime}-1}$ to $i$. Therefore $v_{{n^\prime}-1}$ is on the right of $ \zzff_-$. As we have noted before, since it is a black face, there is only one zig-zag flow which intersects the boundary of $f^{(n^\prime)}$ and has $v^{({n^\prime}-1)}$ on the right. However, this crosses $\zzf^{({n^\prime}-1)}$ from left to right and so can not be $\zzff_-$ since $-\pi < \theta_{{n^\prime}-1} \leq \alpha^- <0$. 
Therefore $\zzff_-$ does not intersect the boundary of $f_{n^\prime}$.

The proof of the other statement follows similarly.
\end{proof}
\begin{corollary} \label{aftervn}
Let $a = \zzf^+_{2n} = \zzff^-_{2k+1}$ be the unique arrow where $\zzff^-$ and $\zzf^+$ intersect. Then $ta$ occurs after $v^{({n^\prime})}$ in $\zzf^+$ and after $w^{({k^\prime})}$ in $\zzff^-$.
\end{corollary}
\begin{proof}
Using the lemma we can see that the path $\zzff^-$ does not intersect the path $p^{({n^\prime})}p^{({n^\prime}-1)} \dots p^{(1)}$ from $v^{({n^\prime}-1)}$ to $i$. Therefore $v^{({n^\prime})}$ is on the right of $ \zzff_-$. We know that $\zzff_-$ crosses $\zzf_+$ from right to left, so $v_{{n^\prime}}$ occurs in $\zzf_+$ before the intersection.
\end{proof}
Using this result we see that  the paths $p$ and $q$ both contain the vertex $ta$:
the arrow $a = \zzf^+_{2n}$ is a zig of $\zzf_+$ and 
by Corollary~\ref{aftervn} we know that $ta$ occurs after $v_{{n^\prime}}$ in $\zzf_+$. Therefore $ta = v^{({n})}$ for some $n>{n^\prime}$, and by construction $v^{({n})}$ is a vertex of $p$. Similarly we note that $ta = w^{(k)}$ for some $k>{k^\prime}$ and so $ta$ is a vertex of $q$.

We define paths $\finitep := p^{({n})}p^{({n}-1)} \dots p^{(1)}$ and $\finiteq := q^{({k})}q^{({k}-1)} \dots q^{(1)}$ from $ta$ to $i$. These are finite pieces of the paths $p$ and $q$ respectively.  We note that since $P(\sigma)$ evaluates to zero on $p$ and $q$, it must also evaluate to zero on $\finitep$ and $\finiteq$.

Similarly we are able define paths $\finiteph := \wh{p}^{({n})}\wh{p}^{({n}-1)} \dots \wh{p}^{(1)}$ and $\finiteqh := \wh{q}^{({k})}\wh{q}^{({k}-1)} \dots \wh{q}^{(1)}$ from $i$ to $ta$.

\subsection{The path from $j$ to $i$}
We want to be able to use F-term relations on the path $\finitep$ to change it into a path which passes through the vertex $j$. We start by showing that the paths $\finitep$ and $\finiteq$ are F-term equivalent.
\begin{lemma}\label{pqhomol}
The elements $\Meq{\finitep}$ and $\Meq{\finiteq}$ are equal, i.e. the paths $\finitep$ and $\finiteq$ are homologous.
\end{lemma}
\begin{proof}
Recall the short exact sequence (\ref{eq:Nmes}):
$$0 \lra{} \Z \lra{} \Mm \lra{H} H_1(T;\Z) \lra{} 0 $$
where the kernel of $H$ is spanned by the class $\square = \del \Meq{\face} $ for any face $\face \in Q_2$. Since $\finitep$ and $\finiteq$ start and finish at the same vertices in the universal cover $\Qcov$, the element $\Meq{\finitep}-\Meq{\finiteq} \in \Mm$ is in the kernel of $H$. Therefore it is some multiple of class $\square$. The perfect matching $P(\sigma)$, by definition, evaluates to 1 on the boundary of every face, and therefore it evaluates to 1 on $\square$. Since it evaluates to zero on $\Meq{\finitep}-\Meq{\finiteq}$, we conclude that this multiple of $\square$ is zero, and so $\Meq{\finitep}=\Meq{\finiteq}$.
\end{proof}
Consider the paths $\finiteph$ and $\finiteqh$. 
These are constructed out of sections of zig-zag flows which have the vertex $j$ on the left and right respectively.
Then the path $\finiteph(\finiteqh)^{-1}$ is constructed of sections of zig-zag flows that have $j$ consistently on the left. Thus this path either passes through $j$ or has a non-zero winding number around $j$. We note that the `boundary path' $\finitep(\finiteq)^{-1}$ has the same property. If $j$ is a vertex of $\finiteph(\finiteqh)^{-1}$ then it is a vertex of one of the faces $f^{(r)}$ or $g^{(s)}$. Neighbouring zig-zag flows in the boundary of this face which pass through $j$, have $j$ on different sides. This forces $j$ to be one of the distinguished vertices on the face, and so it is a vertex of $\finitep(\finiteq)^{-1}$ as well.
Finally, since the boundary of a face either contains a vertex, or has zero winding number about that vertex, we note that either $\finitep(\finiteq)^{-1}$ passes through $j$ or
the winding number of $\finitep(\finiteq)^{-1}$ around $j$ is non-zero.


To complete the proof of Proposition~\ref{EXTREM}, it is sufficient to show that there exists a path which is F-term equivalent to $\finitep$ and $\finiteq$, and which passes through vertex $j$. F-term equivalence implies that $P(\sigma)$ evaluates to zero on this path as well, and we obtain the required path from $j$ to $i$ by looking at the appropriate piece.


\begin{lemma} \label{expath}
Suppose ${p},{q}$ are $F$-term equivalent oriented paths in $\Qcov$ from vertex $v_1$ to $v_2$, which do not pass through vertex $v$.  If $\wind({p}{q}^{-1})$ is non-zero, then there exists an oriented path
${p^\prime}$ which is $F$-term equivalent to ${p}$ and ${q}$ and passes through v.
\end{lemma}
\begin{proof}
Since ${p}$ and ${q}$ are F-term equivalent, there exists a sequence of paths $p_0 = {p}, p_1, \dots , p_{s^\prime} = {q}$ such that $p_s$ and $p_{s+1}$ differ by a single $F$-term relation for $s=0,\dots,{s^\prime}-1$.

Suppose that neither ${p_s}q^{-1}$ nor ${p_{s+1}}q^{-1}$ pass through $v$ for some $s\in \{ 0,\dots,{s^\prime}-1 \}$, i.e. $\wind({p_s}q^{-1})$ and $\wind({p_{s+1}}q^{-1})$ are well defined. 
Since $p_s$ and $p_{s+1}$ differ by a single F-term relation we can write $p_s = \alpha_1 r_s \alpha_2$ and $p_{s+1} = \alpha_1 r_{s+1} \alpha_2$, where $r_s{r_{s+1}}^{-1}$ is the boundary of the union $D$ of the two faces which meet along the arrow dual to the relation. 
Since $\wind({p_s}q^{-1})$ and $\wind({p_{s+1}}q^{-1})$ are well defined, then
$$ \wind({p_s}q^{-1}) - \wind({p_{s+1}}q^{-1}) = \wind({p_s}{p_{s+1}}^{-1}) =  \wind(r_s{r_{s+1}}^{-1}) =0  $$
as there are no vertices 
in the interior of $D$.
However
$$ \wind({p}{q}^{-1}) \neq 0 = \wind({q}{q}^{-1}) $$
and so there must exist $s\in \{ 0,\dots,{s^\prime}-1 \}$ such that ${p_s}q^{-1}$ passes through $v$.  Then  $p^{\prime}:={p_s}$ is F-term equivalent to $p$ and $q$ and passes through $v$.

\end{proof}

\section{Proof of Theorem~\ref{GCACTHM}} \label{prop2thm}
Recall that a dimer model is algebraically consistent if the algebra map $\mathpzc{h}: A \lra{} \C [ \underline{M}^+]$ (\ref{algmap}) from the path algebra of the quiver modulo F-term relations to the toric algebra,
is an isomorphism. We noted in Remark~\ref{hinject} that for a geometrically consistent dimer model injectivity is equivalent to the statement of Theorem~\ref{Uniqueness}. It therefore remains for us to prove that the map $\mathpzc{h}$ is surjective. We now show that this follows from Proposition~\ref{EXTREM}.

To prove surjectivity, we need to show that for $i,j \in Q_0$ and any element $m \in M_{ij}^+ $ there exists a representative path, that is, a path $p$ from $i$ to $j$ in $Q$ such that $\Meq{p} =m$.

First we prove that it is sufficient to show that there exists a representative path for elements of $M_{ij}^+ $ which lie on the boundary of the cone $M^+$. Recall that the cone $M^+$ is the dual cone of the cone $N^+$ which is integrally generated by the perfect matchings. Therefore the elements in the boundary of $M^+$, are precisely those which evaluate to zero on some perfect matching.

Suppose that there exists a representative path for all elements $m \in M_{ij}^+ $ which lie on the boundary of $M^+$. Recall that the coboundary map $d \colon \Z^{Q_1} \lra{}\Z^{Q_2}$ sums the function on the edges around each face and $d(\pf) = \const{1}$ for any perfect matching $\pf$. Thus if $p$ is a path going once around the boundary of any quiver face $\face \in Q_2$ (starting at any vertex of $\face$) and $\pf$ is a perfect matching, then $\langle \pf , p \rangle = 1$. 
Define $\Box := \Meq{p}$ to be the image of $p$ in $\Mg$ which we note is independent of the choice of $p$ since the 2-torus is connected.
Now let $i,j \in Q_0$ and consider any element $m \in M_{ij}^+ $. If we evaluate any perfect matching $\pf$ on $m$ by definition we get an non-negative integer. Let
$$ n := \operatorname{min}\{ \langle \pf , m \rangle \mid \pf \text{ is a perfect matching }\} $$
Then for each perfect matching $\pf$, we observe that $\langle \pf ,(m-n\Box)\rangle = \langle \pf ,m\rangle -n \geq 0$ and by construction there exists at least one perfect matching where the equality holds. In other words $(m-n\Box)$ lies in the boundary of $\Mg^+$. Then by assumption, there is a representative path $q$ of $m-n\Box$, from $i$ to $j$ in $Q$.

Finally we construct a representative path for $m$. Let $\face$ be any face which has $j$ as a vertex, and let $p^{n}$ be the path which starts at $j$ and goes $n$ times around the boundary of $\face$. Then the path $qp^{n}$ from $i$ to $j$ is a well defined path, and
$$\Meq{qp^{n}} = \Meq{q}+n\Meq{p} = (m-n\Box) + n\Box = m$$

Now we prove that 
there exists a representative path for all elements $m \in M_{ij}^+ $ which lie on the boundary of $M^+$.

We start by fixing $i,j \in Q_0$, and let $m \in M_{ij}^+ $ be an element in the boundary of $M^+$. Then some perfect matching $\pf_m$ evaluates to zero on $m$.
Recall the short exact sequence (\ref{eq:Nmes})
$$0 \lra{} \Z \lra{} \Mm \lra{H} H_1(T;\Z) \lra{} 0 $$
where the kernel of $H$ is spanned by the class $\square = \del \Meq{\face} $ for any face $\face \in Q_2$.
Fix some path $p$ from $\wt{j}\in \Qcov_0$ to $\wt{i}\in \Qcov_0$ where $\wt{i}$ and $\wt{j}$ project down to $i$ and $j$ respectively. By Proposition~\ref{EXTREM} there is a path from $\wt{i}$ to every lift of $j$ in the universal cover such that this path is zero on some perfect matching. Therefore there exists a path $q$ which is zero on some perfect matching and such that $pq$ projects down to a closed path with any given homology class. In particular, there exists a path $q$ which is zero on some perfect matching $\pf_q$ and has
$$H(\Meq{p} + \Meq{q}) = H(\Meq{pq})= H(\Meq{p} + m)$$
Therefore $\Meq{q}-m \in \Mm$ is in the kernel of $H$, and so $\Meq{q}-m = k\square$ in $\Mm$, for some $k \in \Z$. Since every perfect matching evaluates to 1 on $\square$, applying $\pf_m$ and $\pf_q$ to $\Meq{q}-m$, we see that
$$ - \langle \pf_q , m \rangle = k = \langle \pf_m ,\Meq{q} \rangle $$
Finally, since $m$ and $\Meq{q}$ are both in $\Mg^+$, we see that $k=0$ and so $\Meq{q}=m $ in $\Mm$. Therefore the map $\mathpzc{h}$ is surjective, and we are done.




\chapter{Calabi-Yau algebras from algebraically consistent dimers}
In this chapter we prove one of the main theorems of this article:
\begin{theorem}\label{AlgCY3}
If a dimer model on a torus is algebraically consistent then the algebra $A$ obtained from it is CY3.
\end{theorem}
This gives a class of superpotential algebras which are Calabi-Yau and which can be written down explicitly.
We start by recalling Ginzburg's definition of a Calabi-Yau algebra \cite{Ginz}. A theorem due to Ginzburg shows that superpotential algebras are CY3 if a particular sequence of maps gives a bimodule resolution of the algebra. We formulate this for algebras coming from dimer models and show that, because they are graded, it is sufficient to prove that a one sided complex of right $A$ modules is exact. We then prove that this is the case for algebras obtained from algebraically consistent dimer models.
\section{Calabi-Yau algebras}
The notion of a Calabi-Yau algebra we use here was introduced by Ginzburg in \cite{Ginz}. 
Consider the contravariant functor $M \mapsto M^! := \RHom_{A-Bimod}(M, A \otimes A)$ on the `perfect' derived category of bounded complexes of finitely generated projective $A$-bimodules.
We use the outer bimodule structure on $A \otimes A$ when taking $\RHom$ and the result $M^!$ is an $A$-bimodule using the inner structure. 
\begin{definition}
An algebra $A$ is said to be a Calabi-Yau algebra of dimension $d\geq1$ if it is homologically smooth, and there exists an $A$-module quasi-isomorphism
$$ f:A \lra{\cong} A^![d] \quad \text{such that} \quad f=f^![d]  $$
\end{definition}

In \cite{Ginz} Ginzburg gives a way of proving that superpotential algebras are CY3 by checking that a particular sequence of maps is a resolution of the algebra. 
He gives an explicit description of this sequence of maps which we follow here.

Let $Q$ be a finite quiver with path algebra $\C Q$. As in Section~\ref{quivalgsec}, let $[ \C Q , \C Q ]$ be the complex vector space in $\C Q$ spanned by commutators and denote by $\C Q_\text{cyc}:= \C Q /[ \C Q , \C Q ] $ the quotient space. This space has a basis of elements corresponding to cyclic paths in the quiver.  For each arrow $a \in Q_1$ there is a linear map
$$\frac{\del}{\del x_a}: \C Q_\text{cyc} \rightarrow \C Q $$
which is a (formal) cyclic derivative. The image of a cyclic path is obtained by taking all the representatives of the path in $\C Q$ which starts with $x_a$, removing this and then summing.
We can write this map in a different way using a formal left derivative $\frac{\del_l}{\del_l x_a}: \C Q \rightarrow \C Q $ defined as follows:\\
Every monomial is of the form $x_b x$ for some $b \in Q_1$. Then we define:
\begin{equation*}
\frac{\del_l}{\del_l x_{a}}x_b x:=
\begin{cases}
x & \text{if }a=b,\\
0 & \text{otherwise}.  \end{cases}
\end{equation*}
We extend this to the whole of $\C Q$ by linearity. The formal right derivative $\frac{\del_r}{\del_r x_{a}}$ is defined similarly. Then we have
\begin{equation}
\frac{\del}{\del x_a}(x)= \sum_{\genfrac{}{}{0pt}{}{x^\prime \in \C Q}{(x^\prime)=(x)}}\frac{\del_l}{\del_l x_{a}}x^\prime = \sum_{\genfrac{}{}{0pt}{}{x^\prime \in \C Q}{(x^\prime)=(x)}}\frac{\del_r}{\del_r x_{a}}x^\prime
\end{equation}
where as before $(x)$ denotes the cyclic element in $\C Q_\text{cyc}$ corresponding to $x \in \C Q$.

For each arrow $a \in Q_1$ there is another linear map which we denote by:
$$\frac{\del}{\del x_a}: \C Q \rightarrow \C Q \otimes \C Q \qquad x \mapsto \left ( \frac{\del x}{\del x_a} \right )^{\prime} \otimes \left ( \frac{\del x}{\del x_a} \right )^{\prime\prime} $$
This is defined on monomials as follows: for each occurrence of $x_a$ in a monomial, the monomial can be written in the form $x x_a y$. This defines an element $x \otimes y \in \C Q \otimes \C Q$ and the sum of these elements over each occurrence of $x_a$ in the monomial is the image of the monomial. We extend this linearly for general elements of $\C Q$.
Following Ginzburg we use the same notation for both the cyclic derivative and this map.

The algebra $S := \bigoplus_{i \in Q_0} \C e_i $ is semi-simple and is the sub-algebra of $\C Q$ generated by paths of length zero. We define $T_1 := \bigoplus_{b \in Q_1} \C x_b $ with the natural structure of an $S,S$-bimodule. For each arrow $b \in Q_1$, there is a relation $ R_b := \frac{\partial}{\partial x_b}W$ which is the cyclic derivative of the superpotential $W$. Let $T_2 := \bigoplus_{b \in Q_1} \C R_b $ be the $S,S$-bimodule generated by these. 
For each vertex $v \in Q_0$, we obtain from the super-potential $W$ a syzygy
$$ W_v:= \sum_{b \in  T_v}  x_b R_b = \sum_{b \in  H_v} R_b x_b$$
Finally we define $T_3 := \bigoplus_{v \in Q_0} \C W_v $ which is isomorphic to $S$ and has an $S,S$-bimodule structure.
We consider the following maps:
$$\mew_0 :  A  \otimes_S A \lra{} A \qquad x \otimes y \mapsto xy$$ is given by the multiplication in $A$.
\begin{align*}
\mew_1 : & A \otimes_S T_1 \otimes_S A \lra{} A  \otimes_S A \\
&x \otimes x_a \otimes y \; \mapsto  xx_a\otimes y - x \otimes x_ay
\end{align*}

\begin{align*}
\mew_2 : & A \otimes_S T_2 \otimes_S A \lra{} A \otimes_S T_1 \otimes_S A \\
&x \otimes R_a \otimes y \; \mapsto \sum_{b \in Q_1} x \left( \frac{\del R_a}{\del x_b} \right )^{\prime}\otimes x_b \otimes \left ( \frac{\del R_a}{\del x_b} \right )^{\prime\prime}y
\end{align*}

\begin{align*}
\mew_3 : & A \otimes_S T_3 \otimes_S A \lra{} A \otimes_S T_2 \otimes_S A \\
&x \otimes W_v \otimes y \; \mapsto \sum_{b \in T_v}x x_b \otimes R_b \otimes y - \sum_{b \in H_v}x \otimes R_b \otimes x_b y
\end{align*}

Piecing these maps together we can write down the following sequence of maps:
\begin{equation}\label{fullcomplex} 0 \lla{}  A \lla{\mew_0}  A  \otimes_S A \lla{\mew_1} A \otimes_S T_1 \otimes_S A \lla{\mew_2} A \otimes_S T_2 \otimes_S A \lla{\mew_3} A \otimes_S T_3 \otimes_S A \lla{} 0
\end{equation}
which is the complex in Proposition~5.1.9 of \cite{Ginz}.
Then Corollary~5.3.3 of \cite{Ginz} includes the following result:
\begin{theorem}
$A$ is a Calabi-Yau algebra of dimension 3 if and only if the complex (\ref{fullcomplex}) is a resolution of $A$.
\end{theorem}
\begin{remark}
It is not actually necessary to check that the complex is exact everywhere as the first part of it is always exact. By Theorem~5.3.1 of \cite{Ginz} we see that it is sufficient to check exactness at  $A \otimes_S T_2 \otimes_S A$ and $A \otimes_S T_3 \otimes_S A$.
\end{remark}
\section{The one sided complex}
Recall from Section~\ref{Perfmatchsec} that any element in the interior of the perfect matching cone $\Ng^+$ defines a (positive) $\Z$-grading of $A$. 
We define the graded radical of $A$ by $\operatorname{Rad} A := \bigoplus_{n \geq 1} A^{(n)} $ and note that $S = A^{(0)}$ where $A^{(n)} $ denotes the $n$th graded piece.
Algebra $S$ is also the quotient of $A$ by the graded radical, and we can use the quotient map $A \rightarrow A/\operatorname{Rad}A \cong S$ to consider $S$ as an $A,A$-bimodule. Using this bimodule structure we consider the functor $\F = S \otimes_A -$ from the category of $A,A$-bimodules to itself. We apply this to the complex (\ref{fullcomplex}) and get the following complex: 
\begin{equation}\label{onesidecomplex} 0 \lla{}  S \lla{\F(\mew_0)}  A \lla{\F(\mew_1)}  T_1 \otimes_S A \lla{\F(\mew_2)} T_2 \otimes_S A \lla{\F(\mew_3)} T_3 \otimes_S A \lla{} 0
\end{equation}
We usually forget the left $A$-module structure, and treat this as a complex of right $A$-modules. We call this the one
sided complex. The maps are:
\begin{align*}
\F(\mew_1) : &  T_1 \otimes_S A \lra{}  A \\
&x_a \otimes y \; \mapsto  -  x_ay
\end{align*}

\begin{align*}
\F(\mew_2) : &  T_2 \otimes_S A \lra{}  T_1 \otimes_S A \\
&R_a \otimes y \; \mapsto \sum_{b \in Q_1}  x_b \otimes \left ( \frac{\del_l R_a}{\del_l x_b} \right )y
\end{align*}

\begin{align*}
\F(\mew_3) : & T_3 \otimes_S A \lra{}  T_2 \otimes_S A \\
& W_v \otimes y \; \mapsto  - \sum_{b \in H_v} R_b \otimes x_b y
\end{align*}


Something stronger is actually true. There is a natural transformation of functors from the identity functor on the category of $A,A$-bimodules to $\F$. In particular we have the following commutative diagram, where $\nat : 
t \mapsto 1_S \otimes_A t$.
\begin{equation*}
\resizebox{12cm}{!}{$
\begin{CD}
A @<{\mew_0}<< A \otimes_S A @<{\mew_1}<< A \otimes_S T_1 \otimes_S A  @<{\mew_2}<< A \otimes_S T_2 \otimes_S A      @<{\mew_3}<< A \otimes_S T_3 \otimes_S A \\
@V{\nat}VV       @V{\nat}VV               @V{\nat}VV          @V{\nat}VV    @V{\nat}VV    \\
S @<{\F(\mew_0)}<< A  @<{\F(\mew_1)}<< T_1 \otimes_S A @<{\F(\mew_2)}<< T_2 \otimes_S A  @<{\F(\mew_3)}<< T_3 \otimes_S A\\ 
\end{CD}$ }
\end{equation*}

We define a grading on all the objects in this diagram. 
Since the super-potential $W$ is a homogeneous element of $\C Q$, it can be seen that the syzygies $\{W_v \mid v \in Q_0\}$ are homogeneous elements
, and furthermore the relations $\{R_a \mid a\in Q_1\}$ are also homogeneous. Therefore we can extend the grading to a grading of $A \otimes_S T_\bullet \otimes_S A$ where the grade of a product of homogeneous elements is given by the sum of the grades in each of the three positions. We call this the total grading of $A \otimes_S T_\bullet \otimes_S A$.
Similarly we can define a total grading of $T_\bullet \otimes_S A$.
We note that the maps $\mew$, $\F(\mew)$ and $\nat$ in the commutative diagram above all respect the total grading.

Finally, using the diagram, we show that $\F(\mew)$ is the `leading term' of $\mew$. 
For any $u \in T_\bullet \otimes_S A$, we note that
$$\mew_\bullet : 1 \otimes u \mapsto 1 \otimes v + \sum x \otimes w $$
for some $v, w \in T_{\bullet-1} \otimes_S A$ and $x \in \operatorname{Rad}A$. Then
$$ v = \nat(\mew_\bullet(1 \otimes u)) = \F(\mew_\bullet)(\nat(1 \otimes u))  =  \F(\mew_\bullet)(u) $$
Using the fact that $\mew_\bullet$ is an $A,A$-bimodule map we see that for $y \in A$:
\begin{equation} \label{differentialmaps}
\mew_\bullet : y \otimes u \mapsto y \otimes \F(\mew_\bullet)(u) + \sum y x \otimes w
\end{equation}
for some $w \in T_{\bullet-1} \otimes_S A$ and $x \in \operatorname{Rad}A$.

\begin{proposition}
The full complex \ref{fullcomplex} is exact if and only if the one-sided complex \ref{onesidecomplex} is exact.
\end{proposition}
\begin{proof}
First we suppose the one-sided complex is exact and prove that the full complex is exact. Since the total grading is respected by all the maps we need only look at the $d$th graded pieces of each space. Let $\phi_0 \in (A \otimes_S T_n \otimes_S A)^{(\tg)}$ be closed with respect to $\mew$, where $n \in \{ 0,1,2,3\}$ and $T_0=S$. We can write $\phi_0$ in the form
$$\sum_{\mob \in Y} \mob \otimes \moa_{\mob} + \{\text{terms with higher grade in the first position}\}$$ where $Y$ is a linearly independent set of monomials in the graded piece $A^{(\tg_0)}$ with least possible grade, and $\moa_{\mob} \in (T_n \otimes_S A)^{(\tg-\tg_0)}$.

Applying the differential $\mew$ and using (\ref{differentialmaps}), we see that closedness translates into the condition:
\begin{align*}
0 &= \sum_\mob \mob \otimes \F(\mew)(u_\mob)  + \{\text{terms with higher grade in the first position}\}
\end{align*}
Since the monomials $\mob \in Y$ are linearly independent this implies that for all $\mob \in Y$
$$ \F(\mew)(\moa_\mob) = 0 $$
Using the exactness of the one-sided complex, we conclude that there exist elements $v_\mob \in (T_{n+1} \otimes_S A)^{(\tg-\tg_0)}$ (where $T_4 := 0$) such that $ \F(\mew)(v_\mob) = \moa_\mob $ for each $\mob \in Y$.
We construct an element:
$$\psi_1 := \sum_{\mob \in Y} \mob \otimes v_{\mob} \in (A \otimes_S T_{n+1} \otimes_S A)^{(\tg)}$$
and apply the differential $\mew$ to get
$$\mew \psi_1 = \sum_{\mob \in Y} \mob \otimes \moa_{\mob} + \{\text{terms with higher grade in the first position}\}$$
We observe that $\phi_1 := \phi_0 -\mew \psi_1$ is in the kernel of $\mew$ and has been constructed such that its terms have strictly higher grade in the first position than $\phi_0$. We iterate the procedure, noting that the grade in the first position is strictly increasing but is bounded above by the total grade $\tg$. Therefore after a finite number of iterations we must get $\phi_r = \phi_0 - \sum_{i=1}^r \mew \psi_i = 0$. We conclude that $\phi_0 = \mew(\sum_{i=1}^r \psi_i)$ and the complex is exact at $ A \otimes_S T_n \otimes_S A$.

Conversely, we note that full complex \ref{fullcomplex} is a projective $A,A$-bimodule resolution of $A$, and that $A$ itself has the structure of a projective left $A$ module. Therefore it is an exact sequence of projective left modules and so is split exact. As a consequence it remains exact when we tensor with $S$ on the left, i.e. when we apply the functor $\F$.


\end{proof}

\section{Key lemma}
The following lemma is going to play an important part in the proof of the main theorem. We recall from Section~\ref{conseqgcsec} that the lattice $\Mg$ is a quotient of $\Z_{Q_1}$ and that the boundary map we get by considering the quiver as a cellular decomposition of the torus, descends to a well defined map $\del : \Mg \rightarrow \Z_{Q_0}$. Furthermore, by summing the arrows, every path $p$ in $Q$ determines a class $\Meq{p} \in \Mg$, which lies in $\Mg_{ij} = \del^{-1}(j-i)$ where $tp=i$ and $hp=j$.
\begin{lemma}\label{posit}
Let $v, j \in Q_0$ be quiver vertices, and consider an element $m \in M_{vj}$. Suppose that for all arrows $b \in Q_1$ with $hb=v$, we have $m+ \Meq{b} \in  M_{ij}^+ $ where $i=tb$. Then $m \in M_{vj}^+$.
\end{lemma}
\begin{proof}
We consider the vertex $v$, and label the outgoing arrows $a_1, \dots , a_{k}$ and the incoming arrows $b_1, \dots , b_{k}$ around $v$ as below. We also label the paths completing the boundary of each face $\bp_{i}$ and $\bq_{i}$ as below.
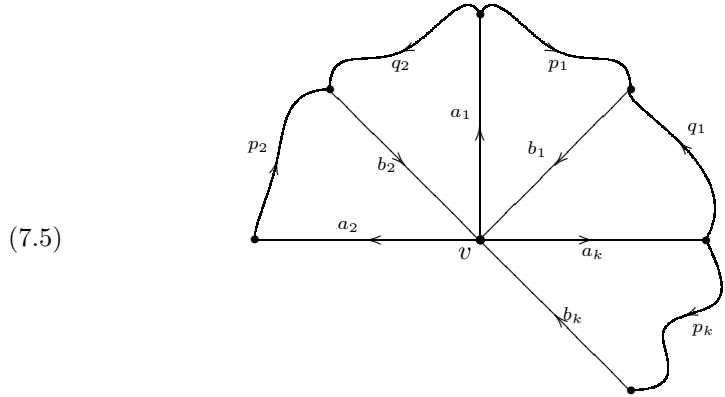
\begin{equation} \label{vpicture} \begin{xy} <20mm,0mm>:
,0*@{*}*^+!RU{v};(1.5,0)*@{*}**@{-}?(0.5)*@{>}*^+!U{\scriptstyle{a_{k}}}
,0*@{*};(1,1)*@{*}**@{-}?(0.5)*@{<}*^+!DR{\scriptstyle{b_1}}
,0*@{*};(0,1.5)*@{*}**@{-}?(0.5)*@{>}*^+!DR{\scriptstyle{a_1}}
,0*@{*};(-1,1)*@{*}**@{-}?(0.5)*@{<}*^+!R{\scriptstyle{b_2}}
,0*@{*};(-1.5,0)*@{*}**@{-}?(0.5)*@{>}*^+!DR{\scriptstyle{a_2}}
,0*@{*};(1,-1)*@{*}**@{-}?(0.5)*@{<}*^+!L{\scriptstyle{b_{k}}}
,(1.5,0);(1,1) **\crv{(1.8,0.5)&(0.9,0.9)}?(0.5)*@{>}*^+!DL{\scriptstyle{\bq_{1}}}
,(1,1);(0,1.5) **\crv{(1.0,1.3)&(0.5,1)&(0.1,1.7)}?(0.5)*@{<}*^+!U{\scriptstyle{\bp_{1}}}
,(0,1.5);(-1,1) **\crv{(-0.1,1.7)&(-0.5,1)&(-1.0,1.3)}?(0.5)*@{>}*^+!U{\scriptstyle{\bq_{2}}}
,(-1,1);(-1.5,0) **\crv{(-1.3,1)&(-1.3,0.5)&(-1.5,0.1)}?(0.5)*@{<}*^+!DR{\scriptstyle{\bp_{2}}}
,(1,-1);(1.5,0) **\crv{(1.4,-1.0)&(1,-0.5)&(1.7,-0.5)&(1.6,-0.2)}?(0.5)*@{<}*^+!UL{\scriptstyle{\bp_{k}}}
\end{xy} \end{equation}
\\
Let $m_i :=m+ \Meq{b_i} \in  M_{tb_ij}^+$. We start by noting that
$
\bp_{i}b_i = \bq_{i+1}b_{i+1}$ is the F-term relation dual to $a_i$, so
$\Meq{\bp_{i}} + \Meq{b_i} = \Meq{\bq_{i+1}} + \Meq{b_{i+1}}$ in $\Mg$.
Adding $m$ to both sides we see that:
\begin{equation}\label{pathseq}
m_i + \Meq{\bp_{i}} = m_{i+1} + \Meq{\bq_{{i+1}}}
\end{equation}
Since $m_i\in  M_{tb_ij}^+$, using algebraic consistency, there exist paths $y_i$ 
such that $\Meq{y_i} = m_i$ for all $i$.
We lift each path $y_i$ to a path in the universal cover $\Qcov$. From equation~(\ref{pathseq}), again using algebraic consistency, we observe that the two paths $\bp_{i}y_{i}$ and $\bq_{i+1}y_{i+1}$ are F-term equivalent.

We seek a path $\sigma$ which is F-term equivalent to $\bp_{\alpha}y_{\alpha}$ for some $\alpha \in \{1, \dots , k \}$, and which passes through the vertex $v$. Either there exists some $\alpha$ such that $y_{\alpha}$ passes through $v$, in which case we define $\sigma := \bp_{\alpha}y_{\alpha}$, or we can consider the closed curves $\gamma_i := y_i^{-1} \bp_{i}^{-1} \bq_{i+1}y_{i+1} $ which have well defined winding numbers $\wind(\gamma_i)$ around $v$ for all $i$. In this case
$$\sum_{i=1}^{k}|\wind(\gamma_i)| \geq |\wind(\gamma_1 \dots \gamma_{k})| = |\wind(\bp_{1}^{-1} \bq_{2} \dots\bp_{{k}}^{-1} \bq_{{1}})| = 1 $$
since the path $\bp_{1}^{-1} \bq_{2} \dots\bp_{{k}}^{-1} \bq_{{1}}$ is the boundary of the union of all faces containing $v$, which is homeomorphic to a disc in the plane with $v$ an interior point. Therefore there exists $\alpha \in \{ 1, \dots ,k \}$ such that $\wind(\gamma_{\alpha})$ is well defined and nonzero. Thus we may apply Corollary~{\ref{expath}} to the paths $\bp_{\alpha}y_{\alpha}$ and $\bq_{\alpha+1}y_{\alpha+1}$ and we obtain an F-term equivalent oriented path $\sigma$ which passes through $v$.

To prove the result we need to show that $  \langle \pf , m \rangle \geq 0$ for every perfect matching $\pf$. Let $\pf$ be a perfect matching and suppose that $\langle \pf , \Meq{b_i} \rangle =0 $ for some $i$. Then
$$  \langle \pf , m \rangle = \langle \pf , m_i - \Meq{b_i} \rangle = \langle \pf , m_i \rangle \geq 0  $$
since $m_i \in \Mg^+$, and we are done. Otherwise $ \langle \pf , \Meq{b_i} \rangle = 1$ for all $i=1, \dots, k$. Since the path $\sigma$ passes through $v$ then by construction it must contain an arrow $b_i$ for some $i$, and
\begin{equation}\label{contq}
\langle \pf , \Meq{\sigma} \rangle 
\geq \langle \pf , \Meq{b_i} \rangle =1
\end{equation}
Since $\sigma$ and $\bp_{\alpha}y_{\alpha}$ are $F$-term equivalent,
\begin{equation}\label{feqres}
\Meq{\sigma} = \Meq{\bp_{\alpha}y_{\alpha}} = 
\Meq{\bp_{\alpha}} + \Meq{b_{\alpha}} + m
\end{equation}
Now $\bp_{\alpha}b_{\alpha}a_{\alpha}$ is the boundary of a face in the quiver, so perfect matching $\pf$ is non-zero on a single arrow of $\bp_{\alpha}b_{\alpha}a_{\alpha}$. We know that $\langle \pf, \Meq{b_{\alpha}}\rangle=1$, so $\langle \pf, \Meq{\bp_{\alpha}}\rangle=0$.
We apply our perfect matching $\pf$ to (\ref{feqres}) and use (\ref{contq}) to see that
$$ 1 \leq \langle \pf , \Meq{\sigma} \rangle = 1 +\langle \pf , m \rangle  $$
Thus we have shown that $\langle \pf , m \rangle \geq 0$ as required. 
\end{proof}

\section{The main result}

We are now in the position to prove the main theorem of this chapter.
\begin{theorem}\label{exact}
If we have an algebraically consistent dimer model on a torus then the sequence of maps
\begin{equation} \label{cmplx} T_1 \otimes_S A \lla{\F(\mew_2)} T_2 \otimes_S A \lla{\F(\mew_3)} T_3 \otimes_S A \lla{} 0
\end{equation}
is exact, and hence $A$ is a CY3 algebra.
\end{theorem}
\begin{proof}
First we prove exactness at $T_2 \otimes_S A$. Consider any element $ \phi \in T_2 \otimes_S A$
which is closed. We can write $\phi$ in the form
$$\phi = \sum_{b \in Q_1}R_b \otimes \moa_{b}$$
We need to show that $\phi$ is in the image of the differential.
We can write any $\psi \in T_3 \otimes_S A$ in the form $\psi :=\sum_{v \in Q_0} W_v \otimes c_{v} \in T_3 \otimes_S A $ and we note that
\begin{equation*}
\mew (\psi) =\sum_{v \in Q_0} \sum_{b \in H_v} R_b \otimes  x_{b} c_{v}
\end{equation*}
Therefore $\phi$ is in the image of the differential if and only if the following statement holds.
{\it For each vertex $v \in Q_0$ there is an element $c_v \in A$ such that $\moa_b = x_b c_v$ for all arrows $b \in H_v$.} \\
We now prove this statement. Applying the differential to $\phi$ we observe that
\begin{equation*}
0 = \sum_{a,b \in Q_1}x_a \otimes \frac{\del_lR_b}{\del_l x_{a}} \moa_{b} = \sum_{a,b \in Q_1}x_a \otimes \frac{\del_rR_a}{\del_r x_{b}} \moa_{b}
\end{equation*}
By construction the set $\{ x_a \mid a \in Q_1 \}$ is linearly independent in $T_1$, so
\begin{equation}\label{inker} \sum_{b \in Q_1}\frac{\del_r R_a}{\del_r x_{b}} \moa_{b} = 0 \qquad \forall a \in Q_1
\end{equation}
The algebra $A = \bigoplus_{j \in Q_0} Ae_j$ naturally splits into pieces using the idempotents. 
Without loss of generality we assume $\moa_b \in Ae_j$ for some $j \in Q_0$, 
so for each $b \in Q_1$ we have $\moa_b \in e_{tb}Ae_j$.

Using algebraic consistency we work 
on the toric algebra side. For $b \in Q_1$, the element corresponding to ${\moa_{b}}$ is of the form
$$ \wt{\moa_{b}} = \sum_{m \in M_{tb,j}^+} \alpha_{b}^{m} z^{m} \in \C[\underline{M}^+]$$
where $\alpha_{b}^{m} \in \C$. 
Because $M_{tb,j}^+ \subset M$ where subtraction is well defined, we can write each element in the form $m = (m-\Meq{b}) + \Meq{b}$. 
Therefore we can re-write $\wt{\moa_{b}}$ as
$$ \wt{\moa_{b}} = z^{\Meq{b}}\sum_{m \in M_{vj}} \alpha_{b}^{m+\Meq{b}} z^{m}$$
where $\alpha_{b}^{m+\Meq{b}}$ is take to be zero where it was not previously defined, i.e. when $m+\Meq{b} \notin M_{tb,j}^+$. We need to show that $\sum_{m \in M_{vj}} \alpha_{b}^{m+\Meq{b}} z^{m}$ is the same for all $b \in Q_1$ which have the same head, and that it is a well defined element of $\C[\underline{M}^+]$.


Let $v \in Q_0$ be any vertex and consider the following diagram, for any arrow $a$ with $ta =v$.
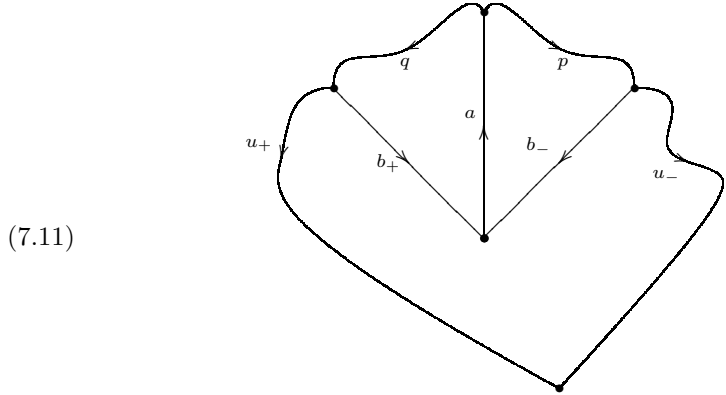
\begin{equation} \label{vandupicture} \begin{xy} <20mm,0mm>:
,0*@{*};(1,1)*@{*}**@{-}?(0.5)*@{<}*^+!DR{\scriptstyle{b_-}}
,0*@{*};(0,1.5)*@{*}**@{-}?(0.5)*@{>}*^+!DR{\scriptstyle{a}}
,0*@{*};(-1,1)*@{*}**@{-}?(0.5)*@{<}*^+!R{\scriptstyle{b_+}}
,(1,1);(0,1.5) **\crv{(1.0,1.3)&(0.5,1)&(0.1,1.7)}?(0.5)*@{<}*^+!U{\scriptstyle{\bp}}
,(0,1.5);(-1,1) **\crv{(-0.1,1.7)&(-0.5,1)&(-1.0,1.3)}?(0.5)*@{>}*^+!U{\scriptstyle{\bq}}
,(-1,1);(0.5,-1)*@{*} **\crv{(-1.3,1)&(-1.3,0.5)&(-1.5,0.1)}?(0.5)*@{>}*^+!DR{\scriptstyle{\moa_{+}}}
,(1,1);(0.5,-1) **\crv{(1.4,1.0)&(1,0.5)&(1.7,0.5)&(1.6,0.2)}?(0.5)*@{>}*^+!UR{\scriptstyle{\moa_{-}}}
\end{xy} \end{equation}
\\
In the diagram we have drawn $\moa_{\pm}:= \moa_{b_\pm} $ as if they were paths, however it should be remembered that they are elements in the algebra and don't necessarily correspond to actual paths in the quiver. The relation dual to $a$ is $R_{a} = x_{\bp_{}}x_{b_-}-x_{\bq_{}}x_{b_{+}}$, and on substitution into (\ref{inker}), we get
\begin{equation}\label{closed}
x_{\bp_{}}\moa_{-}-x_{\bq_{}}\moa_{+} =0
\end{equation} in $A$.
Under the isomorphism to the toric algebra, and using the relation $\Meq{\bp_{}} + \Meq{b_-} = \Meq{\bq_{}} + \Meq{b_{+}}$ we obtain:
$$ 0=z^{\Meq{\bp}}\wt{\moa_{-}}- z^{\Meq{\bq}}\wt{\moa_{+}}= z^{\Meq{\bp}+\Meq{b_-}}\sum_{m \in M_{vj}} (\alpha_{b_-}^{m+\Meq{b_-}}-\alpha_{b_+}^{m+\Meq{b_+}}) z^{m}$$
The set $\{z^{m+\Meq{\bp_{}} + \Meq{b_-}}\mid m \in M_{vj} \}$ is independent, and so we can see that
$$\alpha_{b_-}^{m+\Meq{b_-}}=\alpha_{b_+}^{m+\Meq{b_+}}=: \alpha^m \text{  for all  } m \in M_{vj}$$
Then
$$ \wt{\moa_{b_\pm}} = z^{\Meq{b_\pm}}\sum_{m \in M_{vj}} \alpha^{m} z^{m}$$
We have shown that this formula holds for
a pair of arrows $b_\pm$ with head at $v$. However, recalling that the arrows around $v$ fit together as in \ref{vpicture}, and considering the arrows pairwise, we observe that for any arrow $b \in H_v$.
$$ \wt{\moa_{b}} = z^{\Meq{b}}\sum_{m \in M_{vj}} \alpha^{m} z^{m}$$
Furthermore $\alpha^{m}= 0 $ unless $m+\Meq{b} \in M_{tb,j}^+$ for all $b \in H_v$. It follows from Lemma~{\ref{posit}} that $\alpha^{m}$ can only be non-zero when $m \in M_{vj}^+$. Thus we see that for each $b \in H_v$,
$$ \wt{\moa_{b}} = z^{\Meq{b}} \wt{c_{v}} \quad \text{where} \quad  \wt{c_{v}} := \sum_{m \in M_{vj}^+} \alpha^{m} z^{m} $$ which is a well defined element of $\C [ \underline{M}^+]$. 
Using algebraic consistency again, there exists element $c_{v} \in A$ such that $\moa_{b} = x_{b_i} c_{v}$, and we are done.

To show that the complex is exact at $T_3 \otimes_S A $, suppose that $\phi \in T_3 \otimes_S A$ is closed. We may write $\phi$ in the form $ \phi = \sum_{v \in Q_0} W_v \otimes \moa_{v} $.
Applying the differential we observe that $\sum_{v \in Q_0} \sum_{b \in H_v}R_b \otimes  x_{b} \moa_{v} =0$ and since the relations $R_b$ are linearly independent over $\C$ in $T_2$, this implies that $ x_{b} \moa_{v} =0$ for all $v \in Q_0$ and $b \in H_v$. Because of algebraic consistency we can use the given isomorphism and work in the toric algebra $\C [ \underline{M}^+]$. Let $\sum_{m \in M_{vj}^+} \beta^m z^m$ be the image of $\moa_{v}e_j $, where $\beta^m \in \C$.
Then since $ x_{b} \moa_{v}e_j =0$, for each $j \in Q_0$ we see that
$$z^{\Meq{b}} \sum_{m \in M_{vj}^+} \beta^m z^m = 0  $$
The set $\{z^{m+\Meq{b}}\mid m \in M_{vj} \}$ is independent, therefore
$ \beta^m = 0$  for all  $ m \in M_{vj}$. Mapping back to the superpotential algebra, this implies $ \moa_{v}e_j =0$, for each $j \in Q_0$ and $v \in Q_0$, so $ \moa_{v}=0$ for any $v \in Q_0$.
We conclude that $\phi =0$.
\end{proof}

\chapter{Non-commutative crepant resolutions} \label{NCCRChap}

In this final chapter we use results obtained in the previous chapters to prove that the algebra $A$, obtained from an algebraically consistent dimer model on a torus, is a non-commutative crepant resolution (NCCR) in the sense of Van den Bergh \cite{VdB}, of the commutative ring $R$ associated to the dimer model. 
\section{Reflexivity}
The key step towards proving that $A$ is an NCCR, is to prove the following property.
\begin{proposition} \label{reflexthm}
Let $A \cong \C[\underline{M}^+]$ be the algebra obtained from an algebraically consistent dimer model on a torus, and fix any $i\in Q_0$. Then there is a natural isomorphism of $R:= \C[\Mm^+]$-modules
$$ \Psi :\C[\Mg_{jk}^+] \lra{\cong} \Hom_{R}(\C[\Mg_{ij}^+], \C[\Mg_{ik}^+]) $$
for each $j,k\in Q_0$.
\end{proposition}
This allows us to write $A$ as the endomorphism algebra of a reflexive $R$-module in a explicit way. Together with the Calabi-Yau property proved in Chapter~7 this is sufficient to prove the theorem.


%
We start by proving a localised version of Proposition~\ref{reflexthm} which holds in very general circumstances.
\begin{lemma}\label{homismult}
For any $i,j,k \in Q_0$, there is an isomorphism of $\C[\Mm]$-modules
$$ \wh{\Psi} :\C[\Mg_{jk}] \lra{\cong} \Hom_{\C[\Mm]}(\C[\Mg_{ij}], \C[\Mg_{ik}]) $$
defined by $z \mapsto \{ \varphi_z \colon \wt{z} \mapsto z \wt{z} \}$.
\end{lemma}
\begin{proof}
The map is well defined and injectivity follows from cancellation in $\C[\underline{M}]$. Therefore we just need to show that $\wh{\Psi}$ is surjective. 

Let $ \varphi \in \Hom_{\C[\Mm]}(\C[\Mg_{ij}], \C[\Mg_{ik}]) $ be any element. We note that if $m,m^{\prime} \in M_{ij}$ then their difference is an element of $\Mm$. Using the fact that $ \varphi$ is a $\C[\Mm]$-morphism we see that 
\begin{equation} \label{indeqation}
\varphi(z^{m})= \varphi(z^{m-m^{\prime}}z^{m^{\prime}}) =z^{m-m^{\prime}}\varphi(z^{m^{\prime}})
\end{equation}
Define $\xi:= \varphi(z^{m})/z^{m} \in \C[\Mg_{jk}]$ which we see from equation~(\ref{indeqation}) is independent of the choice of $m$. Then, for every $m \in M_{ij} $
$$ \varphi_\xi(z^m) = \xi z^m =\varphi(z^m) $$
and so $\varphi= \varphi_\xi$.
\end{proof}

We now consider the relationship between the $\C[\Mm]$-module 
$$\mathcal{H}:= \Hom_{\C[\Mm]}(\C[\Mg_{ij}], \C[\Mg_{ik}])$$ 
and the $\C[\Mm^+]$-module 
$$\mathcal{H}^+:= \Hom_{\C[\Mm^+]}(\C[\Mg_{ij}^+], \C[\Mg_{ik}^+])$$
Let $m,m^{\prime} \in M_{ij}^+$ be any two elements and suppose $ \varphi \in \mathcal{H}^+$. By Lemma~\ref{posshift} there exists an element $\zeta \in M_o^+$ such that $ m-m^{\prime} + \zeta \in \Mm^+$. 
Then, using the fact that $ \varphi$ is a $\C[\Mm^+]$-morphism we see that
$$ \varphi(z^{m})= \varphi(z^{m-m^{\prime}+ \zeta} z^{-\zeta}z^{m^{\prime}}) =z^{m-m^{\prime}}z^{\zeta}\varphi(z^{-\zeta}z^{m^{\prime}})=z^{m-m^{\prime}}\varphi(z^{m^{\prime}}) $$
Thus $\xi_\varphi := \varphi(z^{\wt{m}})/z^{\wt{m}}$ is independent of choice of $\wt{m} \in M_{ij}^+$ and there is a well defined map:
$$\iota:\mathcal{H}^+ \hookrightarrow \mathcal{H} \qquad \varphi \mapsto \wh{\varphi} $$
where $ \wh{\varphi}(z^m):= \xi_\varphi  z^m $, which is also independent of the choice of $\wt{m} \in M_{ij}^+$. 
We note that $\iota$ extends each element $\varphi \in \mathcal{H}^+$ to a map $\wh{\varphi}$ on the whole of $\C[\Mg_{ij}]$ whose restriction to $\C[\Mg_{ij}^+]$ is $\varphi$, i.e. 
$$\wh{\varphi}|_{\C[\Mg_{ij}^+]} = \varphi : \C[\Mg_{ij}^+] \lra{} \C[\Mg_{ik}^+]$$
One can readily check that $\iota$ is injective. Furthermore, if we consider $\mathcal{H}$ as a $\C[\Mm^+]$-module by restricting the $\C[\Mm]$-module structure, we see that it is a morphism of $\C[\Mm^+]$-modules. 
\begin{proof}[Proof of Proposition~\ref{reflexthm}]
As in the proof of Lemma~\ref{homismult}, it can be seen that for any $i,j,k \in Q_0$, the map 
$$ \Psi :\C[\Mg_{jk}^+] \lra{} \Hom_{\C[\Mm^+]}(\C[\Mg_{ij}^+], \C[\Mg_{ik}^+]) $$
defined by $z \mapsto \{ \varphi_z \colon \wt{z} \mapsto z \wt{z} \}$ is well defined and injective. We consider the commutative square
\[ \begin{CD}
\mathcal{H}^+ @>{\iota}>> \mathcal{H} \\
@A{\Psi}AA                        @A{\wh{\Psi}}AA       \\
\C[\Mg_{jk}^+] @>{\subseteq}>> \C[\Mg_{jk}]
\end{CD}
\]
To complete the proof we need to show that $ \Psi $ is a surjective map. Let $\varphi \in \mathcal{H}^+$ and let $\wh{\varphi} = \iota(\varphi)$ be its extension in $\mathcal{H}$. By Lemma~\ref{homismult} the map $\wh{\Psi}$ on the right of the diagram is an isomorphism so there exists $z \in \C[M_{jk}]$ such that $\wh{\varphi} = \wh{\varphi}_z : \wt{z} \mapsto z\wt{z}$. Writing $z$ in the form $z=\sum_{m \in M_{jk}}\alpha_m z^m$ we see that for any $m^{\prime} \in M_{ij}^+$,
$$ \varphi(z^{m^{\prime}}) = \wh{\varphi}(z^{m^{\prime}}) = \sum_{m \in M_{jk}}\alpha_m z^{m+m^{\prime}} \in \C[M_{ik}^+] $$
In particular, for each $m \in M_{jk}$ such that $\alpha_m \neq 0$, we note that $m+m^{\prime} \in M_{ik}^+$ for all $m^{\prime} \in M_{ij}^+$. Now because the algebra $A$ is obtained from an algebraically consistent dimer model, we can use Lemma~\ref{posit} to deduce that $m \in M_{jk}^+$. Thus $z \in \C[M_{jk}^+]$, and $\varphi= \Psi(z)$.
\end{proof}
\begin{corollary} \label{reflexcor}
The $R$-module $\C[M_{jk}^+]$ is reflexive for all $j,k \in Q_0$.
\end{corollary}
\begin{proof}
As $\C[M_{ii}^+]=R$, for all $i \in Q_0$, it follows from Proposition~\ref{reflexthm} that $(\C[M_{jk}^+])^\vee \cong \C[M_{kj}^+]$ and hence $\C[M_{jk}^+]$ is reflexive.
\end{proof}

\section{Non-commutative crepant resolutions}
We recall the definition of an NCCR introduced by Van den Bergh in \cite{VdB}. 
\begin{definition}\label{NCCRdefn}
Let $R$ be a normal Gorenstein domain. An NCCR of $R$ is a homologically homogeneous $R$-algebra of the form $A= \End_R(M)$ where $M$ is a reflexive $R$-module.
\end{definition}
If $R$ is equi-dimensional of dimension $n$, which will be the case in our examples, the statement that $A$ is homologically homogeneous is equivalent to saying that all simple $A$-modules have projective dimension $n$. We now prove the main theorem of this chapter.
\begin{theorem}\label{NCCRthm}
%
Given an algebraically consistent dimer model on a torus, the corresponding toric algebra $A \cong \C[\underline{M}^+]$ is an NCCR of the (commutative) ring $R:= \C[\Mm^+]$ associated to that dimer model. 
\end{theorem}
\begin{proof}
By Theorem~\ref{AlgCY3}, if we have an algebraically consistent dimer model on a torus then the corresponding toric algebra $A$ is a Calabi-Yau algebra of global dimension 3. By Proposition~2.6 (see also Remark~3.2(1)) in \cite{VdB2} we see that $A$ is homologically homogeneous. Fix any $i\in Q_0$ and let
$$\mathcal{B}_i:=\bigoplus_{j \in Q_0}\C[\Mg_{ij}^+]$$
Then using Proposition~\ref{reflexthm} we note that:
$$ A = \bigoplus_{j,k \in Q_0} \C[\Mg_{jk}^+] \cong \bigoplus_{j,k \in Q_0} \Hom_{R}(\C[\Mg_{ij}^+], \C[\Mg_{ik}^+]) \cong \operatorname{End}_R(\mathcal{B}_i) $$
where $R$ is the centre of $A$ (see Lemma~\ref{cent}). Therefore to complete the proof we just need to show that $\mathcal{B}_i$ is a reflexive $R$-module. However this follows straight from Corollary~\ref{reflexcor}.
\end{proof}

Finally we recall from Section~\ref{stiengul} that Gulotta in \cite{Gulotta} and Stienstra in \cite{Stienstra2} prove that for any lattice polygon $V$, there exists a geometrically consistent dimer model which has $V$ as its perfect matching polygon. Thus it is possible to associate to every Gorenstein affine toric threefold, a geometrically consistent dimer model. Therefore we have proved the following result.

\begin{theorem}
Every Gorenstein affine toric threefold admits an NCCR, which can be obtained via a geometrically consistent dimer model.
\end{theorem}


\backmatter
\bibliographystyle{amsplain}

\begin{thebibliography}{99}
\bibitem{AltHille} K.~Altmann and L.~Hille, {\em Strong Exceptional Sequences Provided by Quivers},
Algebras and Representation Theory, {\bf 2}: 1-17, 1999.

\bibitem{planetmath} A.~Ambrosio,\\
 \url{http://planetmath.org/encyclopedia/ProofOfBirkoffVonNeumannTheorem.html}.

\bibitem{Birkhoff} G.~Birkhoff, {\em Tres observaciones sobre el algebra lineal}, Univ. Nac. Tucum� Rev, Ser. A, no. 5, pp147-151. (1946) 

\bibitem{Bocklandt} R.~Bocklandt, {\em Graded Calabi Yau algebras of dimension 3}, Journal of pure and applied algebra, 212:1(2008), p. 14-32

\bibitem{Bocklandt2} R.~Bocklandt, {\em Calabi-Yau algebras and weighted quiver polyhedra}, preprint {arXiv:0905.0232v1}.

\bibitem{BKR}
   T.~Bridgeland, A.~King; M.~Reid,
   \emph{The McKay Correspondence as an Equivalence of Derived Categories},
   Journal of the American Mathematical Society, Vol. 14, No. 3. (Jul., 2001), pp. 535-554.

\bibitem{BuchsEisen} D.A.~Buchsbaum and D.~Eisenbud, {\em What makes a complex exact?}, J. Algebra {\bf 25} (1973),
259-268.

\bibitem{Grenoble} D.~Cox, G.~Barthel, {\em Geometry of Toric Varieties}, Lecture Notes from `Summer School 2000', Grenoble, 2000.

\bibitem{Dani} V.~Danilov, {\em The Geometry of Toric Varieties}, Russ. Math. Surveys, {\bf 33}:2 (1978), 97--154.

\bibitem{Davison} B.~Davison, { \em Consistency conditions for brane tilings}, preprint arXiv: 0812:4185v1

\bibitem{Ewald} G.~Ewald, {\em Combinatorial Convexity and Algebraic Geometry},
Graduate Texts in Mathematics, Springer-Verlag, New~York, 1996.

\bibitem{FHKVW} S. Franco, A. Hanany, K. D. Kennaway, D. Vegh, and B. Wecht, {\em Brane dimers and quiver gauge theories}, preprint {arXiv: hep-th/0504110 }.

\bibitem{Fulton} W.~Fulton, {\em Introduction to Toric Varieties},
Princeton University Press, Princeton, 1993.

\bibitem{Ginz}
V.~Ginzburg,
\emph{Calabi-Yau Algebras},
preprint arXiv:math.AG/0612139.

\bibitem{Gulotta} D.~R.~Gulotta  {\em Properly ordered dimers, $R$-charges, and an efficient inverse algorithm}, preprint  arXiv:0807.3012 (July 2008)

\bibitem{Hall35} P.~Hall, {\em On Representatives of Subsets},  J. London Math. Soc. 10, 26-30, 1935.

\bibitem{HHV} A.~Hanany, C.~P.~Herzog, and D.~Vegh, {\em Brane tilings and exceptional collections},
JHEP 07 (2006) 001.

\bibitem{HananyKen} A.~Hanany and K.~D.~Kennaway, {\em Dimer models and toric diagrams}, 
preprint {arXiv: hep-th/0503149}.

\bibitem{HananyVegh} A.~Hanany and D.~Vegh, {\em Quivers, Tilings, Branes and Rhombi},
preprint 	arXiv: hep-th/0511063v1

\bibitem{Hartshorne} R.~Hartshorne, {\em Algebraic Geometry},
Graduate Texts in Mathematics, Springer-Verlag, New~York, 1977.

\bibitem{Hille} L.~Hille, {\em Toric quiver varieties},
Algebras and Modules II (I.C.R.A. 8) (I.~Reiten, S.O.~Smal{\o} and O.~Solberg, eds.), Canadian Math. Soc. Proceedings, vol. 24, 1998, pp. 311-325.

\bibitem{Ishii1} A.~Ishii and K.~Ueda, {\em On moduli spaces of quiver representations associated with brane tilings},
preprint arXiv:0710.1898

\bibitem{Ishii} A.~Ishii and K.~Ueda, {\em Dimer models and the special McKay correspondence},
preprint arXiv:0905.0059

\bibitem{Iversen} B.~Iversen, {\em Cohomology of Sheaves}, Universitext,
Springer-Verlag, Berlin, 1986.

\bibitem{Jarvis}  R.~A.~Jarvis,  {\em On the identification of the convex hull of a finite set of points in the plane}, Information Processing Letters 2: 1821. (1973)

\bibitem{Kennaway} K.~D.~Kennaway, { \em Brane Tilings}, International Journal of Modern Physics A (IJMPA),  	
Volume:22, No:18 Year: 2007 pp. 2977-3038 

\bibitem{Kenyonintro} R.~Kenyon, { \em An introduction to the dimer model}, preprint arXiv: math.CO/0310326

\bibitem{Kenyon} R.~Kenyon and J.M.~Schlenker, { \em Rhombic embeddings of planar quad-graphs}, Transactions of the AMS, Volume 357, Number 9, Pages 3443-3458 (2004)

\bibitem{King} A.~King, {\em Tilting Bundles on some Rational Surfaces},
preprint at \url{http://www.maths.bath.ac.uk/~masadk/papers}, 1997.

\bibitem{Mozgovoy} S.~Mozgovoy and M. Reineke, { \em On the noncommutative Donaldson-Thomas invariants arising from brane tilings}, preprint arXiv: 0809:0117

\bibitem{Mustata} M.~Musta\c{t}\v{a}, {\em Vanishing Theorems on Toric Varieties},  Tohoku Math. J.(2), vol. 54 , 451-470, 2002.

\bibitem{Northcott} D.G.~Northcott, {\em Finite Free Resolutions}, Cambridge Tracts in Math.,
vol. 71, C.U.P., Cambridge, 1976.

\bibitem{VdB2}
J.T.~Stafford and M.~Van~den~Bergh,
\emph{Noncommutative Resolutions and Rational Singularities},
preprint 	arXiv:math/0612032v1

\bibitem{Stienstra} J.~Stienstra, {\em Hypergeometric Systems in two Variables, Quivers, Dimers and Dessins d'Enfants}, Modular Forms and String Duality, Fields Institute Communications volume 54 (2008), pp. 125-161

\bibitem{Stienstra2} J.~Stienstra, { \em  Computation of Principal $\mathcal{A} $-determinants through Dimer Dynamics}, preprint 
arXiv:0901.3681

\bibitem{Balazs} B.~Szendr\"oi {\em Non-commutative Donaldson-Thomas invariants and the conifold}, Geometry and Topology 12 (2008) 1171-1202. 

\bibitem{VdB}
M.~Van~den~Bergh,
\emph{Non-commutative Crepant Resolutions},
The legacy of Niels Henrik Abel,  749--770, Springer, Berlin, 2004.

\end{thebibliography}

\printindex

\end{document}